\documentclass[11pt]{article}
\usepackage[margin=1in]{geometry}

\usepackage{amsmath}
\DeclareMathAlphabet{\mathcal}{OMS}{cmsy}{m}{n}
\usepackage{amssymb}
\usepackage{amsthm}
\usepackage{amsfonts}
\usepackage{mathtools}
\usepackage{mymacros}
\usepackage{braket}
\usepackage{bbm}
\usepackage{graphicx}
\usepackage{caption}
\usepackage{subcaption}
\usepackage{listings}
\usepackage[ruled,vlined,linesnumbered]{algorithm2e}
\usepackage{xcolor}
\usepackage{multirow}
\usepackage{lscape}
\usepackage{pgf,interval}
\usepackage{tabularx}
\usepackage{booktabs}
\usepackage[inline]{enumitem}
\usepackage{hyperref}
\usepackage{pgfplots}
\usepackage[title]{appendix}
\usepackage{mathrsfs}

\usepackage[numbers]{natbib}

\title{
    {\bf Supermodularity and Submodularity in\\Network Interdiction}  
    \author{
        \normalsize {\bf Bo Zhou}, {\bf Ruiwei Jiang}, and {\bf Siqian Shen} \\
        \small Department of Industrial and Operations Engineering,\\
        \small University of Michigan, Ann Arbor, MI 48109\\
        \small Email: \{bozum, ruiwei, siqian\}@umich.edu\\
    }
}
\date{}

\begin{document}

    \maketitle

    \begin{abstract}
        \noindent
        We study a bilevel network interdiction problem, with an attacker interdicting (attacking) certain components of a network and a defender optimizing operations over the ensuing network. We study when the defender's optimal objective is submodular or supermodular with respect to the attacker's interdiction decisions, for optimizing the bilevel integer program more efficiently. We first consider the min-cost flow (MinCF) interdiction problem and derive necessary and sufficient conditions for the supermodularity or submodularity to hold under three types of attacks, respectively on supplies/demands, flow capacities, and cost coefficients. We extend to other variants, including capacitated facility location, maximum flow (MaxF), and shortest path (SP) interdiction. The conditions hold under general network topologies and parameter settings, and depend solely on the locations of the attacks. We further incorporate additional network information (e.g., detailed parameters and special topologies) to establish less restrictive conditions. We also derive necessary and sufficient conditions for supermodularity or submodularity in SP and MaxF interdiction in series-parallel networks. Furthermore, we explore more challenging interdiction problems where the defender may make additional binary decisions (e.g., repairing or reinforcing the network) and identify conditions that preserve submodularity or recover supermodularity. Via extensive numerical studies with diverse types of attacks, we demonstrate an order-of-magnitude computational speedup achieved by exploiting these properties and generating valid inequalities, for solving network interdiction at scale.

        ~

    \end{abstract}

    \section{Introduction}\label{sec:intro}
        In network interdiction problems~\citep{smith2008algorithms}, an attacker attempts to interdict (attack) the efficient operation of a network, while a defender seeks to operate over the ensuing network in spite of the attacks.
       The generic network interdiction problem can be modeled  as a bilevel program:
        \begin{equation}\label{eq:BP}
            \begin{aligned}
                \max_{\substack{x\in X}}~&f(x,y)\\
                \text{s.t.}~&g(x,y)\leq0,\\
                &y\in\arg\min_{y\in Y}\left\{c^{\top}y: Ax+By\leq h\right\},
            \end{aligned}
        \end{equation}
        where $x\in X\subseteq\mathbb{R}^{n_{x}}$ and $y\in Y\subseteq\mathbb{R}^{n_{y}}$ are the decision variables of the attacker and the defender with
        $n_{x}$ and $n_{y}$ denoting the dimensions of $x$ and $y$, respectively;
        $A\in\mathbb{R}^{m\times n_{x}}$, $B\in\mathbb{R}^{m\times n_{y}}$, $c\in\mathbb{R}^{n_{y}}$, and $h\in\mathbb{R}^{m}$ are the parameters of the defender's model with $m$ being the number of rows (constraints);
        $f(\cdot)$ denotes the attacker's objective and maximizing $f(\cdot)$ is generally adversarial to minimizing the defender's objective $c^{\top}y$ (although we allow \(f(x,y) \ne c^{\top}y\)); and
        $g(\cdot)$ denotes the attacker's constraints such as a budget on attacks. Note that both $f(\cdot)$ and $g(\cdot)$ may be linear or nonlinear, depending on the attacker's problem setup.

The above network interdiction is computationally challenging due to the lack of convexity and the bilevel nature. To solve~\eqref{eq:BP}, one can consider a value function representation, which relies on the optimal objective of the defender's problem as a function of the attacker's decisions. We can then reformulate the bilevel program as a monolithic formulation:
        \begin{subequations}\label{eq:VFR}
            \begin{align}
                \max_{\substack{x\in X,y\in Y}}~&f(x,y)\label{eq:VFR1}\\
                \text{s.t.}~&g(x,y)\leq0\\
                &Ax+By\leq h\label{eq:VFR3}\\
                &c^{\top}y\leq\phi(x),\label{eq:optimality}
            \end{align}
        \end{subequations}
        where
        \begin{equation}\label{eq:phi}
            \begin{aligned}
                \phi(x):=\min_{y\in Y}~&c^{\top}y\\
                \text{s.t.}~&By\leq h-Ax
            \end{aligned}
        \end{equation}
        is the optimal value function of the interdiction decisions $x$. Here, constraint~\eqref{eq:optimality} enforces the optimality condition of the defender's model.
        We refer to~\citet{israeli2002shortest,smith2020survey} for an overview of solution methods based on the reformulation~\eqref{eq:VFR}.
        For most network interdiction problems, typically $Y=\mathbb{R}^{n_{y}}$, and hence, when $x$ is fixed, the defender's optimization problem is a linear program (LP).
        As a result, one can replace \(\phi(x)\) in~\eqref{eq:optimality} with either the KKT conditions or the dual of the defender's model, commonly referred to as ``dualize-and-combine''.
        Yet this approach induces extra bilinear terms and thus requires further linearization~\citep[see, e.g.,][]{zhoulearning}.
        More importantly, it fails when the defender's problem involves integer variables.
        Another approach is to replace~\eqref{eq:optimality} with a series of valid inequalities (such as no-good cuts, Benders cuts, intersection cuts, Lagrangian cuts, submodular interdiction cuts~\citep[see, e.g.,][]{tahernejad2020branch, zhou2025bilevel, taninmics2022branch}), which avoids extra bilinear terms and does not require $Y$ to be continuous.
        In this case, the quality of the valid inequalities becomes critical.

        In this paper, we will investigate supermodularity and submodularity properties of the defender’s optimal objective with respect to the attacker’s interdiction decisions, and utilize them to derive valid inequalities for solving network interdiction problems at scale.

        Recall the definitions of supermodularity and submodularity:
        \begin{definition} \label{def:property}
            Consider a sublattice $X\subseteq\mathbb{R}^{n_x}$ and a function $\phi: X\rightarrow \mathbb{R}$.\\
            (i) $\phi(x)$ is supermodular if $\phi(x') + \phi(x'') \leq \phi(x' \vee x'') + \phi(x' \wedge x'')$ for all $x', x'' \in X$.\\
            (ii) $\phi(x)$ is submodular if $\phi(x') + \phi(x'') \geq \phi(x' \vee x'') + \phi(x' \wedge x'')$ for all $x', x'' \in X$.
        \end{definition}
        \noindent Obviously, $\phi(x)$ is supermodular if and only if $-\phi(x)$ is submodular. We also state the following two assumptions throughout this paper.
        \begin{assumption}[Binary interdiction]\label{asp:binary}
            The attack decision is binary-valued, i.e., $X=\{0,1\}^{n_{x}}$.
        \end{assumption}
        \begin{assumption}[Complete recourse]\label{asp:recourse}
            For any possible attack, the defender's problem is feasible and bounded.
            In other words, for any $x\in X$, $\phi(x)$ defined in \eqref{eq:phi} is well-defined and finite.
        \end{assumption}
        \noindent Assumption~\ref{asp:binary} specifies binary interdiction by the attacker and Assumption~\ref{asp:recourse} requires the complete recourse of the defender's model.
        Both assumptions are mild and appear in most real-world network flow applications.
        
 \subsection{Review of Most Relevant Literature}
    \label{sec:lit}
        There has been extensive studies on submodularity (or supermodularity) in integer programming for decades~\citep{nemhauser1978analysis,schrijver2003combinatorial,atamturk2020submodularity}, referring to the diminishing (or magnifying) marginal increments of a function.
        Such structural properties lead to greedy schemes with provable accuracy guarantees or valid inequalities with efficient separation, enabling  efficient algorithms for optimizing or approximating solutions to integer programs~\citep[see, e.g.,][]{atamturk2020submodularity,atamturk2009submodular,zhang2018ambiguous} and bilevel programs~\citep[see, e.g.,][]{zhoulearning,zhou2025bilevel,ljubic2018outer,qi2024sequential}.
        That is, when $\phi(x)$ is either supermodular or submodular (see Definition \ref{def:property}), one can efficiently generate powerful valid inequalities and reformulations for~\eqref{eq:optimality}~\citep[see, e.g.,][]{nemhauser1978analysis, schrijver2003combinatorial, edmonds2003submodular, nemhauser1981maximizing}.
        Specifically, we define a mapping $\phi:2^{[n_{x}]}\rightarrow\mathbb{R}$ as
        \begin{equation}
            \phi(\mathcal{S}):=\phi(s),
        \end{equation}
        where $s\in\{0,1\}^{n_{x}}$ and for all $i\in[n_{x}]$, $s_{i}=1$ if $i\in\mathcal{S}$ and $s_{i}=0$ if $i\notin\mathcal{S}$.
        Proposition \ref{pps:cut} summarizes valid inequalities from the literature.
        \begin{proposition}\label{pps:cut}
            \citep{nemhauser1981maximizing, schrijver2003combinatorial}
            Consider the defender's optimality condition \eqref{eq:optimality}.\\
            (i) If $\phi(x)$ is \emph{supermodular}, for any $\hat{x}\in\{0,1\}^{n_{x}}$, we have a valid inequality for \eqref{eq:optimality} as
            \begin{equation}\label{eq:submodular-cut}
                c^{\top}y\leq \phi(\mathcal{S}_{0})+\sum_{k=1}^{n_{x}}[\phi(\mathcal{S}_{k})-\phi(\mathcal{S}_{k-1})]x_{\sigma_{k}},
            \end{equation}
            where $\sigma$ is a permutation of $[n_{x}]$ such that $\hat{x}_{\sigma_{1}}\geq \hat{x}_{\sigma_{2}}\geq\cdots\geq \hat{x}_{\sigma_{n_{x}}}$; $\mathcal{S}_{0}:=\emptyset$, and for all $k\in[n_x]$, $\mathcal{S}_{k}:=\{\sigma_{1},\ldots,\sigma_{k}\}$ defines the first $k$ entries of $\sigma$.\\
            (ii) If $\phi(x)$ is \emph{submodular}, for any $\hat{x}\in\{0,1\}^{n_{x}}$, we have a valid inequality for \eqref{eq:optimality} as
            \begin{equation}\label{eq:supermodular-cut}
                c^{\top}y\leq \phi(\mathcal{S}_{\hat{x}})-\sum_{i\in \mathcal{S}_{\hat{x}}}\delta([n_{x}]\backslash\{i\},\{i\})(1-x_{i})+\sum_{i\in[n_{x}]\backslash \mathcal{S}_{\hat{x}}}\delta(\mathcal{S}_{\hat{x}},\{i\})x_{i}.
            \end{equation}
            where $\mathcal{S}_{\hat{x}}:=\{i\in[n_{x}]:\hat{x}_{i}=1\}$;
            for all $\mathcal{S}\subseteq[n_{x}]$ and $i\in[n_{x}]\backslash\mathcal{S}$, $\delta(\mathcal{S},\{i\}):=\phi(\mathcal{S}\cup\{i\})-\phi(\mathcal{S})$.
        \end{proposition}
        \noindent When $\phi(x)$ is submodular or supermodular, the right-hand side of \eqref{eq:submodular-cut} or \eqref{eq:supermodular-cut} constructs a facet-defining inequality for the epigraph of $\phi(x)$, respectively.
        The valid inequality-based algorithm in turn speeds up the solution of the network interdiction model~\eqref{eq:VFR}.
        Therefore, in what follows, we focus on identifying the submodularity and supermodularity of $\phi(x)$.

        Our work is also related to the literature on parametric optimization.
        In the context of parametric network flow optimization, \citet{gale1981substitutes} identified pairwise submodularity (or supermodularity) based on the concepts of substitutes (or complements) between two distinct arcs.
        Later, \citet{granot1985substitutes} generalized the results in~\citet{gale1981substitutes} and identified joint submodularity (or supermodularity) by extending the above concepts to a set of arcs.
        However, the results of~\citet{granot1985substitutes} only considered flow balance constraints and were limited to the so-called ``biconnected'' networks.
        For a general parametric optimization model,~\citet{topkis1998supermodularity} provided conditions for the optimal value to preserve the supermodularity of the objective function.
        Building upon this,~\citet{lu2015reliable} established the supermodularity of uncapacitated facility location models.
        Recently,~\citet{LongQi-5297} and~\citet{ChenLong-5298} successfully established conditions for supermodularity preservation of parametric minimization and parametric maximization models, respectively (see Section~\ref{sec:pre} for details).
        However, owing to their generality, these conditions are difficult to verify and are not directly amenable to network optimization problems. These limitations motivate the present work, in which we undertake a systematic investigation of submodularity and supermodularity conditions in network interdiction. To the best of our knowledge, a comprehensive study of this topic remains absent from the literature.
        
 \subsection{Preliminary}
    \label{sec:pre}
        The following two theorems provide necessary and sufficient conditions for the supermodularity and submodularity of a parametric LP as defined in \eqref{eq:phi}.
        \begin{theorem}[Adapted from Theorem 2 of \citet{LongQi-5297}]\label{trm:condition-generalLP-sub}
            The value function $\phi(x)$ defined in \eqref{eq:phi} is \emph{supermodular} for all $c$ and $h$ if and only if $A$ and $B$ satisfy one of the following two conditions:\\
            (i) $\text{rank}(B)=m$;\\
            (ii) For all $\mathcal{I}\subseteq[m]$ with $|\mathcal{I}|=\text{rank}(B)+1$  such that $\text{rank}(B_{\mathcal{I}})=\text{rank}(B)$, for all $x\in\mathbb{R}_{+}^{n_{x}}$ such that $A_{\mathcal{I}}x\in\text{span}(B_{\mathcal{I}})$, we must have that $A_{\mathcal{I},k}x_{k}\in\text{span}(B_{\mathcal{I}})$ holds for any $k\in[n_{x}]$, where $A_{\mathcal{I},k}$ is the $k$th column of $A_{\mathcal{I}}$.
        \end{theorem}
        \begin{theorem}[Adapted from Theorems 10 and 11 of \citet{ChenLong-5298}]\label{trm:condition-generalLP-super}
            The value function $\phi(x)$ defined in \eqref{eq:phi} is \emph{submodular} for all $c$ and $h$ if and only if $A$ and $B$ satisfy one of the following two conditions:\\
            (i) $\text{rank}(B)=m$;\\
            (ii) For all $\mathcal{I}\subseteq[m]$ with $|\mathcal{I}|=\text{rank}(B)+1$  such that $\text{rank}(B_{\mathcal{I}})=\text{rank}(B)$, for all $x\in\mathbb{R}^{n_{x}}$ such that $A_{\mathcal{I}}x\in\text{span}(B_{\mathcal{I}})$, we must have that $A_{\mathcal{I}}x^{+}\in\text{span}(B_{\mathcal{I}})$ holds.
        \end{theorem}
        \noindent According to Theorems \ref{trm:condition-generalLP-sub} and \ref{trm:condition-generalLP-super}, the supermodularity or submodularity of $\phi(x)$ depends on the coefficient matrices $A$ and $B$.
        We note that the two theorems are dedicated to parametric minimization as in~\eqref{eq:phi} but can be readily adapted to parametric maximization.

        The difference between the two conditions derived in Theorems~\ref{trm:condition-generalLP-sub} and~\ref{trm:condition-generalLP-super} lies in their second requirements, demonstrated in
        Figure \ref{fig:comparison} and Example \ref{exp:comparison}.

              \begin{figure}[htbp!]
                \begin{center}
                    \includegraphics[height=0.35\columnwidth]{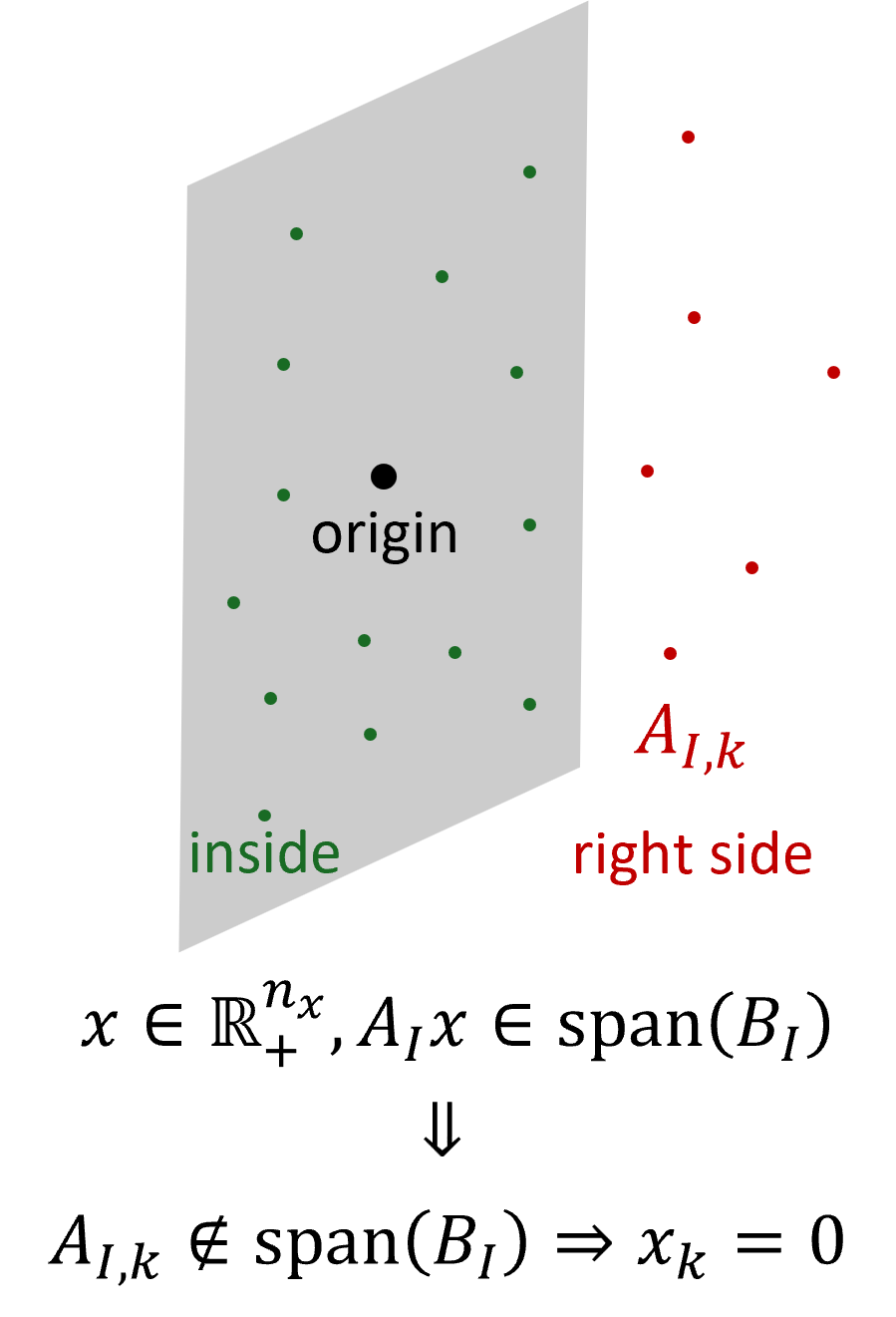}
                    \hspace{2em}
                    \includegraphics[height=0.35\columnwidth]{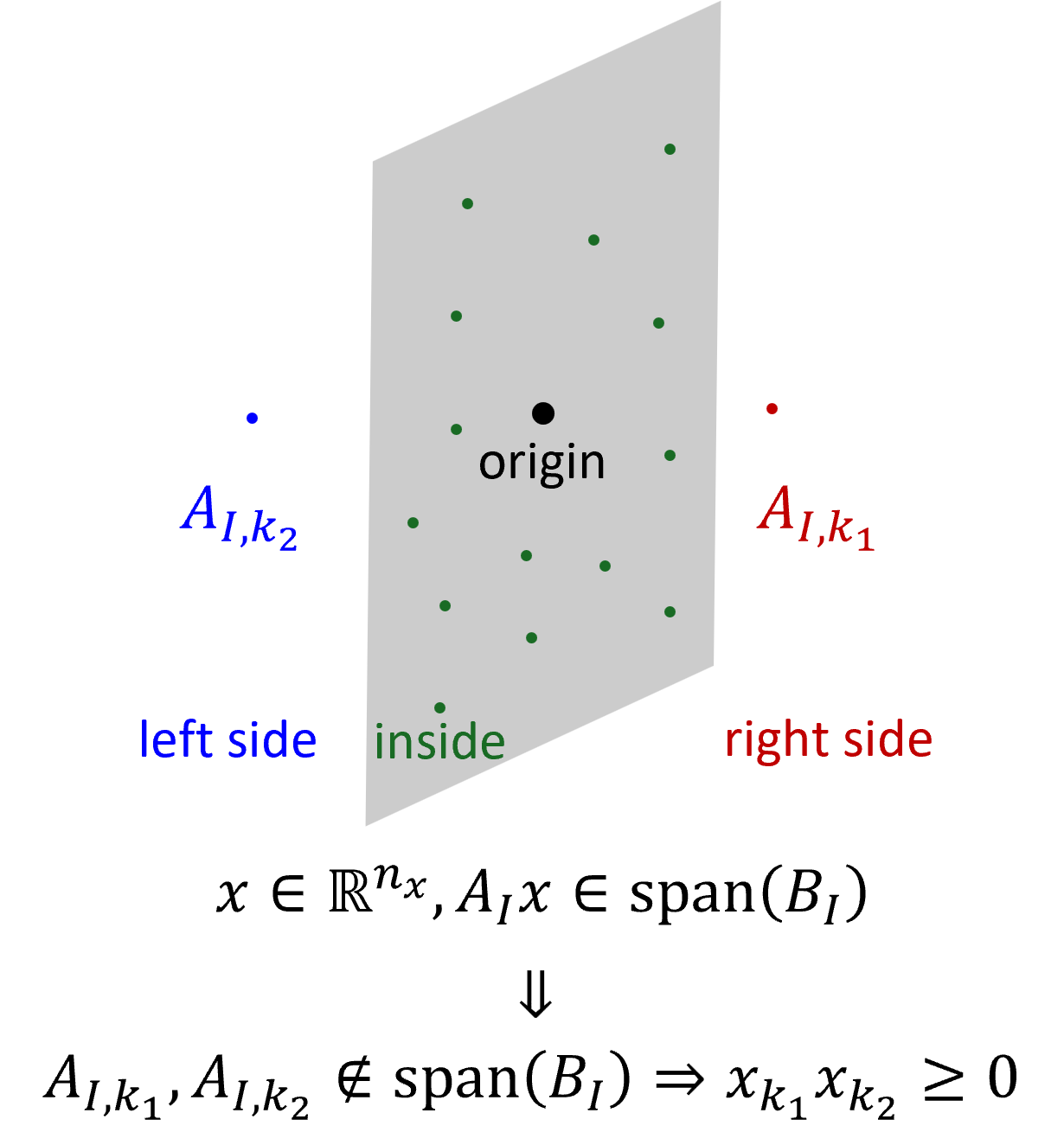}
                \end{center}
                \vspace{-2ex}
                \caption{Difference between Theorem \ref{trm:condition-generalLP-sub} (left) and Theorem \ref{trm:condition-generalLP-super} (right).}
                \label{fig:comparison}
            \end{figure}

        \begin{example}
        \label{exp:comparison}

            Given $A\in\mathbb{R}^{m\times n_{x}}$ and $B\in\mathbb{R}^{m\times n_{y}}$ with $\text{rank}(B)=r<m$, consider an index set $\mathcal{I}\subseteq[m]$ with $|\mathcal{I}|=r+1$  such that $\text{rank}(B_{\mathcal{I}})=r$.
            Then, $\text{span}(B_{\mathcal{I}})$ is a hyperplane in the $(r+1)$-dimensional space and $A_{\mathcal{I},k}$ is a point in the space, as shown in Figure \ref{fig:comparison}.\\
            (1) The second requirement of Theorem \ref{trm:condition-generalLP-sub} designates that ``for all $x\in\mathbb{R}_{+}^{n_{x}}$ such that $A_{\mathcal{I}}x\in\text{span}(B_{\mathcal{I}})$, we must have that $A_{\mathcal{I},k}x_{k}\in\text{span}(B_{\mathcal{I}})$ holds for any $k\in[n_{x}]$.'' Geometrically, this implies that all points $A_{\mathcal{I},k}$ lie either on the hyperplane \(\text{span}(B_{\mathcal{I}})\) (in case of $x_{k}>0$) or on the same side of the hyperplane (in case of $x_{k}=0$), as illustrated in the left panel of Figure~\ref{fig:comparison}.\\
            (2) The second requirement of Theorem \ref{trm:condition-generalLP-super} designates that ``for all $x\in\mathbb{R}^{n_{x}}$ such that $A_{\mathcal{I}}x\in\text{span}(B_{\mathcal{I}})$, we must have that $A_{\mathcal{I}}x^{+}\in\text{span}(B_{\mathcal{I}})$ holds.'' This implies that, on each side of the hyperplane \(\text{span}(B_{\mathcal{I}})\), there exists at most one point from $\{A_{\mathcal{I},k}: k\in[n_{x}]\}$ that does not lie on the hyperplane, as illustrated in the right panel of Figure \ref{fig:comparison}.

        \end{example}

        One approach to applying Theorems~\ref{trm:condition-generalLP-sub} and \ref{trm:condition-generalLP-super} is to describe their conditions using integer programs.
        However, we note that the number of possible index sets $\mathcal{I}$ in either Theorem \ref{trm:condition-generalLP-sub} or Theorem \ref{trm:condition-generalLP-super} grows \emph{exponentially} in \(m\), which renders it computationally prohibitive to verify supermodularity or submodularity through either theorem for practical applications.

\subsection{Contributions of the Paper}

The main contributions of this work are as follows.
\begin{enumerate}
    \item We begin with min-cost flow (MinCF) interdiction, one of the most widely studied network interdiction models, and consider three types of attacks: attacks on supplies/demands, cost coefficients, and flow capacities.  We derive necessary and sufficient conditions for submodularity and supermodularity of the defender's optimal objective $\phi$ with respect to the attacker's interdiction $x$. We further extend these results to several variants of MinCF interdiction, including facility location interdiction, maximum flow (MaxF) interdiction, and shortest path (SP) interdiction, and establish the corresponding necessary and sufficient conditions for each. The proposed conditions hold under general network topologies and parameter settings, depending solely on the locations of the attacks. Moreover, these conditions can be verified by solving a polynomial number of linear programs.
\item We incorporate additional network information, such as detailed parameter values and special network topologies, to derive less restrictive conditions for submodularity and supermodularity in network interdiction. On the parameter side, we introduce the concept of interdiction-robust redundant arcs, which reduces network complexity and relaxes the conditions. On the topology side, we investigate series-parallel networks (SPNs) as a practically relevant special case, and derive necessary and sufficient conditions for submodularity and supermodularity in SP interdiction and MaxF interdiction. The proposed conditions can be verified by identifying at most $N$ interdiction-robust redundant or non-binding arcs, where $N$ denotes the number of attackable arcs.

\item We extend the framework to more challenging settings involving a mixed-integer defender that may undertake additional binary decisions, such as repairing or reinforcing the network. We derive sufficient conditions under which the submodularity is preserved or the supermodularity is recovered for the value function of the original defender in the absence of such additional decisions. We further conduct extensive numerical experiments and demonstrate an order-of-magnitude speedup achieved by exploiting these structural properties to generate valid inequalities for solving interdiction problems at scale.
\end{enumerate}

\subsection{Structure of the Paper}

    The remainder of the paper is organized as follows.
    Section~\ref{sec:general} derives conditions for submodularity and supermodularity in network interdiction under general settings. Section~\ref{sec:specific} refines these results by incorporating additional network structure, while Section~\ref{sec:extension} extends the framework to the case of a mixed-integer defender. Section~\ref{sec:numerical} presents numerical experiments to validate the theoretical findings, and Section~\ref{sec:conclusions} concludes the paper and discusses directions for future research.

~\\
\noindent \emph{Notations:}
        For a positive integer $n$, we define $[n]:=\{1,2,\ldots,n\}$.
        For $x', x'' \in\mathbb{R}^{n}$, we define $x' \vee x'' := [\max\{ x'_1, x''_1\}, \ldots, \max\{x'_n, x''_n\}]^{\top}$ and $x' \wedge x'' := [\min\{x'_1, x''_1\}, \ldots, \min\{x'_n, x''_n\}]^{\top}$.
        We denote $x^{\vee}=x' \vee x''$ and $x^{\wedge}=x' \wedge x''$.
        For $x\in\mathbb{R}^n$, we define $x^{+}:=x\vee0$ and $x^{-}:=x\wedge0$.
        For a matrix $A\in\mathbb{R}^{m\times n}$ and an index set $\mathcal{I}\subseteq[m]$, $A_{\mathcal{I}}$ consists of the rows indexed by $\mathcal{I}$.
        For $k\in[n]$, $A_{\mathcal{I},k}$ is the $k$th column of $A_{\mathcal{I}}$.
        We use $I$ to denote an identity matrix of appropriate dimension.
        For $i\in\mathbb{Z}_{+}$, $e_{i}$ denotes the $i$th standard basis vector with appropriate dimension.

    \section{Super/Submodularity Conditions in Network Interdiction}\label{sec:general}

            Consider network interdiction on a directed graph $\mathcal{G}:=(\mathcal{V},\mathcal{A})$, where $\mathcal{V}$ denotes the set of nodes and $\mathcal{A}$ denotes the set of arcs, with $n_{v}:=|\mathcal{V}|$ and $n_{a}:=|\mathcal{A}|$. The special structure of matrices $A$ and $B$ in~\eqref{eq:phi} for network interdiction, allows us to derive simpler conditions for supermodularity or submodularity.
            We denote by $T\in\{-1,0,1\}^{n_{v}\times n_{a}}$ the node-arc incidence matrix of $\mathcal{G}$ and $t_{ij}$ is the entry of $T$, such that
            $t_{ij}=-1$ or $t_{ij}=1$ indicates that the $i$th node is the tail node or the head node of the $j$th arc, respectively.
            \begin{definition}\label{def:undirected}
                A \emph{path} $P$ in $\mathcal{G}$ is a sequence of nodes $(v_{0},v_{1},...,v_{k})$ with arcs $(a_{1}, ..., a_{k})$ and $k\geq1$ such that $v_{0},v_{1},...,v_{k-1}\in\mathcal{V}$ are distinct nodes, $v_k\in\mathcal{V}$ is distinct from $v_0$ if $k\leq2$ and is distinct from $v_{1},...,v_{k-1}$ if $k\geq3$, and $a_{i}=(v_{i-1},v_{i})\in\mathcal{A}$ or $a_{i}=(v_{i},v_{i-1})\in\mathcal{A}$ for all $i\in[k]$.\\
                In addition, $\mathcal{A}_P:=\{a_{i}:i\in[k]\}$ denotes the set of arcs in $P$.
            \end{definition}
            \begin{definition}\label{def:undirected}
                (i) Given a node set $\mathcal{V}_{s}\subseteq\mathcal{V}$, a \emph{$\mathcal{V}_{s}$-excluded path} in $\mathcal{G}$ is a special path $(v_{0},v_{1},...,v_{k})$ such that $v_{i}\notin\mathcal{V}_{s}$ for $i=1,2,...,k-1$.\\
                (ii) A \emph{cycle} in $\mathcal{G}$ is a special path $(v_{0},v_{1},...,v_{k})$ such that $v_{k}=v_{0}$.
            \end{definition}
            \begin{definition}\label{def:direction}
                Consider a path $P=(v_{0},v_{1},...,v_{k})$ in $\mathcal{G}$, an arc $a_{i}$ between $v_{i-1}$ and $v_{i}$ and an arc $a_{j}$ between $v_{j-1}$ and $v_{j}$ with $i,j\in[k]$ and $i\neq j$
                (note that $a_{i}$ may be $(v_{i-1},v_{i})$ or $(v_{i},v_{i-1})$, and $a_{j}$ may be $(v_{j-1},v_{j})$ or $(v_{j},v_{j-1})$).
                We say that arcs $a_{i}$ and $a_{j}$ are in $P$ and
                \\
                (i) $a_{i}$ and $a_{j}$ have the \emph{same direction} in $P$ if both $a_{i}$ and $a_{j}$ align with either the forward or reverse order of $P$, i.e., $a_{i}=(v_{i-1},v_{i}),a_{j}=(v_{j-1},v_{j})$ or $a_{i}=(v_{i},v_{i-1}),a_{j}=(v_{j},v_{j-1})$.\\
                (ii) $a_{i}$ and $a_{j}$ have \emph{opposite directions} in $P$ if the two arcs align with the forward and reverse order of $P$ respectively, i.e., $a_{i}=(v_{i-1},v_{i}),a_{j}=(v_{j},v_{j-1})$ or $a_{i}=(v_{i},v_{i-1}),a_{j}=(v_{j-1},v_{j})$.
            \end{definition}
            Then, considering that the conditions for submodularity and supermodularity vary across defender's models, we narrow our focus to MinCF interdiction, one of the most classic network interdiction problems.
            A general MinCF problem is formulated as
            \begin{equation}\label{eq:MCF}
                \begin{aligned}
                    \min_{y\in\mathbb{R}^{n_{a}}}~&c^{\top}y\\
                    \text{s.t.}~&Ty\geq d,\\
                    &0\leq y\leq f,
                \end{aligned}
            \end{equation}
            where $y\in\mathbb{R}^{n_{a}}$ denotes the arc flow;
            $c\in\mathbb{R}^{n_{a}}$ and $f\in\mathbb{R}_{+}^{n_{a}}$ represent the transportation cost coefficients and flow capacities of arcs, respectively;
            $d\in\mathbb{R}^{n_{v}}$ represents nodal supplies/demands
            ($d_{i}\geq0$ indicates a demand of $d_{i}$ arising from node \(i\) and $d_{i}<0$ indicates a supply of $-d_{i}$ arising from node \(i\)).
            We note that many network-related problems can reduce to the above MinCF formulation, such as assignment problems, transportation problems, shortest path problems, maximum flow problems, which are also discussed in this paper.

            In the context of MinCF interdiction, we use \eqref{eq:MCF} as the base model and consider three types of attacks, which interdict MinCF by modifying supplies/demands $d$, reducing flow capacities $f$, or inflating cost coefficients $c$.
            In the following three sections, we examine the three types of attacks, respectively.

        \subsection{Attacks on Supplies/Demands}
            We first consider attacks on supplies/demands and modify the base model~\eqref{eq:MCF} as
            \begin{equation}\label{eq:MCF-demand}
                \begin{aligned}
                    \phi(x)=\min_{y\in\mathbb{R}^{n_{a}}}~&c^{\top}y\\
                    \text{s.t.}~&Ty\geq d+\text{diag}(\delta)x\\
                    &0\leq y\leq f,
                \end{aligned}
            \end{equation}
            where $x\in\{0,1\}^{n_{v}}$ denotes the attacker's interdiction, with $x_{i}=1$ indicates that the supply or demand of the $i$th node is attacked.
            In addition, $\delta\in\mathbb{R}_{+}^{n_{v}}$ denotes the impacts of the attack, which either increase the demand or reduce the supply capacity.

            Then, Proposition \ref{pps:general-MCF-demand} shows that $\phi(x)$ defined in \eqref{eq:MCF-demand} is supermodular.
            Based on Theorem \ref{trm:condition-generalLP-sub}, we provide a detailed proof in Appendix \ref{apd:pps:general-MCF-demand}.
            \begin{proposition}\label{pps:general-MCF-demand}
                Given a network with incidence matrix $T$ and a vector $\delta\in\mathbb{R}_{+}^{n_{v}}$, the function $\phi(x)$ defined in~\eqref{eq:MCF-demand} is supermodular for all $c$, $d$, and $f$.
            \end{proposition}
            \noindent
            Proposition \ref{pps:general-MCF-demand} establishes the supermodularity of the minimum cost of MinCF with respect to attacks on supplies/demands, indicating that the marginal cost of each additional attack is non-decreasing.
            This supermodularity holds for arbitrary combinations of supplies/demands, cost coefficients, and flow capacities, thereby highlighting its generality and facilitating its extension to many relevant network interdiction problems, such as facility location interdiction.

            \subsubsection{Special Case: Capacitated Facility Location Interdiction}
                We apply Proposition~\ref{pps:general-MCF-demand} to capacitated facility location (CFL) interdiction, which is a special case of MinCF interdiction with attacks on supplies/demands.
                In CFL interdiction, the defender seeks to flow commodities from \(m\) facilities to satisfy the demands at \(n\) demand sites with the minimum transportation cost. Formally, the defender's model is formulated as
                \begin{equation}\label{eq:facility-supply}
                    \begin{aligned}
                        \phi(x,z)=\min_{y\in\mathbb{R}^{mn}}~&\sum_{i=1}^m \sum_{j=1}^n c_{ij}y_{ij}\\
                        \text{s.t.}~& \sum_{j=1}^n t_{ij} y_{ij} \leq s_i - \gamma_i z_i, \quad \forall i \in [m], \\
                        & \sum_{i=1}^m t_{ij} y_{ij} \geq d_j + \delta_j x_j, \quad \forall j \in [n], \\
                        &0 \leq y_{ij} \leq f_{ij}, \quad \forall i \in [m], j \in [n],
                    \end{aligned}
                \end{equation}
                where \(y_{ij}\) denotes the flow from facility \(i\) to demand site \(j\),
                \(t_{ij} \in \{0,1\}\) denotes the connectivity between facility \(i\) and demand site \(j\) such that they are connected if and only if \(t_{ij} = 1\), and
                \(f_{ij}\) denotes the capacity of flow \(y_{ij}\).
                Each facility \(i\) has an original capacity of \(s_i\) and an attack (with \(z_i=1\)) reduces the capacity by \(\gamma_i\in\mathbb{R}_{+}\), and
                each site \(j\) has an original demand of \(d_j\) and an attack (with \(x_j = 1\)) increases the demand by \(\delta_j\in\mathbb{R}_{+}\).

                Comparing \eqref{eq:MCF-demand} and \eqref{eq:facility-supply}, we find that \eqref{eq:facility-supply} can reduce to \eqref{eq:MCF-demand} by considering a bipartite network. Therefore, it follows from Proposition~\ref{pps:general-MCF-demand} that $\phi(x,z)$ is \emph{jointly} supermodular in \((x,z)\).
                \begin{proposition}\label{pps:general-facility-supply}
                    Given $t\in\{0,1\}^{mn}$, $\gamma\in\mathbb{R}^{m}_{+}$, and $\delta\in\mathbb{R}^{n}_{+}$, the function $\phi(x,z)$ defined by \eqref{eq:facility-supply} is supermodular for all $c$, $s$, $d$, and $f$.
                \end{proposition}
                \noindent
                Proposition \ref{pps:general-facility-supply} establishes the supermodularity of the minimum transportation cost with respect to attacks on facilities and demand sites, and such a property holds for arbitrary combinations of supplies/demands and cost coefficients.
                This extends the existing result for \emph{uncapacitated} facility location problems (see, e.g., Proposition~9 of \citet{LongQi-5297}) to more general CFL problems.

        \subsection{Attacks on Flow Capacities}\label{sec:general-MCF-capacity}
            Next, we consider attacks on flow capacities and modify the base model \eqref{eq:MCF} as
            \begin{equation}\label{eq:MCF-capacity}
                \begin{aligned}
                    \phi(x)=\min_{y\in\mathbb{R}^{n_{a}}}~&c^{\top}y\\
                    \text{s.t.}~&Ty\geq d\\
                    &0\leq y\leq f-\text{diag}(\delta)x,
                \end{aligned}
            \end{equation}
            where $\delta\in\mathbb{R}_{+}^{n_{a}}$ denotes the impacts of the attacks and $x\in\{0,1\}^{n_{a}}$ denotes the attacker's interdiction decision.
            $x_{j}=1$ indicates that the flow capacity of the $j$th arc is decreased.
            We let $\mathcal{A}_{a}\subseteq\mathcal{A}$ denote the set of attackable arcs, so $\delta\in\Delta(\mathcal{A}_{a})$, where $\Delta(\mathcal{A}_{a}):=\{\delta\in\mathbb{R}_{+}^{n_{a}}:\delta_{j}=0 \text{~if~} a_{j}\notin{\mathcal{A}_{a}},\forall j\in[n_{a}]\}$ and $a_{j}$ is the $j$th arc in $\mathcal{A}$.

            Proposition \ref{pps:general-MCF-capacity} identifies the conditions for the submodularity and supermodularity of $\phi(x)$ defined by \eqref{eq:MCF-capacity}.
            Based on Theorems \ref{trm:condition-generalLP-sub} and \ref{trm:condition-generalLP-super}, we provide a detailed proof in Appendix \ref{apd:pps:general-MCF-capacity}.
            \begin{proposition}\label{pps:general-MCF-capacity}
                Given a network $\mathcal{G}=(\mathcal{V},\mathcal{A})$ with incidence matrix $T$, a set $\mathcal{A}_{a}\subseteq\mathcal{A}$ of attackable arcs, and a vector $\delta\in\Delta(\mathcal{A}_{a})$, it holds that\\
                (i) the function $\phi(x)$ defined by~\eqref{eq:MCF-capacity} is \emph{supermodular} for any $c$, $d$, and $f$ if and only if
                for any path $P$ in $\mathcal{G}$, for any two arcs $a',a''\in\mathcal{A}_{a}\cap\mathcal{A}_{P}$, $a'$ and $a''$ have \emph{opposite} directions in $P$;\\
                (ii) the function $\phi(x)$ defined by~\eqref{eq:MCF-capacity} is \emph{submodular} for any $c$, $d$, and $f$ if and only if
                for any path $P$ in $\mathcal{G}$, for any two arcs $a',a''\in\mathcal{A}_{a}\cap\mathcal{A}_{P}$, $a'$ and $a''$ have the \emph{same} direction in $P$.
            \end{proposition}
            \noindent
            Proposition \ref{pps:general-MCF-capacity} implies that the minimum cost of MinCF with respect to attacks on flow capacities may be supermodular or submodular, depending on the topological relationship of attackable arcs.
            When any two attackable arcs in the same path have opposite directions, the incremental cost of each additional attack is non-decreasing.
            In contrast, when all attackable arcs in the same path have the same direction, the incremental cost of each additional attack is non-increasing.
            The following examples illustrate the topological conditions of Proposition~\ref{pps:general-MCF-capacity}.
            \begin{example}\label{exp:example-path}
                We consider a directed network $\mathcal{G}=(\mathcal{V},\mathcal{A})$ with the node set $\mathcal{V}=\{v_{1},v_{2},v_{3},v_{4},v_{5}\}$ and the arc set $\mathcal{A}=\{(v_{1},v_{2}),(v_{2},v_{3}),(v_{4},v_{3}),(v_{4},v_{5}),(v_{3},v_{5}),(v_{1},v_{5})\}$.
                Figure~\ref{fig:example-path} depicts the network and the attackable arcs are marked in gray.
                We discuss the topological conditions of Proposition~\ref{pps:general-MCF-capacity} in the following three cases.
                \begin{figure}[!htbp]
                    \begin{center}
                        \includegraphics[width=1\columnwidth]{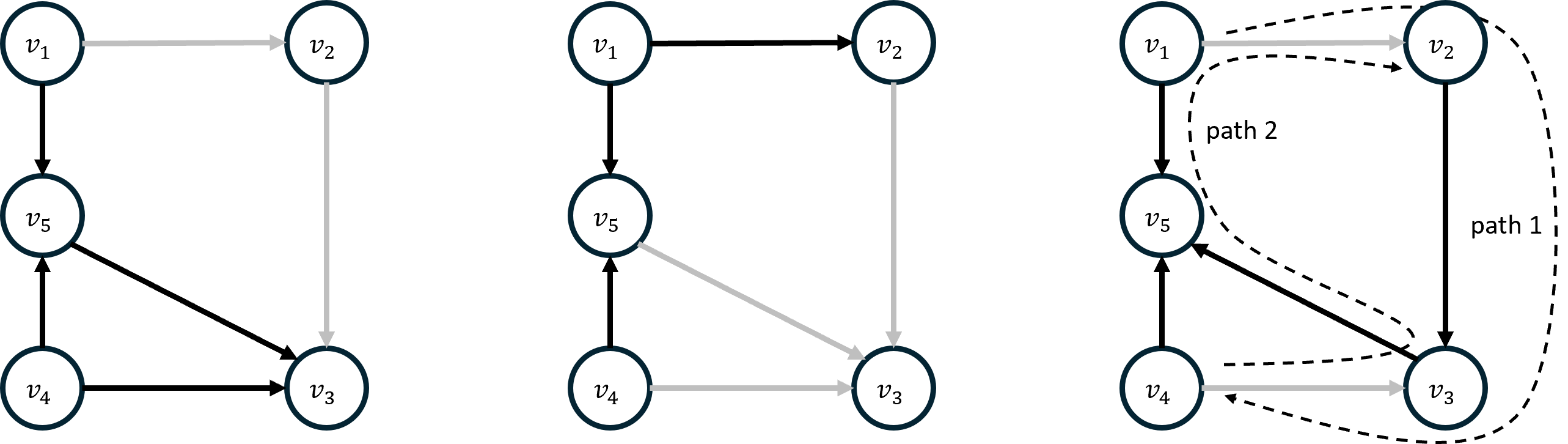}
                    \end{center}
                    \vspace{-2ex}
                    \caption{Topological relationship with respect to paths.}
                    \label{fig:example-path}
                \end{figure}

                (1) Suppose that we may attack the arcs in set $\mathcal{A}_{a}=\{(v_{1},v_{2}),(v_{2},v_{3})\}$, as highlighted in the left panel of Figure~\ref{fig:example-path}.
                Note that the arcs in $\mathcal{A}_{a}$ always have the same direction in all paths of \(\mathcal{G}\).
                As a result, if we attack the flow capacities of these arcs, the value function \(\phi(x)\) defined by~\eqref{eq:MCF-capacity} is submodular.

                (2) Suppose that we may attack the arcs in set $\mathcal{A}_{a}=\{(v_{2},v_{3}),(v_{4},v_{3}),(v_{5},v_{3})\}$, as highlighted in the middle panel of Figure \ref{fig:example-path}.
                Note that any two arcs in $\mathcal{A}_{a}$ have opposite directions, regardless of the path containing them we find in \(\mathcal{G}\).
                Consequently, if we attack the flow capacities of the arcs in $\mathcal{A}_{a}$, the value function \(\phi(x)\) defined by~\eqref{eq:MCF-cost} is supermodular.

                (3) Suppose that we may attack the arcs in set $\mathcal{A}_{a}=\{(v_{1},v_{2}),(v_{4},v_{3})\}$, as highlighted in the right panel of Figure~\ref{fig:example-path}.
                In path 1, the two arcs $(v_{1},v_{2})$ and $(v_{4},v_{3})$ have opposite directions, while in path 2, they have the same direction.
                Hence, the two arcs $(v_{1},v_{2})$ and $(v_{4},v_{3})$ do not have clear topological relationship and the value function \(\phi(x)\) with respect to \(\mathcal{A}_{a}\) is neither supermodular nor submodular.
            \end{example}

            The proposed conditions in Proposition ~\ref{pps:general-MCF-capacity} hold for arbitrary combinations of supplies/demands, cost coefficients, and flow capacities, thereby highlighting its generality and facilitating its extensions to other network interdiction problems such as maximum flow interdiction.

            \subsubsection{Special Case: Maximum Flow Interdiction}
                We extend Proposition \ref{pps:general-MCF-capacity} to maximum flow (MaxF) interdiction, which is a special case of MinCF interdiction with attacks on flow capacities.
                In MaxF interdiction, the defender seeks to maximize the total flow from sources to sinks.
                Note that we allow multiple sources and sinks and denote the number of sources/sinks by $n_{s}$.
                Formally, the defender's model is formulated as
                \begin{equation}\label{eq:MF-capacity}
                    \begin{aligned}
                        \phi(x)=\max_{y\in\mathbb{R}^{n_{a}},\eta\in\mathbb{R}^{n_{s}}}~&(1^{\top}\tilde{T})^{+}\eta\\
                        \text{s.t.}~&Ty=\tilde{T}\eta\\
                        &0\leq y\leq f-\text{diag}(\delta)x
                    \end{aligned}
                \end{equation}
                where $y\in\mathbb{R}^{n_{a}}$ means arc flow and $\eta\in\mathbb{R}^{n_{s}}$ indicates the sending/receiving flow at sources/sinks.
                $\tilde{T}\in\{-1,0,1\}^{n_{v}\times n_{s}}$ is the node-source/sink incidence matrix and its entry $\tilde{t}_{ij}=-1$ (or $\tilde{t}_{ij}=1$) indicates that the $j$th source (or sink) locates at node $i$.
                The objective $(1^{\top}\tilde{T})^{+}\eta$ means the total receiving flow of all sinks.
                For each arc $j$, an attack ($x_{j}=1$) reduces the capacity by $\delta_{j}$.
                As in Section \ref{sec:general-MCF-capacity}, $\mathcal{A}_{a}\subseteq\mathcal{A}$ denotes the subset of attackable arcs and $\delta\in\Delta(\mathcal{A}_{a})$.

                Comparing \eqref{eq:MCF-capacity} and \eqref{eq:MF-capacity}, we find that \eqref{eq:MF-capacity} can be rewritten in a similar formula to \eqref{eq:MCF-capacity} by considering variable $[y^\top,\eta^\top]^\top$ and incidence matrix $[T,-\tilde{T}]$.
                Proposition \ref{pps:general-MF-capacity} identifies the conditions for the supermodularity or submodularity of $\phi(x)$ defined by \eqref{eq:MF-capacity}.
                Based on Theorems \ref{trm:condition-generalLP-sub} and \ref{trm:condition-generalLP-super}, we provide a detailed proof in Appendix \ref{apd:pps:general-MF-capacity}.
                \begin{proposition}\label{pps:general-MF-capacity}
                    Given a network $\mathcal{G}=(\mathcal{V},\mathcal{A})$ with incidence matrix $T$, a set $\mathcal{V}_{s}\subseteq\mathcal{V}$ of source/sink nodes with incidence matrix $\tilde{T}$, a set $\mathcal{A}_{a}\subseteq\mathcal{A}$ of attackable arcs, and a vector $\delta\in\Delta(\mathcal{A}_{a})$, it holds that\\
                    (i) the function $\phi(x)$ defined by \eqref{eq:MF-capacity} is \emph{submodular} for any $f$ if and only if
                    for any $\mathcal{V}_{s}$-excluded path $P$ in $\mathcal{G}$, for any two arcs $a',a''\in\mathcal{A}_{a}\cap\mathcal{A}_{P}$, $a'$ and $a''$ have \emph{opposite} directions in $P$;\\
                    (ii) the function $\phi(x)$ defined by \eqref{eq:MF-capacity} is \emph{supermodular} for any $f$ if and only if
                    for any $\mathcal{V}_{s}$-excluded path $P$ in $\mathcal{G}$, for any two arcs $a',a''\in\mathcal{A}_{a}\cap\mathcal{A}_{P}$, $a'$ and $a''$ have the \emph{same} direction in $P$.
                \end{proposition}
                \noindent
                Proposition \ref{pps:general-MF-capacity} establishes the submodularity and supermodularity of the maximum flow with respect to attacks on arc capacities, and such a property holds for arbitrary flow capacities.
                There are some similarities between Proposition \ref{pps:general-MF-capacity} and Proposition \ref{pps:general-MCF-capacity}.
                First, Proposition \ref{pps:general-MF-capacity} switches the conditions (the same direction or opposite directions) for submodularity and supermodularity in Proposition \ref{pps:general-MCF-capacity}.
                The reason is that \eqref{eq:MF-capacity} is a parametric maximization while \eqref{eq:MCF-capacity} is a parametric minimization.
                Second, the description of conditions in Proposition \ref{pps:general-MF-capacity} and Proposition \ref{pps:general-MCF-capacity} are almost identical.
                The only distinction is that the topological relationship of Proposition \ref{pps:general-MF-capacity} considers $\mathcal{V}_{s}$-excluded paths.
                The following examples illustrate the topological relationship of Proposition \ref{pps:general-MF-capacity}.
                \begin{example}\label{exp:example-path_terminal}
                    We still consider the network topology in Example \ref{exp:example-path} and assume $\mathcal{V}_{s}=\{v_{1},v_{3}\}$ as the set of source/sink nodes, as shown in Figure \ref{fig:example-path_terminal}.
                    The attackable arcs and source/sink nodes are marked in gray.
                    We then discuss the topological conditions of Proposition \ref{pps:general-MF-capacity} in the following two cases.
                    \begin{figure}[!htbp]
                        \begin{center}
                            \vspace{-0ex}
                            \includegraphics[width=0.65\columnwidth]{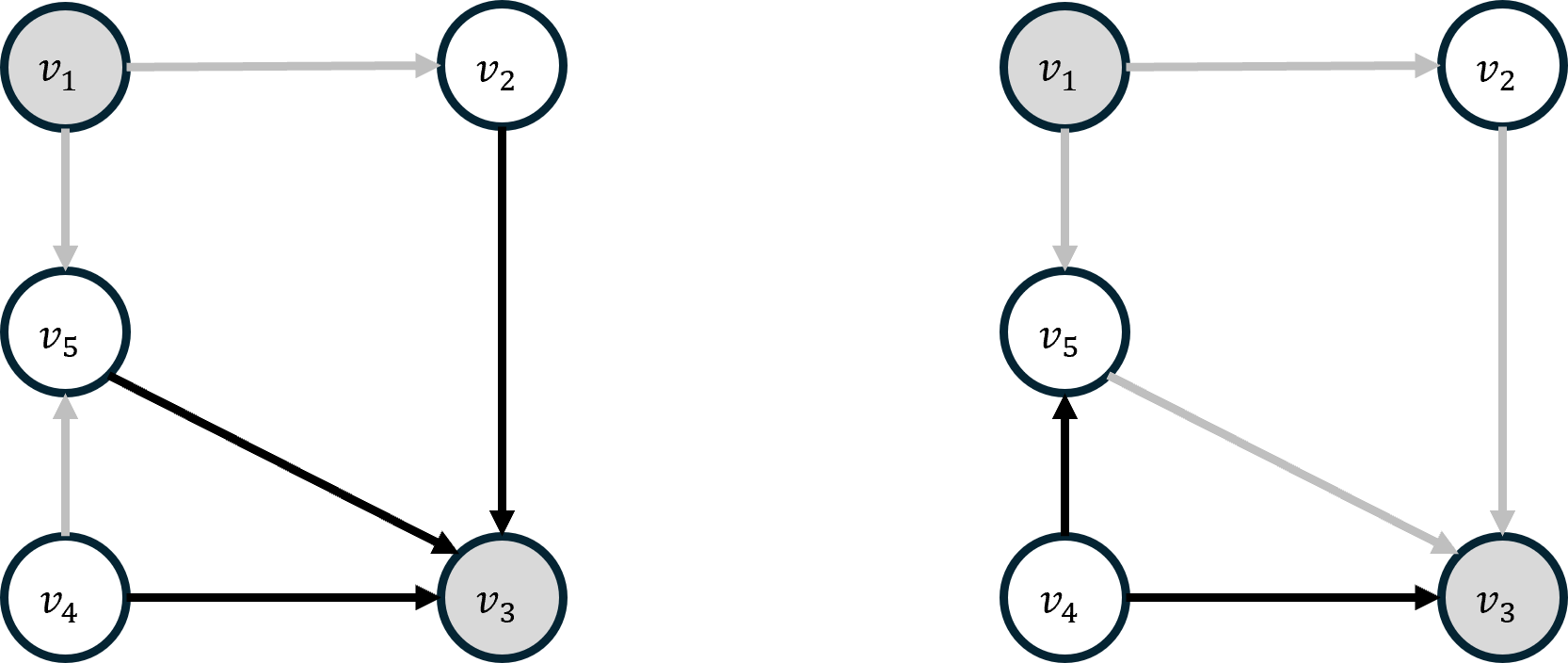}
                        \end{center}
                        \vspace{-2ex}
                        \caption{Topological relationship with respect to $\mathcal{V}_{s}$-excluded paths.}
                        \vspace{-2ex}
                        \label{fig:example-path_terminal}
                    \end{figure}

                    (1) Suppose that we may attack the arcs in set $\mathcal{A}_{a}=\{(v_{1},v_{2}),(v_{1},v_{5}),(v_{4},v_{5})\}$, as highlighted in the left panel of Figure \ref{fig:example-path_terminal}.
                    Note that although some arcs may have the same direction in a certain path, such as $(v_{1},v_{2})$ and $(v_{4},v_{5})$, these arcs are not in any $\mathcal{V}_{s}$-excluded path.
                    Hence, the arcs in set $\mathcal{A}_{a}$ always have opposite directions in all $\mathcal{V}_{s}$-excluded paths of $\mathcal{G}$.
                    As a result, if we attack the flow capacities of these arcs, the value function $\phi(x)$ defined by \eqref{eq:MF-capacity} is submodular.

                    (2) Suppose that we may attack the arcs in set $\mathcal{A}_{a}=\{(v_{1},v_{2}),(v_{2},v_{3}),(v_{1},v_{5}),(v_{5},v_{3})\}$, as highlighted in the right panel of Figure \ref{fig:example-path_terminal}.
                    Note that although some arcs may have opposite directions in a certain path, such as $(v_{1},v_{2})$ and $(v_{1},v_{5})$, these arcs are not in any $\mathcal{V}_{s}$-excluded path.
                    Hence, the arcs in set $\mathcal{A}_{a}$ always have the same direction in all $\mathcal{V}_{s}$-excluded paths of $\mathcal{G}$.
                    As a result, if we attack the flow capacities of these arcs, the value function $\phi(x)$ defined by \eqref{eq:MF-capacity} is supermodular.

                    In this example, because we only need to check $\mathcal{V}_{s}$-excluded paths, we can first break the network into two subnetworks for analysis, i.e., one subnetwork with $v_{1},v_{2},v_{3}$ and the other subnetwork with $v_{1},v_{4},v_{5},v_{3}$.
                    Any two attackable arcs in different subnetworks are independent for the conditions in Proposition \ref{pps:general-MF-capacity} and we do not need to check their topological relationship.
                \end{example}

                From the example, we observe that the conditions in Proposition \ref{pps:general-MF-capacity} are less restrictive than those in Proposition \ref{pps:general-MCF-capacity}.
                This observation follows directly from the definitions of paths and $\mathcal{V}_{s}$-excluded paths, where a path is not necessarily a $\mathcal{V}_{s}$-excluded path.
                This is reasonable because \eqref{eq:MF-capacity} can be seen as a special case of \eqref{eq:MCF-capacity} with source/sink information $\tilde{T}$.
                Proposition \ref{pps:general-MF-capacity} is actually for a network with fixed sources/sinks rather than a network with arbitrary sources/sinks, which leads to the less restrictive conditions.

        \subsection{Attacks on Cost Coefficients}\label{sec:general-MCF-cost}
            Finally, we consider attacks on cost coefficients and modify the base model~\eqref{eq:MCF} as
            \begin{equation}\label{eq:MCF-cost}
                \begin{aligned}
                    \phi(x)=\min_{y\in\mathbb{R}^{n_{a}}}~&(c+\text{diag}(\delta)x)^{\top}y\\
                    \text{s.t.}~&Ty\geq d\\
                    &0\leq y\leq f,
                \end{aligned}
            \end{equation}
            where $\delta\in\mathbb{R}_{+}^{n_{a}}$ indicates the impacts of the attacks and $x\in\{0,1\}^{n_{a}}$ is the attacker's interdiction decision.
            $x_{j}=1$ indicates that the cost coefficient of arc $j$ is increased.
            As in Section~\ref{sec:general-MCF-capacity}, $\mathcal{A}_{a}\subseteq\mathcal{A}$ denotes the set of attackable arcs and $\delta\in\Delta(\mathcal{A}_{a})$.

            Proposition \ref{pps:general-MCF-cost} identifies conditions for the $\phi(x)$ defined by~\eqref{eq:MCF-cost} being submodular or supermodular.
            Based on Theorems \ref{trm:condition-generalLP-sub} and \ref{trm:condition-generalLP-super}, we provide a detailed proof in Appendix \ref{apd:pps:general-MCF-cost}.
            \begin{proposition}\label{pps:general-MCF-cost}
                Given a network $\mathcal{G}=(\mathcal{V},\mathcal{A})$ with incidence matrix $T$, a set $\mathcal{A}_{a}\subseteq\mathcal{A}$ of attackable arcs, and a vector $\delta\in\Delta(\mathcal{A}_{a})$, it holds that\\
                (i) the function $\phi(x)$ defined by \eqref{eq:MCF-cost} is \emph{submodular} for any $c$, $d$, and $f$ if and only if
                for any path $P$ in $\mathcal{G}$, for any two arcs $a',a''\in\mathcal{A}_{a}\cap\mathcal{A}_{P}$, $a'$ and $a''$ have the \emph{same} direction in $P$;\\
                (ii) the function $\phi(x)$ defined by \eqref{eq:MCF-cost} is \emph{supermodular} for any $c$, $d$, and $f$ if and only if
                for any path $P$ in $\mathcal{G}$, for any two arcs $a',a''\in\mathcal{A}_{a}\cap\mathcal{A}_{P}$, $a'$ and $a''$ have \emph{opposite} directions in $P$.
            \end{proposition}
            \noindent
            Proposition \ref{pps:general-MCF-cost} implies that, depending on the topological relationship of attackable arcs, the minimum cost of MinCF with respect to attacks on cost coefficients may be supermodular or submodular.

            We note that the conditions in Proposition \ref{pps:general-MCF-cost} coincide with those in Proposition \ref{pps:general-MCF-capacity}.
            This is reasonable because both Propositions~\ref{pps:general-MCF-capacity} and~\ref{pps:general-MCF-cost} establish the conditions for attacks on arcs, either on their cost coefficients or flow capacities.
            We shows in Corollary \ref{crl:MCF-imply-capacity} that Proposition \ref{pps:general-MCF-cost} can imply Proposition~\ref{pps:general-MCF-capacity}, and the detailed proof is provided in Appendix \ref{apd:crl:MCF-imply-capacity}.
            \begin{corollary}\label{crl:MCF-imply-capacity}
                Proposition \ref{pps:general-MCF-cost} implies Proposition \ref{pps:general-MCF-capacity}.
            \end{corollary}
            \noindent
            Figure \ref{fig:example-imply-capacity} provides the intuition of Corollary \ref{crl:MCF-imply-capacity}.
            For an attackable arc $j$ with cost coefficient $c_j$ and flow capacity $f_j$, we assume that an attack decreases its flow capacity by $\delta_j$.
            We modify the flow capacity of original arc $j$ into $\delta_j$ and add a parallel dummy arc with cost coefficient $c_j$ and flow capacity $f_j-\delta_j$.
            We assume an attack on the modified arc $j$ that increases its cost coefficient by $M$, where $M$ is a sufficiently large constant.
            We observe that the attack on the flow capacity of original arc $j$ is equivalent to the attack on the cost coefficient of modified arc $j$.
            From Figure \ref{fig:example-imply-capacity}, we illustrate that attacks on flow capacities can be transformed into attacks on cost coefficients.
            Therefore, we can use the conditions in Proposition \ref{pps:general-MCF-cost} to identify the submodularity or supermodularity of $\phi(x)$ defined in \eqref{eq:MCF-capacity} and thus derive the conditions in Proposition \ref{pps:general-MCF-capacity}.
            \begin{figure}[!htbp]
                \begin{center}
                    \vspace{-0ex}
                    \includegraphics[width=0.5\columnwidth]{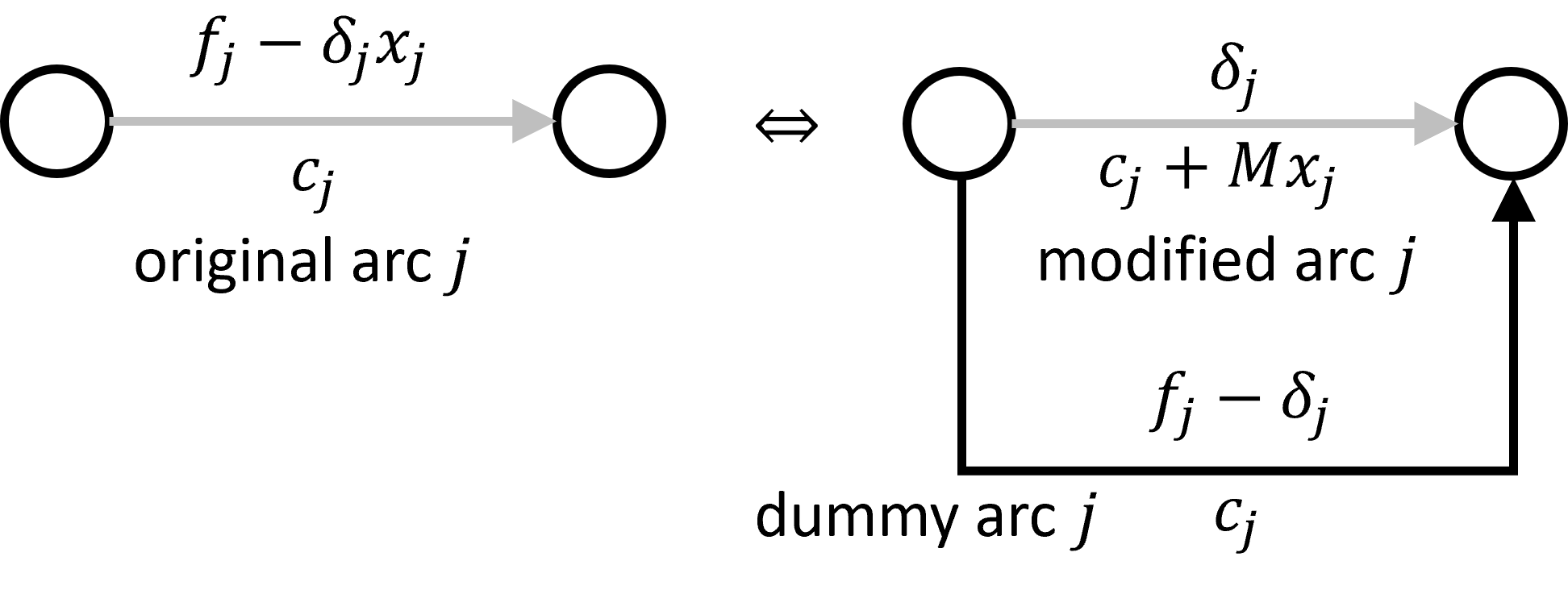}
                \end{center}
                \vspace{-2ex}
                \caption{Transformation from attacks on flow capacities to attacks on cost coefficients.}
                \vspace{-2ex}
                \label{fig:example-imply-capacity}
            \end{figure}

            Furthermore, we shows in Corollary \ref{crl:MCF-imply-demand} that Proposition \ref{pps:general-MCF-cost} can also imply Proposition~\ref{pps:general-MCF-demand}, and the detailed proof is provided in Appendix \ref{apd:crl:MCF-imply-demand}.
            \begin{corollary}\label{crl:MCF-imply-demand}
                Proposition \ref{pps:general-MCF-cost} implies Proposition~\ref{pps:general-MCF-demand}.
            \end{corollary}
            \noindent
            Figure \ref{fig:example-imply-demand} provides the intuition of Corollary \ref{crl:MCF-imply-demand}.
            As shown in the upper left part of Figure \ref{fig:example-imply-demand}, for an attackable demand node $v_i$ with demand $d_i$, we assume that an attack increases its demand by $\delta_i$.
            We modify the demand of node $v_i$ into $d_i+\delta_i$ and add a dummy supply node $v'_i$ with supply capacity $\delta_i$.
            There is a dummy arc $j$ from $v'_i$ to $v_i$ with zero cost coefficient and flow capacity $\delta_j$.
            We assume an attack on the dummy arc $j$ that increases its cost coefficient from zero to $M$, where $M$ is a sufficiently large constant.
            We observe that the attack on the demand of node $v_i$ is equivalent to the attack on the cost coefficient of dummy arc $j$.
            The case of an attackable supply node is shown in the lower left part of Figure \ref{fig:example-imply-demand} and we have a similar observation.
            Hence, we illustrate that attacks on supplies/demands can be transformed into attacks on the cost coefficients of dummy arcs.
            Therefore, we can use the conditions in Proposition \ref{pps:general-MCF-cost} to identify the submodularity or supermodularity of $\phi(x)$ defined in \eqref{eq:MCF-demand}.
            Moreover, we observe that any two dummy arcs have opposite directions in any path in the modified network, as shown in the right part of Figure \ref{fig:example-imply-demand}.
            According to Proposition \ref{pps:general-MCF-cost}, $\phi(x)$ defined in \eqref{eq:MCF-demand} is supermodular, which implies Proposition~\ref{pps:general-MCF-demand}.
            \begin{figure}[!htbp]
                \begin{center}
                    \vspace{-0ex}
                    \includegraphics[width=0.95\columnwidth]{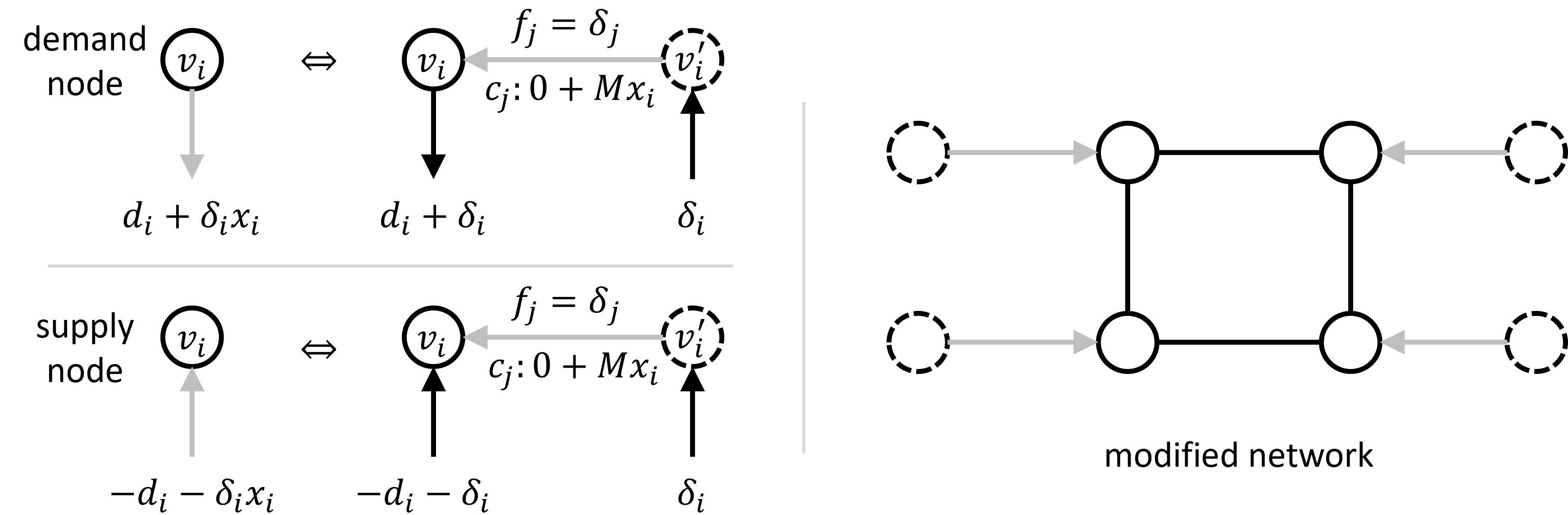}
                \end{center}
                \vspace{-2ex}
                \caption{Transformation from attacks on supplies/demands to attacks on cost coefficients.}
                \vspace{-2ex}
                \label{fig:example-imply-demand}
            \end{figure}

            Figure \ref{fig:example-imply-capacity} and Figure \ref{fig:example-imply-demand} tells that either attacks on flow capacities or attacks on supplies/demands can be transformed into attacks on cost coefficients.
            The observation is reasonable because both \eqref{eq:MCF-demand} and \eqref{eq:MCF-capacity} can reduce to \eqref{eq:MCF-cost}.
            Therefore, Proposition \ref{pps:general-MCF-cost} provides more general conditions to identify submodularity and supermodularity in MinCF interdiction.
            In addition, Proposition \ref{pps:general-MCF-cost} can be readily extended to other network interdiction problems such as shortest path interdiction.

            \subsubsection{Mixed Attacks on Flow Capacities and Cost Coefficients}
                The above observation inspires our investigation of mixed attacks on flow capacities and cost coefficients.
                We modify the base model~\eqref{eq:MCF} as
                \begin{equation}\label{eq:MCF-mixed}
                    \begin{aligned}
                        \phi(x^c,x^f)=\min_{y\in\mathbb{R}^{n_{a}}}~&(c+\text{diag}(\delta^c)x^c)^{\top}y\\
                        \text{s.t.}~&Ty\geq d\\
                        &0\leq y\leq f-\text{diag}(\delta^f)x^f,
                    \end{aligned}
                \end{equation}
                where $\delta^c,\delta^f\in\mathbb{R}_{+}^{n_{a}}$ indicate the impacts of the attacks on cost coefficients and flow capacities, respectively;
                $x^c,x^f\in\{0,1\}^{n_{a}}$ is the attacker's interdiction decision on cost coefficients and flow capacities, respectively.
                $x^c_{j}=1$ (or $x^f_j=1$) indicate that the cost coefficient (or flow capacity) of arc $j$ is increased (or decreased).
                As in Section~\ref{sec:general-MCF-cost}, $\mathcal{A}^c_{a},\mathcal{A}^f_{a}\subseteq\mathcal{A}$ denote the set of cost-attackable and capacity-attackable arcs, respectively, and $\delta^c\in\Delta(\mathcal{A}^c_{a})$, $\delta^f\in\Delta(\mathcal{A}^f_{a})$.

                Developed from Proposition \ref{pps:general-MCF-cost}, Proposition \ref{pps:general-MCF-mixed} identifies conditions for the $\phi(x)$ defined by~\eqref{eq:MCF-mixed} being submodular or supermodular.
                We provide a detailed proof in Appendix \ref{apd:pps:general-MCF-mixed}.
                \begin{proposition}\label{pps:general-MCF-mixed}
                    Given a network $\mathcal{G}=(\mathcal{V},\mathcal{A})$ with incidence matrix $T$, sets $\mathcal{A}^{c}_{a},\mathcal{A}^{f}_{a}\subseteq\mathcal{A}$ of attackable arcs with $\mathcal{A}^{c}_{a}\cap\mathcal{A}^{f}_{a}=\emptyset$, and vectors $\delta^c\in\Delta(\mathcal{A}^c_{a}),\delta^f\in\Delta(\mathcal{A}^f_{a})$, it holds that\\
                    (i) the function $\phi(x)$ defined by \eqref{eq:MCF-cost} is \emph{submodular} for any $c$, $d$, and $f$ if and only if
                    for any path $P$ in $\mathcal{G}$, for any two arcs $a',a''\in(\mathcal{A}^c_{a}\cup\mathcal{A}^f_{a})\cap\mathcal{A}_{P}$, $a'$ and $a''$ have the \emph{same} direction in $P$;\\
                    (ii) the function $\phi(x)$ defined by \eqref{eq:MCF-cost} is \emph{supermodular} for any $c$, $d$, and $f$ if and only if
                    for any path $P$ in $\mathcal{G}$, for any two arcs $a',a''\in(\mathcal{A}^c_{a}\cup\mathcal{A}^f_{a})\cap\mathcal{A}_{P}$, $a'$ and $a''$ have \emph{opposite} directions in $P$.
                \end{proposition}
                \noindent
                Proposition \ref{pps:general-MCF-mixed} provides the conditions for general attacks on arcs and implies that, depending on the topological relationship of attackable arcs, the minimum cost of MinCF with respect to mixed attacks on cost coefficients and flow capacities may be supermodular or submodular.

            \subsubsection{Special Case: Shortest Path Interdiction}
                We extend Proposition \ref{pps:general-MCF-cost} to shortest path (SP) interdiction, which is a special case of MinCF interdiction with attacks on cost coefficients.
                In SP interdiction, the defender seeks to find a path from the origin node $i_{o}$ to the destination node $i_{d}$ with the shortest distance (or time).
                Formally, the defender's model is formulated as
                \begin{equation}\label{eq:SP-distance}
                    \begin{aligned}
                        \phi(x)=\min_{y\in\mathbb{R}^{n_{a}}}~&(c+\text{diag}(\delta)x)^{\top}y\\
                        \text{s.t.}~&Ty=d\\
                        &y\geq0,
                    \end{aligned}
                \end{equation}
                where $c\in\mathbb{R}_{+}^{n_{a}}$ is the vector of arc distances (or time),
                $d\in\{-1,0,1\}^{n_{v}}$ indicates the origin and destination nodes with $d_{i_{o}}=-1$, $d_{i_{d}}=1$, and $d_{i}=0,\forall i\in[n_{v}]\backslash\{i_{o},i_{d}\}$.
                $y\in\mathbb{R}^{n_{a}}$ indicates the shortest path such that arc $j$ is in the shortest path if and only if $y_{j}=1$.
                Note that $y$ is originally binary-valued but can be relaxed as $y\geq0$ without loss of optimality in SP problems.
                For each arc $j$, an attack (with $x_{j}=1$) increases the distance (or time) by $\delta_{j}$.
                As in Section~\ref{sec:general-MCF-cost}, $\mathcal{A}_{a}\subseteq\mathcal{A}$ denotes the subset of attackable arcs and $\delta\in\Delta(\mathcal{A}_{a})$.

                Comparing \eqref{eq:MCF-cost} and \eqref{eq:SP-distance}, we find that \eqref{eq:MCF-cost} can extend to \eqref{eq:SP-distance} by imposing equality $Ty=d$ and relaxing $y\leq f$.
                Proposition \ref{pps:general-SP-distance} identifies the conditions for the submodularity and supermodularity of $\phi(x)$ defined by \eqref{eq:SP-distance}.
                Based on Theorems \ref{trm:condition-generalLP-sub} and \ref{trm:condition-generalLP-super}, we provide a detailed proof in Appendix \ref{apd:pps:general-SP-distance}.
                \begin{proposition}\label{pps:general-SP-distance}
                    Given a network $\mathcal{G}=(\mathcal{V},\mathcal{A})$ with incidence matrix $T$, a set $\mathcal{A}_{a}\subseteq\mathcal{A}$ of attackable arcs, and a vector $\delta\in\Delta(\mathcal{A}_{a})$, it holds that\\
                    (i) the function $\phi(x)$ defined by \eqref{eq:SP-distance} is \emph{submodular} for any $c$ and $d$ if and only if
                    for any cycle $P$ in $\mathcal{G}$, for any two arcs $a',a''\in\mathcal{A}_{a}\cap\mathcal{A}_{P}$, $a'$ and $a''$ have the \emph{same} direction in $P$;\\
                    (ii) the function $\phi(x)$ defined by \eqref{eq:SP-distance} is \emph{supermodular} for any $c$ and $d$ if and only if
                    for any cycle $P$ in $\mathcal{G}$, for any two arcs $a',a''\in\mathcal{A}_{a}\cap\mathcal{A}_{P}$, $a'$ and $a''$ have \emph{opposite} directions in $P$.
                \end{proposition}
                \noindent
                Proposition \ref{pps:general-SP-distance} establishes the submodularity and supermodularity of the shortest path distance with respect to attacks on arc distance, and such a property holds for arbitrary combinations of arc distances and origin/destination locations.
                The description of Proposition \ref{pps:general-SP-distance} and Proposition \ref{pps:general-MCF-cost} are almost identical.
                The only distinction is that the topological relationship of Proposition \ref{pps:general-SP-distance} considers cycles.
                The following examples illustrate the topological relationship of Proposition \ref{pps:general-SP-distance}.
                \begin{example}\label{exp:example-cycle}
                    We still consider the network topology in Example \ref{exp:example-path}.
                    The attackable arcs are marked in gray, as shown in Figure \ref{fig:example-cycle}.
                    We then discuss the topological conditions of Proposition \ref{pps:general-SP-distance} in the following two cases.

                    \begin{figure}[!htbp]
                        \begin{center}
                            \includegraphics[width=0.65\columnwidth]{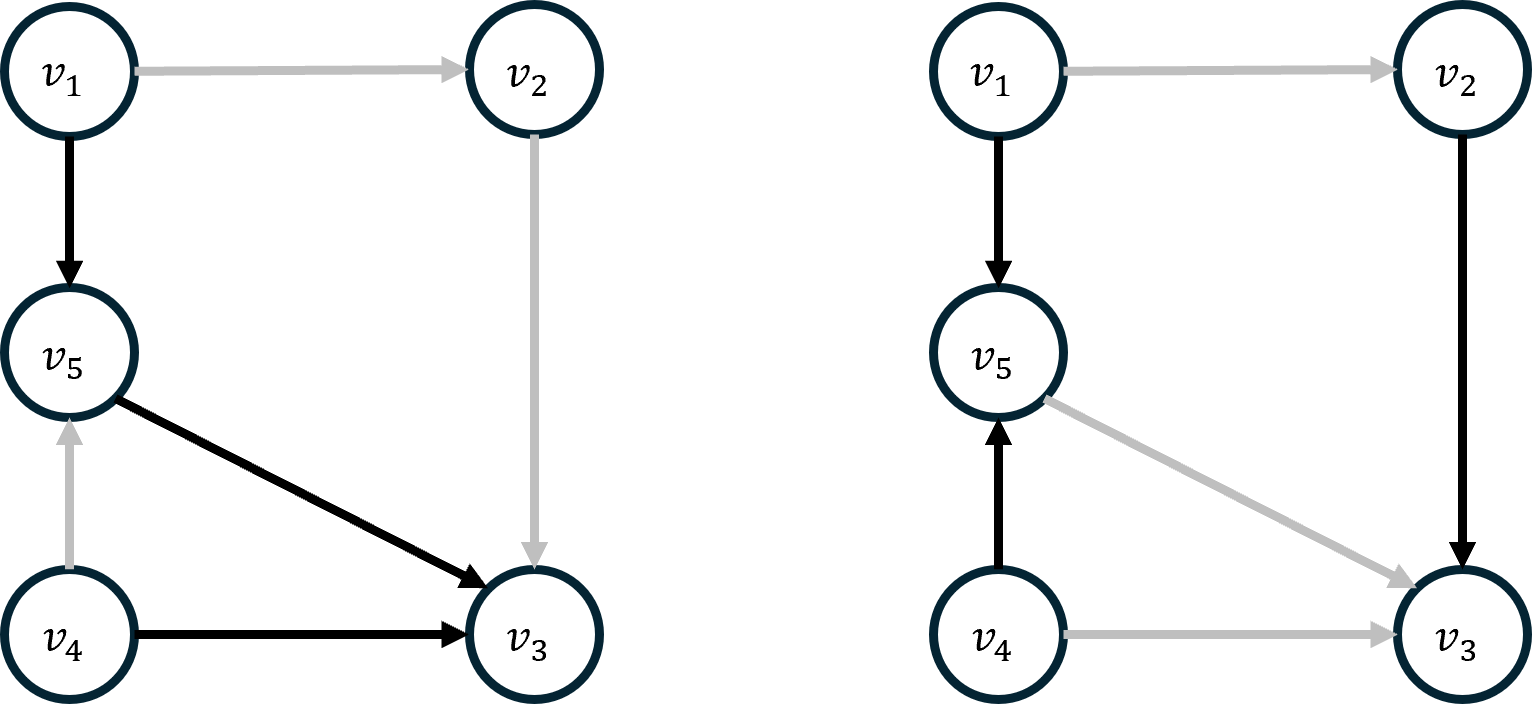}
                        \end{center}
                        \vspace{-2ex}
                        \caption{Topological relationship with respect to cycles.}
                        \label{fig:example-cycle}
                    \end{figure}

                    (1) Suppose that we may attack the arcs in set $\mathcal{A}_{a}=\{(v_{1},v_{2}),(v_{2},v_{3}),(v_{4},v_{5})\}$, as highlighted in the left panel of Figure \ref{fig:example-cycle}.
                    Note that the arcs in $\mathcal{A}_{a}$ always have the same direction in all cycles of $\mathcal{G}$.
                    As a result, if we attack the distances of the arcs in $\mathcal{A}_{a}$, the value function $\phi(x)$ defined by \eqref{eq:SP-distance} is submodular.

                    (2) Suppose that we may attack the arcs in set $\mathcal{A}_{a}=\{(v_{1},v_{2}),(v_{4},v_{3}),(v_{5},v_{3})\}$, as highlighted in the right panel of Figure \ref{fig:example-cycle}.
                    We recall that arcs $(v_{1},v_{2})$ and $(v_{4},v_{3})$ do not have clear topological relationship with respect to paths as considered in Proposition \ref{pps:general-MCF-cost} (see Example \ref{exp:example-path}).
                    However, there exist no cycle in $\mathcal{G}$ that contains path 2 as depicted in Figure \ref{fig:example-path}.
                    Hence, arcs $(v_{1},v_{2})$ and $(v_{4},v_{3})$ have clear topological relationship (opposite directions) with respect to cycles as considered in Proposition \ref{pps:general-SP-distance}.
                    Furthermore, any two arcs in $\mathcal{A}_{a}$ always have opposite directions regardless of the cycle containing them we find in $\mathcal{G}$.
                    Consequently, if we attack the distances of the arcs in $\mathcal{A}_{a}$, the value function $\phi(x)$ defined by \eqref{eq:SP-distance} is supermodular.
                \end{example}

                From Example \ref{exp:example-cycle}, we observe that the condition in Proposition \ref{pps:general-SP-distance} is less restrictive than those in Proposition \ref{pps:general-MCF-cost}.
                The observation follows directly from the definitions of paths and cycles, where a cycle must be a path while a path is not necessarily a cycle.
                This is reasonable because \eqref{eq:SP-distance} is a special case of \eqref{eq:MCF-cost} with the imposition of equality $Ty=d$.
                Under Assumption \ref{asp:recourse}, this imposition implicitly incorporates the restriction $1^{\top}d=0$ for feasibility.
                Proposition \ref{pps:general-SP-distance} is actually for any $d$ with $1^{\top}d=0$, which accounts for the less restrictive conditions.

        \subsection{Check Conditions}
            In this subsection, we explore how to check the conditions in Propositions \ref{pps:general-MCF-capacity}--\ref{pps:general-MCF-mixed}.
            We first use the conditions with respect to paths for example and then extend to $\mathcal{V}_{s}$-excluded paths and cycles.

            For any path $P$, for any two arcs $a',a''\in\mathcal{A}_{a}\cap\mathcal{A}_{P}$, $a'$ and $a''$ have opposite directions (or the same derection) in $P$, if and only if
            for any two arcs $a',a''\in\mathcal{A}_{a}$, there exists no path $P$ in $\mathcal{G}$ such that $a'$ and $a''$ have the same direction (or opposite directions) in $P$.
            Therefore, we need to check the topological relationship of each pair of $a',a''\in\mathcal{A}_{a}$.
            We note that in trivial cases where $a'$ and $a''$ are connected directly or through an arc, we can easily identify their topological relationship.
            Hence, we only consider non-trivial cases in the following.

            We denote by $\mathcal{V}_{\text{inc}}$ the set of incident nodes to $a'$ and $a''$.
            Then, we can check the topological relationship between $a'$ and $a''$ through Lemma \ref{lem:check-path}.
            The proof follows directly from the network structure and is thus omitted.
            \begin{lemma}\label{lem:check-path}
                Given two arcs $a',a''\in\mathcal{A}_{a}$.\\
                (i) There exists a path such that $a'$ and $a''$ have the same direction in the path, if and only if
                there exists a connected subnetwork $g$ such that $g$ contains the head node of one arc and the tail node of the other arc but does not contain the other two incident nodes in $\mathcal{V}_{\text{inc}}$;\\
                (ii) There exists a path such that $a'$ and $a''$ have opposite directions in the path, if and only if
                there exists a connected subnetwork $g$ such that $g$ contains the two head nodes or the two tail nodes but does not contain the other two incident nodes in $\mathcal{V}_{\text{inc}}$.
            \end{lemma}

            We denote by $\mathcal{J}_{\text{inc}}(\mathcal{V})$ the index set of arcs incident to the nodes in $\mathcal{V}$.
            Then, by Lemma \ref{lem:check-path2}, we can verify whether there exists a path $P$ as described in Lemma \ref{lem:check-path}.
            \begin{lemma}\label{lem:check-path2}
                There exists a connected subnetwork $g$ such that $g$ contains two incident nodes $v_{i_1},v_{i_2}\in\mathcal{V}_{\text{inc}}$ but does not contain the other two incident nodes in $\mathcal{V}_{\text{inc}}$, if and only if
                the linear system
                \begin{equation}\label{eq:check-path}
                    \left\{\begin{aligned}
                        &Ty=d\\
                        &y_{j}=0,\forall j\in\mathcal{J}_{\text{inc}}(\mathcal{V}_{\text{inc}}\backslash\{v_{i_1},v_{i_2}\})
                    \end{aligned}\right.
                \end{equation}
                is feasible, where $y\in\mathbb{R}^{n_a}$ and $d\in\{-1,0,1\}^{n_v}$ indicates the two incident nodes $v_{i_1},v_{i_2}$.
                Without loss of generality, we set $d_{i_1}=-1$, $d_{i_2}=1$, and $d_{i}=0$ for all $i\in[n_v]\backslash\{i_1,i_2\}$.
            \end{lemma}
            \noindent
            According to Lemma \ref{lem:check-path2}, if \eqref{eq:check-path} is feasible, we can find a path $P$ as described in Lemma \ref{lem:check-path}, and hence, $a'$ and $a''$ may have either the same direction or opposite directions in a certain path.

            Based on Lemma \ref{lem:check-path} and Lemma \ref{lem:check-path2}, we develop Algorithm \ref{alg:check-MCF} to check the proposed conditions with respect to paths, as stated in Proposition \ref{pps:check-MCF}.
            The proof is provided in Appendix \ref{apd:pps:check-MCF}.
            \begin{proposition}\label{pps:check-MCF}
                For any path $P$ in $\mathcal{G}$, for any two arcs $a',a''\in\mathcal{A}_{a}\cap\mathcal{A}_{P}$, $a'$ and $a''$ have the same direction (or opposite directions) in $P$, if Algorithm \ref{alg:check-MCF} outputs 1 (or $-1$).
            \end{proposition}
            \begin{algorithm}[!htbp]
                \caption{Algorithm to check the topological relationship between arcs in $\mathcal{A}_{a}$}
                \label{alg:check-MCF}
                \KwIn{Network $\mathcal{G}$, set $\mathcal{A}_{a}$ of attackable arcs, and set $\mathcal{V}_{s}$ of soruce/sink nodes if applicable.
                    Denote by $N=|\mathcal{A}_{a}|(|\mathcal{A}_{a}|-1)/2$ the number of pairs of arcs in $\mathcal{A}_{a}$
                }
                Initialize indicator $\chi=0$ and indicator vectors $\eta=\tilde{\eta}=\eta^+=\eta^-=0_{N}$\;
                \For{$k\in[N]$}{
                    Denote the $k$th pair of arcs in $\mathcal{A}_{a}$ by $a'=(v'_t,v'_h)$ and $a''=(v''_t,v''_h)$\;
                    \If{\eqref{eq:check-path} is feasible when either $v_{i_1}=v'_h,v_{i_2}=v''_t$ or $v_{i_1}=v''_h,v_{i_2}=v'_t$}{
                        update $\eta^+_k=1$\;
                    }
                    \If{\eqref{eq:check-path} is feasible when either $v_{i_1}=v'_h,v_{i_2}=v''_h$ or $v_{i_1}=v'_t,v_{i_2}=v''_t$}{
                        update $\eta^-_k=1$\;
                    }
                    update $\eta_k=\eta^+_k-\eta^-_k$ and $\tilde{\eta}_k=\eta^+_k \eta^-_k$\;
                }
                \If{$\tilde{\eta}=0$}{
                    \If{$\eta\geq0$}{
                        $\chi=1$\;
                    }
                    \If{$\eta\leq0$}{
                        $\chi=-1$\;
                    }
                }
                \KwOut{Indicator $\chi$.}
            \end{algorithm}
            \noindent
            We check the feasibility of \eqref{eq:check-path} by checking the feasibility of an optimization problem with a constant objective and constraint \eqref{eq:check-path}.
            We note that the optimization problem is an LP and its feasibility can be tractably identified.
            \begin{remark}\label{rmk:check-MCF}
                According to Proposition \ref{pps:check-MCF} and Algorithm \ref{alg:check-MCF}, to identify the supermodularity and submodularity in MinCF interdiction, we only need to solve a polynomial number of LPs, which significantly reduces the computational burden compared to directly applying Theorems \ref{trm:condition-generalLP-sub} and \ref{trm:condition-generalLP-super} by integer programs.
                Specifically, for the conditions with respect to paths, the number of LPs in Algorithm \ref{alg:check-MCF} is $2|\mathcal{A}_{a}|(|\mathcal{A}_{a}|-1)$.
            \end{remark}

            We next consider checking the conditions with respect to $\mathcal{V}_{s}$-excluded paths.
            Lemmas \ref{lem:check-path}--\ref{lem:check-path2} can be modified as Lemmas \ref{lem:check-excluded}--\ref{lem:check-excluded2}.
            The proofs are straightforward and are thus omitted.
            \begin{lemma}\label{lem:check-excluded}
                Given two arcs $a',a''\in\mathcal{A}_{a}$ and a node set $\mathcal{V}_{s}$.\\
                (i) There exists a $\mathcal{V}_{s}$-excluded path such that $a'$ and $a''$ have the same direction in the path, if and only if
                there exists a connected subnetwork $g$ such that $g$ contains the head node of one arc and the tail node of the other arc but does not contain the other two incident nodes in $\mathcal{V}_{\text{inc}}$ and any nodes in $\mathcal{V}_{s}$;\\
                (ii) There exists a $\mathcal{V}_{s}$-excluded path such that $a'$ and $a''$ have opposite directions in the path, if and only if
                there exists a connected subnetwork $g$ such that $g$ contains the two head nodes or the two tail nodes but does not contain the other two incident nodes in $\mathcal{V}_{\text{inc}}$ and any nodes in $\mathcal{V}_{s}$.
            \end{lemma}
            \begin{lemma}\label{lem:check-excluded2}
                There exists a connected subnetwork $g$ such that $g$ contains two incident nodes $v_{i_1},v_{i_2}\in\mathcal{V}_{\text{inc}}$ but does not contain the other two incident nodes in $\mathcal{V}_{\text{inc}}$ and any nodes in $\mathcal{V}_{s}$, if and only if
                the linear system
                \begin{equation}\label{eq:check-excluded}
                    \left\{\begin{aligned}
                        &Ty=d\\
                        &y_{j}=0,\forall j\in\mathcal{J}_{\text{inc}}(\mathcal{V}_{\text{inc}}\backslash\{v_{i_1},v_{i_2}\})\cup\mathcal{J}_{\text{inc}}(\mathcal{V}_{s})
                    \end{aligned}\right.
                \end{equation}
                is feasible.
            \end{lemma}
            \noindent
            Based on Lemma \ref{lem:check-excluded} and Lemma \ref{lem:check-excluded2}, we can develop an algorithm similar to Algorithm \ref{alg:check-MCF} to check the proposed conditions with respect to $\mathcal{V}_{s}$-excluded paths.
            The only difference is to replace \eqref{eq:check-path} by \eqref{eq:check-excluded} in step 4 and step 6 of Algorithm \ref{alg:check-MCF}.
            We  note \eqref{eq:check-excluded} is linear and hence Remark \ref{rmk:check-MCF} still holds.

            We next consider checking the conditions with respect to cycles.
            Lemmas \ref{lem:check-path}--\ref{lem:check-path2} can be modified as Lemmas \ref{lem:check-cycle}--\ref{lem:check-cycle2}.
            The inequality in \eqref{eq:check-cycle} limits that there is no shared node between $g_1$ and $g_2$.
            The proofs are straightforward and are thus omitted.
            \begin{lemma}\label{lem:check-cycle}
                Given two arcs $a',a''\in\mathcal{A}_{a}$.\\
                (i) There exists a cycle such that $a'$ and $a''$ have the same direction in the cycle, if and only if
                there exists a connected subnetwork $g_1$ containing the head node of one arc and the tail node of the other arc and a connected subnetwork $g_2$ containing the other two incident nodes such that $g_1$ and $g_2$ do not share node;\\
                (ii) There exists a cycle such that $a'$ and $a''$ have opposite directions in the cycle, if and only if
                there exists a connected subnetwork $g_1$ containing the two head nodes and a connected subnetwork $g_2$ containing the two tail nodes such that $g_1$ and $g_2$ do not share node.
            \end{lemma}
            \begin{lemma}\label{lem:check-cycle2}
                There exists a connected subnetwork $g_1$ containing two incident nodes $v_{i_1},v_{i_2}\in\mathcal{V}_{\text{inc}}$ and a connected subnetwork $g_2$ containing the other two incident nodes $v_{i_3},v_{i_4}\in\mathcal{V}_{\text{inc}}$ such that $g_1$ and $g_2$ do not share node, if and only if
                the linear system
                \begin{equation}\label{eq:check-cycle}
                    \left\{\begin{aligned}
                        &Ty=d,Ty'=d'\\
                        &|T|(|y|+|y'|)\leq2
                    \end{aligned}\right.
                \end{equation}
                is feasible, where $y,y'\in\mathbb{R}^{n_a}$ and $d,d'\in\{-1,0,1\}^{n_v}$ indicate the incident nodes.
            \end{lemma}
            \noindent
            Based on Lemma \ref{lem:check-cycle} and Lemma \ref{lem:check-cycle2}, we can develop an algorithm similar to Algorithm \ref{alg:check-MCF} to check the proposed conditions with respect to cycles.
            The only difference is to replace the if-condition in step 4 and step 6 of Algorithm \ref{alg:check-MCF} by ``\eqref{eq:check-cycle} is feasible when $v_{i_1}=v'_h,v_{i_2}=v''_t$ and $v_{i_3}=v''_h,v_{i_4}=v'_t$'' and ``\eqref{eq:check-cycle} is feasible when $v_{i_1}=v'_h,v_{i_2}=v''_h$ and $v_{i_3}=v'_t,v_{i_4}=v''_t$'', respectively.
            We note that the inequality in \eqref{eq:check-cycle} can be linearized by introducing auxiliary variables $\bar{y}$ and $\bar{y}'$ such that $|y|\leq\bar{y}$, $|y'|\leq\bar{y}'$, and $|T|(\bar{y}+\bar{y}')\leq2$.
            Hence, Remark \ref{rmk:check-MCF} still holds.

        \subsection{Discussion on Violated Cases}
            To close this section, we consider cases where the above established conditions are violated.
            It is normal that the conditions may not always hold in practical problems, but we can still exploit the conditions to generate valid inequalities for the optimality condition \eqref{eq:optimality}, which may help to improve the computational performance in solving network interdiction problems.
            We provide two approaches in the following.

            \subsubsection{Type 1}
                The basic idea of this approach is to remove certain arcs from the network so that, within the reduced network, the remaining attackable arcs satisfy one of the established conditions.
                We provide the following lemma.
                \begin{lemma}\label{lem:partition}
                    Given a network $\mathcal{G}=(\mathcal{V},\mathcal{A})$ and a set $\mathcal{A}_{a}\subseteq\mathcal{A}$ of attackable arcs.
                    There exist a subset $\mathcal{A}_r\subseteq\mathcal{A}$ such that, within the reduced network $\mathcal{G}\backslash\mathcal{A}_r$, arcs in $\mathcal{A}_a\backslash\mathcal{A}_r$ satisfy the condition for either submodularity or supermodularity in Proposition \ref{pps:general-MCF-capacity} or \ref{pps:general-MCF-cost}, where $\mathcal{G}\backslash\mathcal{A}_r$ indicates a reduced network by removing arcs in $\mathcal{A}_r$ from $\mathcal{G}$.
                \end{lemma}
                \noindent
                Lemma \ref{lem:partition} is straightforward because we can always consider an extreme case where $|\mathcal{A}_a\backslash\mathcal{A}_r|=1$.

                Then, to generate valid inequalities, we need to derive upper bounds of $\phi(x)$.
                We consider attacks on cost coefficients, and for all $\mathcal{A}_r$ satisfying Lamma \ref{lem:partition}, we rewrite \eqref{eq:MCF-cost} as
                \begin{equation}\label{eq:MCF-cost-partition3}
                    \begin{aligned}
                        \phi(x)=\min_{y_{\mathcal{J}_r}\in\mathbb{R}^{|\mathcal{J}_r|}}\left\{
                        \sum_{j\in\mathcal{J}_r}(c_j+\delta_j x_j)y_j+\varphi(x_{[n_a]\backslash\mathcal{J}_r},y_{\mathcal{J}_r}):~
                        0\leq y_j\leq f_j,~\forall j\in\mathcal{J}_r
                        \right\}
                    \end{aligned}
                \end{equation}
                where $\mathcal{J}_r\subseteq[n_a]$ is the index set of arcs corresponding to $\mathcal{A}_r$ and
                \begin{equation}\label{eq:MCF-cost-partition4}
                    \begin{aligned}
                        \varphi(x_{[n_a]\backslash\mathcal{J}_r},y_{\mathcal{J}_r}):=\min_{y_j\in\mathbb{R}:j\in[n_a]\backslash\mathcal{J}_r}~&\sum_{j\in[n_a]\backslash\mathcal{J}_r}(c_j+\delta_j x_j)y_j\\
                        \text{s.t.}~&\sum_{j\in[n_a]\backslash\mathcal{J}_r}T_j y_j\geq d-\sum_{j\in\mathcal{J}_r}T_j y_j\\
                        &0\leq y_j\leq f_j,~\forall j\in[n_a]\backslash\mathcal{J}_r.
                    \end{aligned}
                \end{equation}
                Following \eqref{eq:MCF-cost-partition3}, we propose valid inequalities in Proposition \ref{pps:partition-cost}.
                \begin{proposition}\label{pps:partition-cost}
                    Consider the defender's model $\phi(x)$ defined in \eqref{eq:MCF-cost}.
                    For all $\mathcal{A}_r$ satisfying Lamma \ref{lem:partition}, for all $\hat{y}_{\mathcal{J}_r}\in\mathbb{R}^{|\mathcal{J}_r|}$ such that $0\leq\hat{y}_{\mathcal{J}_r}\leq f_{\mathcal{J}_r}$ and $\hat{y}_{\mathcal{J}_r}$ is feasible to $\varphi(x_{[n_a]\backslash\mathcal{J}_r},y_{\mathcal{J}_r})$ defined in \eqref{eq:MCF-cost-partition4}, we have a valid inequality for the optimality condition \eqref{eq:optimality} as
                    \begin{equation}\label{eq:MCF-cost-partition5}
                        \begin{aligned}
                            \phi(x)\leq\sum_{j\in\mathcal{J}_r}(c_j+\delta_j x_j)\hat{y}_j+\varphi(x_{[n_a]\backslash\mathcal{J}_r},\hat{y}_{\mathcal{J}_r}).
                        \end{aligned}
                    \end{equation}
                \end{proposition}

                According to Lemma \ref{lem:partition}, $\varphi(\cdot)$ in Proposition \ref{pps:partition-cost}
                is either submodular or supermodular in $x_{[n_a]\backslash\mathcal{J}_r}$ after fixing $y_{\mathcal{J}_r}$ and can thus be replaced by valid inequalities \eqref{eq:submodular-cut} or \eqref{eq:supermodular-cut}.
                We note that there exist multiple sets $\mathcal{A}_r$ satisfying Lamma \ref{lem:partition}, and different $\mathcal{A}_r$ and $\hat{y}_{\mathcal{J}_r}$ induce valid inequalities of different strengths.
                Generally, we select a set $\mathcal{A}_r$ as small as possible for higher strength and calculate the $\hat{y}_{\mathcal{J}_r}$ optimal to incumbent $x$ for tightness.
                However, more details on the selection of $\mathcal{A}_r$ are beyond the scope of this paper and is left for future work.

            \subsubsection{Type 2}
                The basic idea of this approach is to duplicate the network and distribute the attackable arcs across different duplicated networks such that the attackable arcs in each duplicated network satisfy one of the established conditions.
                We provide the following lemma.
                \begin{lemma}\label{lem:partition2}
                    Given a network $\mathcal{G}=(\mathcal{V},\mathcal{A})$ and a set $\mathcal{A}_{a}\subseteq\mathcal{A}$ of attackable arcs.
                    There exist multiple ($N$) subsets, denoted by $\mathcal{A}^k_a\subseteq\mathcal{A}_a$ with $k\in[N]$, such that $\mathcal{A}_a=\cup_{k\in[N]}\mathcal{A}^k_a$, and for all $k\in[N]$, arcs in $\mathcal{A}^k_a$ satisfy the condition for either submodularity or supermodularity in Proposition \ref{pps:general-MCF-capacity} or \ref{pps:general-MCF-cost}.
                \end{lemma}
                \noindent
                Lemma \ref{lem:partition2} is also straightforward because we can always consider an extreme case where $|\mathcal{A}^k_a|=1$ for all $k\in[N]$.

                Then, to generate valid inequalities, we need to derive upper bounds of $\phi(x)$.
                We consider attacks on cost coefficients, and for all groups of $\mathcal{A}^k_a$ satisfying Lamma \ref{lem:partition2}, we rewrite \eqref{eq:MCF-cost} as
                \begin{equation}\label{eq:MCF-cost-partition23}
                    \begin{aligned}
                        \phi(x)=\min_{\alpha\in\mathbb{R}^N}\left\{
                        \sum_{k\in[N]}\alpha_k\tilde{\phi}^k(x):~
                        1^{\top}\alpha=1,~\alpha\geq0
                        \right\},
                    \end{aligned}
                \end{equation}
                where $y$ is replaced by $\sum_{k\in[N]}\tilde{y}^k$ and
                \begin{equation}\label{eq:MCF-cost-partition24}
                    \begin{aligned}
                        \tilde{\phi}^k(x):=&\min_{\tilde{y}^k\in\mathbb{R}^{n_a}}\left\{
                            \sum_{j\in[n_a]}(c_j+\delta_j x_j)\tilde{y}^k_j:~
                            T\tilde{y}^k\geq d,~
                            0\leq \tilde{y}^k\leq f
                        \right\}.
                    \end{aligned}
                \end{equation}
                We denote by $\mathcal{J}^k_a$ the index set of arcs in $\mathcal{A}^k_a$.
                We restrict $x_j=1$ for all $j\in[n_a]\backslash\mathcal{J}^k_a$ and have that
                \begin{equation}\label{eq:MCF-cost-partition21}
                    \begin{aligned}
                        \tilde{\phi}^k(x)\leq\bar{\phi}^k(x_{\mathcal{J}^k_a}):=&\min_{\tilde{y}^k\in\mathbb{R}^{n_a}}
                        \left\{
                            \sum_{j\in\mathcal{J}^k_a}(c_j+\delta_j x_j)\tilde{y}^k_j+\sum_{j\in[n_a]\backslash\mathcal{J}^k_a}(c_j+\delta_j)\tilde{y}^k_j:~
                            T\tilde{y}^k\geq d,~
                            0\leq \tilde{y}^k\leq f
                        \right\}.
                    \end{aligned}
                \end{equation}
                Following \eqref{eq:MCF-cost-partition23} and \eqref{eq:MCF-cost-partition21}, we propose valid inequalities in Proposition \ref{pps:partition2-cost}.
                \begin{proposition}\label{pps:partition2-cost}
                    Consider the defender's model $\phi(x)$ defined in \eqref{eq:MCF-cost}.
                    For all groups of $\mathcal{A}^k_a$ satisfying Lamma \ref{lem:partition2}, for all $\hat{\alpha}\in\mathbb{R}^{N}$ such that $\hat{\alpha}\geq0$ and $1^{\top}\hat{\alpha}=1$, we have a valid inequality for the optimality condition \eqref{eq:optimality} as
                    \begin{equation}\label{eq:MCF-cost-partition5}
                        \begin{aligned}
                            \phi(x)\leq\sum_{k\in[N]}\hat{\alpha}_k\bar{\phi}^k(x_{\mathcal{J}^k_a}),
                        \end{aligned}
                    \end{equation}
                    where $\bar{\phi}^k(x_{\mathcal{J}^k_a})$ is defined in \eqref{eq:MCF-cost-partition21}.
                \end{proposition}

                Moreover, according to Lemma \ref{lem:partition2}, $\bar{\phi}(\cdot)$ in Proposition \ref{pps:partition2-cost}
                is either submodular or supermodular in $x_{\mathcal{J}^k_a}$ and can thus be replaced by valid inequalities \eqref{eq:submodular-cut} or \eqref{eq:supermodular-cut}.
                We note that there exist multiple groups of $\mathcal{A}^k_a$ satisfying Lamma \ref{lem:partition2}, and different groups of $\mathcal{A}^k_a$ and $\hat{\alpha}$ induce valid inequalities of different strengths.
                Generally, we select as few subsets $\mathcal{A}^k_a$ as possible and calculate the $\hat{\alpha}$ optimal to incumbent $x$ for higher strength.
                However, more details on the selection of $\mathcal{A}^k_a$ are beyond the scope of this paper and is left for future work.

    \section{Incorporation of More Network Information}\label{sec:specific}
        The analysis in Section \ref{sec:general} is dedicated to general network interdiction.
        Proposition \ref{pps:general-MCF-demand} provides a general conclusion that attacks on nodal supplies/demands will lead to supermodularity.
        Yet for attacks on arcs, there is no such general conclusion.
        Though Propositions \ref{pps:general-MCF-capacity}--\ref{pps:general-MCF-mixed} give corresponding conditions for supermodularity and submodularity, these conditions are somewhat restrictive.
        Moreover, when additional network information is given, such as detailed network parameters and special network structures, these conditions may be sufficient but not necessary.
        Therefore, we incorporate more network information in this section and establish sufficient and less restrictive conditions for supermodularity and submodularity under attacks on arcs.

        \subsection{Network Reduction under Detailed Network Parameters}
            We first incorporate detailed network parameters into analysis, which may reduce the network topology and thus relax the conditions.
            The core idea is to remove redundant arcs under detailed network parameters so that we do not need to check all paths/$\mathcal{V}_s$-excluded paths/cycles as described in Propositions \ref{pps:general-MCF-capacity}--\ref{pps:general-MCF-mixed}.
            We give the definition of redundant arcs as follows.
            \begin{definition}\label{def:interdiction-independent}
                Given a network $\mathcal{G}=(\mathcal{V},\mathcal{A})$ and a set $\mathcal{A}_{a}\subseteq\mathcal{A}$ of attackable arcs.
                Let $\psi_{\mathcal{G}}:2^{\mathcal{A}_{a}}\rightarrow\mathbb{R}$ denote the optimal objective of the defender's problem on network $\mathcal{G}$ under attacks on arcs in set $\mathcal{A}_{a}$.
                Then, an arc $a_j\in\mathcal{A}$ is $\mathcal{A}_{a}$-robust redundant if
                \begin{equation}\label{eq:interdiction-independent}
                    \psi_{\mathcal{G}}(\mathcal{X})=\psi_{\mathcal{G}\backslash\{a_j\}}(\mathcal{X}\backslash\{a_j\}),\forall \mathcal{X}\subseteq\mathcal{A}_{a}.
                \end{equation}
            \end{definition}
            \noindent
            From Definition \ref{def:interdiction-independent}, the removal of redundant arcs does not change the optimal objective of the defender's problem.
            Then, we can remove redundant arcs under detailed network parameters and relax the conditions in Propositions \ref{pps:general-MCF-capacity}--\ref{pps:general-MCF-mixed}.
            As analyzed in Section \ref{sec:general-MCF-cost}, \eqref{eq:MCF-cost} with attacks on cost coefficients shows a more general formulation.
            Hence, we consider \eqref{eq:MCF-cost} for example and the results in this section can be easily extended to other interdiction problems.

            We derive Proposition \ref{pps:specific-MCF-arc-sufficient} based on Proposition \ref{pps:general-MCF-cost} and the proof is provided in Appendix \ref{apd:pps:specific-MCF-arc-sufficient}.
            \begin{proposition}\label{pps:specific-MCF-arc-sufficient}
                Given a network $\mathcal{G}=(\mathcal{V},\mathcal{A})$ with incidence matrix $T$, a set $\mathcal{A}_{a}\subseteq\mathcal{A}$ of attackable arcs, and a vector $\delta\in\Delta(\mathcal{A}_{a})$.
                Under some known network parameters ($c$, $d$, or $f$), arcs in set $\mathcal{A}_{r}$ are $\mathcal{A}_{a}$-robust redundant.
                Then, it holds that\\
                (i) the function $\phi(x)$ defined by \eqref{eq:MCF-cost} is \emph{submodular}, if
                for any path $P$ in $\mathcal{G}\backslash\mathcal{A}_{r}$, for any two arcs $a',a''\in(\mathcal{A}_{a}\backslash\mathcal{A}_r)\cap\mathcal{A}_{P}$, $a'$ and $a''$ have the \emph{same} direction in $P$;\\
                (ii) the function $\phi(x)$ defined by \eqref{eq:MCF-cost} is \emph{supermodular}, if
                for any path $P$ in $\mathcal{G}\backslash\mathcal{A}_{r}$, for any two arcs $a',a''\in(\mathcal{A}_{a}\backslash\mathcal{A}_r)\cap\mathcal{A}_{P}$, $a'$ and $a''$ have \emph{opposite} directions in $P$.
            \end{proposition}
            \noindent
            Proposition \ref{pps:specific-MCF-arc-sufficient} provides sufficient conditions and check the topological relationship of remaining attackable arcs in the reduced network $\mathcal{G}\backslash\mathcal{A}_{r}$, as illustrated in Example \ref{exp:example-specific-arc-sufficient}.
            \begin{example}\label{exp:example-specific-arc-sufficient}
                We still consider the network topology in Example \ref{exp:example-path}.
                We set $v_{1},v_{2}$ to be supply nodes and $v_{3}$ to be a demand node, as shown in Figure \ref{fig:example-specific-arc-sufficient}.
                We suppose cost coefficients to be positive and suppose flow capacities to be sufficiently large.
                We then discuss the following two cases.
                \begin{figure}[!htbp]
                    \begin{center}
                        \vspace{-0ex}
                        \includegraphics[width=0.6\columnwidth]{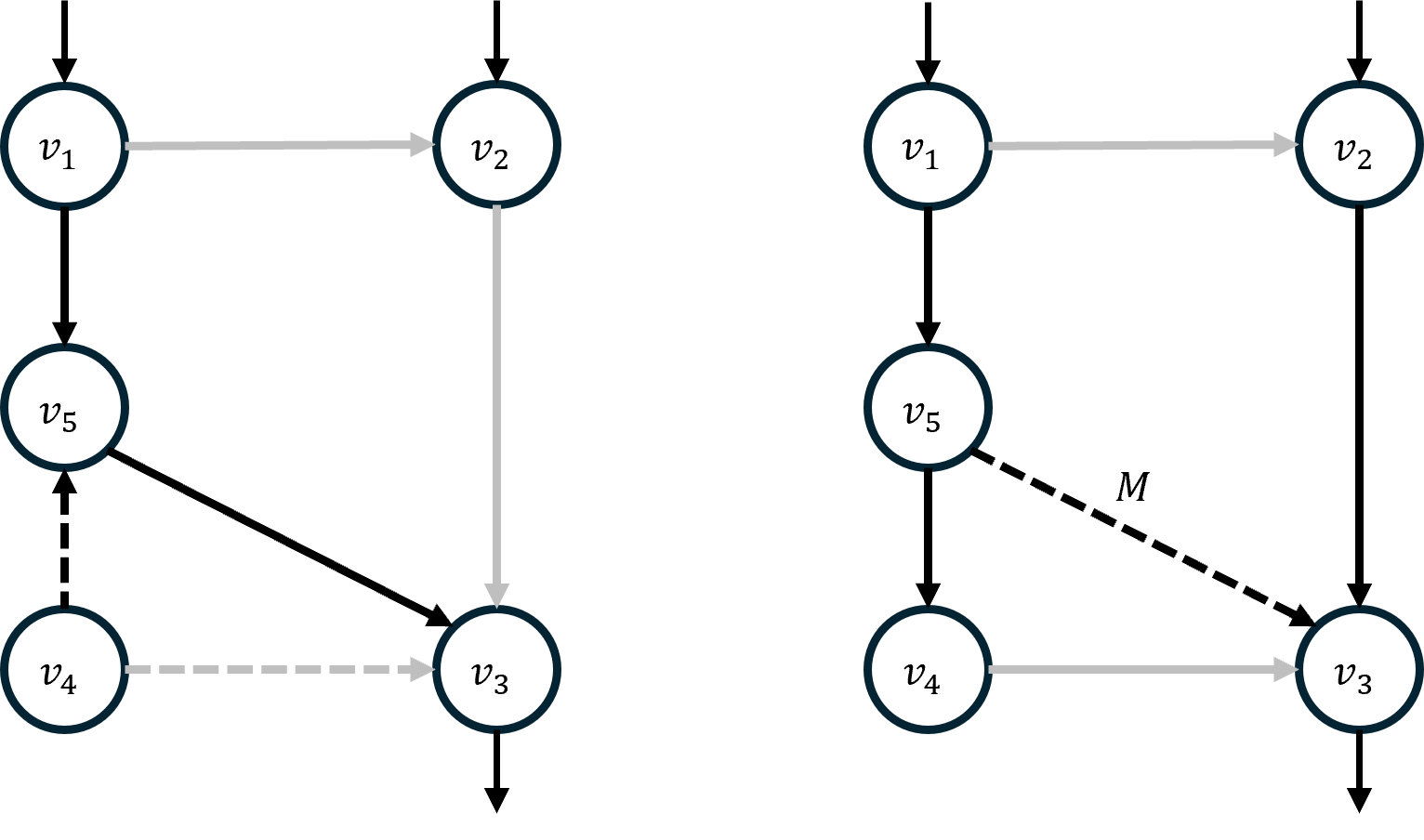}
                    \end{center}
                    \vspace{-4ex}
                    \caption{Topological relationship after removing redundant arcs.}
                    \vspace{-2ex}
                    \label{fig:example-specific-arc-sufficient}
                \end{figure}

                (1) Suppose that the arc between $v_4$ and $v_5$ is from $v_4$ to $v_5$, and we may attack the arcs in set $\mathcal{A}_{a}=\{(v_{1},v_{2}),(v_{2},v_{3}),(v_4,v_3)\}$, as shown in the left panel of Figure \ref{fig:example-specific-arc-sufficient}.
                According to the above settings, there must be no flow on arcs $(v_{4},v_{3})$ and $(v_{4},v_{5})$.
                Hence, arcs $(v_{4},v_{3})$ and $(v_{4},v_{5})$ are redundant, and we remove them from the network.
                Then, we do not need to check the topological relationship between $(v_{4},v_{3})$ and the remaining attackable arcs, and the remaining attackable arcs $(v_{1},v_{2})$ and $(v_{2},v_{3})$ always have the same direction in all paths of the reduced network.
                Therefore, although $(v_{4},v_{3})$ and $(v_{4},v_{5})$ have opposite directions in a certain path of the original network, the value function $\phi(x)$ defined by \eqref{eq:MCF-cost} is submodular.

                (2) Suppose that the arc between $v_4$ and $v_5$ is from $v_5$ to $v_4$, the cost coefficient on arc $(v_5,v_3)$ is sufficiently large, and we may attack the arcs in set $\mathcal{A}_{a}=\{(v_{1},v_{2}),(v_4,v_3)\}$, as shown in the right panel of Figure \ref{fig:example-specific-arc-sufficient}.
                According to the above settings, there must be no flow on arc $(v_{5},v_{3})$.
                Hence, arc $(v_{5},v_{3})$ is redundant, and we remove it from the network.
                Then, we do not need to check any path containing $(v_{5},v_{3})$, and the attackable arcs $(v_{1},v_{2})$ and $(v_{4},v_{3})$ always have opposite directions in all paths of the reduced network.
                Therefore, although $(v_{1},v_{2})$ and $(v_{4},v_{3})$ have the same direction in path $(v_4,v_3,v_5,v_1,v_2)$ of the original network, the value function $\phi(x)$ defined by \eqref{eq:MCF-cost} is supermodular.
            \end{example}

            From the example, we see that the conditions in Proposition \ref{pps:specific-MCF-arc-sufficient} are less restrictive than those in Proposition \ref{pps:general-MCF-cost}, and it is easier to check the conditions in a reduced network.
            A larger set $\mathcal{A}_{r}$ of redundant arcs leads to less restrictive conditions in Proposition \ref{pps:specific-MCF-arc-sufficient}.
            When $\mathcal{A}_{r}=\emptyset$, the conditions in Proposition \ref{pps:specific-MCF-arc-sufficient} reduces to those in Proposition \ref{pps:general-MCF-cost}.
            Hence, we hope to identify an $\mathcal{A}_{r}$ as large as possible.
            Generally, we can quickly identify parts of redundant arcs, such as arcs opposite to the expected flow direction and arcs with extremely high cost coefficients, as considered in Example \ref{exp:example-specific-arc-sufficient}.
            Yet, it may be expensive to identify all redundant arcs in a network with given detailed parameters.
            One may design algorithms to find all redundant arcs but it is not the focus of this paper.

        \subsection{Special Case: Series-Parallel Network}
            We then incorporate special network structure, series-parallel networks (SPNs), into the analysis of SP interdiction and MaxF interdiction, which enables more tractable conditions.
            The definitions related to SPNs are provided in Appendix \ref{apd:SPN-definition}.
            And due to the SPN structure, we can further simplify the network topology by series-parallel reduction, which merges most of the non-attackable arcs for ease of analysis.
            The series-parallel reduction for SP interdiction and MaxF interdiction is introduced in Appendix \ref{apd:SPN-reduction} and all SPNs considered in this section have been simplified.

            \subsubsection{Shortest Path Interdiction}
                We consider SP interdiction and the defender's model is provided in \eqref{eq:SP-distance}.
                Before presenting the conditions, we first introduce the following definitions.
                \begin{definition}\label{def:conditional parallel}
                    Given an SPN $\mathcal{G}=(\mathcal{V},\mathcal{A},s,t,T)$ and two arc sets $\mathcal{A}_1,\mathcal{A}_2\subseteq\mathcal{A}$.
                    Then, \\
                    (1) an arc $a$ is $(\mathcal{A}_1|\mathcal{A}_2)$-parallel if $a$ is in parallel with all arcs in $\mathcal{A}_1$ but is not in parallel with any arc in $\mathcal{A}_2\backslash\mathcal{A}_1$;\\
                    (2) if $|\mathcal{A}_1|=1$, an $(\mathcal{A}_1|\mathcal{A}_2)$-parallel subnetwork $g$ consists of all $(\mathcal{A}_1|\mathcal{A}_2)$-parallel arcs as well as their incident nodes.
                    A node $v$ in $g$ is the source or sink node of $g$ if every $s$--$t$ path containing an arc in $g$ passes through $v$.
                \end{definition}

                We then extend Proposition \ref{pps:general-SP-distance} and establish the conditions in SPNs with detailed parameters, as stated in Proposition \ref{pps:SPN-SP-arc}.
                The proofs is placed in Appendix \ref{apd:pps:SPN-SP-arc}.
                \begin{proposition}\label{pps:SPN-SP-arc}
                    Given an SPN $\mathcal{G}=(\mathcal{V},\mathcal{A},s,t,T)$ with parameters $c$, a set $\mathcal{A}_{a}\subseteq\mathcal{A}$ of attackable arcs, and a vector $\delta\in\Delta(\mathcal{A}_{a})$.
                    Denote by $\mathcal{A}_{P_s}$ the set of the arcs consisting the shortest path $P_s$ under no attacks.
                    Then, \\
                    (i) the function $\phi(x)$ defined by \eqref{eq:SP-distance} is supermodular, if and only if $\mathcal{A}_a\cap\mathcal{A}_{P_s}=\emptyset$ or
                    the following conditions are satisfied:\\
                    \indent
                    1) for all $\mathcal{A}'\subseteq\mathcal{A}_a\cap\mathcal{A}_{P_s}$ with $|\mathcal{A}'|=2$, all $(\mathcal{A}'|\emptyset)$-parallel arcs are $\mathcal{A}_a$-robust redundant, and\\
                    \indent
                    2) for all $a_k\in\mathcal{A}_a\cap\mathcal{A}_{P_s}$ and its corresponding $(\{a_k\}|(\mathcal{A}_a\cap\mathcal{A}_{P_s}))$-parallel subnetwork $g_k$, if $g_k$ exists, the shortest distance between the source and sink nodes of $g_k$ is supermodular in $x_{g_k}$ when $a_k$ is attacked, where $x_{g_k}$ consists of the entries of $x$ related to the attackable arcs in $g_k$;\\
                    (ii) the function $\phi(x)$ defined by \eqref{eq:SP-distance} is submodular, if and only if $\mathcal{A}_a\cap\mathcal{A}_{P_s}=\emptyset$ or
                    all arcs in $\mathcal{A}_a\backslash\mathcal{A}_{P_s}$ are $\mathcal{A}_a$-robust redundant.
                \end{proposition}
                \noindent
                Proposition \ref{pps:SPN-SP-arc} implies that,  depending on both the original shortest path $P_s$ and the redundant arcs, the shortest path distance with respect to attacks on arc distance under detailed network parameters may be supermodular or submodular.
                If the shortest path under any attack $x$ consists of $(\{a\}|(\mathcal{A}_a\cap\mathcal{A}_{P_s}))$-parallel subnetworks for all $a\in\mathcal{A}_a\cap\mathcal{A}_{P_s}$ and the shortest distance of each subnetwork is supermodular, the marginal distance of each additional attack is non-decreasing.
                If all attackable arcs outside $P_s$ are redundant, the marginal distance of each additional attack is non-increasing.
                The following examples illustrate the established conditions in Proposition \ref{pps:SPN-SP-arc}.
                \begin{example}\label{exp:example-SPN-SP-arc}
                    We provide examples for supermodularity and submodularity in Figure \ref{fig:example-SPN-SP-arc}(a) and Figure \ref{fig:example-SPN-SP-arc}(b), respectively.

                    \begin{figure}[!htbp]
                        \begin{center}
                            \vspace{-0ex}
                            \begin{subfigure}[b]{0.9\textwidth}
                                \centering
                                \includegraphics[width=\columnwidth]{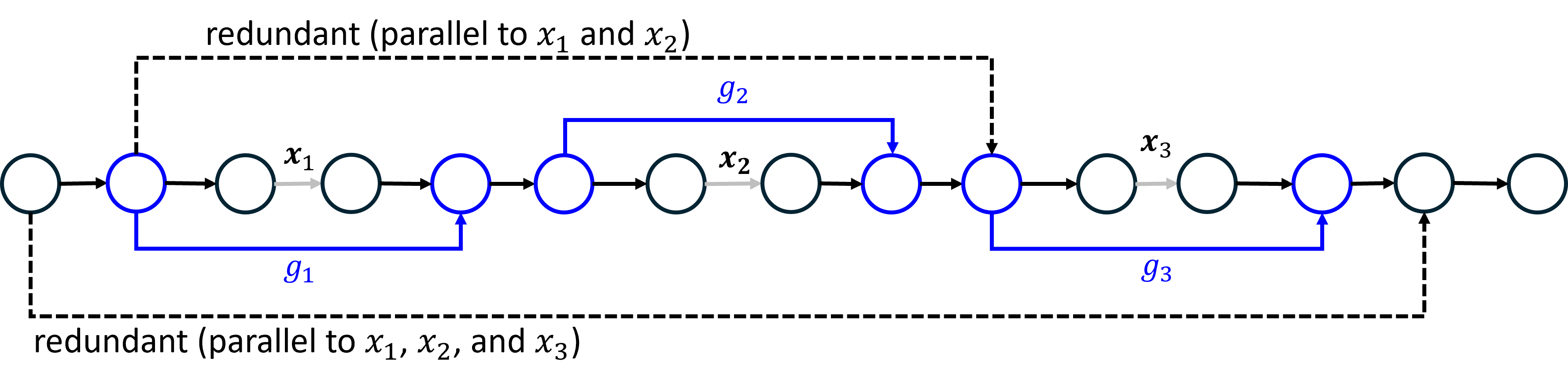}
                                \caption{Supermodular}
                                \vspace{1ex}
                                \label{fig:example-SPN-SP-arc-super}
                            \end{subfigure}\\
                            \begin{subfigure}[b]{0.9\textwidth}
                                \centering
                                \includegraphics[width=\columnwidth]{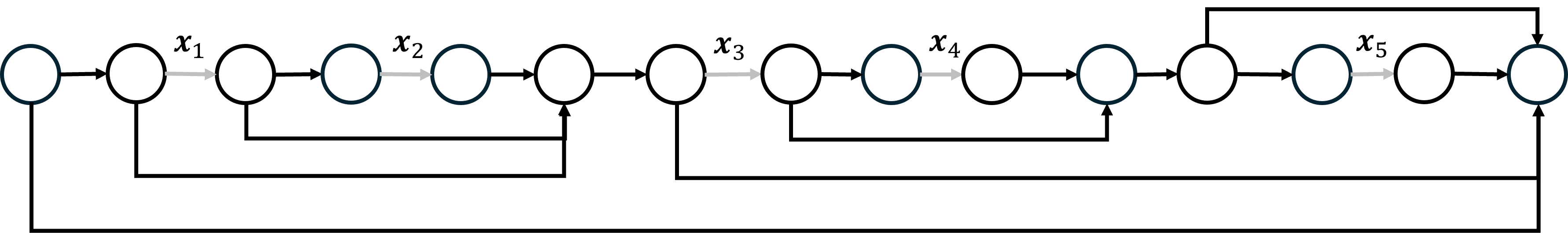}
                                \caption{Submodular}
                                \label{fig:example-SPN-SP-arc-sub}
                            \end{subfigure}
                        \end{center}
                        \vspace{-4ex}
                        \caption{Examples for SP interdiction in SPNs.}
                        \vspace{-2ex}
                        \label{fig:example-SPN-SP-arc}
                    \end{figure}

                    (i) In Figure \ref{fig:example-SPN-SP-arc}(a), we assume that the central main path is the shortest path in the original network $\mathcal{G}$, denoted by $P_s$, and it includes three attackable arcs corresponding to $x_1$, $x_2$, and $x_3$.
                    The attackable arcs in $P_s$ are marked in gray, all arcs in parallel with no less than two attackable arcs are marked in dashed lines, and for each single attack $x_k$, its corresponding subnetwork $g_k$ is marked in blue.
                    When the first condition for supermodularity holds, all gray arcs can be removed and we get the reduced network $\mathcal{G}'$.
                    Then, we can observe that the new shortest path $P'_s$ must pass through the source and sink nodes of each $g_k$.
                    Therefore, we can calculate the length of $P'_s$ by identifying the shortest path between the source and sink nodes of each $g_k$.
                    According to the second condition, we can guarantee the supermodularity of the shortest distance between the source and sink nodes of each $g_k$, and thus ensure the supermodularity of the shortest distance in $\mathcal{G}'$.

                    (ii) In Figure \ref{fig:example-SPN-SP-arc}(b), we assume that the central main path is the shortest path in the original network $\mathcal{G}$, denoted by $P_s$, and it includes a few attackable arcs (marked in gray).
                    Following the condition for submodularity, all attackable arcs outside $P_s$ have been removed in the reduced network $\mathcal{G}'$.
                    We can observe that if there exists a cycle containing two attackable arcs in $P_s$, the two arcs must have the same direction in the cycle, which aligns with the condition for submodularity in Proposition \ref{pps:general-SP-distance}.
                    Therefore, the shortest distance in $\mathcal{G}'$ in is submodular.
                \end{example}

                We note that the conditions in Proposition \ref{pps:SPN-SP-arc} align with those in Proposition \ref{pps:general-SP-distance}.
                For supermodularity, when all redundant arcs are removed, there exists no cycle containing two attackable arcs in $P_s$ or in different $g_k$.
                As illustrated in Figure \ref{fig:example-SPN-SP-arc}(a), an attackable arc in $P_s$ or in $g_k$ and another attackable arc in $P_s$ or in $g_{k'}$ ($k'\neq k$) must not be contained in a cycle.
                Furthermore, an attackable arc $a_k$ in $P_s$ must have the opposite direction to any attackable arc in $g_k$, and the topological relationship among arcs in each $g_k$ will be recursively limited by the second condition for supermodularity in Proposition \ref{pps:SPN-SP-arc}.
                For submodularity, as illustrated in Figure \ref{fig:example-SPN-SP-arc}(b), when all redundant arcs are removed, any two remaining attackable arcs must have the same direction if the two arcs are in the same cycle.

                From the above examples and analysis, the conditions in Proposition \ref{pps:SPN-SP-arc} are less restrictive than those in Proposition \ref{pps:general-SP-distance}.
                It is reasonable because Proposition \ref{pps:SPN-SP-arc} incorporates additional information, including the SPN structure and the detailed parameters.
                To check either of the conditions in Proposition \ref{pps:SPN-SP-arc}, we need to identify the $\mathcal{A}_{a}$-robust redundancy of at most $|\mathcal{A}_a|$ arcs.
                Following Definition \ref{def:interdiction-independent}, we can identify the redundancy of an arc through an LP as stated in Corollary \ref{crl:redundancy-SP} and the proof is provided in Appendix \ref{apd:crl:redundancy-SP}.
                \begin{corollary}\label{crl:redundancy-SP}
                    Consider an SPN $\mathcal{G}$.
                    Then, for all arc $a_j\in\mathcal{A}$, $a_j$ is $\mathcal{A}_{a}$-robust redundant with respect to \eqref{eq:SP-distance} if the equality
                    \begin{equation}\label{eq:redundancy-SP}
                        \begin{aligned}
                            \hat{y}^*_j=0
                        \end{aligned}
                    \end{equation}
                    holds, where $\hat{y}_j^*$ is optimal to $\phi(\hat{x})$;
                    $\hat{x}_j=0$, $\hat{x}_k=0$ for all $k\in\mathcal{K}_s(a_j)$, and $\hat{x}_k=1$ for all $k\in\mathcal{K}_p(a_j)$;
                    $\mathcal{K}_s(a_j)$ and $\mathcal{K}_p(a_j)$ are the index sets of attackable arcs that are in series with and in parallel with $a_j$, respectively.
                \end{corollary}

            \subsubsection{Maximum Flow Interdiction}
                We consider MaxF interdiction and the defender's model is provided in \eqref{eq:MF-capacity}.
                We first introduce the concept of non-binding arcs in the following.
                \begin{definition}\label{def:non-binding}
                    Given a network $\mathcal{G}=(\mathcal{V},\mathcal{A})$, a set $\mathcal{A}_{a}\subseteq\mathcal{A}$ of attackable arcs, and a vector $\delta\in\Delta(\mathcal{A}_{a})$.
                    Then, an arc $a\in\mathcal{A}$ is $\mathcal{A}_{a}$-robust non-binding if the flow on $a$ optimal to $\phi(x)$ defined in \eqref{eq:MF-capacity} does not reach its flow capacity for all $x\in\{0,1\}^{n_a}$.
                \end{definition}
                \noindent
                According to Definition \ref{def:non-binding}, attacks on non-binding arcs can be ignored in MaxF interdiction.
                We can treat the non-binding attackable arcs as non-attackable arcs and further simplify the network topology by series-parallel reduction.
                Therefore, all remaining attackable arcs in the simplified SPN are possibly binding.

                We then extend Proposition \ref{pps:general-MF-capacity} and establish the conditions in simplified SPNs with detailed parameters, as stated in Proposition \ref{pps:SPN-MF-arc}.
                The proofs is placed in Appendix \ref{apd:pps:SPN-MF-arc}.
                \begin{proposition}\label{pps:SPN-MF-arc}
                    Given a simplified SPN $\mathcal{G}=(\mathcal{V},\mathcal{A},s,t,T)$ with parameters $f$, a set $\mathcal{A}_{a}\subseteq\mathcal{A}$ of attackable arcs, and a vector $\delta\in\Delta(\mathcal{A}_{a})$.
                    Denote by $\mathcal{A}_{C_m}$ the set of the arcs in the minimum cut $C_m$ under no attacks.
                    Then, \\
                    (i) the function $\phi(x)$ defined by \eqref{eq:MF-capacity} is submodular, if and only if $\mathcal{A}_{a}\subseteq\mathcal{A}_{C_m}$ or
                    all arcs in $\mathcal{A}_{a}$ are in parallel with each other;\\
                    (ii) the function $\phi(x)$ defined by \eqref{eq:MF-capacity} is supermodular, if and only if $\mathcal{A}_{a}\subseteq\mathcal{A}_{C_m}$ or
                    for all $a\in\mathcal{A}_a\backslash\mathcal{A}_{C_m}$, if $a$ is in series with at least two arcs in $\mathcal{A}_{C_m}$, $a$ must be binding whenever attacked;
                    otherwise, $a$ is strictly in series with an arc in $\mathcal{A}_{C_m}$.
                \end{proposition}
                \noindent
                Proposition \ref{pps:SPN-MF-arc} implies that, depending on both the original minimum cut $C_m$ and the topological relationship of attackable arcs, the maximum flow with respect to attacks on arc capacity under detailed network parameters may be supermodular or submodular.
                If all attackable arcs in the simplified SPN are in parallel with each other, the marginal flow of each additional attack is non-decreasing.
                If all attackable arcs outside $C_m$ are binding whenever attacked or are strictly in series, the marginal distance of each additional attack is non-increasing.
                The following examples illustrate the established conditions in Proposition \ref{pps:SPN-MF-arc}.
                \begin{example}\label{exp:example-SPN-SP-arc}
                    We provide examples for supermodularity and submodularity in Figure \ref{fig:example-SPN-MF-arc}(a) and Figure \ref{fig:example-SPN-MF-arc}(b), respectively.

                    (i) In Figure \ref{fig:example-SPN-MF-arc}(a), we assume that the dashed line indicates the minimum cut in the original network $\mathcal{G}$, denoted by $C_m$.
                    The attackable arcs are marked in gray.
                    We can observe that all attackable arcs are in parallel with each other.
                    According to the condition, we can guarantee the submodularity of the maximum flow in $\mathcal{G}$.

                    (ii) In Figure \ref{fig:example-SPN-MF-arc}(b), we assume that the dashed line indicates the minimum cut in the original network $\mathcal{G}$, denoted by $C_m$.
                    The attackable arcs are marked in gray.
                    We can observe that arc $a_j$ is in series with two arcs in $C_m$, but its capacity reduced to zero when attacked and thus must be binding.
                    Each of the other attackable arcs outside $C_m$ is strictly in series with an arc in $C_m$.
                    Following the condition for submodularity, the maximum flow in $\mathcal{G}$ in is supermodular.

                    \begin{figure}[!htbp]
                        \begin{center}
                            \vspace{-0ex}
                            \begin{subfigure}[b]{0.48\textwidth}
                                \centering
                                \includegraphics[width=\columnwidth]{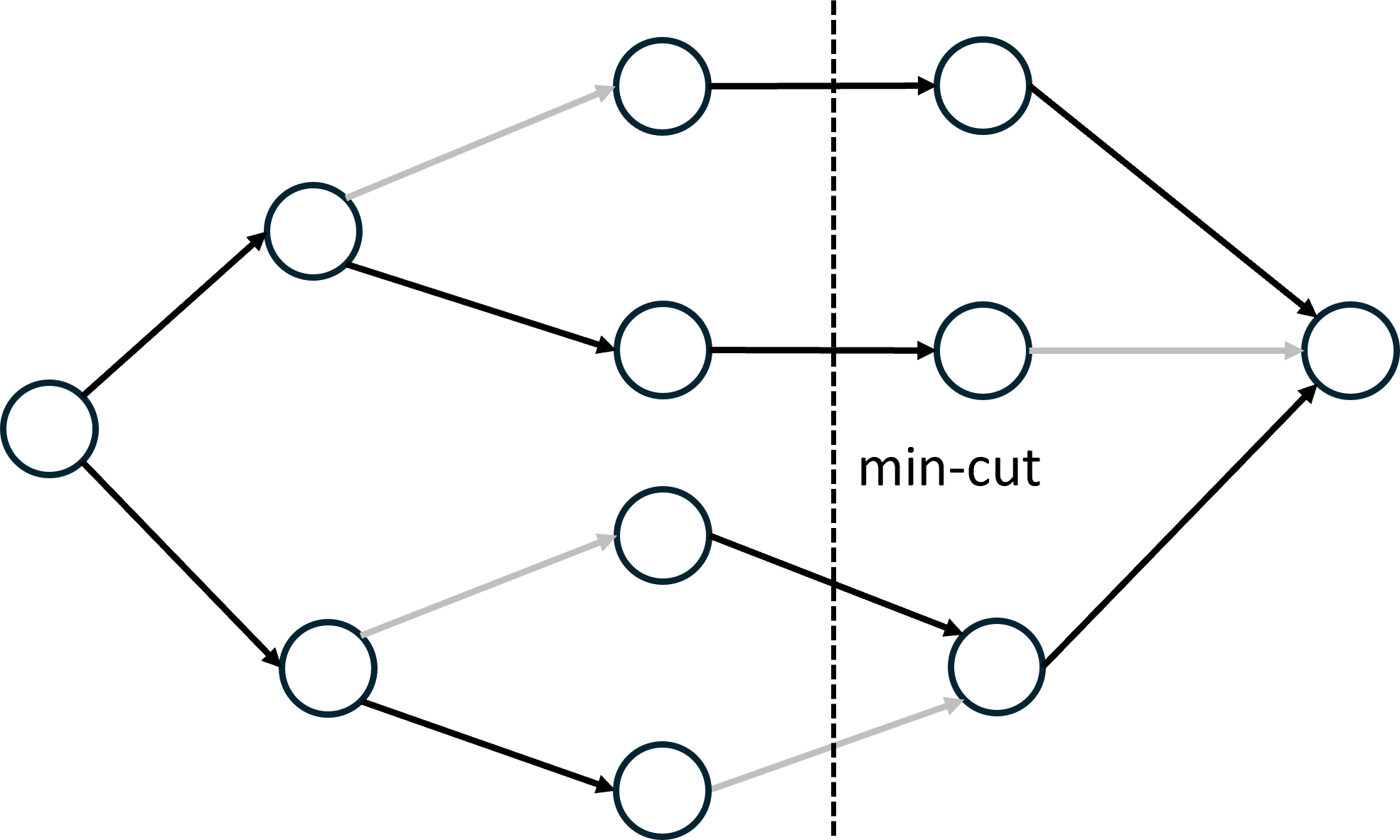}
                                \caption{Submodular}
                                \vspace{-1ex}
                                \label{fig:example-SPN-MF-arc-sub}
                            \end{subfigure}
                            \hfill
                            \begin{subfigure}[b]{0.48\textwidth}
                                \centering
                                \includegraphics[width=\columnwidth]{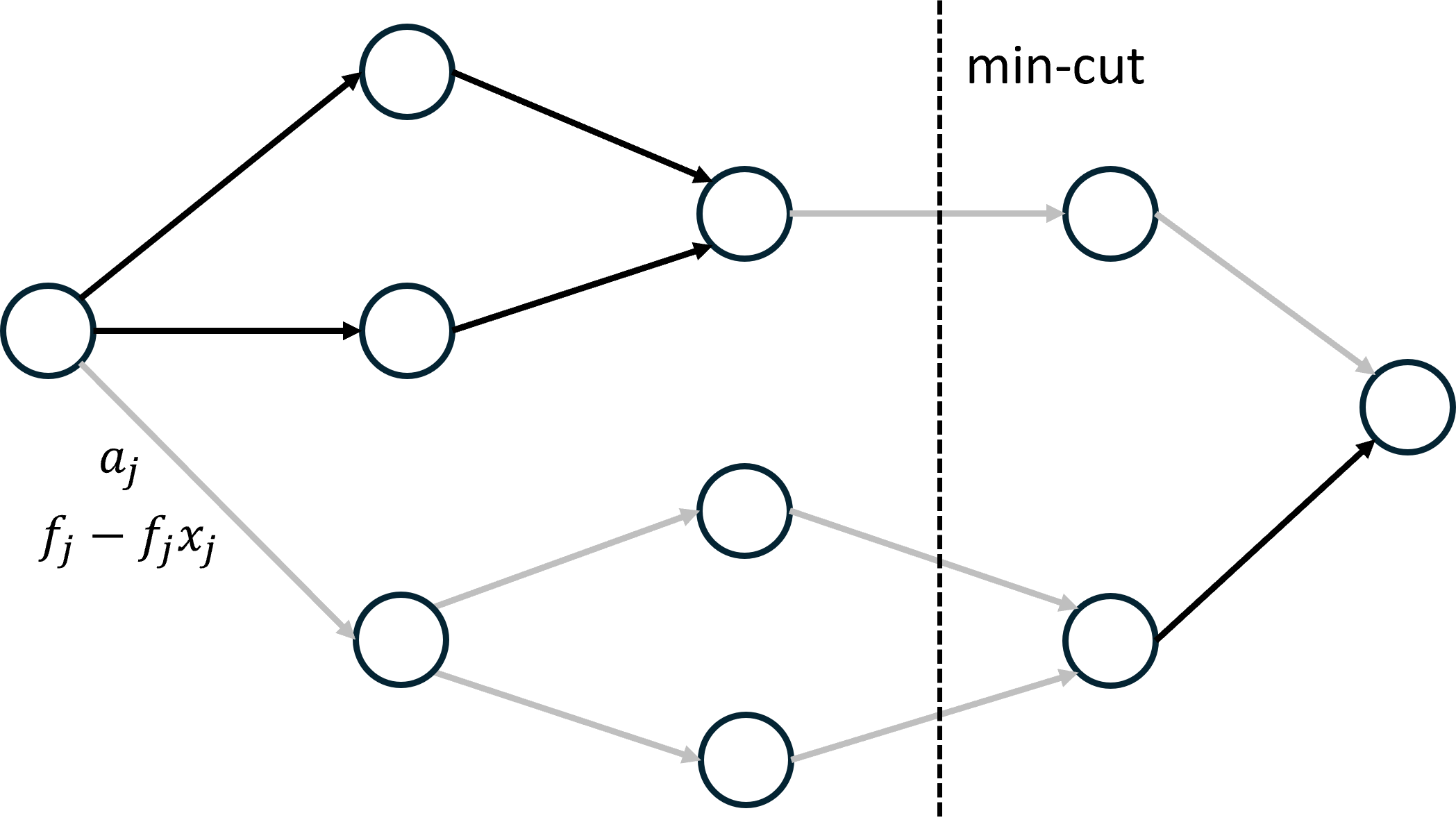}
                                \caption{Supermodular}
                                \vspace{-1ex}
                                \label{fig:example-SPN-MF-arc-super}
                            \end{subfigure}
                        \end{center}
                        \vspace{-2ex}
                        \caption{Examples for MaxF interdiction in SPNs.}
                        \vspace{-2ex}
                        \label{fig:example-SPN-MF-arc}
                    \end{figure}
                \end{example}

                We note that the conditions in Proposition \ref{pps:SPN-MF-arc} align with those in Proposition \ref{pps:general-MF-capacity}.
                For submodularity, all arcs that are in series with at least two arcs in $C_m$ are non-binding, and hence, each attackable arc can be analyzed independently.
                As illustrated in Figure \ref{fig:example-SPN-MF-arc}(a), each pair of attackable arcs naturally has opposite directions in any $s$--$t$ path in the network $\mathcal{G}$.
                For supermodularity, as illustrated in Figure \ref{fig:example-SPN-SP-arc}(b), except arc $a_j$, the other attackable arcs always have the same direction in any $s$--$t$ path in the network $\mathcal{G}$.

                From the above examples and analysis, the conditions in Proposition \ref{pps:SPN-MF-arc} are less restrictive than those in Proposition \ref{pps:general-MF-capacity}.
                It is reasonable because Proposition \ref{pps:SPN-MF-arc} incorporates additional information, including the SPN structure and the detailed parameters.
                To check either of the conditions in Proposition \ref{pps:SPN-MF-arc}, we need to identify the $\mathcal{A}_{a}$-robust non-binding of at most $|\mathcal{A}_a|$ arcs.
                Following Definition \ref{def:non-binding}, we can identify the non-binding of an arc through an LP as stated in Corollary \ref{crl:non-binding-MF} and the proof is provided in Appendix \ref{apd:crl:non-binding-MF}

                \begin{corollary}\label{crl:non-binding-MF}
                    Consider an SPN $\mathcal{G}$.
                    Then, for all arc $a_j\in\mathcal{A}$, $a_j$ is $\mathcal{A}_{a}$-robust non-binding with respect to \eqref{eq:MF-capacity} if the inequality
                    \begin{equation}\label{eq:non-binding-MF}
                        \begin{aligned}
                            \hat{y}_j^*<f_j-\delta_j\hat{x}_j
                        \end{aligned}
                    \end{equation}
                    holds, where $\hat{y}_j^*$ is optimal to $\phi(\hat{x})$;
                    $\hat{x}_j=1$, $\hat{x}_k=0$ for all $k\in\mathcal{K}_s(a_j)$, and $\hat{x}_k=1$ for all $k\in\mathcal{K}_p(a_j)$;
                    $\mathcal{K}_s(a_j)$ and $\mathcal{K}_p(a_j)$ are the index sets of attackable arcs that are in series with and in parallel with $a_j$, respectively.
                \end{corollary}

    \section{Mixed-Integer Defender}\label{sec:extension}
        In this section, we extend to a more challenging network interdiction problem where the defender may make additional binary decisions to repair (or reinforce) the network by, e.g., offsetting the attacker's interdiction on the supply/demand, decreasing the cost coefficients of arcs, or increasing the flow capacities.
        The model is formulated as
        \begin{equation}\label{eq:MCF-integer}
            \begin{aligned}
                \phi'(x)=&\min_{\substack{z\in\{0,1\}^{|X|}}} \ \ r^{\top}z+\phi(x-z),
            \end{aligned}
        \end{equation}
        where $z$ is the additional binary variable and $r>0$ indicates the repair cost;
        $\phi(\cdot)$ is a parametric LP defined by either of \eqref{eq:MCF-demand}--\eqref{eq:SP-distance}.
        Compared to $\phi(\cdot)$, model~\eqref{eq:MCF-integer} defines $\phi'(\cdot)$ by a parametric mixed-integer LP.
        Consequently, Theorems \ref{trm:condition-generalLP-sub}--\ref{trm:condition-generalLP-super} and the propositions in the above sections are no longer applicable to \eqref{eq:MCF-integer}.
        Nevertheless, $\phi'(\cdot)$ is likely to exhibit approximate submodularity or supermodularity under the above identified conditions.

        \subsection{Preserve Submodularity}
            To analyze the submodularity of $\phi'$, we equivalently rewrite \eqref{eq:MCF-integer} as
            \begin{equation}\label{eq:MCF-integer2}
                \begin{aligned}
                    \phi'(x)=&\min_{\substack{\tilde{z}\in\{0,1\}^{|X|}}} \ \ r^{\top}(1-\tilde{z})+\varphi(x,\tilde{z}),
                \end{aligned}
            \end{equation}
            where $\tilde{z}=1-z$ indicates \emph{not} repairing cost coefficients on arcs and
            \begin{equation}\label{eq:MCF-cost-repair}
                \begin{aligned}
                    \varphi(x,\tilde{z}):=&\phi\big(x-(1-\tilde{z})\big)\\
                    =&\min_{y\in\mathbb{R}^{n_{a}}}\big(c+\text{diag}(\delta)(x+\tilde{z}-1)\big)^{\top}y\\
                    &~~~\text{s.t.}~Ty\geq d\\
                    &~~~~~~~~0\leq y\leq f.
                \end{aligned}
            \end{equation}
            Here, we consider $\phi(\cdot)$ defined in \eqref{eq:MCF-cost} for example.

            Then, we shows in Proposition \ref{pps:MCF-cost-repair-sub} that the condition for submodularity in Proposition \ref{pps:general-MCF-cost} also applies to \eqref{eq:MCF-integer2}.
            The proof is provided in Appendix \ref{apd:pps:MCF-cost-repair-sub}.
            \begin{proposition}\label{pps:MCF-cost-repair-sub}
                Given a network $\mathcal{G}$ with incidence matrix $T$, a set $\mathcal{A}_{a}\subseteq\mathcal{A}$ of attackable arcs, and a vector $\delta\in\Delta(\mathcal{A}_{a})$.
                Then, the function $\phi'(x)$ defined by~\eqref{eq:MCF-integer2} is \emph{submodular} for any $c$, $d$, $f$, and $r$ if
                for any path $P$ in $\mathcal{G}$, for any two arcs $a',a''\in\mathcal{A}_{a}\cap\mathcal{A}_{P}$, $a'$ and $a''$ have the \emph{same} direction in $P$.
            \end{proposition}
            \noindent
            Proposition \ref{pps:MCF-cost-repair-sub} implies that the minimum cost of repair-incorporated MinCF with respect to attacks on cost coefficients may be submodular, depending on the topological relationship of attackable arcs.
            This submodularity results from the joint submodularity of $\varphi(x,\tilde{z})$, which is intuitive since the unrepaired arc corresponding to $\tilde{z}_j$ implicitly has the same direction with the attacked arc corresponding to $x_j$.
            Figure \ref{fig:example-imply-cost-repair} provides the intuition of the joint submodularity.
            For an attackable arc j with cost coefficient $c_j+\delta_j(x_j+\tilde{z}_j-1)$, we assume that the arc is broken into two segments.
            Without loss of generality, we assume that the former segment has cost coefficient $c_j+\delta_j x_j$ and the latter segment has cost coefficient $-\delta_j+\delta_j \tilde{z}_j$.
            We observe that attacking or not repairing original arc $j$ is equivalent to attacking its former segment or not repairing its latter segment.
            From Figure \ref{fig:example-imply-cost-repair}, we illustrate that the former and latter segments of an original attackable arc always have the same direction.
            Therefore, we can use the conditions in Proposition \ref{pps:general-MCF-cost} to identify the joint submodularity of $\varphi(x,\tilde{z})$ and thus derive the condition in Proposition \ref{pps:MCF-cost-repair-sub}.

            \begin{figure}[!htbp]
                \begin{center}
                    \vspace{-0ex}
                    \includegraphics[width=0.7\columnwidth]{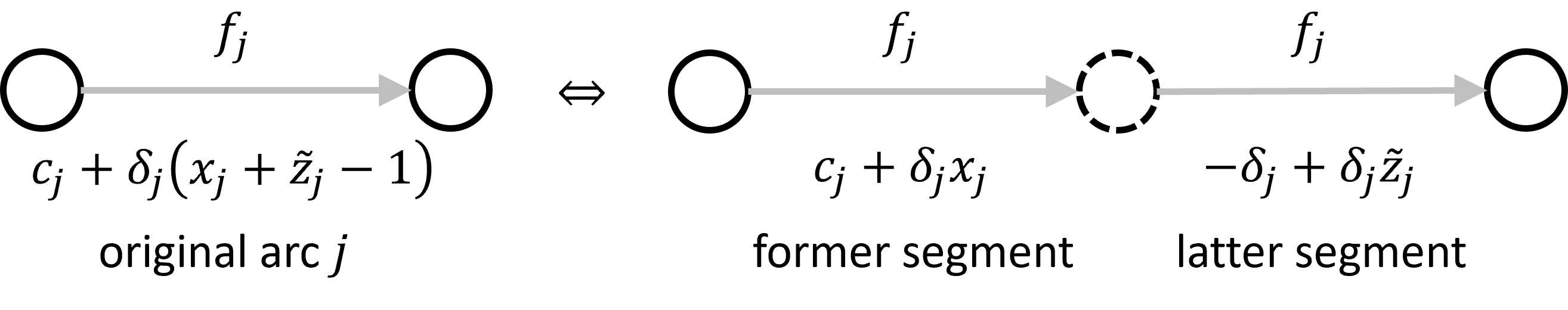}
                \end{center}
                \vspace{-2ex}
                \caption{Transformation from repairs to attacks.}
                \vspace{-2ex}
                \label{fig:example-imply-cost-repair}
            \end{figure}

            \begin{remark}
                We note that the intuition shown in Figure \ref{fig:example-imply-cost-repair} holds in general, and hence, Proposition \ref{pps:general-MCF-cost} can be readily extended to $\phi'(\cdot)$ with $\phi(\cdot)$ defined in \eqref{eq:MCF-capacity}, \eqref{eq:MF-capacity}, \eqref{eq:MCF-mixed}, or \eqref{eq:SP-distance}, with conditions corresponding to Proposition \ref{pps:general-MCF-capacity}, Proposition \ref{pps:general-MF-capacity}, Proposition \ref{pps:general-MCF-mixed}, or Proposition \ref{pps:general-SP-distance}, respectively.
            \end{remark}

        \subsection{Recover Supermodularity}
            We then analyze the supermodularity of $\phi'(x)$.
            Unfortunately, as shown in Figure \ref{fig:example-imply-cost-repair}, the unrepaired arc corresponding to $\tilde{z}_j$ implicitly has the same direction with the attacked arc corresponding to $x_j$, which violates the condition ``opposite directions'' for supermodularity.
            Moreover, even if $\varphi(x,\tilde{z})$ is jointly supermodular in $x$ and $\tilde{z}$ under certain conditions, the supermodularity of $\phi'(x)$ defined by~\eqref{eq:MCF-integer2} cannot be guaranteed.
            Therefore, the condition for supermodularity in Proposition \ref{pps:general-MCF-cost} does not apply to \eqref{eq:MCF-integer2}.
            Nevertheless, when we look back to the defender's model~\eqref{eq:MCF-integer}, we observe that $\phi$ may be supermodular in $x-z$ under the corresponding condition in Proposition \ref{pps:general-MCF-cost}.
            Hence, $\phi'(x)$ defined by \eqref{eq:MCF-integer} may exhibit approximate supermodularity.

            In light of the approximate supermodularity of $\phi'(x)$, we consider recovering supermodularity through a correction term.
            We first represent $\phi'(x)$ through the difference of two supermodular functions as
            \begin{equation}\label{eq:MCF-cost-repair-correct}
                \phi'(x)=\tilde{\phi}(x)-\rho\zeta(x),
            \end{equation}
            where $\rho>0$ and $\zeta(x)$ is strictly supermodular.
            Also, \eqref{eq:MCF-cost-repair-correct} can be rewritten as $\tilde{\phi}(x)=\phi'(x)+\rho\zeta(x)$.
            It means that the non-supermodular $\phi'(x)$ can be transformed into a supermodular $\tilde{\phi}(x)$ by introducing the correction term $\rho\zeta(x)$, as formalized in Proposition \ref{pps:MCF-cost-repair-super-corrected}.
            The proof is provided in Appendix \ref{apd:pps:MCF-cost-repair-super-corrected}.
            \begin{proposition}\label{pps:MCF-cost-repair-super-corrected}
                Suppose that $\phi(\cdot)$ is supermodular with respect to the set $\mathcal{A}_{a}\subseteq\mathcal{A}$ of attackable arcs.
                Then, there exist a $\rho\in\mathbb{R}_{+}$ such that $\tilde{\phi}(x):=\phi'(x)+\rho\zeta(x)$ is supermodular, where $\zeta(x)$ is a strictly supermodular function.
                The following is two examples for the selection of $\zeta(x)$ and $\rho$:\\
                \indent
                (i) If $\zeta(x)=(1^\top x)^2$, $\rho\geq\hat{\rho}/2$ guarantees the supermodularity of $\tilde{\phi}(x)$;\\
                \indent
                (ii) If $\zeta(x)=a^{|x|}$ with $a>1$, $\rho\geq\hat{\rho}/(a-1)^{2}$ guarantees the supermodularity of $\tilde{\phi}(x)$.\\
                Here, we denote
                \begin{equation}\label{eq:rho}
                    \hat{\rho}:=\sum_{k=1}^{|\mathcal{A}_a|-2}(\Lambda_{\sigma_k}-r_{\sigma_k})^+,
                \end{equation}
                where $\Lambda_j:=\phi(1)-\phi(1-e_j)$ indicates the maximal cost increment of an attack on arc $j$, and $\sigma$ is a permutation of $\{j\in[n_a]:a_j\in\mathcal{A}_a\}$ such that $\Lambda_{\sigma_1}-r_{\sigma_1}\geq \Lambda_{\sigma_2}-r_{\sigma_2}\geq\ldots\geq \Lambda_{\sigma_{n_a}}-r_{\sigma_{n_a}}$.
            \end{proposition}

            According to Proposition \ref{pps:MCF-cost-repair-super-corrected}, we can find a special case where $\phi'(x)$ directly exhibits supermodularity, as stated in Corollary \ref{crl:MCF-cost-repair-super}, and the proof is provided in Appendix \ref{apd:crl:MCF-cost-repair-super}.
            \begin{corollary}\label{crl:MCF-cost-repair-super}
                Suppose that $\phi(\cdot)$ is supermodular with respect to the set $\mathcal{A}_{a}\subseteq\mathcal{A}$ of attackable arcs.
                Then, the function $\phi'(x)$ defined by~\eqref{eq:MCF-integer} is \emph{supermodular} if $r_j\geq\Lambda_j$ holds for any arc $j$ in $\mathcal{A}_a$.
            \end{corollary}
            \noindent
            Corollary \ref{crl:MCF-cost-repair-super} implies that the minimum cost of repair-incorporated MinCF with respect to attacks on cost coefficients may be supermodular, depending on both the supermodularity of $\phi(\cdot)$ and the relationship between $\Lambda$ and repair cost $r$.
            The intuition of Corollary \ref{crl:MCF-cost-repair-super} is that when the repair costs of all attackable arcs are relatively high, no arc is repaired ($z=0$) under any possible attacks and $\phi'(x)=\phi(x)$ preserves the supermodularity.

        \subsection{Solution Algorithm}
            For network interdiction with mixed-integer defenders, the updated optimality condition is
            \begin{equation}\label{eq:optimality2}
                \begin{aligned}
                    \theta\leq\phi'(x),
                \end{aligned}
            \end{equation}
            with $\theta$ is the defender's objective.

            When $\phi(\cdot)$ exhibits submodularity or supermodularity, we consider two approaches to handle the optimality condition \eqref{eq:optimality2}.
            First, with the above preserved submodularity or recovered supermodularity, we can directly replace \eqref{eq:optimality2} by valid inequalities \eqref{eq:submodular-cut} or \eqref{eq:supermodular-cut}.
            To construct each valid inequality \eqref{eq:submodular-cut} or \eqref{eq:supermodular-cut}, we need to solve a series of MILPs.
            Second, we note that for each $\hat{z}\in\{0,1\}^{|X|}$, inequality
            \begin{equation}\label{eq:optimality3}
                \theta\leq r^{\top}\hat{z}+\phi(x-\hat{z})
            \end{equation}
            is valid for the optimality condition~\eqref{eq:optimality2}.
            We select the $\hat{z}=z^*$, where $z^*$ is optimal to the incumbent $x$, and hence, the valid inequality \eqref{eq:optimality3} becomes tight.
            We then replace $\phi(x-z^*)$ in the right-hand side of \eqref{eq:optimality3} by valid inequalities~\eqref{eq:submodular-cut} or \eqref{eq:supermodular-cut}.
            In this approach, because $\phi(\cdot)$ is an LP, to construct each valid inequality, we need to solve one MILP (to get $z^*$) and a series of LPs.
            Comparing the characteristics of the above two approaches, we adopt the latter one in this paper and design algorithms to solve the network interdiction problem.

    \section{Numerical Tests}\label{sec:numerical}
        In this section,we provide numerical experiment results to demonstrate the benefits of identifying and using submodularity or supermodularity in network interdiction.
        All experiments are implemented in Python 3.7 on a computer equipped with an Intel i7-10870H CPU and 16 GB of RAM.
        A one-hour time limit is applied.

        \subsection{MinCF Interdiction on Nodal Supplies/Demands}
            We first consider MinCF interdiction with attacks on nodal supplies/demands and assume that there exist repair decisions in the defender's problem.
            The attacker aims to maximize the defender's minimum cost, subject to a budget of attacks.
            Based on \eqref{eq:MCF-demand}, the considered MinCF interdiction is formulated as
            \begin{equation}\label{eq:case-MCF}
                \begin{aligned}
                    \max_{x\in\{0,1\}^{n_v}}&r^{\top}z+c^{\top}y\\
                    \text{s.t.}~&1^{\top}x\leq\Gamma\\
                    &(y,z)\in\arg\min_{y\in\mathbb{R}^{n_{a}},z\in\{0,1\}^{n_v}}r^{\top}z+\phi(x-z)\\
                    &~~~~~~~~~~~~~~~~~~~~~~~~~~~~~\text{s.t.}~Ty\geq d+\text{diag}(\delta)(x-z)\\
                    &~~~~~~~~~~~~~~~~~~~~~~~~~~~~~~~~~~0\leq y\leq f,
                \end{aligned}
            \end{equation}
            where $\Gamma$ is the budget of the number of attacks.
            Because the lower level involves binary variables, the interdiction problem \eqref{eq:case-MCF} is a bilevel mixed-integer linear program (MILP) and the duality method is not applicable.
            Therefore, algorithms for bilevel MILPs are required, and in this paper, we consider the state-of-the-art bilevel-MILP solver \texttt{MibS} as the benchmark method to solve \eqref{eq:case-MCF}.

            Fortunately, we observe that the lower level of \eqref{eq:case-MCF} aligns with \eqref{eq:MCF-integer}.
            Hence, following Proposition \ref{pps:general-MCF-demand}, we reformulate \eqref{eq:case-MCF} as
            \begin{equation}\label{eq:case-MCF2}
                \begin{aligned}
                    \max_{x\in\{0,1\}^{n_v},y\in\mathbb{R}^{n_{a}},z\in\{0,1\}^{n_v}}&r^{\top}z+c^{\top}y\\
                    \text{s.t.}~&1^{\top}x\leq\Gamma\\
                    &Ty\geq d+\text{diag}(\delta)(x-z)\\
                    &0\leq y\leq f\\
                    &r^{\top}z+c^{\top}y=\theta\\
                    &\theta\leq r^{\top}\hat{z}+\phi(x-\hat{z}),~\forall \hat{z}\in\{0,1\}^{n_v},
                \end{aligned}
            \end{equation}
            where $\theta$ is an auxiliary variable and $\phi(x-\hat{z})$ is supermodular in $x$.
            We note that given the incumbent $\hat{x}$ and the $\hat{z}$ optimal to $\hat{x}$, the last constraint can be replaced by inequalities \eqref{eq:submodular-cut}.
            Therefore, we can solve \eqref{eq:case-MCF2} by off-the-shelf solvers such as \texttt{Gurobi} and by adding inequalities \eqref{eq:submodular-cut} as lazy constraints.

            \subsubsection{Instance Settings}
                We consider the real-world networks in~\citet{TransportationNetworks}, where the number of nodes $n_v$ ranges from 24 (the Sioux Falls network) to 7,388 (the Austin network).
                The incidence matrix $T$ depends on the topology of the considered networks.
                For each network, we generate 10 instances that satisfy Assumption \ref{asp:recourse}.
                In each instance, we uniformly sample $c\in[5,10]^{n_a}$ and $d\in[-10,40]^{n_v}$, and scale $d^-$ such that $1^{\top}d^-=3\times1^{\top}d^+$;
                we set $f_j=1.5\max\{\bar{d},y^*_j\}$ for all $j\in[n_a]$, where $y^*$ is the optimal flow under sufficiently large $f$ and no attack, and $\bar{d}$ is the mean value of demands.
                We set $\Gamma=5$ and randomly select $N$ nodes to be attackable, where $N:=\max\{2\Gamma,\lceil n_v/100\rceil\}$.
                For any attackable node $i$, we set $\delta_i=|d_i|/2$, and $r_i=\max\{\phi(0)-\phi(-e_i),\Lambda_i\zeta_i\}$, where $\Lambda$ follows Proposition \ref{pps:MCF-cost-repair-super-corrected} and $\zeta_i$ is uniformly sampled from $[0,2]$.
                We require $r_i\geq\phi(0)-\phi(-e_i)$ to avoid trivial cases where certain nodes are reinforced under no attack.

            \subsubsection{Performance Comparison}
                We compare our approach with \texttt{MibS}~\citep[see][]{tahernejad2020branch}, a state-of-the-art solver for bilevel program.
                Figure \ref{fig:MCF-node-repair} reports the numerical performance of the two approaches.
                In Figure \ref{fig:MCF-node-repair}, the left axis represents the solution time (in logarithmic scale) and the right axis represents the optimality gap of the best solutions provided by MibS (in contrast, our approach proved global optimum in all tested instances).
                The curves represent the mean time or gap across the 10 instances for each network, while the shaded regions indicate the corresponding ranges of these instances.
                \begin{figure}[!htbp]
                    \begin{center}
                        \includegraphics[width=0.65\columnwidth]{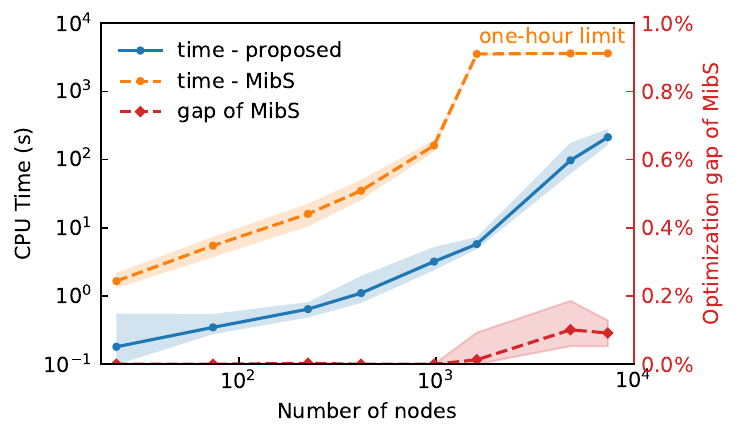}
                    \end{center}
                    \vspace{-4ex}
                    \caption{Comparison in MinCF interdiction with attacks on nodes.}
                    \label{fig:MCF-node-repair}
                \end{figure}

                From Figure \ref{fig:MCF-node-repair}, we observe that by exploiting supermodularity, the proposed approach achieves an order-of-magnitude speed-up as compared to \texttt{MibS}.
                More specifically, \texttt{MibS} reaches the time limit when the number of nodes reaches 1603 (the Terrassa network), while the proposed approach was able to obtain an optimal solution in less than 300 seconds even for the Austin network (with 7,388 nodes).
                This demonstrates the benefits of identifying and using supermodularity in solving MinCF interdiction problems.

        \subsection{MaxF Interdiction on Arc Capacities}
            We then consider MaxF interdiction with attacks on arc capacities over SPNs and assume that there exist repair decisions in the defender's problem.
            The attacker aims to minimize the defender's maximum flow, subject to a budget of attacks.
            Based on \eqref{eq:MF-capacity}, the considered MaxF interdiction is formulated as
            \begin{equation}\label{eq:case-MF}
                \begin{aligned}
                    \min_{x\in\{0,1\}^{n_a}}&\eta-r^{\top}z\\
                    \text{s.t.}~&1^{\top}x\leq\Gamma\\
                    &(y,\eta,z)\in\arg\max_{y\in\mathbb{R}^{n_{a}},\eta\in\mathbb{R},z\in\{0,1\}^{n_a}}~\eta-r^{\top}z\\
                    &~~~~~~~~~~~~~~~~~~~~~~~~~~~~~~~~~~~~~~\text{s.t.}~Ty=d\eta\\
                    &~~~~~~~~~~~~~~~~~~~~~~~~~~~~~~~~~~~~~~~~~~~0\leq y\leq f-\text{diag}(\delta)(x-z),
                \end{aligned}
            \end{equation}
            where $\Gamma$ is the budget of the number of attacks and $r$ is the discounted penalty coefficient associated with the repair cost.
            Because an SPN has a single source node and a single sink node, we replace $\tilde{T}$ in \eqref{eq:MF-capacity} by $d\in\mathbb{R}^{n_v}$ that indicates the locations of source and sink nodes.
            Correspondingly, $\eta\in\mathbb{R}$ represents the total flow from the source node to the sink node.
            Because the lower level involves binary variables, the interdiction problem \eqref{eq:case-MF} is a bilevel MILP and duality methods are not applicable.
            Similar to the MinCF interdiction case above, we consider the state-of-the-art bilevel-MILP solver \texttt{MibS} as the benchmark method to solve \eqref{eq:case-MF}.

            We observe that the lower level of \eqref{eq:case-MF} aligns with \eqref{eq:MCF-integer}.
            Hence, we reformulate \eqref{eq:case-MF} as
            \begin{equation}\label{eq:case-MF2}
                \begin{aligned}
                    \min_{x\in\{0,1\}^{n_a},y\in\mathbb{R}^{n_{a}},\eta\in\mathbb{R},z\in\{0,1\}^{n_a}}&\eta-r^{\top}z\\
                    \text{s.t.}~&1^{\top}x\leq\Gamma\\
                    &Ty=d\eta\\
                    &0\leq y\leq f-\text{diag}(\delta)(x-z)\\
                    &\eta-r^{\top}z=\theta\\
                    &\theta\geq -r^{\top}\hat{z}+\phi(x-\hat{z}),~\forall \hat{z}\in\{0,1\}^{n_v},
                \end{aligned}
            \end{equation}
            where $\theta$ is an auxiliary variable and $\phi(x-\hat{z})$ follows \eqref{eq:MF-capacity}.
            When the attackable arcs satisfy the first condition in Proposition \ref{pps:SPN-MF-arc}, $\phi(x-\hat{z})$ is submodular in $x$.
            Then, given the incumbent $\hat{x}$ and the $\hat{z}$ optimal to $\hat{x}$, the last constraint in \eqref{eq:case-MF2} can be replaced by inequalities \eqref{eq:submodular-cut}.
            Therefore, we can solve \eqref{eq:case-MF2} by off-the-shelf solvers such as \texttt{Gurobi} and by adding inequalities \eqref{eq:submodular-cut} as lazy constraints.

            \subsubsection{Instance Settings}
                We follow Definition \ref{def:SPN} to randomly generate SPNs with the number of arcs $n_a$ ranging from 100 to 10,000.
                For each $n_a$, we generate 10 SPN instances that satisfy Assumption \ref{asp:recourse}.
                The incidence matrix $T$ and vector $d$ depend on the topology of the generated SPNs.
                In each instance, we uniformly sample $f\in[10,50]^{n_a}$ and set $\Gamma=5$.
                We randomly select a cut in the network and further randomly select $N$ arcs from this cut to be attackable, where $N:=\max\{2\Gamma,\lceil n_a/100\rceil\}$.
                Because all attackable arcs are in a cut, they must satisfy the first condition in Proposition \ref{pps:SPN-MF-arc}.
                For any attackable arc $j$, we set $\delta_j=0.8f_j$ and $r_i=\max\{\phi(-e_j)-\phi(0),\Lambda_j\zeta_j\}$, where $\Lambda_j:=\phi(1-e_j)-\phi(1)$ and $\zeta_i$ is uniformly sampled from $[0,2]$.
                We require $r_j\geq\phi(-e_j)-\phi(0)$ to avoid trivial cases where certain arcs are reinforced under no attack.

            \subsubsection{Performance Comparison}
                We still compare our approach with \texttt{MibS}~\citep{tahernejad2020branch}.
                Figure \ref{fig:MF-arc-repair} reports the numerical performance of the two approaches.
                In Figure \ref{fig:MF-arc-repair}, the left axis represents the solution time (in logarithmic scale) and the right axis represents the optimality gap of the best solutions provided by MibS (in contrast, our approach proved global optimum in all tested instances).
                The curves represent the mean time or gap across the 10 instances for each network, while the shaded regions indicate the corresponding ranges of these instances.
                \begin{figure}[!htbp]
                    \begin{center}
                        \includegraphics[width=0.65\columnwidth]{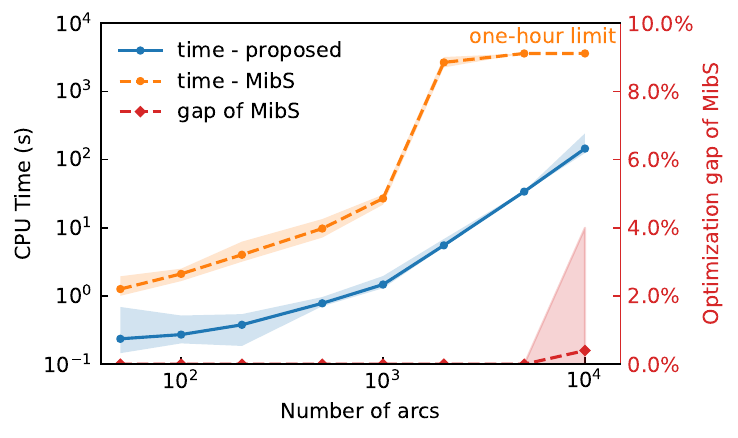}
                    \end{center}
                    \vspace{-4ex}
                    \caption{Comparison in MaxF interdiction with attacks on arcs.}
                    \label{fig:MF-arc-repair}
                \end{figure}

                From Figure \ref{fig:MF-arc-repair}, we observe that by exploiting submodularity, the proposed approach achieves an order-of-magnitude speed-up as compared to \texttt{MibS}.
                More specifically, \texttt{MibS} reaches the time limit when the number of arcs reaches 5000, while the proposed approach was able to obtain an optimal solution in less than 250 seconds even for the network with 10,000 nodes.
                This demonstrates the benefits of identifying and using submodularity in solving MaxF interdiction problems.

    \section{Conclusions}\label{sec:conclusions}
        This paper presents a systematic study of conditions for submodularity and supermodularity in network interdiction.
        The proposed necessary and sufficient conditions hold under general network topologies and parameter settings, and depend solely on the locations of the attacks.
        We also establish sufficient and less restrictive conditions by incorporating more network information, such as detailed parameters and special topologies.
        We further extend to mixed-integer defenders and identify conditions that preserve submodularity or recover supermodularity.
        Numerical experiments on both real-world networks (up to 7,388 nodes) and randomly generated networks (up to 10,000 arcs) show an order-of-magnitude speedup achieved by exploiting these properties to generate valid inequalities, demonstrating the benefits of identifying and using sub/supermodularity in solving network interdiction problems.

    \bibliographystyle{unsrtnat}
	\bibliography{reference}

    \clearpage
    \begin{appendices}
        \section{Proofs}
            \subsection{Proof of Proposition \ref{pps:general-MCF-demand}}\label{apd:pps:general-MCF-demand}
                \begin{proof}\color{black}
                    We rewrite \eqref{eq:MCF-demand} as
                    \begin{equation}\label{eq:MCF-demand2}
                        \begin{aligned}
                            \phi(x)=\min_{y\in\mathbb{R}^{n_{a}}}~&c^{\top}y\\
                            \text{s.t.}~&Ax+By\leq h,
                        \end{aligned}
                    \end{equation}
                    with
                    \begin{equation*}
                        A=\left[\begin{matrix}
                            \delta'\\
                            ~\\
                            ~
                        \end{matrix}\right]\in\mathbb{R}^{m\times n_{v}},
                        B=\left[\begin{matrix}
                            -T\\
                            -I\\
                            I
                        \end{matrix}\right]\in\mathbb{R}^{m\times n_{a}},
                        h=\left[\begin{matrix}
                            -d\\
                            ~\\
                            f
                        \end{matrix}\right]\in\mathbb{R}^{m},
                    \end{equation*}
                    where $m=n_{v}+2n_{a}$, $\delta':=\text{diag}(\delta)$, and $I$ is the identity matrix of appropriate dimension.
                    Then, we have $\text{rank}(B)=n_{a}<m$.
                    Hence, we consider the second condition of Theorem \ref{trm:condition-generalLP-sub} in this proof.
                    Specifically, we consider a set $\mathcal{I}\subseteq[m]$ with $|\mathcal{I}|=\text{rank}(B)+1=n_{a}+1$, and denote $A_{\mathcal{I}}$ and $B_{\mathcal{I}}$ as
                    \begin{equation*}
                        A_{\mathcal{I}}=\left[\begin{matrix}
                            \delta'_{\mathcal{I}_{1}}\\
                            ~\\
                            ~
                        \end{matrix}\right]\in\mathbb{R}^{(n_{a}+1)\times n_{v}},
                        B_{\mathcal{I}}=\left[\begin{matrix}
                            -T_{\mathcal{I}_{1}}\\
                            -I_{\mathcal{I}_{2}}\\
                            I_{\mathcal{I}_{3}}
                        \end{matrix}\right]\in\mathbb{R}^{(n_{a}+1)\times n_{a}},
                    \end{equation*}
                    where $\mathcal{I}_{1}\subseteq[n_{v}]$, $\mathcal{I}_{2}\subseteq[n_{a}]$, and $\mathcal{I}_{3}\subseteq[n_{a}]$ are corresponding to the subsets of $\mathcal{I}$.

                    Following the second condition of Theorem \ref{trm:condition-generalLP-sub}, we suppose $\text{rank}(B_{\mathcal{I}})=\text{rank}(B)=n_{a}$.
                    It indicates that $B_{\mathcal{I}}$ is of full column rank, and hence, $B_{\mathcal{I}}y=0$ with variable $y\in\mathbb{R}^{n_{a}}$ has a unique solution $y=0$.
                    Specifically, the linear system
                    \begin{equation}\label{eq:MCF-demand-I1}
                        \left\{\begin{aligned}
                            &T_{s}y=0,\forall s\in\mathcal{S}(\mathcal{I})\\
                            &y_{j}=0, \forall j\in\mathcal{J}(\mathcal{I})
                        \end{aligned}\right.
                    \end{equation}
                    has a unique solution $y=0$, where $\mathcal{S}(\mathcal{I})=\mathcal{I}_{1}$ and $\mathcal{J}(\mathcal{I})=\mathcal{I}_{2}\cup\mathcal{I}_{3}$.
                    Then, we introduce an auxiliary variable $p\in\mathbb{R}^{n_{v}}$ and rewrite \eqref{eq:MCF-demand-I1} as
                    \begin{equation}\label{eq:MCF-demand-I2}
                        \left\{\begin{aligned}
                            &Ty=-p\\
                            &p_{s}=0,\forall s\in\mathcal{S}(\mathcal{I})\\
                            &y_{j}=0, \forall j\in\mathcal{J}(\mathcal{I})
                        \end{aligned}\right..
                    \end{equation}
                    Obviously, \eqref{eq:MCF-demand-I1} has a unique solution $y=0$ if and only if \eqref{eq:MCF-demand-I2} has a unique solution $y=0,p=0$.
                    From the perspective of network flow, $y$ means arc flow and $p$ indicates external injection.
                    $Ty=-p$ restricts the nodal flow conservation under external injection $p$.
                    For any $s\in\mathcal{S}(\mathcal{I})$, $p_{s}=0$ means no external injection into node $s$.
                    For any $s\in[n_{v}]\backslash\mathcal{S}(\mathcal{I})$, the external injection into node $s$ is unlimited and is thus free.
                    Such nodes with free external injection are called slack nodes and we denote the set of slack nodes as $\mathcal{V}_{\text{slack}}(\mathcal{I})$.
                    For any $j\in\mathcal{J}(\mathcal{I})$, $y_{j}=0$ means that arc $j$ does not admit flow, and hence, we can remove arcs $j,\forall j\in\mathcal{J}(\mathcal{I})$ from the network.
                    We denote the set of remaining arcs as $\mathcal{A}'(\mathcal{I})$ and obtain a new directed network $\mathcal{G}'(\mathcal{I})=\big(\mathcal{V},\mathcal{A}'(\mathcal{I})\big)$.
                    With the above understanding in mind, a unique solution $y=0,p=0$ of \eqref{eq:MCF-demand-I2} indicates:
                    under nodal flow conservation limits, whenever there is no external injection into nodes in $\mathcal{V}\backslash\mathcal{V}_{\text{slack}}(\mathcal{I})$, there exists no flow in $\mathcal{G}'(\mathcal{I})$.
                    Therefore, according to Lemma \ref{lem:pps:general-MCF-demand}, each component of $\mathcal{G}'(\mathcal{I})$ is acyclic and contains at most one slack node in $\mathcal{V}_{\text{slack}}(\mathcal{I})$.

                    Following the second condition of Theorem \ref{trm:condition-generalLP-sub}, we suppose an $x\in\mathbb{R}^{n_{v}}_{+}$ such that $A_{\mathcal{I}}x\in\text{span}(B_{\mathcal{I}})$.
                    It implies that there exists a $y\in\mathbb{R}^{n_{a}}$ such that $A_{\mathcal{I}}x=B_{\mathcal{I}}y$.
                    Specifically, the linear system
                    \begin{equation}\label{eq:MCF-demand-x}
                        \left\{\begin{aligned}
                            &Ty=-p\\
                            &p_{s}=\delta_{s}x_{s},\forall s\in\mathcal{S}(\mathcal{I})\\
                            &y_{j}=0, \forall j\in\mathcal{J}(\mathcal{I})
                        \end{aligned}\right.
                    \end{equation}
                    is feasible.
                    According to the second condition of Theorem \ref{trm:condition-generalLP-sub}, to prove the supermodularity of \eqref{eq:MCF-demand2}, we need to prove $A_{\mathcal{I},k}x_{k}\in\text{span}(B_{\mathcal{I}})$ for any $k\in[n_{v}]$.
                    It implies that for any $k\in[n_{v}]$, there exists a $y'\in\mathbb{R}^{n_{a}}$ such that $A_{\mathcal{I},k}x_{k}=B_{\mathcal{I}}y'$.
                    If $k\notin\mathcal{S}(\mathcal{I})$, according to the structure of $A_{\mathcal{I}}$, we have $A_{\mathcal{I},k}=0$, and hence, there exists a $y'=0$.
                    If $k\in\mathcal{S}(\mathcal{I})$, $A_{\mathcal{I},k}x_{k}=B_{\mathcal{I}}y'$ implies that the linear system
                    \begin{equation}\label{eq:MCF-demand-y}
                        \left\{\begin{aligned}
                            &Ty'=-p'\\
                            &p'_{k}=\delta_{k}x_{k}\\
                            &p'_{s}=0,\forall s\in\mathcal{S}(\mathcal{I})\backslash\{k\}\\
                            &y'_{j}=0, \forall j\in\mathcal{J}(\mathcal{I})
                        \end{aligned}\right.
                    \end{equation}
                    is feasible.
                    Comparing \eqref{eq:MCF-demand-x} and \eqref{eq:MCF-demand-y} to \eqref{eq:MCF-demand-I2} from the perspective of network flow, we find:
                    \eqref{eq:MCF-demand-x} modifies the injection at node $s$ from 0 to $\delta_{s}x_{s}$ for all $s\in\mathcal{S}(\mathcal{I})$;
                    \eqref{eq:MCF-demand-y} modifies the injection at the only node $k$ from 0 to $\delta_{k}x_{k}$.
                    Either \eqref{eq:MCF-demand-x} or \eqref{eq:MCF-demand-y} is feasible if and only if under such modifications, the nodal conservation still holds for each component in $\mathcal{G}'(\mathcal{I})$.
                    Specifically, for a certain component $g$ in $\mathcal{G}'(\mathcal{I})$, we denote its node index set by $\mathcal{S}_{g}\subseteq[n_{v}]$ and its arc index set by $\mathcal{J}_{g}\subseteq[n_{a}]\backslash\mathcal{J}(\mathcal{I})$.
                    Considering that the component is acyclic, its nodal conservation holds when
                    \begin{equation}\label{eq:MCF-demand-nodal}
                        \sum_{s\in\mathcal{S}_{g}}p_{s}=0.
                    \end{equation}
                    Given \eqref{eq:MCF-demand-x}, \eqref{eq:MCF-demand-nodal} reduces to
                    \begin{equation}\label{eq:MCF-demand-nodal-x}
                        \sum_{s\in\mathcal{S}_{g}\cap\mathcal{S}(\mathcal{I})}\delta_{s}x_{s}+\sum_{s\in\mathcal{S}_{g}\backslash\mathcal{S}(\mathcal{I})}p_{s}=0.
                    \end{equation}
                    Given \eqref{eq:MCF-demand-y}, \eqref{eq:MCF-demand-nodal} reduces to
                    \begin{equation}\label{eq:MCF-demand-nodal-y}
                        \sum_{s\in\mathcal{S}_{g}\cap\{k\}}\delta_{s}x_{s}+\sum_{s\in\mathcal{S}_{g}\backslash\mathcal{S}(\mathcal{I})}p'_{s}=0.
                    \end{equation}

                    From the above analysis, to prove the submodularity of \eqref{eq:MCF-demand2}, we need to prove that \eqref{eq:MCF-demand-nodal-y} holds for any $k\in\mathcal{S}(\mathcal{I})$ whenever \eqref{eq:MCF-demand-nodal-x} holds.
                    When $k\in\mathcal{S}(\mathcal{I})\backslash\mathcal{S}_{g}$, \eqref{eq:MCF-demand-nodal-y} always holds due to free $p'_{s},\forall s\in\mathcal{S}_{g}\backslash\mathcal{S}(\mathcal{I})$.
                    Hence, we only consider $k\in\mathcal{S}(\mathcal{I})\cap\mathcal{S}_{g}$ in the following, and \eqref{eq:MCF-demand-nodal-y} reduces to
                    \begin{equation}\label{eq:MCF-demand-nodal-yy}
                        \delta_{k}x_{k}+\sum_{s\in\mathcal{S}_{g}\backslash\mathcal{S}(\mathcal{I})}p'_{s}=0.
                    \end{equation}
                    Considering that the component contains at most one slack node, we consider two scenarios.
                    First, if component $g$ contains one slack node, indexed by $s_{0}$, we have $\mathcal{S}_{g}\backslash\mathcal{S}(\mathcal{I})=\{s_{0}\}$.
                    Then, \eqref{eq:MCF-demand-nodal-yy} reduces to
                    \begin{equation*}
                        \delta_{k}x_{k}+p_{s_{0}}=0,
                    \end{equation*}
                    which always holds due to free $p_{s_{0}}$.
                    Second, if component $g$ contains no slack node, we have $\mathcal{S}_{g}\backslash\mathcal{S}(\mathcal{I})=\emptyset$.
                    Then, \eqref{eq:MCF-demand-nodal-x} reduces to
                    \begin{equation*}
                        \sum_{s\in\mathcal{S}_{g}\cap\mathcal{S}(\mathcal{I})}\delta_{s}x_{s}=0.
                    \end{equation*}
                    Because $\delta\geq0,x\geq0$, we have $\delta_{s}x_{s}=0$ for all $s\in\mathcal{S}_{g}\cap\mathcal{S}(\mathcal{I})$ .
                    Hence, we have $\delta_{k}x_{k}=0$ and \eqref{eq:MCF-demand-nodal-yy} holds.
                    We finish this part of proof.

                    Through the above analysis and by the second condition of Theorem \ref{trm:condition-generalLP-sub}, we prove that $\phi(x)$ defined in \eqref{eq:MCF-demand2} is supermodular for any $c$ and $h$.
                    Therefore, $\phi(x)$ defined in \eqref{eq:MCF-demand} is supermodular for any $c$, $f$, and $d$.
                \end{proof}

                \begin{lemma}\label{lem:pps:general-MCF-demand}
                    Given a directed network $\mathcal{G}=(\mathcal{V},\mathcal{A})$ and a node set $\mathcal{V}_{\text{slack}}\subseteq\mathcal{V}$.
                    Under nodal flow conservation limits, whenever there is no external injection into nodes in $\mathcal{V}\backslash\mathcal{V}_{\text{slack}}$, there exists no flow in $\mathcal{G}$,
                    if and only if
                    each component of $\mathcal{G}$ is acyclic and contains at most one node in $\mathcal{V}_{\text{slack}}$.
                \end{lemma}
                \begin{proof}
                    ``$\Leftarrow$'':
                    Because different components of $\mathcal{G}$ are independent, we consider a certain component and discuss two cases.\\
                    First, if the component does not include any node in $\mathcal{V}_{\text{slack}}$, when there is no external injection into nodes in $\mathcal{V}\backslash\mathcal{V}_{\text{slack}}$, there is no external injection into the component.
                    Because the component is acyclic and nodal flow conservation holds, there is no flow in the component.\\
                    Second, if the component includes one slack node in $\mathcal{V}_{\text{slack}}$, when there is no external injection into nodes in $\mathcal{V}\backslash\mathcal{V}_{\text{slack}}$, because nodal flow conservation holds, the external injection into the only slack node must be zero.
                    Hence, there is no external injection into the component, and the case reduces to the former one.

                    ``$\Rightarrow$'': Suppose the contrary that in a certain component of $\mathcal{G}$, either there exists a cycle, or there exists at least two nodes in $\mathcal{V}_{\text{slack}}$.\\
                    For the former case, there may exist cyclic flow that is independent of external injection.\\
                    For the latter case, because these nodes are in the same component, there must exist a path from one slack node to another slack node, and hence, there may exist flow on the path.\\
                    Therefore, both cases contradict that no flow exists in $\mathcal{G}$.
                \end{proof}

            \subsection{Proof of Proposition \ref{pps:general-MCF-capacity}}\label{apd:pps:general-MCF-capacity}
                \begin{proof}\color{black}
                    We rewrite \eqref{eq:MCF-capacity} as
                    \begin{equation}\label{eq:MCF-capacity2}
                        \begin{aligned}
                            \phi(x)=\min_{y\in\mathbb{R}^{n_{a}}}~&c^{\top}y\\
                            \text{s.t.}~&Ax+By\leq h
                        \end{aligned}
                    \end{equation}
                    with
                    \begin{equation*}
                        A=\left[\begin{matrix}
                            ~\\
                            ~\\
                            \delta'
                        \end{matrix}\right]\in\mathbb{R}^{m\times n_{a}},
                        B=\left[\begin{matrix}
                            -T\\
                            -I\\
                            I
                        \end{matrix}\right]\in\mathbb{R}^{m\times n_{a}},
                        h=\left[\begin{matrix}
                            -d\\
                            ~\\
                            f
                        \end{matrix}\right]\in\mathbb{R}^{m},
                    \end{equation*}
                    where $m=n_{v}+2n_{a}$, $\delta':=\text{diag}(\delta)$, and $I$ is an identity matrix of appropriate dimension.
                    Then, we have $\text{rank}(B)=n_{a}<m$.
                    Hence, we consider the second condition of Theorem \ref{trm:condition-generalLP-sub} or Theorem \ref{trm:condition-generalLP-super} in this proof.
                    Specifically, we consider a set $\mathcal{I}\subseteq[m]$ with $|\mathcal{I}|=\text{rank}(B)+1=n_{a}+1$, and denote $A_{\mathcal{I}}$ and $B_{\mathcal{I}}$ as
                    \begin{equation*}
                        A_{\mathcal{I}}=\left[\begin{matrix}
                            ~\\
                            ~\\
                            \delta'_{\mathcal{I}_{3}}
                        \end{matrix}\right]\in\mathbb{R}^{(n_{a}+1)\times n_{a}},
                        B_{\mathcal{I}}=\left[\begin{matrix}
                            -T_{\mathcal{I}_{1}}\\
                            -I_{\mathcal{I}_{2}}\\
                            I_{\mathcal{I}_{3}}
                        \end{matrix}\right]\in\mathbb{R}^{(n_{a}+1)\times n_{a}},
                    \end{equation*}
                    where $\mathcal{I}_{1}\subseteq[n_{v}]$, $\mathcal{I}_{2}\subseteq[n_{a}]$, $\mathcal{I}_{3}\subseteq[n_{a}]$ are corresponding to the subsets of $\mathcal{I}$.

                    Following the second condition of Theorem \ref{trm:condition-generalLP-sub} or Theorem \ref{trm:condition-generalLP-super}, we suppose $\text{rank}(B_{\mathcal{I}})=\text{rank}(B)=n_{a}$.
                    It indicates that $B_{\mathcal{I}}$ is of full column rank, and hence, $B_{\mathcal{I}}y=0$ with variable $y\in\mathbb{R}^{n_{a}}$ has a unique solution $y=0$.
                    Specifically, the linear system
                    \begin{equation}\label{eq:MCF-capacity-I1}
                        \left\{\begin{aligned}
                            &T_{s}y=0, \forall s\in\mathcal{S}(\mathcal{I})\\
                            &y_{j}=0,\forall j\in\mathcal{J}_{1}(\mathcal{I})\cup\mathcal{J}_{2}(\mathcal{I})
                        \end{aligned}\right.
                    \end{equation}
                    has a unique solution $y=0$, where $\mathcal{S}(\mathcal{I})=\mathcal{I}_{1}$, $\mathcal{J}_{1}(\mathcal{I})=\mathcal{I}_{2}$, and $\mathcal{J}_{2}(\mathcal{I})=\mathcal{I}_{3}$.
                    Then, we introduce a slack variable $p\in\mathbb{R}^{n_{v}}$ and reformulate \eqref{eq:MCF-capacity-I1} as
                    \begin{equation}\label{eq:MCF-capacity-I2}
                        \left\{\begin{aligned}
                            &Ty=-p\\
                            &p_{s}=0, \forall s\in\mathcal{S}(\mathcal{I})\\
                            &y_{j}=0,\forall j\in\mathcal{J}(\mathcal{I})
                        \end{aligned}\right.
                    \end{equation}
                    where $\mathcal{J}(\mathcal{I})=\mathcal{J}_{1}(\mathcal{I})\cup\mathcal{J}_{2}(\mathcal{I})$.
                    Obviously, \eqref{eq:MCF-capacity-I1} has a unique solution $y=0$ if and only if \eqref{eq:MCF-capacity-I2} has a unique solution $y=0,p=0$.
                    From the perspective of network flow, $y$ means arc flow and $p$ indicates external injection.
                    $Ty=-p$ restricts the nodal flow conservation under external flow injection $p$.
                    For $s\in\mathcal{S}(\mathcal{I})$, $p_{s}=0$ means no external injection into node $s$.
                    For $s\notin\mathcal{S}(\mathcal{I})$, the external injection is unlimited and is thus free.
                    Such nodes with free external injection are called slack nodes and we denote the set of slack nodes as $\mathcal{V}_{\text{slack}}(\mathcal{I})$.
                    For $j\in\mathcal{J}(\mathcal{I})$, $y_{j}=0$ admits no flow on arc $j$, and hence, we can remove arcs $j,\forall j\in\mathcal{J}(\mathcal{I})$ from the network.
                    We denote the set of remaining arcs as $\mathcal{A}'(\mathcal{I})$ and obtain a new directed graph $\mathcal{G}'(\mathcal{I})=\big(\mathcal{V},\mathcal{A}'(\mathcal{I})\big)$.
                    With the above understanding in mind, a unique solution $y=0,p=0$ of \eqref{eq:MCF-capacity-I2} indicates:
                    under nodal flow conservation limits, whenever there is no external injection into nodes in $\mathcal{V}\backslash\mathcal{V}_{\text{slack}}(\mathcal{I})$, there exists no flow in $\mathcal{G}'(\mathcal{I})$.
                    Therefore, according to Lemma \ref{lem:pps:general-MCF-demand}, each component of $\mathcal{G}'(\mathcal{I})$ is acyclic and contains at most one slack node in $\mathcal{V}_{\text{slack}}(\mathcal{I})$.

                    In the following, we derive the necessary and sufficient conditions for the submodularity and supermodularity of $\phi(x)$ defined by \eqref{eq:MCF-capacity}, respectively.

                    \vspace{2ex}
                    \noindent\textbf{(i) Supermodularity.}

                    Following the second condition of Theorem \ref{trm:condition-generalLP-sub}, we suppose an $x\in\mathbb{R}^{n_{a}}_{+}$ such that $A_{\mathcal{I}}x\in\text{span}(B_{\mathcal{I}})$.
                    It implies that there exist a $y\in\mathbb{R}^{n_a}$ such that $A_{\mathcal{I}}x=B_{\mathcal{I}}y$.
                    Specifically, the linear system
                    \begin{equation}\label{eq:MCF-capacity-sub-x0}
                        \left\{\begin{aligned}
                            &Ty=-p\\
                            &p_{s}=0, \forall s\in\mathcal{S}(\mathcal{I})\\
                            &y_{j}=0,\forall j\in\mathcal{J}_{1}(\mathcal{I})\\
                            &y_{j}=\delta_{j}x_{j},\forall j\in\mathcal{J}_{d}(\mathcal{I})
                        \end{aligned}\right.
                    \end{equation}
                    is feasible, where $\mathcal{J}_{d}(\mathcal{I}):=\mathcal{J}_{2}(\mathcal{I})\backslash\mathcal{J}_{1}(\mathcal{I})$.
                    According to the second condition of Theorem \ref{trm:condition-generalLP-sub}, to guarantee the supermodularity of \eqref{eq:MCF-capacity2}, we need to guarantee $A_{\mathcal{I},k}x_{k}\in\text{span}(B_{\mathcal{I}})$ for any $k\in[n_{a}]$.
                    It implies that for any $k\in[n_a]$, there exists $y'\in\mathbb{R}^{n_{a}}$ such that $A_{\mathcal{I},k}x_{k}=B_{\mathcal{I}}y'$.
                    Specifically, the linear system
                    \begin{equation}\label{eq:MCF-capacity-sub-y0}
                        \left\{\begin{aligned}
                            &Ty'=-p'\\
                            &p'_{s}=0, \forall s\in\mathcal{S}(\mathcal{I})\\
                            &y'_{j}=0,\forall j\in\mathcal{J}_{1}(\mathcal{I})\\
                            &y'_{k}=\delta_{k}x_{k},\text{if}~k\in\mathcal{J}_{d}(\mathcal{I})\\
                            &y'_{j}=0,\forall j\in\mathcal{J}_{d}(\mathcal{I})\backslash\{k\}
                        \end{aligned}\right.
                    \end{equation}
                    is feasible.
                    If $k\notin\mathcal{J}_{d}(\mathcal{I})$, $y'=0,p'=0$ is always feasible.
                    If $k\notin\mathcal{J}_{a}$, where $\mathcal{J}_{a}:=\{j\in[n_{a}]:\delta_{j}>0\}$ is corresponding to the set $\mathcal{A}_{a}$ of attackable arcs, we have $\delta_{k}=0$ and thus $y'=0,p'=0$ is feasible.
                    Hence, we only consider $k\in\mathcal{J}_{d}(\mathcal{I})\cap\mathcal{J}_{a}$ in the following.
                    From the perspective of network flow, \eqref{eq:MCF-capacity-sub-x0} and \eqref{eq:MCF-capacity-sub-y0} do not align with the topology of $\mathcal{G}'(\mathcal{I})$ due to $y_{j}=\delta_{j}x_{j}$ for all $ j\in\mathcal{J}_{2}(\mathcal{I})$ or $y'_{k}=\delta_{k}x_{k}$.
                    To remain the topology of $\mathcal{G}'(\mathcal{I})$, for any $ j\in\mathcal{J}_{d}(\mathcal{I})$, we replace flow on arc $j$ with a pair of dummy external injection at its head and tail nodes.
                    We note the dummy injection does not affect slack nodes, into which the total external injection remains free, and hence, the set $\mathcal{V}_{\text{slack}}(\mathcal{I})$ remains unchanged.
                    Then, we reformulate \eqref{eq:MCF-capacity-sub-x0} and \eqref{eq:MCF-capacity-sub-y0} as
                    \begin{equation}\label{eq:MCF-capacity-sub-x}
                        \left\{\begin{aligned}
                            &Ty=-p\\
                            &p_{s}=\sum_{j\in\mathcal{J}_{d}(\mathcal{\mathcal{I}})\cap\mathcal{J}_{a}}t_{sj}\delta_{j}x_{j}, \forall s\in\mathcal{S}(\mathcal{I})\\
                            &y_{j}=0,\forall j\in\mathcal{J}(\mathcal{I}),
                        \end{aligned}\right.
                    \end{equation}
                    \begin{equation}\label{eq:MCF-capacity-sub-y}
                        \left\{\begin{aligned}
                            &Ty'=-p'\\
                            &p'_{s}=t_{sk}\delta_{k}x_{k}, \forall s\in\mathcal{S}(\mathcal{I})\\
                            &y'_{j}=0,\forall j\in\mathcal{J}(\mathcal{I}),
                        \end{aligned}\right.
                    \end{equation}
                    where $t_{sj}\in\{-1,0,1\}$ is the element of $T$.
                    Comparing \eqref{eq:MCF-capacity-sub-x} and \eqref{eq:MCF-capacity-sub-y} to \eqref{eq:MCF-capacity-I2}, we find:
                    \eqref{eq:MCF-capacity-sub-x} modifies the external injection at node $s$ from 0 to $\sum_{j\in\mathcal{J}_{d}(\mathcal{\mathcal{I}})\cap\mathcal{J}_a}t_{sj}\delta_{j}x_{j}$ for all $s\in\mathcal{S}(\mathcal{I})$;
                    \eqref{eq:MCF-capacity-sub-y} modifies the external injection at node $s$ from 0 to $t_{sk}\delta_{k}x_{k}$ for all $s\in\mathcal{S}(\mathcal{I})$.
                    Either \eqref{eq:MCF-capacity-sub-x} or \eqref{eq:MCF-capacity-sub-y} is feasible, if and only if under such modifications, their nodal conservation still holds for each component in $\mathcal{G}'(\mathcal{I})$.
                    Specifically, for a certain component $g$ in $\mathcal{G}'(\mathcal{I})$, we denote its node index set by $\mathcal{S}_{g}\subseteq[n_{v}]$ and its arc index set by $\mathcal{J}_{g}\subseteq[n_{a}]\backslash\mathcal{J}(\mathcal{I})$.
                    Considering that the component is acyclic, the nodal conservation holds when
                    \begin{equation}\label{eq:MCF-capacity-sub-nodal}
                        \sum_{s\in\mathcal{S}_{g}}p_{s}=0.
                    \end{equation}
                    Given \eqref{eq:MCF-capacity-sub-x}, \eqref{eq:MCF-capacity-sub-nodal} reduces to
                    \begin{equation}\label{eq:MCF-capacity-sub-nodal-x}
                        \sum_{s\in\mathcal{S}_{g}\cap\mathcal{S}(\mathcal{I})}\sum_{j\in\mathcal{J}_{d}(\mathcal{\mathcal{I}})\cap\mathcal{J}_{a}}t_{sj}\delta_{j}x_{j}+
                        \sum_{s\in\mathcal{S}_{g}\backslash\mathcal{S}(\mathcal{I})}p_{s}=0.
                    \end{equation}
                    Given \eqref{eq:MCF-capacity-sub-y}, \eqref{eq:MCF-capacity-sub-nodal} reduces to
                    \begin{equation}\label{eq:MCF-capacity-sub-nodal-y}
                        \sum_{s\in\mathcal{S}_{g}\cap\mathcal{S}(\mathcal{I})}t_{sk}\delta_{k}x_{k}+
                        \sum_{s\in\mathcal{S}_{g}\backslash\mathcal{S}(\mathcal{I})}p'_{s}=0.
                    \end{equation}

                    From the above analysis, to guarantee the supermodularity of \eqref{eq:MCF-capacity2}, we need to guarantee that \eqref{eq:MCF-capacity-sub-nodal-y} holds for any $k\in\mathcal{J}_{d}(\mathcal{I})\cap\mathcal{J}_{a}$ whenever \eqref{eq:MCF-capacity-sub-nodal-x} holds.
                    We note each component in $\mathcal{G}'(\mathcal{I})$ contains at most one slack node.
                    If component $g$ contains one slack node, indexed by $s_{0}$, we have $\mathcal{S}_{g}\backslash\mathcal{S}(\mathcal{I})=\{s_{0}\}$, and hence, \eqref{eq:MCF-capacity-sub-nodal-y} always holds due to free $p'_{s_{0}}$.
                    Hence, we focus on the case that the component contains no slack node, i.e., $\mathcal{S}_{g}\subseteq\mathcal{S}(\mathcal{I})$.
                    Then, \eqref{eq:MCF-capacity-sub-nodal-x} and \eqref{eq:MCF-capacity-sub-nodal-y} reduce to
                    \begin{equation}\label{eq:MCF-capacity-sub-nodal-xx}
                        \sum_{j\in\mathcal{J}_{d}(\mathcal{\mathcal{I}})\cap\mathcal{J}_{a}\cap\mathcal{J}'_{g}}t'_{j}(\mathcal{S}_{g})\delta_{j}x_{j}=0  %
                    \end{equation}
                    \begin{equation}\label{eq:MCF-capacity-sub-nodal-yy}
                        t'_{k}(\mathcal{S}_{g})\delta_{k}x_{k}=0
                    \end{equation}
                    where $t'_{j}(\mathcal{S}_{g}):=\sum_{s\in\mathcal{S}_{g}}t_{sj}$ indicates the relationship between arc $j$ and component $g$.
                    $t'_{j}(\mathcal{S}_{g})=-1$ or $t'_{j}(\mathcal{S}_{g})=1$ means that $\mathcal{S}_{g}$ contains only the tail or head node of arc $j$, respectively;
                    $t'_{j}(\mathcal{S}_{g})=0$ means that $\mathcal{S}_{g}$ contains either both the head and tail nodes of arc $j$ or neither of them.
                    $\mathcal{J}'_{g}:=\{j\in[n_{a}]:t'_{j}(\mathcal{S}_{g})\neq0\}$ indicates the index set of ``hanging" arcs connected to component $g$.
                    Only one incident node of each ``hanging" arc is in component $g$.
                    If $k\notin\mathcal{J}'_{g}$, we have $t'_{k}(\mathcal{S}_{g})=0$ and \eqref{eq:MCF-capacity-sub-nodal-yy} always holds.
                    We next prove that \eqref{eq:MCF-capacity-sub-nodal-yy} holds for any $k\in\mathcal{J}_{d}(\mathcal{I})\cap\mathcal{J}_{a}\cap\mathcal{J}'_{g}$ whenever \eqref{eq:MCF-capacity-sub-nodal-xx} holds,
                    if and only if
                    for any two arcs in $\mathcal{J}_{d}(\mathcal{I})\cap\mathcal{J}_{a}$, if each arc has one incident node in component $g$, both of the two incident nodes are either head nodes or tail nodes.

                    \noindent
                    $\Leftarrow$:
                    If the condition holds, we have $t'_{j}(\mathcal{S}_{g})=1$ for all $j\in\mathcal{J}_{d}(\mathcal{\mathcal{I}})\cap\mathcal{J}_{a}\cap\mathcal{J}'_{g}$ or $t'_{j}(\mathcal{S}_{g})=-1$ for all $j\in\mathcal{J}_{d}(\mathcal{\mathcal{I}})\cap\mathcal{J}_{a}\cap\mathcal{J}'_{g}$.
                    Then, \eqref{eq:MCF-capacity-sub-nodal-xx} reduces to
                    \begin{equation*}
                        \sum_{j\in\mathcal{J}_{d}(\mathcal{\mathcal{I}})\cap\mathcal{J}_{a}\cap\mathcal{J}'_{g}}\delta_{j}x_{j}=0.
                    \end{equation*}
                    Because $\delta\geq0,x\geq0$, we have $\delta_{j}x_{j}=0$ for all $j\in\mathcal{J}_{d}(\mathcal{\mathcal{I}})\cap\mathcal{J}_{a}\cap\mathcal{J}'_{g}$.
                    Hence, for any $k\in\mathcal{J}_{a}\cap\mathcal{J}_{d}(\mathcal{\mathcal{I}})\cap\mathcal{J}'_{g}$, $\delta_{k}x_{k}=0$ and \eqref{eq:MCF-capacity-sub-nodal-yy} holds.

                    \noindent
                    $\Rightarrow$:
                    Suppose the contrary that there exist two arcs in $\mathcal{J}_{d}(\mathcal{\mathcal{I}})\cap\mathcal{J}_{a}$, indexed by $j_{1},j_{2}$, and their incident nodes in component $g$ contain a head node and a tail node.
                    We have $t'_{j_{1}}(\mathcal{S}_{g})=-t'_{j_{2}}(\mathcal{S}_{g})\neq0$.
                    Then, for any $\xi\neq0$, \eqref{eq:MCF-capacity-sub-nodal-xx} holds when $x_{j_{1}}=\xi/\delta_{j_{1}}$, $x_{j_{2}}=\xi/\delta_{j_{2}}$, and $x_{j}=0,\forall j\in(\mathcal{J}_{a}\cap\mathcal{J}_{d}(\mathcal{\mathcal{I}})\cap\mathcal{J}'_{g})\backslash\{j_{1},j_{2}\}$.
                    Yet, when either $k=j_{1}$ or $k=j_{2}$, \eqref{eq:MCF-capacity-sub-nodal-yy} does not hold, which contradicts that \eqref{eq:MCF-capacity-sub-nodal-yy} holds for any $k\in\mathcal{J}_{a}\cap\mathcal{J}_{d}(\mathcal{\mathcal{I}})\cap\mathcal{J}'_{g}$ whenever \eqref{eq:MCF-capacity-sub-nodal-xx} holds.

                    When $\mathcal{I}$ varies, a component $g$ of $\mathcal{G}'(\mathcal{I})$ may be any connected subgraph of $\mathcal{G}$ and $\mathcal{J}_{d}(\mathcal{\mathcal{I}})$ may contains any two arcs in $\mathcal{J}_{a}\cap\mathcal{J}'_{g}$.
                    Therefore, according to Lemma \ref{lem:pps:general-MCF-capacity-opposite}, $\forall\mathcal{I}\subseteq[m]$ with $|\mathcal{I}|=\text{rank}(B)+1$ such that $\text{rank}(B_{\mathcal{I}})=\text{rank}(B)$, for any component $g$ of $\mathcal{G}'(\mathcal{I})$, for any two arcs in $\mathcal{J}_{d}(\mathcal{\mathcal{I}})\cap\mathcal{J}_{a}$, if each arc has one incident node in component $g$, both of the two incident nodes are either head nodes or tail nodes,
                    if and only if
                    \begin{condition}\label{cdt:capacity-super}
                        for any path $P$ in $\mathcal{G}$, any two arcs in $\mathcal{A}_{a}\cap\mathcal{A}_{P}$ have opposite directions in $P$.
                    \end{condition}

                    Through the above proof and by the second condition of Theorem \ref{trm:condition-generalLP-sub}, we prove that
                    $\phi(x)$ defined in \eqref{eq:MCF-capacity2} is supermodular for any $c$ and $h$, if and only if Condition \ref{cdt:capacity-super} holds.
                    Therefore, $\phi(x)$ defined by \eqref{eq:MCF-capacity} is supermodular for any $c$, $d$, and $f$, if and only if Condition \ref{cdt:capacity-super} holds.

                    \vspace{2ex}
                    \noindent\textbf{(ii) Submodularity.}

                    Following the second condition of Theorem \ref{trm:condition-generalLP-super}, we suppose an $x\in\mathbb{R}^{n_{a}}$ such that $A_{\mathcal{I}}x\in\text{span}(B_{\mathcal{I}})$.
                    It implies that there exist a $y\in\mathbb{R}^{n_{a}}$ such that $A_{\mathcal{I}}x=B_{\mathcal{I}}y$.
                    Specifically, the linear system
                    \begin{equation}\label{eq:MCF-capacity-super-x0}
                        \left\{\begin{aligned}
                            &Ty=-p\\
                            &p_{s}=0, \forall s\in\mathcal{S}(\mathcal{I})\\
                            &y_{j}=0,\forall j\in\mathcal{J}_{1}(\mathcal{I})\\
                            &y_{j}=\delta_{j}x_{j},\forall j\in\mathcal{J}_{d}(\mathcal{I}).
                        \end{aligned}\right.
                    \end{equation}
                    is feasible, where $\mathcal{J}_{d}(\mathcal{I})=\mathcal{J}_{2}(\mathcal{I})\backslash\mathcal{J}_{1}(\mathcal{I})$.
                    According to the second condition of Theorem \ref{trm:condition-generalLP-super}, to guarantee the submodularity of \eqref{eq:MCF-capacity2}, we need to guarantee $A_{\mathcal{I}}x^{+}\in\text{span}(B_{\mathcal{I}})$.
                    It implies that there a $y'\in\mathbb{R}^{n_{a}}$ such that $A_{\mathcal{I}}x^{+}=B_{\mathcal{I}}y'$.
                    Specifically, the linear system
                    \begin{equation}\label{eq:MCF-capacity-super-y0}
                        \left\{\begin{aligned}
                            &Ty'=-p'\\
                            &p'_{s}=0, \forall s\in\mathcal{S}(\mathcal{I})\\
                            &y'_{j}=0,\forall j\in\mathcal{J}_{1}(\mathcal{I})\\
                            &y'_{j}=\delta_{j}x^{+}_{j},\forall j\in\mathcal{J}_{2}(\mathcal{I})
                        \end{aligned}\right.
                    \end{equation}
                    is feasible.
                    From the perspective of network potential, \eqref{eq:MCF-capacity-super-x0} or \eqref{eq:MCF-capacity-super-y0} do not align with the topology of $\mathcal{G}'(\mathcal{I})$, due to $y_{j}=\delta_{j}x_{j}$ or $y'_{j}=\delta_{j}x^{+}_{j}$ for $ j\in\mathcal{J}_{2}(\mathcal{I})\backslash\mathcal{J}_{1}(\mathcal{I})$.
                    To remain the topology of $\mathcal{G}'(\mathcal{I})$, for any $j\in\mathcal{J}_{d}(\mathcal{I})$, we replace arc flow $y_{j}=\delta_{j}x_{j}$ with a pair of dummy external injection at its head and tail nodes.
                    We note the dummy injection does not affect slack nodes, into which the total external flow injection remains free, and hence, the set $\mathcal{V}_{\text{slack}}(\mathcal{I})$ remains unchanged.
                    Then, we reformulate \eqref{eq:MCF-capacity-super-x0} and \eqref{eq:MCF-capacity-super-y0} as
                    \begin{equation}\label{eq:MCF-capacity-super-x}
                        \left\{\begin{aligned}
                            &Ty=-p\\
                            &p_{s}=\sum_{j\in\mathcal{J}_{d}(\mathcal{\mathcal{I}})\cap\mathcal{J}_{a}}t_{sj}\delta_{j}x_{j}, \forall s\in\mathcal{S}(\mathcal{I})\\
                            &y_{j}=0,\forall j\in\mathcal{J}(\mathcal{I}),
                        \end{aligned}\right.
                    \end{equation}
                    \begin{equation}\label{eq:MCF-capacity-super-y}
                        \left\{\begin{aligned}
                            &Ty'=-p'\\
                            &p'_{s}=\sum_{j\in\mathcal{J}_{d}(\mathcal{\mathcal{I}})\cap\mathcal{J}_{a}}t_{sj}\delta_{j}x^{+}_{j}, \forall s\in\mathcal{S}(\mathcal{I})\\
                            &y'_{j}=0,\forall j\in\mathcal{J}(\mathcal{I}),
                        \end{aligned}\right.
                    \end{equation}
                    where $t_{sj}\in\{-1,0,1\}$ is the entry of the incidence matrix $T$ and $\mathcal{J}_{a}:=\{j\in[n_{a}]:\delta_{j}>0\}$ is corresponding to the set $\mathcal{A}_{a}$ of attackable arcs.
                    Comparing \eqref{eq:MCF-capacity-super-x} and \eqref{eq:MCF-capacity-super-y} to \eqref{eq:MCF-capacity-I2}, we find:
                    \eqref{eq:MCF-capacity-super-x} modifies the external injection at node $s$ from 0 to $\sum_{j\in\mathcal{J}_{d}(\mathcal{\mathcal{I}})\cap\mathcal{J}_a}t_{sj}\delta_{j}x_{j}$ for all $s\in\mathcal{S}(\mathcal{I})$;
                    \eqref{eq:MCF-capacity-super-y} modifies the external injection at node $s$ from 0 to $\sum_{j\in\mathcal{J}_{d}(\mathcal{\mathcal{I}})\cap\mathcal{J}_a}t_{sj}\delta_{j}x^{+}_{j}$ for all $s\in\mathcal{S}(\mathcal{I})$.
                    Either \eqref{eq:MCF-capacity-super-x} or \eqref{eq:MCF-capacity-super-y} is feasible if and only if under such modifications, the nodal conservation still holds for each component of $\mathcal{G}'(\mathcal{I})$.
                    Specifically, for a certain component $g$ of $\mathcal{G}'(\mathcal{I})$, we denote its node index set by $\mathcal{S}_{g}\subseteq[n_{v}]$ and its arc index set by $\mathcal{J}_{g}\subseteq[n_{a}]\backslash\mathcal{J}(\mathcal{I})$.
                    Considering that the component is acyclic, the nodal conservation holds when
                    \begin{equation}\label{eq:MCF-capacity-super-nodal}
                        \sum_{s\in\mathcal{S}_{g}}p_{s}=0.
                    \end{equation}
                    Given \eqref{eq:MCF-capacity-super-x}, \eqref{eq:MCF-capacity-super-nodal} reduces to
                    \begin{equation}\label{eq:MCF-capacity-super-nodal-x}
                        \sum_{s\in\mathcal{S}_{g}\cap\mathcal{S}(\mathcal{I})}\sum_{j\in\mathcal{J}_{d}(\mathcal{\mathcal{I}})\cap\mathcal{J}_{a}}t_{sj}\delta_{j}x_{j}+
                        \sum_{s\in\mathcal{S}_{g}\backslash\mathcal{S}(\mathcal{I})}p_{s}=0.
                    \end{equation}
                    Given \eqref{eq:MCF-capacity-super-y}, \eqref{eq:MCF-capacity-super-nodal} reduces to
                    \begin{equation}\label{eq:MCF-capacity-super-nodal-y}
                        \sum_{s\in\mathcal{S}_{g}\cap\mathcal{S}(\mathcal{I})}\sum_{j\in\mathcal{J}_{d}(\mathcal{\mathcal{I}})\cap\mathcal{J}_{a}}t_{sj}\delta_{j}x^{+}_{j}+
                        \sum_{s\in\mathcal{S}_{g}\backslash\mathcal{S}(\mathcal{I})}p'_{s}=0.
                    \end{equation}

                    From the above analysis, to guarantee the submodularity of \eqref{eq:MCF-capacity2}, we need to guarantee that \eqref{eq:MCF-capacity-super-nodal-y} holds whenever \eqref{eq:MCF-capacity-super-nodal-x} holds.
                    We note each component in $\mathcal{G}'(\mathcal{I})$ contains at most one slack node.
                    If component $g$ contains one slack node, indexed by $s_{0}$, we have $\mathcal{S}_{g}\backslash\mathcal{S}(\mathcal{I})=\{s_{0}\}$, and hence, \eqref{eq:MCF-capacity-super-nodal-y} always holds due to free $p'_{s_{0}}$.
                    Hence, we focus on the case that the component contains no slack node in the following, i.e., $\mathcal{S}_{g}\subseteq\mathcal{S}(\mathcal{I})$.
                    Then, \eqref{eq:MCF-capacity-super-nodal-x} and \eqref{eq:MCF-capacity-super-nodal-y} reduce to
                    \begin{equation}\label{eq:MCF-capacity-super-nodal-xx}
                        \sum_{j\in\mathcal{J}_{d}(\mathcal{\mathcal{I}})\cap\mathcal{J}_{a}\cap\mathcal{J}'_{g}}t'_{j}(\mathcal{S}_{g})\delta_{j}x_{j}=0,
                    \end{equation}
                    \begin{equation}\label{eq:MCF-capacity-super-nodal-yy}
                        \sum_{j\in\mathcal{J}_{d}(\mathcal{\mathcal{I}})\cap\mathcal{J}_{a}\cap\mathcal{J}'_{g}}t'_{j}(\mathcal{S}_{g})\delta_{j}x^{+}_{j}=0,
                    \end{equation}
                    where $t'_{j}(\mathcal{S}_{g}):=\sum_{s\in\mathcal{S}_{g}}t_{sj}$ indicates the relationship between arc $j$ and component $g$.
                    $t'_{j}=-1$ or $t'_{j}=1$ means that $\mathcal{S}_{g}$ contains the tail or head node of arc $j$, respectively;
                    $t'_{j}=0$ means that $\mathcal{S}_{g}$ contains either both the head and tail nodes of arc $j$ or neither of them.
                    $\mathcal{J}'_{g}:=\{j\in[n_{a}]:t'_{j}(\mathcal{S}_{g})\neq0\}$ indicates the index set of ``hanging" arcs connected to component $g$.
                    Only one incident node of each ``hanging" arc is in component $g$.
                    We next prove that \eqref{eq:MCF-capacity-super-nodal-yy} holds whenever \eqref{eq:MCF-capacity-super-nodal-xx} holds, if and only if for any two arcs in $\mathcal{J}_{d}(\mathcal{\mathcal{I}})\cap\mathcal{J}_{a}$, if each arc has one incident node in component $g$, there exist a head node and a tail node in the two incident nodes.

                    \noindent
                    $\Leftarrow$:
                    If the condition holds, there are at most two arcs in $\mathcal{J}_{d}(\mathcal{\mathcal{I}})\cap\mathcal{J}_{a}$, each of which has one incident node in component $g$.\\
                    If there is no such arc, we have $\mathcal{J}_{d}(\mathcal{\mathcal{I}})\cap\mathcal{J}_{a}\cap\mathcal{J}'_{g}=\emptyset$ and thus \eqref{eq:MCF-capacity-super-nodal-yy} holds.\\
                    If there is one such arc, indexed by $j_{0}$, we have $\mathcal{J}_{d}(\mathcal{\mathcal{I}})\cap\mathcal{J}_{a}\cap\mathcal{J}'_{g}=\{j_0\}$.
                    \eqref{eq:MCF-capacity-super-nodal-xx} reduces to $\delta_{j_{0}}x_{j_{0}}=0$.
                    Because $\delta_{j_{0}}>0$, we have $x_{j_{0}}=0$.
                    Hence, \eqref{eq:MCF-capacity-super-nodal-yy} reduces to $\delta_{j_{0}}x^+_{j_{0}}=0$ and holds.\\
                    If there are two such arcs, indexed by $j_{1},j_{2}$, we have $\mathcal{J}_{d}(\mathcal{\mathcal{I}})\cap\mathcal{J}_{a}\cap\mathcal{J}'_{g}=\{j_1,j_2\}$.
                    According to the condition, we have $t'_{j_{1}}(\mathcal{S}_g)=-t'_{j_{2}}(\mathcal{S}_g)\neq0$.
                    Then, \eqref{eq:MCF-capacity-super-nodal-xx} reduces to $\delta_{j_{1}}x_{j_{1}}=\delta_{j_{2}}x_{j_{2}}$.
                    Because $\delta_{j_{1}}>0,\delta_{j_{2}}>0$, we have $x_{j_{1}}\geq0,x_{j_{2}}\geq0$ or $x_{j_{1}}\leq0,x_{j_{2}}\leq0$.
                    \eqref{eq:MCF-capacity-super-nodal-yy} reduces to $\delta_{j_{1}}x^{+}_{j_{1}}=\delta_{j_{2}}x^{+}_{j_{2}}$ and holds when either $x_{j_{1}}\geq0,x_{j_{2}}\geq0$ or $x_{j_{1}}\leq0,x_{j_{2}}\leq0$.

                    \noindent
                    $\Rightarrow$:
                    Suppose the contrary that there exist two arcs in $\mathcal{J}_{a}\cap\mathcal{J}_{d}(\mathcal{\mathcal{I}})$, indexed by $j_{1},j_{2}$, each of the two arcs has one incident node in component $g$, and both of the two incident nodes are either head nodes or tail nodes.
                    Hence, we have $t'_{j_{1}}(\mathcal{S}_g)=t'_{j_{2}}(\mathcal{S}_g)\neq0$.
                    For any $\xi\neq0$, \eqref{eq:MCF-capacity-super-nodal-xx} holds when $x_{j_{1}}=\xi/\delta_{j_{1}}$, $x_{j_{2}}=-\xi/\delta_{j_{2}}$, and $x_{j}=0,\forall j\in(\mathcal{J}_{d}(\mathcal{\mathcal{I}})\cap\mathcal{J}_{a}\cap\mathcal{J}'_{g})\backslash\{j_{1},j_{2}\}$.
                    Yet, because $\delta_{j_{1}}>0,\delta_{j_{2}}>0$, \eqref{eq:MCF-capacity-super-nodal-yy} does not hold, which contradicts that \eqref{eq:MCF-capacity-super-nodal-yy} holds whenever \eqref{eq:MCF-capacity-super-nodal-xx} holds.

                    When $\mathcal{I}$ varies, a component of $\mathcal{G}'(\mathcal{I})$ may be any connected subgraph of $\mathcal{G}$ and $\mathcal{J}_{d}(\mathcal{I})$ may contains any two arcs in $\mathcal{J}_{a}\cap\mathcal{J}'_{g}$.
                    Therefore, according to Lemma \ref{lem:pps:general-MCF-capacity-same}, $\forall\mathcal{I}\subseteq[m]$ with $|\mathcal{I}|=\text{rank}(B)+1$ such that $\text{rank}(B_{\mathcal{I}})=\text{rank}(B)$, for any component $g$ of $\mathcal{G}'(\mathcal{I})$, for any two arcs in $\mathcal{J}_{d}(\mathcal{\mathcal{I}})\cap\mathcal{J}_{a}$, if each arc has one incident node in component $g$, there exist a head node and a tail node in the two incident nodes,
                    if and only if
                    \begin{condition}\label{cdt:capacity-sub}
                        for any path $P$ in $\mathcal{G}$, any two arcs in $\mathcal{A}_{a}\cap\mathcal{A}_{P}$ have the same direction.
                    \end{condition}

                    Through the above proof and by the second condition of Theorem \ref{trm:condition-generalLP-super}, we prove that
                    $\phi(x)$ in \eqref{eq:MCF-capacity2} is submodular for any $c$ and $h$, if and only if Condition \ref{cdt:capacity-sub} holds.
                    Therefore, $\phi(x)$ defined by \eqref{eq:MCF-capacity} is submodular for any $c$, $d$, and $f$, if and only if Condition \ref{cdt:capacity-sub} holds.
                \end{proof}

                \begin{lemma}\label{lem:pps:general-MCF-capacity-opposite}
                    Given a directed network $\mathcal{G}=(\mathcal{V},\mathcal{A})$ and an arc set $\mathcal{A}_{a}\subseteq\mathcal{A}$.
                    For any connected subgraph $g$ of $\mathcal{G}$, for any two arcs in $\mathcal{A}_{a}$, if each arc has one incident node in subgraph $g$, both of the two incident nodes are either head nodes or tail nodes,
                    if and only if
                    for any path $P$ in $\mathcal{G}$, any two arcs in $\mathcal{A}_{a}\cap\mathcal{A}_{P}$ have opposite directions in $P$.
                \end{lemma}
                \begin{proof}
                    ``$\Leftarrow$'':
                    For any connected subgraph $g$ in $\mathcal{G}$ and any two arcs in $\mathcal{A}_{a}$, if each of the two arcs has one incident node in subgraph $g$, there must exist a path $P$ that contains the two arcs and goes through subgraph $g$.
                    According to the condition, the two arcs have opposite directions in $P$.
                    Without loss of generality, we assume the former arc in $P$ aligns with the forward sequence of $P$ and the latter aligns with the reverse sequence of $P$.
                    Therefore, for either the former or the latter arc, the incident node in subgraph $g$ is a head node.

                    ``$\Rightarrow$'':
                    Suppose the contrary that there exist a path $P$ and two arcs in $\mathcal{A}_{a}\cap\mathcal{A}_{P}$ with the same direction in $P$.
                    We can always find a connected subgraph $g$ of $\mathcal{G}$, which contains all arcs that is in $P$ and between the two arcs.
                    Hence, each of the two arcs has one incident node in subgraph $g$.
                    Without loss of generality, we assume the two arcs align with the forward sequence of $P$.
                    We find that for the former arc, its incident node in subgraph $g$ is its head node, and for the latter arc, its incident node in subgraph $g$ is its tail node.
                    This contradicts with that both of the two incident nodes are either head nodes or tail nodes.
                \end{proof}

                \begin{lemma}\label{lem:pps:general-MCF-capacity-same}
                    Given a directed network $\mathcal{G}=(\mathcal{V},\mathcal{A})$ and an arc set $\mathcal{A}_{a}\subseteq\mathcal{A}$.
                    For any connected subgraph $g$ of $\mathcal{G}$, for any two arcs in $\mathcal{A}_{a}$, if each of the two arcs has one incident node in subgraph $g$, there exist a head node and a tail node in the two incident nodes,
                    if and only if
                    for any path $P$ in $\mathcal{G}$, any two arcs in $\mathcal{A}_{a}\cap\mathcal{A}_{P}$ have the same direction.
                \end{lemma}
                \begin{proof}
                    ``$\Leftarrow$'':
                    For any connected subgraph $g$ in $\mathcal{G}$ and any two arcs in $\mathcal{A}_{a}$, if each of the two arcs has one incident node in subgraph $g$, there must exist a path $P$ that contains the two arcs and goes through subgraph $g$.
                    According to the condition, the two arcs have the same direction in $P$.
                    Without loss of generality, we assume the two arcs align with the forward sequence of $P$.
                    Therefore, for the former arc, its incident node in subgraph $g$ is its head node, and for the latter arc, its incident node in subgraph $g$ is its tail node.

                    ``$\Rightarrow$'':
                    Suppose the contrary that the there exist a path $P$ and two arcs in $\mathcal{A}_{a}\cap\mathcal{A}_{P}$ with opposite directions in $P$.
                    We can always find a connected subgraph $g$ of $\mathcal{G}$, which contains all arcs that is in $P$ and between the two arcs.
                    Hence, each of the two arcs has one incident node in subgraph $g$.
                    Without loss of generality, we assume the former arc in $P$ aligns with the forward sequence of $P$ and the latter aligns with the reverse sequence of $P$.
                    We find that for either the former or the latter arc, the incident node in subgraph $g$ is a head node.
                    This contradicts that there exist a head node and a tail node in the two incident nodes.
                \end{proof}

            \subsection{Proof of Proposition \ref{pps:general-MF-capacity}}\label{apd:pps:general-MF-capacity}
                \begin{proof}\color{black}
                    We rewrite \eqref{eq:MF-capacity} as
                    \begin{equation}\label{eq:MF-capacity2}
                        \begin{aligned}
                            -\phi(x)=\min_{y\in\mathbb{R}^{n_{a}},\eta\in\mathbb{R}^{n_s}}~&-[0,(1^{\top}\tilde{T})^{+}]\left[\begin{matrix}y\\\eta\end{matrix}\right]\\
                            \text{s.t.}~&Ax+B\left[\begin{matrix}y\\\eta\end{matrix}\right]\leq h,
                        \end{aligned}
                    \end{equation}
                    where
                    \begin{equation*}
                        A=\left[\begin{matrix}
                            ~\\
                            ~\\
                            ~\\
                            \delta'
                        \end{matrix}\right]\in\mathbb{R}^{m\times n_{a}},
                        B=\left[\begin{matrix}
                            T & -\tilde{T}\\
                            -T & \tilde{T}\\
                            -I &\\
                            I &
                        \end{matrix}\right]\in\mathbb{R}^{m\times (n_{a}+n_{s})},
                        h=\left[\begin{matrix}
                            ~\\
                            ~\\
                            ~\\
                            f
                        \end{matrix}\right]\in\mathbb{R}^{m},
                    \end{equation*}
                    where $m=2n_{v}+2n_{a}$, $\delta':=\text{diag}(\delta)$, and $I$ is an identity matrix of appropriate dimension.
                    Then, we have $\text{rank}(B)=n_{a}+n_{s}<m$.
                    Hence, we consider the second condition of Theorem \ref{trm:condition-generalLP-sub} or Theorem \ref{trm:condition-generalLP-super} in this proof.
                    Specifically, we consider a set $\mathcal{I}\subseteq[m]$ with $|\mathcal{I}|=\text{rank}(B)+1=n_{a}+n_{s}+1$, and denote $A_{\mathcal{I}}$ and $B_{\mathcal{I}}$ as
                    \begin{equation*}
                        A_{\mathcal{I}}=\left[\begin{matrix}
                            ~\\
                            ~\\
                            ~\\
                            \delta'_{\mathcal{I}_{4}}
                        \end{matrix}\right]\in\mathbb{R}^{(n_{a}+n_{s}+1)\times n_{a}},
                        B_{\mathcal{I}}=\left[\begin{matrix}
                            T_{\mathcal{I}_{1}} & -\tilde{T}_{\mathcal{I}_{1}}\\
                            -T_{\mathcal{I}_{2}} & \tilde{T}_{\mathcal{I}_{2}}\\
                            -I_{\mathcal{I}_{3}} & \\
                            I_{\mathcal{I}_{4}} &
                        \end{matrix}\right]\in\mathbb{R}^{(n_{a}+n_{s}+1)\times (n_{a}+n_{s})},
                    \end{equation*}
                    where $\mathcal{I}_{1},\mathcal{I}_{2}\subseteq[n_{v}]$ and $\mathcal{I}_{3},\mathcal{I}_{4}\subseteq[n_{a}]$ are corresponding to the subsets of $\mathcal{I}$.

                    Following the second condition of Theorem \ref{trm:condition-generalLP-sub} or Theorem \ref{trm:condition-generalLP-super}, we suppose $\text{rank}(B_{\mathcal{I}})=\text{rank}(B)=n_{a}+n_{s}$.
                    It indicates that $B_{\mathcal{I}}$ is of full column rank, and hence, $B_{\mathcal{I}}[y^{\top},\eta^{\top}]^{\top}=0$ with variables $y\in\mathbb{R}^{n_a}, \eta\in\mathbb{R}^{n_s}$ has a unique solution $y=0,\eta=0$.
                    Specifically, the linear system
                    \begin{equation}\label{eq:MF-capacity-I1}
                        \left\{\begin{aligned}
                            &T_{s}y=\tilde{T}_{s}\eta, \forall s\in\mathcal{S}(\mathcal{I})\\
                            &y_{j}=0,\forall j\in\mathcal{J}_{1}(\mathcal{I})\cup\mathcal{J}_{2}(\mathcal{I})
                        \end{aligned}\right.
                    \end{equation}
                    has a unique solution $y=0,\eta=0$, where $\mathcal{S}(\mathcal{I})=\mathcal{I}_{1}\cup\mathcal{I}_{2}$, $\mathcal{J}_{1}(\mathcal{I})=\mathcal{I}_{3}$, and $\mathcal{J}_{2}(\mathcal{I})=\mathcal{I}_{4}$.
                    Then, we introduce a slack variable $p\in\mathbb{R}^{n_{v}}$ and reformulate \eqref{eq:MF-capacity-I1} as
                    \begin{subequations}\label{eq:MF-capacity-I2}
                        \begin{gather}
                            \left\{\begin{aligned}
                                &Ty=-p\\
                                &p_{s}=0, \forall s\in\mathcal{S}(\mathcal{I})\backslash\mathcal{S}_{s}\\
                                &y_{j}=0,\forall j\in\mathcal{J}(\mathcal{I})
                            \end{aligned}\right.\label{eq:MF-capacity-I2-1}\\
                            \tilde{T}_{s}\eta=-p_{s}, \forall s\in\mathcal{S}(\mathcal{I})\cap\mathcal{S}_{s}
                        \end{gather}
                    \end{subequations}
                    where $\mathcal{J}(\mathcal{I})=\mathcal{J}_{1}(\mathcal{I})\cup\mathcal{J}_{2}(\mathcal{I})$ and $\mathcal{S}_{s}$ is the index set of source/sink nodes in $\mathcal{V}_{s}$.
                    Obviously, \eqref{eq:MF-capacity-I1} has a unique solution $y=0, \eta=0$ if and only if \eqref{eq:MF-capacity-I2-1} has a unique solution $y=0,p=0$ and $\mathcal{S}_{s}\subseteq\mathcal{S}(\mathcal{I})$.
                    In the following, we focus on the analysis of \eqref{eq:MF-capacity-I2-1}.
                    From the perspective of network flow, $y$ means arc flow and $p$ indicates external injection.
                    $Ty=-p$ restricts the nodal flow conservation with external injection $p$.
                    For $s\in\mathcal{S}(\mathcal{I})\backslash\mathcal{S}_{s}$, $p_{s}=0$ means no external injection into node $s$.
                    For $s\notin\mathcal{S}(\mathcal{I})\backslash\mathcal{S}_{s}$, the external injection is unlimited and is thus free.
                    Such nodes with free external injection are called slack nodes and we denote the set of slack nodes as $\mathcal{V}_{\text{slack}}(\mathcal{I})$.
                    For $j\in\mathcal{J}(\mathcal{I})$, $y_{j}=0$ admits no flow on arc $j$, and hence, we remove arcs $j,\forall j\in\mathcal{J}(\mathcal{I})$ from the network.
                    We denote the set of remaining arcs as $\mathcal{A}'(\mathcal{I})$ and obtain a new directed graph $\mathcal{G}'(\mathcal{I})=\big(\mathcal{V},\mathcal{A}'(\mathcal{I})\big)$.
                    With the above understanding in mind, a unique solution $y=0, p=0$ of \eqref{eq:MF-capacity-I2-1} indicates:
                    under nodal flow conservation limits, whenever there is no external injection into nodes in $\mathcal{V}\backslash\mathcal{V}_{\text{slack}}(\mathcal{I})$, there exists no flow in $\mathcal{G}'(\mathcal{I})$.
                    Therefore, according to Lemma \ref{lem:pps:general-MCF-demand}, each component of $\mathcal{G}'(\mathcal{I})$ is acyclic and contains at most one slack node in $\mathcal{V}_{\text{slack}}(\mathcal{I})$.

                    In the following, we derive the necessary and sufficient conditions for the submodularity and supermodularity of $\phi(x)$ defined by \eqref{eq:MF-capacity}, respectively.

                    \vspace{2ex}
                    \noindent\textbf{(i) Submodularity.}

                    $\phi(x)$ defined by \eqref{eq:MF-capacity} is submodular if and only if $-\phi(x)$ defined by \eqref{eq:MF-capacity2} is supermodular.
                    Following the second condition of Theorem \ref{trm:condition-generalLP-sub}, we suppose an $x\in\mathbb{R}^{n_{a}}_{+}$ such that $A_{\mathcal{I}}x\in\text{span}(B_{\mathcal{I}})$.
                    It implies that there exist a $y\in\mathbb{R}^{n_{v}}$ and a $\eta\in\mathbb{R}^{n_{s}}$ such that $A_{\mathcal{I}}x=B_{\mathcal{I}}[y^{\top},\eta^{\top}]^{\top}$.
                    Specifically, the linear system
                    \begin{subequations}\label{eq:MF-capacity-sub-x0}
                        \begin{gather}
                            \left\{\begin{aligned}
                                &Ty=-p\\
                                &p_{s}=0, \forall s\in\mathcal{S}(\mathcal{I})\backslash\mathcal{S}_{s}\\
                                &y_{j}=0,\forall j\in\mathcal{J}_{1}(\mathcal{I})\\
                                &y_{j}=\delta_{j}x_{j},\forall j\in\mathcal{J}_{d}(\mathcal{I})
                            \end{aligned}\right.\label{eq:MF-capacity-sub-x0-1}\\
                            \tilde{T}_{s}\eta=-p_{s}, \forall s\in\mathcal{S}_{s}
                        \end{gather}
                    \end{subequations}
                    is feasible, where $\mathcal{J}_{d}(\mathcal{I}):=\mathcal{J}_{2}(\mathcal{I})\backslash\mathcal{J}_{1}(\mathcal{I})$.
                    Obviously, \eqref{eq:MF-capacity-sub-x0} is feasible if and only if \eqref{eq:MF-capacity-sub-x0-1} is feasible.
                    According to the second condition of Theorem \ref{trm:condition-generalLP-sub}, to guarantee the supermodularity of \eqref{eq:MF-capacity2}, we need to guarantee $A_{\mathcal{I},k}x_{k}\in\text{span}(B_{\mathcal{I}})$ for any $k\in[n_{a}]$.
                    It implies that for any $k\in[n_{a}]$, there exist a $y'\in\mathbb{R}^{n_{a}}$ and a $\eta'\in\mathbb{R}^{n_{s}}$ such that $A_{\mathcal{I},k}x_{k}=B_{\mathcal{I}}[y'^{\top},\eta'^{\top}]^{\top}$.
                    Specifically, the linear system
                    \begin{subequations}\label{eq:MF-capacity-sub-y0}
                        \begin{gather}
                            \left\{\begin{aligned}
                                &Ty'=-p'\\
                                &p'_{s}=0, \forall s\in\mathcal{S}(\mathcal{I})\backslash\mathcal{S}_{s}\\
                                &y'_{j}=0,\forall j\in\mathcal{J}_{1}(\mathcal{I})\\
                                &y'_{k}=\delta_{k}x_{k},\text{~if~} k\in\mathcal{J}_{d}(\mathcal{I})\\
                                &y'_{j}=0,\forall j\in\mathcal{J}_{d}(\mathcal{I})\backslash\{k\}
                            \end{aligned}\right.\label{eq:MF-capacity-sub-y0-1}\\
                            \tilde{T}_{s}\eta'=-p'_{s}, \forall s\in\mathcal{S}_{s}
                        \end{gather}
                    \end{subequations}
                    is feasible.
                    Obviously, \eqref{eq:MF-capacity-sub-y0} is feasible if and only if \eqref{eq:MF-capacity-sub-y0-1} is feasible.
                    If $k\notin\mathcal{J}_{d}(\mathcal{I})$, $y'=0,p'=0$ is always feasible.
                    If $k\notin\mathcal{J}_{a}$, where $\mathcal{J}_{a}:=\{j\in[n_{a}]:\delta_{j}>0\}$ is corresponding to the set $\mathcal{A}_{a}$ of attackable arcs, we have $\delta_{k}=0$ and thus $y'=0,p'=0$ is feasible.
                    Hence, we only consider $k\in\mathcal{J}_{d}(\mathcal{I})\cap\mathcal{J}_{a}$ in the following.
                    From the perspective of network flow, \eqref{eq:MF-capacity-sub-x0-1} or \eqref{eq:MF-capacity-sub-y0-1} do not align with the topology of $\mathcal{G}'(\mathcal{I})$ due to $y_{j}=\delta_{j}x_{j}$ for $ j\in\mathcal{J}_{d}(\mathcal{I})$ or $y'_{k}=\delta_{k}x_{k}$.
                    To remain the topology of $\mathcal{G}'(\mathcal{I})$, for any $j\in\mathcal{J}_{d}(\mathcal{I})$, we replace flow on arc $j$ with a pair of dummy external injection at its head and tail nodes.
                    We note the dummy injection does not affect slack nodes, into which the total external injection remains free, and hence, the set $\mathcal{V}_{\text{slack}}(\mathcal{I})$ remains unchanged.
                    Then, we reformulate \eqref{eq:MF-capacity-sub-x0-1} and \eqref{eq:MF-capacity-sub-y0-1} as
                    \begin{equation}\label{eq:MF-capacity-sub-x}
                        \left\{\begin{aligned}
                            &Ty=-p\\
                            &p_{s}=\sum_{j\in\mathcal{J}_{d}(\mathcal{\mathcal{I}})\cap\mathcal{J}_a}t_{sj}\delta_{j}x_{j}, \forall s\in\mathcal{S}(\mathcal{I})\backslash\mathcal{S}_{s}\\
                            &y_{j}=0,\forall j\in\mathcal{J}(\mathcal{I}),
                        \end{aligned}\right.
                    \end{equation}
                    \begin{equation}\label{eq:MF-capacity-sub-y}
                        \left\{\begin{aligned}
                            &Ty'=-p'\\
                            &p'_{s}=t_{sk}\delta_{k}x_{k}, \forall s\in\mathcal{S}(\mathcal{I})\backslash\mathcal{S}_{s}\\
                            &y'_{j}=0,\forall j\in\mathcal{J}(\mathcal{I}),
                        \end{aligned}\right.
                    \end{equation}
                    where $t_{sj}\in\{-1,0,1\}$ is the element of $T$.
                    Comparing \eqref{eq:MF-capacity-sub-x} and \eqref{eq:MF-capacity-sub-y} to \eqref{eq:MF-capacity-I2-1}, we find:
                    \eqref{eq:MF-capacity-sub-x} modifies the external injection at node $s$ from 0 to $\sum_{j\in\mathcal{J}_{d}(\mathcal{\mathcal{I}})\cap\mathcal{J}_a}t_{sj}\delta_{j}x_{j}$ for all $s\in\mathcal{S}(\mathcal{I})\backslash\mathcal{S}_{s}$;
                    \eqref{eq:MF-capacity-sub-y} modifies the external injection at node $s$ from 0 to $t_{sk}\delta_{k}x_{k}$ for all $s\in\mathcal{S}(\mathcal{I})\backslash\mathcal{S}_{s}$.
                    Either \eqref{eq:MF-capacity-sub-x} or \eqref{eq:MF-capacity-sub-y} is feasible, if and only if under such modifications, their nodal conservation still holds for each component in $\mathcal{G}'(\mathcal{I})$.
                    Specifically, for a certain component $g$ in $\mathcal{G}'(\mathcal{I})$, we denote its node index set by $\mathcal{S}_{g}\subseteq[n_{v}]$ and its arc index set by $\mathcal{J}_{g}\subseteq[n_{a}]\backslash\mathcal{J}(\mathcal{I})$.
                    Considering that the component is acyclic, the nodal conservation holds when
                    \begin{equation}\label{eq:MF-capacity-sub-nodal}
                        \sum_{s\in\mathcal{S}_{g}}p_{s}=0.
                    \end{equation}
                    Given \eqref{eq:MF-capacity-sub-x}, \eqref{eq:MF-capacity-sub-nodal} reduces to
                    \begin{equation}\label{eq:MF-capacity-sub-nodal-x}
                        \sum_{s\in\mathcal{S}_{g}\cap(\mathcal{S}(\mathcal{I})\backslash\mathcal{S}_{s})}\sum_{j\in\mathcal{J}_{d}(\mathcal{\mathcal{I}})\cap\mathcal{J}_{a}}t_{sj}\delta_{j}x_{j}+
                        \sum_{s\in\mathcal{S}_{g}\backslash(\mathcal{S}(\mathcal{I})\backslash\mathcal{S}_{s})}p_{s}=0.
                    \end{equation}
                    Given \eqref{eq:MF-capacity-sub-y}, \eqref{eq:MF-capacity-sub-nodal} reduces to
                    \begin{equation}\label{eq:MF-capacity-sub-nodal-y}
                        \sum_{s\in\mathcal{S}_{g}\cap(\mathcal{S}(\mathcal{I})\backslash\mathcal{S}_{s})}t_{sk}\delta_{k}x_{k}+
                        \sum_{s\in\mathcal{S}_{g}\backslash(\mathcal{S}(\mathcal{I})\backslash\mathcal{S}_{s})}p'_{s}=0.
                    \end{equation}

                    From the above analysis, to guarantee the supermodularity of \eqref{eq:MF-capacity2}, we need to guarantee that \eqref{eq:MF-capacity-sub-nodal-y} holds for any $k\in\mathcal{J}_{d}(\mathcal{I})\cap\mathcal{J}_{a}$ whenever \eqref{eq:MF-capacity-sub-nodal-x} holds.
                    We note each component in $\mathcal{G}'(\mathcal{I})$ contains at most one slack node.
                    If component $g$ contains one slack node, indexed by $s_{0}$, we have $\mathcal{S}_{g}\backslash(\mathcal{S}(\mathcal{I})\backslash\mathcal{S}_{s})=\{s_{0}\}$, and hence, \eqref{eq:MF-capacity-sub-nodal-y} always holds due to free $p'_{s_{0}}$.
                    Hence, we only need to consider that the component contains no slack node, i.e., $\mathcal{S}_{g}\subseteq\mathcal{S}(\mathcal{I})\backslash\mathcal{S}_{s}$.
                    Then, \eqref{eq:MF-capacity-sub-nodal-x} and \eqref{eq:MF-capacity-sub-nodal-y} reduce to
                    \begin{equation}\label{eq:MF-capacity-sub-nodal-xx}
                        \sum_{j\in\mathcal{J}_{d}(\mathcal{\mathcal{I}})\cap\mathcal{J}_{a}\cap\mathcal{J}'_{g}}t'_{j}(\mathcal{S}_{g})\delta_{j}x_{j}=0,
                    \end{equation}
                    \begin{equation}\label{eq:MF-capacity-sub-nodal-yy}
                        t'_{k}(\mathcal{S}_{g})\delta_{k}x_{k}=0,
                    \end{equation}
                    where $t'_{j}(\mathcal{S}_{g}):=\sum_{s\in\mathcal{S}_{g}}t_{sj}$ indicates the relationship between arc j and component $g$.
                    $t'_{j}(\mathcal{S}_{g})=-1$ or $t'_{j}(\mathcal{S}_{g})=1$ means that $\mathcal{S}_{g}$ contains only the tail or head node of arc $j$, respectively.
                    $t'_{j}(\mathcal{S}_{g})=0$ means that $\mathcal{S}_{g}$ contains either both the head and tail nodes of arc $j$ or neither of them.
                    $\mathcal{J}'_{g}:=\{j\in[n_{a}]:t'_{j}(\mathcal{S}_{g})\neq0\}$ indicates the index set of ``hanging" arcs connected to component $g$.
                    Only one incident node of each ``hanging" arc is in component $g$.
                    If $k\notin\mathcal{J}'_{g}$, we have $t'_{k}(\mathcal{S}_{g})=0$ and \eqref{eq:MF-capacity-sub-nodal-yy} always holds.
                    We can prove that
                    \eqref{eq:MF-capacity-sub-nodal-yy} holds for any $k\in\mathcal{J}_{d}(\mathcal{\mathcal{I}})\cap\mathcal{J}_{a}\cap\mathcal{J}'_{g}$ whenever \eqref{eq:MF-capacity-sub-nodal-xx} holds
                    if and only if
                    for any two arcs in $\mathcal{J}_{d}(\mathcal{\mathcal{I}})\cap\mathcal{J}_{a}$, if each arc has one incident node in component $g$, both of the two incident nodes are either head nodes or tail nodes.
                    The proof is similar to that in the proof of Proposition \ref{pps:general-MCF-capacity} and is thus omitted.

                    When $\mathcal{I}$ varies, a component $g$ of $\mathcal{G}'(\mathcal{I})$ with $\mathcal{S}_{g}\subseteq\mathcal{S}(\mathcal{I})\backslash\mathcal{S}_{s}$ may be any connected subgraph of $\mathcal{G}$ that contains no node in $\mathcal{V}_{s}$, and $\mathcal{J}_d(\mathcal{I})$ may contains any two arcs in $\mathcal{J}_{a}\cap\mathcal{J}'_{g}$.
                    Therefore, according to Lemma \ref{lem:pps:general-MF-capacity-opposite}, $\forall \mathcal{I}\subseteq[m]$ with $|\mathcal{I}|=\text{rank}(B)+1$ such that $\text{rank}(B_{\mathcal{I}})=\text{rank}(B)$, for any component $g$ of $\mathcal{G}'(\mathcal{I})$ with $\mathcal{S}_{g}\subseteq\mathcal{S}(\mathcal{I})\backslash\mathcal{S}_{s}$, for any two arcs in $\mathcal{J}_{d}(\mathcal{\mathcal{I}})\cap\mathcal{J}_{a}$, if each arc has one incident node in component $g$, both of the two incident nodes are either head nodes or tail nodes, if and only if
                    \begin{condition}\label{cdt:MF-super}
                        for any $\mathcal{V}_{s}$-excluded path $P$ in $\mathcal{G}$, any two arcs in $\mathcal{A}_{a}\cap\mathcal{A}_{P}$ have opposite directions.
                    \end{condition}

                    Through the above proof and by the second condition of Theorem \ref{trm:condition-generalLP-sub}, we prove that $-\phi(x)$ defined by \eqref{eq:MF-capacity2} is supermodular for any $h$, if and only if Condition \ref{cdt:MF-super} holds.
                    Therefore, $\phi(x)$ defined by \eqref{eq:MF-capacity} is submodular for any $f$, if and only if Condition \ref{cdt:MF-super} holds.

                    \vspace{2ex}
                    \noindent\textbf{(ii) Supermodularity.}

                    $\phi(x)$ defined by \eqref{eq:MF-capacity} is supermodular if and only if $-\phi(x)$ defined by \eqref{eq:MF-capacity2} is submodular.
                    Following the second condition of Theorem \ref{trm:condition-generalLP-super}, we suppose an $x\in\mathbb{R}^{n_{a}}$ such that $A_{\mathcal{I}}x\in\text{span}(B_{\mathcal{I}})$.
                    It implies that there exist a $y\in\mathbb{R}^{n_{v}}$ and a $\eta\in\mathbb{R}^{n_{s}}$ such that $A_{\mathcal{I}}x=B_{\mathcal{I}}[y^{\top},\eta^{\top}]^{\top}$.
                    Specifically, the linear system
                    \begin{subequations}\label{eq:MF-capacity-super-x0}
                        \begin{gather}
                            \left\{\begin{aligned}
                                &Ty=-p\\
                                &p_{s}=0, \forall s\in\mathcal{S}(\mathcal{I})\backslash\mathcal{S}_{s}\\
                                &y_{j}=0,\forall j\in\mathcal{J}_{1}(\mathcal{I})\\
                                &y_{j}=\delta_{j}x_{j},\forall j\in\mathcal{J}_{d}(\mathcal{I})
                            \end{aligned}\right.\label{eq:MF-capacity-super-x0-1}\\
                            \tilde{T}_{s}\eta=-p_{s}, \forall s\in\mathcal{S}_{s}
                        \end{gather}
                    \end{subequations}
                    is feasible, where $\mathcal{J}_{d}(\mathcal{I}):=\mathcal{J}_{2}(\mathcal{I})\backslash\mathcal{J}_{1}(\mathcal{I})$.
                    Obviously, \eqref{eq:MF-capacity-super-x0} is feasible if and only if \eqref{eq:MF-capacity-super-x0-1} is feasible.
                    According to the second condition of Theorem \ref{trm:condition-generalLP-super}, to guarantee the submodularity of \eqref{eq:MF-capacity2}, we need to guarantee $A_{\mathcal{I}}x^+\in\text{span}(B_{\mathcal{I}})$.
                    It implies that there exist a $y'\in\mathbb{R}^{n_{v}}$ and a $\eta'\in\mathbb{R}^{n_{s}}$ such that $A_{\mathcal{I}}x^+=B_{\mathcal{I}}[y'^{\top},\eta'^{\top}]^{\top}$.
                    Specifically, the linear system
                    \begin{subequations}\label{eq:MF-capacity-super-y0}
                        \begin{gather}
                            \left\{\begin{aligned}
                                &Ty'=-p'\\
                                &p'_{s}=0, \forall s\in\mathcal{S}(\mathcal{I})\backslash\mathcal{S}_{s}\\
                                &y'_{j}=0,\forall j\in\mathcal{J}_{1}(\mathcal{I})\\
                                &y'_{j}=\delta_{j}x^+_{j},\forall j\in\mathcal{J}_{d}(\mathcal{I})
                            \end{aligned}\right.\label{eq:MF-capacity-super-y0-1}\\
                            \tilde{T}_{s}\eta'=-p'_{s}, \forall s\in\mathcal{S}_{s}
                        \end{gather}
                    \end{subequations}
                    is feasible.
                    Obviously, \eqref{eq:MF-capacity-super-y0} is feasible if and only if \eqref{eq:MF-capacity-super-y0-1} is feasible.
                    From the perspective of network flow, \eqref{eq:MF-capacity-super-x0} or \eqref{eq:MF-capacity-super-y0} do not align with the topology of $\mathcal{G}'(\mathcal{I})$ due to $y_{j}=\delta_{j}x_{j}$ or $y'_{j}=\delta_{j}x^{+}_{j}$ for $ j\in\mathcal{J}_{d}(\mathcal{I})$.
                    To remain the topology of $\mathcal{G}'(\mathcal{I})$, for any $j\in\mathcal{J}_{d}(\mathcal{I})$, we replace flow on arc $j$ with a pair of dummy external injection at its head and tail nodes.
                    We note the dummy injection does not affect slack nodes, into which the total external injection remains free, and hence, the set $\mathcal{V}_{\text{slack}}(\mathcal{I})$ remains unchanged.
                    Then, we reformulate \eqref{eq:MF-capacity-super-x0-1} and \eqref{eq:MF-capacity-super-y0-1} as
                    \begin{equation}\label{eq:MF-capacity-super-x}
                        \left\{\begin{aligned}
                            &Ty=-p\\
                            &p_{s}=\sum_{j\in\mathcal{J}_{d}(\mathcal{\mathcal{I}})\cap\mathcal{J}_{a}}t_{sj}\delta_{j}x_{j}, \forall s\in\mathcal{S}(\mathcal{I})\backslash\mathcal{S}_{s}\\
                            &y_{j}=0,\forall j\in\mathcal{J}(\mathcal{I}),
                        \end{aligned}\right.
                    \end{equation}
                    \begin{equation}\label{eq:MF-capacity-super-y}
                        \left\{\begin{aligned}
                            &Ty'=-p'\\
                            &p'_{s}=\sum_{j\in\mathcal{J}_{d}(\mathcal{\mathcal{I}})\cap\mathcal{J}_{a}}t_{sj}\delta_{j}x^{+}_{j}, \forall s\in\mathcal{S}(\mathcal{I})\backslash\mathcal{S}_{s}\\
                            &y'_{j}=0,\forall j\in\mathcal{J}(\mathcal{I}),
                        \end{aligned}\right.
                    \end{equation}
                    where $t_{sj}\in\{-1,0,1\}$ is the element of the incidence matrix $T$ and $\mathcal{J}_{a}:=\{j\in[n_{a}]:\delta_{j}>0\}$ is corresponding to the set $\mathcal{A}_{a}$ of attackable arcs.
                    Comparing \eqref{eq:MF-capacity-super-x} and \eqref{eq:MF-capacity-super-y} to \eqref{eq:MF-capacity-I2-1}, we find:
                    \eqref{eq:MF-capacity-super-x} modifies the external injection at node $s$ from 0 to $\sum_{j\in\mathcal{J}_{d}(\mathcal{\mathcal{I}})\cap\mathcal{J}_{a}}t_{sj}\delta_{j}x_{j}$ for all $s\in\mathcal{S}(\mathcal{I})$;
                    \eqref{eq:MF-capacity-super-y} modifies the external injection at node $s$ from 0 to $\sum_{j\in\mathcal{J}_{d}(\mathcal{\mathcal{I}})\cap\mathcal{J}_{a}}t_{sj}\delta_{j}x^{+}_{j}$ for all $s\in\mathcal{S}(\mathcal{I})$.
                    Either \eqref{eq:MF-capacity-super-x} or \eqref{eq:MF-capacity-super-y} is feasible, if and only if under such modifications, their nodal conservation still holds for each connected component in $\mathcal{G}'(\mathcal{I})$.
                    Specifically, for a certain component $g$ in $\mathcal{G}'(\mathcal{I})$, we denote its node index set by $\mathcal{S}_{g}\subseteq[n_{v}]$ and its arc index set by $\mathcal{J}_{g}\subseteq[n_{a}]\backslash\mathcal{J}(\mathcal{I})$.
                    Considering that the component is acyclic, the nodal conservation holds when
                    \begin{equation}\label{eq:MF-capacity-super-nodal}
                        \sum_{s\in\mathcal{S}_{g}}p_{s}=0.
                    \end{equation}
                    Given \eqref{eq:MF-capacity-super-x}, \eqref{eq:MF-capacity-super-nodal} reduces to
                    \begin{equation}\label{eq:MF-capacity-super-nodal-x}
                        \sum_{s\in\mathcal{S}_{g}\cap(\mathcal{S}(\mathcal{I})\backslash\mathcal{S}_{s})}\sum_{j\in\mathcal{J}_{d}(\mathcal{\mathcal{I}})\cap\mathcal{J}_{a}}t_{sj}\delta_{j}x_{j}+
                        \sum_{s\in\mathcal{S}_{g}\backslash(\mathcal{S}(\mathcal{I})\backslash\mathcal{S}_{s})}p_{s}=0.
                    \end{equation}
                    Given \eqref{eq:MF-capacity-super-y}, \eqref{eq:MF-capacity-super-nodal} reduces to
                    \begin{equation}\label{eq:MF-capacity-super-nodal-y}
                        \sum_{s\in\mathcal{S}_{g}\cap(\mathcal{S}(\mathcal{I})\backslash\mathcal{S}_{s})}\sum_{j\in\mathcal{J}_{d}(\mathcal{\mathcal{I}})\cap\mathcal{J}_{a}}t_{sj}\delta_{j}x^{+}_{j}+
                        \sum_{s\in\mathcal{S}_{g}\backslash(\mathcal{S}(\mathcal{I})\backslash\mathcal{S}_{s})}p_{s}=0.
                    \end{equation}

                    From the above analysis, to guarantee the submodularity of \eqref{eq:MF-capacity2}, we need to guarantee that \eqref{eq:MF-capacity-super-nodal-y} holds whenever \eqref{eq:MF-capacity-super-nodal-x} holds.
                    We note each component in $\mathcal{G}'(\mathcal{I})$ contains at most one slack node.
                    If component $g$ contains one slack node, indexed by $s_{0}$, we have $\mathcal{S}_{g}\backslash(\mathcal{S}(\mathcal{I})\backslash\mathcal{S}_{s})=\{s_{0}\}$, and hence, \eqref{eq:MF-capacity-super-nodal-y} always holds due to free $p'_{s_{0}}$.
                    Hence, we only need to consider that the component contains no slack node, i.e., $\mathcal{S}_{g}\subseteq\mathcal{S}(\mathcal{I})\backslash\mathcal{S}_{s}$.
                    Then, \eqref{eq:MF-capacity-super-nodal-x} and \eqref{eq:MF-capacity-super-nodal-y} reduce to
                    \begin{equation}\label{eq:MF-capacity-super-nodal-xx}
                        \sum_{j\in\mathcal{J}_{d}(\mathcal{\mathcal{I}})\cap\mathcal{J}_{a}\cap\mathcal{J}'_{g}}t'_{j}(\mathcal{S}_{g})\delta_{j}x_{j}=0,
                    \end{equation}
                    \begin{equation}\label{eq:MF-capacity-super-nodal-yy}
                        \sum_{j\in\mathcal{J}_{d}(\mathcal{\mathcal{I}})\cap\mathcal{J}_{a}\cap\mathcal{J}'_{g}}t'_{j}(\mathcal{S}_{g})\delta_{j}x^{+}_{j}=0,
                    \end{equation}
                    where $t'_{j}(\mathcal{S}_{g}):=\sum_{s\in\mathcal{S}_{g}}t_{sj}$ indicates the relationship between arc j and component $g$.
                    $t'_{j}(\mathcal{S}_{g})=-1$ or $t'_{j}(\mathcal{S}_{g})=1$ means that $\mathcal{S}_{g}$ contains only the tail or head node of arc $j$, respectively.
                    $t'_{j}(\mathcal{S}_{g})=0$ means that $\mathcal{S}_{g}$ contains either both the head and tail nodes of arc $j$ or neither of them.
                    $\mathcal{J}'_{g}:=\{j\in[n_{a}]:t'_{j}(\mathcal{S}_{g})\neq0\}$ indicates the index set of ``hanging" arcs connected to component $g$.
                    Only one incident node of each ``hanging" arc is in component $g$.
                    We can prove that
                    \eqref{eq:MF-capacity-super-nodal-yy} holds whenever \eqref{eq:MF-capacity-super-nodal-xx} holds
                    if and only if
                    for any two arcs in $\mathcal{J}_{d}(\mathcal{\mathcal{I}})\cap\mathcal{J}_{a}$, if each arc has one incident node in component $g$, there exist a head node and a tail node in the two incident nodes.
                    The proof is similar to that in the proof of Proposition \ref{pps:general-MCF-capacity} and is thus omitted.

                    When $\mathcal{I}$ varies, a component $g$ of $\mathcal{G}'(\mathcal{I})$ with $\mathcal{S}_{g}\subseteq\mathcal{S}(\mathcal{I})\backslash\mathcal{S}_{s}$ may be any connected subgraph of $\mathcal{G}$ that contains no node in $\mathcal{V}_{s}$, and $\mathcal{J}_d(\mathcal{I})$ may contains any two arcs in $\mathcal{J}_{a}\cap\mathcal{J}'_{g}$.
                    Therefore, according to Lemma \ref{lem:pps:general-MF-capacity-same}, $\forall \mathcal{I}\subseteq[m]$ with $|\mathcal{I}|=\text{rank}(B)+1$ such that $\text{rank}(B_{\mathcal{I}})=\text{rank}(B)$, for any component $g$ of $\mathcal{G}'(\mathcal{I})$ with $\mathcal{S}_{g}\subseteq\mathcal{S}(\mathcal{I})\backslash\mathcal{S}_{s}$, for any two arcs in $\mathcal{J}_{d}(\mathcal{\mathcal{I}})\cap\mathcal{J}_{a}$, if each arc has one incident node in component $g$, there exist a head node and a tail node in the two incident nodes, if and only if
                    \begin{condition}\label{cdt:MF-sub}
                        for any $\mathcal{V}_{s}$-excluded path $P$ in $\mathcal{G}$, any two arcs in $\mathcal{A}_{a}\cap\mathcal{A}_{P}$ have the same direction.
                    \end{condition}

                    Through the above proof and by the second condition of Theorem \ref{trm:condition-generalLP-super}, we prove that $-\phi(x)$ defined by \eqref{eq:MF-capacity2} is submodular for any $h$, if and only if Condition \ref{cdt:MF-sub} holds.
                    Therefore, $\phi(x)$ defined by \eqref{eq:MF-capacity} is supermodular for any $f$, if and only if Condition \ref{cdt:MF-sub} holds.
                \end{proof}

                \begin{lemma}\label{lem:pps:general-MF-capacity-opposite}
                    Given a directed network $\mathcal{G}=(\mathcal{V},\mathcal{A})$, a node set $\mathcal{V}_s\subseteq\mathcal{V}$, and an arc set $\mathcal{A}_{a}\subseteq\mathcal{A}$.
                    For any connected subgraph $g$ of $\mathcal{G}$ that contains no node in $\mathcal{V}_s$, for any two arcs in $\mathcal{A}_{a}$, if each arc has one incident node in subgraph $g$, both of the two incident nodes are either head nodes or tail nodes, if and only if
                    for any $\mathcal{V}_s$-excluded path $P$ in $\mathcal{G}$, any two arcs in $\mathcal{A}_{a}\cap\mathcal{A}_{P}$ have opposite directions.
                \end{lemma}
                \begin{proof}
                    ``$\Leftarrow$'':
                    For any connected subgraph $g$ in $\mathcal{G}$ that contains no node in $\mathcal{V}_s$ and any two arcs in $\mathcal{A}_{a}$, if each of the two arcs has one incident node in subgraph $g$, there must exist a $\mathcal{V}_s$-excluded path $P$ that contains the two arcs and goes through subgraph $g$.
                    According to the condition, the two arcs have opposite directions in $P$.
                    Without loss of generality, we assume the former arc in $P$ aligns with the forward sequence of $P$ and the latter aligns with the reverse sequence of $P$.
                    Therefore, for either the former or the latter arc, the incident node in subgraph $g$ is a head node.

                    ``$\Rightarrow$'':
                    Suppose the contrary that there exist a $\mathcal{V}_s$-excluded path $P$ and two arcs in $\mathcal{A}_{a}\cap\mathcal{A}_{P}$ with the same direction in $P$.
                    We can always find a connected subgraph $g$ of $\mathcal{G}$, which contains no node in $\mathcal{V}_s$ but contains all arcs that is in $P$ and between the two arcs.
                    Hence, each of the two arcs has one incident node in subgraph $g$.
                    Without loss of generality, we assume the two arcs align with the forward sequence of $P$.
                    We find that for the former arc, its incident node in subgraph $g$ is its head node, and for the latter arc, its incident node in subgraph $g$ is its tail node.
                    This contradicts with that both of the two incident nodes are either head nodes or tail nodes.
                \end{proof}

                \begin{lemma}\label{lem:pps:general-MF-capacity-same}
                    Given a directed network $\mathcal{G}=(\mathcal{V},\mathcal{A})$, a node set $\mathcal{V}_s\subseteq\mathcal{V}$, and an arc set $\mathcal{A}_{a}\subseteq\mathcal{A}$.
                    For any connected subgraph $g$ in $\mathcal{G}$ that contains no node in $\mathcal{V}_s$, for any two arcs in $\mathcal{A}_{a}$, if each arc has one incident node in subgraph $g$, there exist a head node and a tail node in the two incident nodes, if and only if
                    for any $\mathcal{V}_s$-excluded path $P$ in $\mathcal{G}$, any two arcs in $\mathcal{A}_{a}\cap\mathcal{A}_{P}$ have the same direction.
                \end{lemma}
                \begin{proof}
                    ``$\Leftarrow$'':
                    For any connected subgraph $g$ in $\mathcal{G}$ that contains no node in $\mathcal{V}_s$ and any two arcs in $\mathcal{A}_{a}$, if each of the two arcs has one incident node in subgraph $g$, there must exist a $\mathcal{V}_s$-excluded path $P$ that contains the two arcs and goes through subgraph $g$.
                    According to the condition, the two arcs have the same direction in $P$.
                    Without loss of generality, we assume the two arcs align with the forward sequence of $P$.
                    Therefore, for the former arc, its incident node in subgraph $g$ is its head node, and for the latter arc, its incident node in subgraph $g$ is its tail node.

                    ``$\Rightarrow$'':
                    Suppose the contrary that the there exist a $\mathcal{V}_s$-excluded path $P$ and two arcs in $\mathcal{A}_{a}\cap\mathcal{A}_{P}$ with opposite directions in $P$.
                    We can always find a connected subgraph $g$ of $\mathcal{G}$, which contains no node in $\mathcal{V}_s$ but contains all arcs that is in $P$ and between the two arcs.
                    Hence, each of the two arcs has one incident node in subgraph $g$.
                    Without loss of generality, we assume the former arc in $P$ aligns with the forward sequence of $P$ and the latter aligns with the reverse sequence of $P$.
                    We find that for either the former or the latter arc, the incident node in subgraph $g$ is a head node.
                    This contradicts that there exist a head node and a tail node in the two incident nodes.
                \end{proof}

            \subsection{Proof of Proposition \ref{pps:general-MCF-cost}}\label{apd:pps:general-MCF-cost}
                \begin{proof}\color{black}
                    Because \eqref{eq:MCF-cost} is a parametric LP, we reformulate \eqref{eq:MCF-cost} as
                    \begin{equation}\label{eq:MCF-cost2}
                        \begin{aligned}
                            -\phi(x)=-\max_{\lambda\in\mathbb{R}^{n_{v}},\pi\in\mathbb{R}^{n_{a}}}~&d^{\top}\lambda+f^{\top}\pi\\
                            \text{s.t.}~&T^{\top}\lambda+\pi\leq c+\text{diag}(\delta)x\\
                            &\lambda\geq0, \pi\leq0\\
                            =\min_{\lambda\in\mathbb{R}^{n_{v}},\pi\in\mathbb{R}^{n_{a}}}~&-[d^{\top},f^{\top}]\left[\begin{matrix}\lambda\\\pi\end{matrix}\right]\\
                            \text{s.t.}~&Ax+B\left[\begin{matrix}\lambda\\\pi\end{matrix}\right]\leq h,
                        \end{aligned}
                    \end{equation}
                    with dual variables $\lambda\in\mathbb{R}^{n_{v}},\pi\in\mathbb{R}^{n_{a}}$ and
                    \begin{equation*}
                        A=\left[\begin{matrix}
                            -\delta'\\
                            ~\\
                            ~
                        \end{matrix}\right]\in\mathbb{R}^{m\times n_{a}},
                        B=\left[\begin{matrix}
                            T' & I_{n_{a}}\\
                            -I_{n_{v}} & \\
                             & I_{n_{a}}
                        \end{matrix}\right]\in\mathbb{R}^{m\times (n_{v}+n_{a})},
                        h=\left[\begin{matrix}
                            c\\
                            ~\\
                            ~
                        \end{matrix}\right]\in\mathbb{R}^{m},
                    \end{equation*}
                    where $m=n_{v}+2n_{a}$, $\delta':=\text{diag}(\delta)$, $T':=T^{\top}$, and $I_{n}$ means an $n\times n$ identity matrix.
                    Then, we have $\text{rank}(B)=n_{v}+n_{a}<m$.
                    Hence, we consider the second condition of Theorem \ref{trm:condition-generalLP-sub} or Theorem \ref{trm:condition-generalLP-super} in this proof.
                    Specifically, we consider a set $\mathcal{I}\subseteq[m]$ with $|\mathcal{I}|=\text{rank}(B)+1=n_{v}+n_{a}+1$, and denote $A_{\mathcal{I}}$ and $B_{\mathcal{I}}$ as
                    \begin{equation*}
                        A_{\mathcal{I}}=\left[\begin{matrix}
                            -\delta'_{\mathcal{I}_{1}}\\
                            ~\\
                            ~
                        \end{matrix}\right]\in\mathbb{R}^{(n_{v}+n_{a}+1)\times n_{a}},
                        B_{\mathcal{I}}=\left[\begin{matrix}
                            T'_{\mathcal{I}_{1}} & I_{n_{a},\mathcal{I}_{1}}\\
                            -I_{n_{v},\mathcal{I}_{2}} & \\
                            & I_{n_{a},\mathcal{I}_{3}}
                        \end{matrix}\right]\in\mathbb{R}^{(n_{v}+n_{a}+1)\times (n_{v}+n_{a})},
                    \end{equation*}
                    where $\mathcal{I}_{1}\subseteq[n_{a}]$, $\mathcal{I}_{2}\subseteq[n_{v}]$, $\mathcal{I}_{3}\subseteq[n_{a}]$ are corresponding to the subsets of $\mathcal{I}$.

                    Following the second condition of Theorem \ref{trm:condition-generalLP-sub} or Theorem \ref{trm:condition-generalLP-super}, we suppose $\text{rank}(B_{\mathcal{I}})=\text{rank}(B)=n_{v}+n_{a}$.
                    It indicates that $B_{\mathcal{I}}$ is of full column rank, and hence, $B_{\mathcal{I}}[\lambda^{\top},\pi^{\top}]^{\top}=0$ with variables $\lambda\in\mathbb{R}^{n_{v}},\pi\in\mathbb{R}^{n_{a}}$ has a unique solution $\lambda=0,\pi=0$.
                    Specifically, the linear system
                    \begin{equation}\label{eq:MCF-cost-I1}
                        \left\{\begin{aligned}
                            &T'_{j}\lambda+\pi_{j}=0,\forall j\in\mathcal{J}_{1}(\mathcal{I})\\
                            &\lambda_{s}=0, \forall s\in\mathcal{S}(\mathcal{I})\\
                            &\pi_{j}=0, \forall j\in\mathcal{J}_{2}(\mathcal{I})
                        \end{aligned}\right.
                    \end{equation}
                    has a unique solution $\lambda=0,\pi=0$, where $\mathcal{S}(\mathcal{I})=\mathcal{I}_{2}$, $\mathcal{J}_{1}(\mathcal{I})=\mathcal{I}_{1}$, and $\mathcal{J}_{2}(\mathcal{I})=\mathcal{I}_{3}$.
                    Then, we introduce an auxiliary variable $p\in\mathbb{R}^{n_{a}}$ and reformulate \eqref{eq:MCF-cost-I1} as
                    \begin{subequations}\label{eq:MCF-cost-I2}
                        \begin{gather}
                            \left\{\begin{aligned}
                                &T'\lambda=-p\\
                                &p_{j}=0, \forall j\in\mathcal{J}(\mathcal{I})\\
                                &\lambda_{s}=0,\forall s\in\mathcal{S}(\mathcal{I})
                            \end{aligned}\right.\label{eq:MCF-cost-I2-1}\\
                            \left\{\begin{aligned}
                                &\pi_{j}=p_{j}, \forall j\in\mathcal{J}_{1}(\mathcal{I})\\
                                &\pi_{j}=0, \forall j\in\mathcal{J}_{2}(\mathcal{I})
                            \end{aligned}\right.\label{eq:MCF-cost-I2-2}
                        \end{gather}
                    \end{subequations}
                    where $\mathcal{J}(\mathcal{I})=\mathcal{J}_{1}(\mathcal{I})\cap\mathcal{J}_{2}(\mathcal{I})$.
                    Obviously, \eqref{eq:MCF-cost-I1} has a unique solution $\lambda=0,\pi=0$, if and only if \eqref{eq:MCF-cost-I2-1} has a unique solution $\lambda=0,p=0$ and $\mathcal{J}_{1}(\mathcal{I})\cup\mathcal{J}_{2}(\mathcal{I})=[n_{a}]$.
                    From the perspective of network potential, $\lambda$ means node potential and $p$ indicates arc potential difference.
                    $T'\lambda=-p$ restricts the cycle potential consistency under potential difference $p$.
                    For $j\in\mathcal{J}(\mathcal{I})$, $p_{j}=0$ means zero potential difference on arc $j$.
                    For $j\notin\mathcal{J}(\mathcal{I})$, the potential difference on arc $j$ is unlimited and is thus free.
                    Hence, we can remove arcs $j,\forall j\notin\mathcal{J}(\mathcal{I})$ from the network.
                    We denote the set of remaining arcs as $\mathcal{A}'(\mathcal{I})$ and obtain a directed network $\mathcal{G}'(\mathcal{I})=\big(\mathcal{V},\mathcal{A}'(\mathcal{I})\big)$.
                    For $s\in\mathcal{S}(\mathcal{I})$, $\lambda_{s}=0$ means zero potential at node $s$.
                    Such nodes with zero potential are called reference nodes and we denote the set of reference nodes as $\mathcal{V}_{\text{ref}}(\mathcal{I})$.
                    With the above understanding in mind, a unique solution $\lambda=0,p=0$ of \eqref{eq:MCF-cost-I2-1} indicates:
                    under cycle potential consistency limits, given the reference node set $\mathcal{V}_{\text{ref}}(\mathcal{I})$, whenever there is no arc potential difference in $\mathcal{G}'(\mathcal{I})$, the potential of each node in $\mathcal{G}'(\mathcal{I})$ is zero.
                    Therefore, according to Lemma \ref{lem:pps:general-MCF-cost}, each component of $\mathcal{G}'(\mathcal{I})$ contains at least one reference node in $\mathcal{V}_{\text{ref}}(\mathcal{I})$.

                    In the following, we derive the necessary and sufficient conditions for the submodularity and supermodularity of $\phi(x)$ defined by \eqref{eq:MCF-cost}, respectively.

                    \vspace{2ex}
                    \noindent\textbf{(i) Submodularity.}

                    $\phi(x)$ defined by \eqref{eq:MCF-cost} is submodular if and only if $-\phi(x)$ in \eqref{eq:MCF-cost2} is supermodular.
                    Following the second condition of Theorem \ref{trm:condition-generalLP-sub}, we suppose an $x\in\mathbb{R}^{n_{a}}_{+}$ such that $A_{\mathcal{I}}x\in\text{span}(B_{\mathcal{I}})$.
                    It implies that there exist a $\lambda\in\mathbb{R}^{n_{v}}$ and a $\pi\in\mathbb{R}^{n_{a}}$ such that $A_{\mathcal{I}}x=B_{\mathcal{I}}[\lambda^{\top},\pi^{\top}]^{\top}$.
                    Specifically, the linear system
                    \begin{subequations}\label{eq:MCF-cost-sub-x}
                        \begin{gather}
                            \left\{\begin{aligned}
                                &T'\lambda=-p\\
                                &p_{j}=\delta_{j}x_{j},\forall j\in\mathcal{J}(\mathcal{I})\\
                                &\lambda_{s}=0, \forall s\in\mathcal{S}(\mathcal{I})
                            \end{aligned}\right.\label{eq:MCF-cost-sub-x-1}\\
                            \left\{\begin{aligned}
                                &\pi_{j}=p_{j}-\delta_{j}x_{j},\forall j\in\mathcal{J}_{1}(\mathcal{I})\\
                                &\pi_{j}=0, \forall j\in\mathcal{J}_{2}(\mathcal{I}),
                            \end{aligned}\right.\label{eq:MCF-cost-sub-x-2}
                        \end{gather}
                    \end{subequations}
                    is feasible.
                    Obviously, \eqref{eq:MCF-cost-sub-x} is feasible if and only if \eqref{eq:MCF-cost-sub-x-1} is feasible.
                    According to the second condition of Theorem \ref{trm:condition-generalLP-sub}, to guarantee the supermodularity of \eqref{eq:MCF-cost2}, we need to guarantee $A_{\mathcal{I},k}x_{k}\in\text{span}(B_{\mathcal{I}})$ for any $k\in[n_{a}]$.
                    It implies that for any $k\in[n_{a}]$, there exist a $\lambda'\in\mathbb{R}^{n_{v}}$ and a $\pi'\in\mathbb{R}^{n_{a}}$ such that $A_{\mathcal{I},k}x_{k}=B_{\mathcal{I}}[\lambda'^{\top},\pi'^{\top}]^{\top}$.
                    If $k\notin\mathcal{J}_{1}(\mathcal{I})$, according to the structure of $A_{\mathcal{I}}$, we have $A_{\mathcal{I},k}=0$, and hence, $\lambda'=0,\pi'=0$ is feasible.
                    If $k\in\mathcal{J}_{1}(\mathcal{I})$, $A_{\mathcal{I},k}x_{k}=B_{\mathcal{I}}[\lambda'^{\top},\pi'^{\top}]^{\top}$ implies that the linear system
                    \begin{subequations}\label{eq:MCF-cost-sub-y}
                        \begin{gather}
                            \left\{\begin{aligned}
                                &T'\lambda'=-p'\\
                                &p'_{k}=\delta_{k}x_{k},\text{if}~k\in\mathcal{J}(\mathcal{I})\\
                                &p'_{j}=0,\forall j\in\mathcal{J}(\mathcal{I})\backslash\{k\}\\
                                &\lambda'_{s}=0, \forall s\in\mathcal{S}(\mathcal{I})
                            \end{aligned}\right.\label{eq:MCF-cost-sub-y-1}\\
                            \left\{\begin{aligned}
                                &\pi'_{k}=p'_{k}-\delta_{k}x_{k}\\
                                &\pi'_{j}=p'_{j},\forall j\in\mathcal{J}_{1}(\mathcal{I})\backslash\{k\}\\
                                &\pi'_{j}=0, \forall j\in\mathcal{J}_{2}(\mathcal{I}),
                            \end{aligned}\right.\label{eq:MCF-cost-sub-y-2}
                        \end{gather}
                    \end{subequations}
                    if feasible.
                    Obviously, \eqref{eq:MCF-cost-sub-y} is feasible if and only if \eqref{eq:MCF-cost-sub-y-1} is feasible.
                    If $k\notin\mathcal{J}(\mathcal{I})$, $\lambda'=0,p'=0$ is always feasible.
                    If $k\notin\mathcal{J}_{a}$, where $\mathcal{J}_{a}:=\{j\in[n_{a}]:\delta_{j}>0\}$ is corresponding to the set $\mathcal{A}_{a}$ of attackable arcs, we have $\delta_{k}=0$ and thus $\lambda'=0,p'=0$ is feasible.
                    Therefore, we only consider $k\in\mathcal{J}_{a}\cap\mathcal{J}(\mathcal{I})$ in the following.
                    Comparing \eqref{eq:MCF-cost-sub-x-1} and \eqref{eq:MCF-cost-sub-y-1} to \eqref{eq:MCF-cost-I2-1}, we find:
                    \eqref{eq:MCF-cost-sub-x-1} modifies the potential difference at arc $j$ from 0 to $\delta_{j}x_{j}$ for all $j\in\mathcal{J}(\mathcal{I})$;
                    \eqref{eq:MCF-cost-sub-y-1} modifies the potential difference at the only arc $k$ from 0 to $\delta_{k}x_{k}$.
                    Either \eqref{eq:MCF-cost-sub-x-1} or \eqref{eq:MCF-cost-sub-y-1} is feasible if and only if under such modifications, the cycle consistency still holds for each component of $\mathcal{G}'(\mathcal{I})$.
                    Specifically, for a certain component $g$ in $\mathcal{G}'(\mathcal{I})$, we denote its node index set by $\mathcal{S}_{g}\subseteq[n_{v}]$ and arc index set by $\mathcal{J}_{g}\subseteq\mathcal{J}(\mathcal{I})$.
                    For a path $P$ in component $g$, we denote its origin node and destination node by $v_{o}^P$ and $v_{d}^P$, respectively.
                    Considering that there may exist multiple reference nodes in component $g$, the cycle potential consistency holds when for any path $P$ whose origin and destination nodes are the same ($v_{o}^P=v_{d}^P$) or are both reference nodes ($v_{o}^P,v_{d}^P\in\mathcal{V}_{\text{ref}}(\mathcal{I})$), we have
                    \begin{equation}\label{eq:MCF-cost-sub-cycle}
                        \sum_{j\in\mathcal{J}_{P}}\alpha_{j}(P)p_{j}=0,
                    \end{equation}
                    where $\mathcal{J}_{P}\subseteq\mathcal{J}_{g}$ is the index set of arcs in $P$;
                    $\alpha_{j}\in\{-1,0,1\}$ indicates the direction of arc $j$.
                    $\alpha_{j}(P)=0$ means that arc $j$ is not in $P$;
                    $\alpha_{j}(P)=1$ and $\alpha_{j}(P)=-1$ indicate that arc $j$ aligns with the forward and reverse sequence of $P$, respectively.
                    Given \eqref{eq:MCF-cost-sub-x-1}, \eqref{eq:MCF-cost-sub-cycle} reduces to
                    \begin{equation}\label{eq:MCF-cost-sub-cycle-x}
                        \sum_{j\in\mathcal{J}_{a}\cap\mathcal{J}_{P}}\alpha_{j}(P)\delta_{j}x_{j}=0.
                    \end{equation}
                    Given \eqref{eq:MCF-cost-sub-y-1}, \eqref{eq:MCF-cost-sub-cycle} reduces to
                    \begin{equation}\label{eq:MCF-cost-sub-cycle-y}
                        \sum_{j\in\mathcal{J}_{a}\cap\mathcal{J}_{P}\cap\{k\}}\alpha_{j}(P)\delta_{j}x_{j}=0.
                    \end{equation}

                    From the above analysis, to guarantee the supermodularity of \eqref{eq:MCF-cost2}, we need to guarantee that
                    \eqref{eq:MCF-cost-sub-cycle-y} holds for any $k\in\mathcal{J}_a\cap\mathcal{J}(\mathcal{I})$ whenever \eqref{eq:MCF-cost-sub-cycle-x} holds.
                    When $k\notin\mathcal{J}_{P}$, \eqref{eq:MCF-cost-sub-cycle-y} always holds.
                    Hence, we only consider $k\in\mathcal{J}_{a}\cap\mathcal{J}_{P}$ in the following, and \eqref{eq:MCF-cost-sub-cycle-y} reduces to
                    \begin{equation}\label{eq:MCF-cost-sub-cycle-yy}
                        \delta_{k}x_{k}=0.
                    \end{equation}
                    We next prove that \eqref{eq:MCF-cost-sub-cycle-yy} holds for any $k\in\mathcal{J}_a\cap\mathcal{J}_P$ whenever \eqref{eq:MCF-cost-sub-cycle-x} holds if and only if any two arcs in $\mathcal{A}_{a}\cap\mathcal{A}_{P}$ have the same direction.

                    \noindent$\Leftarrow$:
                    If any two arcs in $\mathcal{A}_{a}\cap\mathcal{A}_{P}$ have the same direction, we have that all arcs in $\mathcal{A}_{a}\cap\mathcal{A}_{P}$ have the same direction, i.e., either $\alpha_{j}(P)=1,\forall j\in\mathcal{J}_{a}\cap\mathcal{J}_{P}$ or $\alpha_{j}(P)=-1,\forall j\in\mathcal{J}_{a}\cap\mathcal{J}_{P}$.
                    Then, \eqref{eq:MCF-cost-sub-cycle-x} reduces to
                    \begin{equation*}
                        \sum_{j\in\mathcal{J}_{a}\cap\mathcal{J}_{P}}\delta_{j}x_{j}=0.
                    \end{equation*}
                    Because $x\geq0,\delta\geq0$, we have $\delta_{j}x_{j}=0$ for all $j\in\mathcal{J}_{a}\cap\mathcal{J}_{P}$.
                    Hence, $\delta_{k}x_{k}=0$ and \eqref{eq:MCF-cost-sub-cycle-yy} holds.

                    \noindent$\Rightarrow$:
                    Suppose the contrary that there exist two arcs in $\mathcal{A}_{a}\cap\mathcal{A}_{P}$ with opposite directions, indexed by $j_{1},j_{2}\in\mathcal{J}_{a}\cap\mathcal{J}_{P}$.
                    We have $\alpha_{j_{1}}(P)=-\alpha_{j_{2}}(P)\neq0$.
                    Hence, for any $\xi>0$, \eqref{eq:MCF-cost-sub-cycle-x} holds when $x_{j_{1}}=\xi/\delta_{j_{1}}$, $x_{j_{2}}=\xi/\delta_{j_{2}}$, and $x_{j}=0,\forall j\in(\mathcal{J}_{a}\cap\mathcal{J}_{P})\backslash\{j_{1},j_{2}\}$.
                    Yet, when either $k=j_{1}$ or $k=j_{2}$, \eqref{eq:MCF-cost-sub-cycle-yy} does not hold, which contradicts that \eqref{eq:MCF-cost-sub-cycle-yy} holds for any $k\in\mathcal{J}_a\cap\mathcal{J}_P$ whenever \eqref{eq:MCF-cost-sub-cycle-x} holds.

                    When $\mathcal{I}$ varies, a component $g$ of $\mathcal{G}'(\mathcal{I})$ may be any connected subgraph of $\mathcal{G}$ and $\mathcal{V}_{\text{ref}}(\mathcal{I})$ may contains any node in $\mathcal{G}$.
                    Hence, a path $P$ in component $g$ with $v_{o}^P=v_{d}^P$ or $v_{o}^P,v_{d}^P\in\mathcal{V}_{\text{ref}}(\mathcal{I})$ may be any path in $\mathcal{G}$.
                    Therefore, $\forall\mathcal{I}\subseteq[m]$ with $|\mathcal{I}|=\text{rank}(B)+1$ such that $\text{rank}(B_{\mathcal{I}})=\text{rank}(B)$, for any path $P$ in any component $g$ of $\mathcal{G}'(\mathcal{I})$ with $v_{o}^P=v_{d}^P$ or $v_{o}^P,v_{d}^P\in\mathcal{V}_{\text{ref}}(\mathcal{I})$, any two arcs in $\mathcal{A}_{a}\cap\mathcal{A}_{P}$ have the same direction,
                    if and only if
                    \begin{condition}\label{cdt:cost-sub}
                        for any path $P$ in $\mathcal{G}$, any two arcs in $\mathcal{A}_{a}\cap\mathcal{A}_{P}$ have the same direction in $P$.
                    \end{condition}

                    Through the above proof and by the second condition of Theorem \ref{trm:condition-generalLP-sub}, we prove that $-\phi(x)$ defined by \eqref{eq:MCF-cost2} is supermodular for any $-[d^{\top},f^{\top}]^{\top}$ and $h$, if and only if Condition \ref{cdt:cost-sub} holds.
                    Therefore, $\phi(x)$ defined by \eqref{eq:MCF-cost} is submodular for any $c$, $d$, and $f$, if and only if for Condition \ref{cdt:cost-sub} holds.

                    \vspace{2ex}
                    \noindent\textbf{(ii) Supermodularity.}

                    $\phi(x)$ defined by \eqref{eq:MCF-cost} is supermodular if and only if $-\phi(x)$ in \eqref{eq:MCF-cost2} is submodular.
                    Following the second condition of Theorem \ref{trm:condition-generalLP-super}, we suppose an $x\in\mathbb{R}^{n_{a}}$ such that $A_{\mathcal{I}}x\in\text{span}(B_{\mathcal{I}})$.
                    It implies that there exist a $\lambda\in\mathbb{R}^{n_{v}}$ and a $\pi\in\mathbb{R}^{n_{a}}$ such that $A_{\mathcal{I}}x=B_{\mathcal{I}}[\lambda^{\top},\pi^{\top}]^{\top}$.
                    Specifically, the linear system
                    \begin{subequations}\label{eq:MCF-cost-super-x}
                        \begin{gather}
                            \left\{\begin{aligned}
                                &T'\lambda=-p\\
                                &p_{j}=\delta_{j}x_{j},\forall j\in\mathcal{J}(\mathcal{I})\\
                                &\lambda_{s}=0, \forall s\in\mathcal{S}(\mathcal{I})
                            \end{aligned}\right.\label{eq:MCF-cost-super-x-1}\\
                            \left\{\begin{aligned}
                                &\pi_{j}=p_{j}-\delta_{j}x_{j},\forall j\in\mathcal{J}_{1}(\mathcal{I})\\
                                &\pi_{j}=0, \forall j\in\mathcal{J}_{2}(\mathcal{I})
                            \end{aligned}\right.\label{eq:MCF-cost-super-x-2}
                        \end{gather}
                    \end{subequations}
                    is feasible.
                    Obviously, \eqref{eq:MCF-cost-super-x} is feasible if and only if \eqref{eq:MCF-cost-super-x-1} is feasible.
                    According to the second condition of Theorem \ref{trm:condition-generalLP-super}, to guarantee the submodularity of \eqref{eq:MCF-cost2}, we need to guarantee $A_{\mathcal{I}}x^{+}\in\text{span}(B_{\mathcal{I}})$.
                    It implies that there exist a $\lambda'\in\mathbb{R}^{n_{v}}$ and a $\pi'\in\mathbb{R}^{n_{a}}$ such that $A_{\mathcal{I}}x^{+}=B_{\mathcal{I}}[\lambda'^{\top},\pi'^{\top}]^{\top}$.
                    Specifically, the linear system
                    \begin{subequations}\label{eq:MCF-cost-super-y}
                        \begin{gather}
                            \left\{\begin{aligned}
                                &T'\lambda'=-p'\\
                                &p'_{j}=\delta_{j}x^{+}_{j}, \forall j\in\mathcal{J}(\mathcal{I})\\
                                &\lambda'_{s}=0, \forall s\in\mathcal{S}(\mathcal{I})
                            \end{aligned}\right.\label{eq:MCF-cost-super-y-1}\\
                            \left\{\begin{aligned}
                                &\pi'_{j}=p'_{j}-\delta_{j}x^{+}_{j},\forall j\in\mathcal{J}_{1}(\mathcal{I})\\
                                &\pi'_{j}=0, \forall j\in\mathcal{J}_{2}(\mathcal{I}),
                            \end{aligned}\right.\label{eq:MCF-cost-super-y-2}
                        \end{gather}
                    \end{subequations}
                    is feasible.
                    Obviously, \eqref{eq:MCF-cost-super-y} is feasible if and only if \eqref{eq:MCF-cost-super-y-1} is feasible.
                    Comparing \eqref{eq:MCF-cost-super-x-1} and \eqref{eq:MCF-cost-super-y-1} to \eqref{eq:MCF-cost-I2-1}, we find:
                    \eqref{eq:MCF-cost-super-x-1} modifies the potential difference at arc $j$ from 0 to $\delta_{j}x_{j}$ for all $j\in\mathcal{J}(\mathcal{I})$;
                    \eqref{eq:MCF-cost-super-y-1} modifies the potential difference at arc $j$ from 0 to $\delta_{j}x^{+}_{j}$ for all $j\in\mathcal{J}(\mathcal{I})$.
                    Either \eqref{eq:MCF-cost-super-x-1} or \eqref{eq:MCF-cost-super-y-1} is feasible if and only if under such modifications, the cycle potential consistency still holds for each component of $\mathcal{G}'(\mathcal{I})$.
                    Specifically, for a certain component $g$ in $\mathcal{G}'(\mathcal{I})$, we denote its node index set by $\mathcal{S}_{g}\subseteq[n_{v}]$ and arc index set by $\mathcal{J}_{g}\subseteq\mathcal{J}(\mathcal{I})$.
                    For a path $P$ in component $g$, we denote its origin node and destination node by $v_{o}^P$ and $v_{d}^P$, respectively.
                    Considering that there may exist multiple reference nodes in component $g$, the cycle potential consistency holds when for any path $P$ whose origin node and destination node are the same ($v_{o}^P=v_{d}^P$) or are both reference nodes ($v_{o}^P,v_{d}^P\in\mathcal{V}_{\text{ref}}(\mathcal{I})$), we have
                    \begin{equation}\label{eq:MCF-cost-super-cycle}
                        \sum_{j\in\mathcal{J}_{P}}\alpha_{j}(P)p_{j}=0,
                    \end{equation}
                    where $\mathcal{J}_{P}\subseteq\mathcal{J}_{g}$ is the index set of arcs in $P$;
                    $\alpha_{j}\in\{-1,0,1\}$ indicates the direction of arc $j$.
                    $\alpha_{j}(P)=0$ means that arc $j$ is not in $P$.
                    $\alpha_{j}(P)=1$ and $\alpha_{j}(P)=-1$ indicate that arc $j$ aligns with the forward and reverse sequence of $P$, respectively.
                    Given \eqref{eq:MCF-cost-super-x-1}, \eqref{eq:MCF-cost-super-cycle} reduces to
                    \begin{equation}\label{eq:MCF-cost-super-cycle-x}
                        \sum_{j\in\mathcal{J}_a\cap\mathcal{J}_{P}}\alpha_{j}(P)\delta_{j}x_{j}=0,
                    \end{equation}
                    where $\mathcal{J}_{a}:=\{j\in[n_{a}]:\delta_{j}>0\}$ is corresponding to the set $\mathcal{A}_{a}$ of attackable arcs.
                    Given \eqref{eq:MCF-cost-super-y-1}, \eqref{eq:MCF-cost-super-cycle} reduces to
                    \begin{equation}\label{eq:MCF-cost-super-cycle-y}
                        \sum_{j\in\mathcal{J}_a\cap\mathcal{J}_{P}}\alpha_{j}(P)\delta_{j}x^{+}_{j}=0.
                    \end{equation}

                    From the above analysis, to guarantee the submodularity of \eqref{eq:MCF-cost2}, we need to guarantee that
                    \eqref{eq:MCF-cost-super-cycle-y} holds whenever \eqref{eq:MCF-cost-super-cycle-x} holds.
                    We next prove that \eqref{eq:MCF-cost-super-cycle-y} holds whenever \eqref{eq:MCF-cost-super-cycle-x} holds, if and only if any two arcs in $\mathcal{A}_{a}\cap\mathcal{A}_{P}$ have opposite directions.

                    \noindent$\Leftarrow$:
                    We note that $\mathcal{A}_{a}\cap\mathcal{A}_{P}$ consists of at most two arcs.\\
                    If $\mathcal{A}_{a}\cap\mathcal{A}_{P}=\emptyset$, \eqref{eq:MCF-cost-super-cycle-y} holds.\\
                    If $\mathcal{A}_{a}\cap\mathcal{A}_{P}$ consists of one arc, indexed by $j_{0}$, we have $\mathcal{J}_a\cap\mathcal{J}_{P}=\{j_0\}$.
                    Then, \eqref{eq:MCF-cost-super-cycle-x} reduces to $\delta_{j_{0}}x_{j_{0}}=0$.
                    Because $\delta_{j_{0}}>0$, we have $x_{j_{0}}=0$, and hence, \eqref{eq:MCF-cost-super-cycle-y} reduces to $\delta_{j_{0}}x^+_{j_{0}}=0$ and holds.\\
                    If $\mathcal{A}_{a}\cap\mathcal{A}_{P}$ consists of two arcs with opposites directions, indexed by $j_{1},j_{2}$, we have $\mathcal{J}_a\cap\mathcal{J}_{P}=\{j_{1},j_{2}\}$ and  $\alpha_{j_{1}}(P)=-\alpha_{j_{2}}(P)\neq0$.
                    Then, \eqref{eq:MCF-cost-super-cycle-x} reduces to $\delta_{j_{1}}x_{j_{1}}=\delta_{j_{2}}x_{j_{2}}$.
                    Because $\delta_{j_{1}}>0,\delta_{j_{2}}>0$, we have $x_{j_{1}}\geq0,x_{j_{2}}\geq0$ or $x_{j_{1}}\leq0,x_{j_{2}}\leq0$.
                    \eqref{eq:MCF-cost-super-cycle-y} reduces to $\delta_{j_{1}}x^{+}_{j_{1}}=\delta_{j_{2}}x^{+}_{j_{2}}$, which holds when either $x_{j_{1}}\geq0,x_{j_{2}}\geq0$ or $x_{j_{1}}\leq0,x_{j_{2}}\leq0$.

                    \noindent$\Rightarrow$:
                    Suppose the contrary that there exist two arcs in $\mathcal{A}_{a}\cap\mathcal{A}_{P}$ with the same direction, indexed by $j_{1},j_{2}$.
                    We have $\mathcal{J}_a\cap\mathcal{J}_{P}=\{j_{1},j_{2}\}$ and $\alpha_{j_{1}}(P)=\alpha_{j_{2}}(P)\neq0$.
                    Hence, for any $\xi\neq0$, \eqref{eq:MCF-cost-super-cycle-x} holds when $x_{j_{1}}=\xi/\delta_{j_{1}}$, $x_{j_{2}}=-\xi/\delta_{j_{2}}$, and $x_{j}=0,\forall j\in(\mathcal{J}_a\cap\mathcal{J}_{P})\backslash\{j_{1},j_{2}\}$.
                    Yet, because $\delta_{j_{1}}>0,\delta_{j_{2}}>0$, \eqref{eq:MCF-cost-super-cycle-y} does not hold, which contradicts that \eqref{eq:MCF-cost-super-cycle-y} holds whenever \eqref{eq:MCF-cost-super-cycle-x} holds.

                    When $\mathcal{I}$ varies, a component $g$ of $\mathcal{G}'(\mathcal{I})$ may be any connected subgraph of $\mathcal{G}$ and $\mathcal{V}_{\text{ref}}(\mathcal{I})$ may contains any node in $\mathcal{G}$.
                    Hence, a path $P$ in component $g$ with $v_{o}^P=v_{d}^P$ or $v_{o}^P,v_{d}^P\in\mathcal{V}_{\text{ref}}(\mathcal{I})$ may be any path in $\mathcal{G}$.
                    Therefore, $\forall\mathcal{I}\subseteq[m]$ with $|\mathcal{I}|=\text{rank}(B)+1$ such that $\text{rank}(B_{\mathcal{I}})=\text{rank}(B)$, for any path $P$ in any component $g$ of $\mathcal{G}'(\mathcal{I})$ with $v_{o}^P=v_{d}^P$ or $v_{o}^P,v_{d}^P\in\mathcal{V}_{\text{ref}}(\mathcal{I})$, any two arcs in $\mathcal{A}_{a}\cap\mathcal{A}_{P}$ have opposites directions,
                    if and only if
                    \begin{condition}\label{cdt:cost-super}
                        for any path $P$ in $\mathcal{G}$, any two arcs in $\mathcal{A}_{a}\cap\mathcal{A}_{P}$ have opposites directions in $P$.
                    \end{condition}

                    Through the above proof and by the second condition of Theorem \ref{trm:condition-generalLP-super}, we prove that $-\phi(x)$ defined by \eqref{eq:MCF-cost2} is submodular for any $-[d^{\top},f^{\top}]^{\top}$ and $h$, if and only if Condition \ref{cdt:cost-super} holds.
                    Therefore, $\phi(x)$ defined by \eqref{eq:MCF-cost} is supermodular for any $c$, $d$, and $f$, if and only if Condition \ref{cdt:cost-super} holds.
                \end{proof}

                \begin{lemma}\label{lem:pps:general-MCF-cost}
                    Given a directed network $\mathcal{G}=(\mathcal{V},\mathcal{A})$ and a node set $\mathcal{V}_{\text{ref}}\subseteq\mathcal{V}$.
                    Under cycle potential consistency limits, whenever there is no arc potential difference in $\mathcal{G}$ and the potential of each node in $\mathcal{V}_{\text{ref}}$ is zero, the potential of each node in $\mathcal{G}$ is zero,
                    if and only if
                    each component of $\mathcal{G}$ contains at least one node in $\mathcal{V}_{\text{ref}}$.
                \end{lemma}

                \begin{proof}
                    ``$\Leftarrow$'':
                    Because different components of $\mathcal{G}$ are independent, we consider a certain component $g$ for example.
                    Under cycle potential consistency limits, when there is no arc potential difference in $\mathcal{G}$, the potential of each nodes in $g$ is the same.
                    According to the condition, component $g$ contains at least one node in $\mathcal{V}_{\text{ref}}$.
                    When the potential of each node in $\mathcal{V}_{\text{ref}}$ is zero, component $g$ contains at least one zero-potential node.
                    Therefore, the potential of each node in component $g$ is zero.

                    ``$\Rightarrow$'':
                    Suppose the contrary that a certain component $g$ of $\mathcal{G}$ contains no node in $\mathcal{V}_{\text{ref}}$.
                    Under cycle potential consistency limits, when there is no arc potential difference in $\mathcal{G}$, the potential of each nodes in $g$ is the same.
                    Because component $g$ contains no node in $\mathcal{V}_{\text{ref}}$, even when the potential of each node in $\mathcal{V}_{\text{ref}}$ is zero, the potential of nodes in $g$ is not necessarily zero.
                    This contradicts that the potential of each node in $\mathcal{G}$ is zero.
                \end{proof}

            \subsection{Proof of Corollary \ref{crl:MCF-imply-capacity}}\label{apd:crl:MCF-imply-capacity}
                \begin{proof}
                    Starting from \eqref{eq:MCF-capacity}, we replace $y$ by $y+y'$ and rewrite \eqref{eq:MCF-capacity} as
                    \begin{equation}\label{eq:MCF-capacity-imply1}
                        \begin{aligned}
                            \phi(x)=\min_{y,y'\in\mathbb{R}^{n_{a}}}~&c^{\top}(y+y')\\
                            \text{s.t.}~&T(y+y')\geq d\\
                            &0\leq y+y'\leq f-\text{diag}(\delta)x,
                        \end{aligned}
                    \end{equation}
                    which is equivalent to
                    \begin{equation}\label{eq:MCF-capacity-imply2}
                        \begin{aligned}
                            \phi(x)=\min_{y,y'\in\mathbb{R}^{n_{a}}}~&c^{\top}(y+y')\\
                            \text{s.t.}~&T(y+y')\geq d\\
                            &0\leq y\leq f-\delta\\
                            &0\leq y'\leq \delta-\text{diag}(\delta)x.
                        \end{aligned}
                    \end{equation}
                    Given $x\in\{0,1\}^{n_a}$, we have $\{y'\in\mathbb{R}^{n_a}:0\leq y'\leq \delta-\text{diag}(\delta)x\}=\{y'\in\mathbb{R}^{n_a}:x_{j}y'_{j}=0, \forall i\in\mathcal{J}_a,0\leq y'\leq \delta\}$, where $\mathcal{J}_a:=\{j\in[n_a]:a_j\in\mathcal{A}_a\}$.
                    Hence, \eqref{eq:MCF-capacity-imply2} is equivalent to
                    \begin{equation}\label{eq:MCF-capacity-imply3}
                        \begin{aligned}
                            \phi(x)=\min_{y,y'\in\mathbb{R}^{n_{a}}}~&c^{\top}(y+y')\\
                            \text{s.t.}~&T(y+y')\geq d\\
                            &0\leq y\leq f-\delta\\
                            &0\leq y'\leq \delta\\
                            &x_{j}y'_{j}=0, \forall j\in\mathcal{J}_a.
                        \end{aligned}
                    \end{equation}
                    Then, we relax the constraint $x_{i}y'_{i}=0, \forall i\in[n_a]$ and modify \eqref{eq:MCF-capacity-imply3} as
                    \begin{equation}\label{eq:MCF-capacity-imply4}
                        \begin{aligned}
                            \phi(x)=\min_{y,y'\in\mathbb{R}^{n_{a}}}~&c^{\top}(y+y')+\sum_{j\in\mathcal{J}_a}Mx_{j}y'_{j}\\
                            \text{s.t.}~&T(y+y')\geq d\\
                            &0\leq y\leq f-\delta\\
                            &0\leq y'\leq \delta,
                        \end{aligned}
                    \end{equation}
                    where $M\in\mathbb{R}$ is a constant.
                    When $M$ is sufficiently large, \eqref{eq:MCF-capacity-imply4} is equivalent to \eqref{eq:MCF-capacity-imply3}.
                    We rewrite \eqref{eq:MCF-capacity-imply4} as
                    \begin{equation}\label{eq:MCF-capacity-imply}
                        \begin{aligned}
                            \phi(x)=\min_{y,y'\in\mathbb{R}^{n_{a}}}~&c^{\top}y+(c+\text{diag}(\delta')x)^{\top}y'\\
                            \text{s.t.}~&T(y+y')\geq d\\
                            &0\leq y\leq f-\delta\\
                            &0\leq y'\leq\delta,
                        \end{aligned}
                    \end{equation}
                    where $\delta'\in\Delta(\mathcal{A}_a)$ with $\delta'_j=M$ if $j\in\mathcal{J}_a$ and $\delta'_j=0$ if $j\notin\mathcal{J}_a$.
                    From the above analysis, we have that \eqref{eq:MCF-capacity-imply} is equivalent to \eqref{eq:MCF-capacity}.

                    Furthermore, we observe that \eqref{eq:MCF-capacity-imply} aligns with the form of \eqref{eq:MCF-cost}.
                    Topologically, in \eqref{eq:MCF-capacity-imply}, for each original arc $j$ in network $\mathcal{G}$, a parallel dummy arc is added.
                    Without loss of generality, we assume that $y$ indicate the flow on dummy arcs with cost coefficient $c$ and flow capacity $f-\delta$;
                    $y'$ indicate the flow on original arcs with cost coefficient $c$ and flow capacity $\delta$.
                    For all $j\in[n_a]$, $x_j=1$ indicates that the cost coefficient of original arc $j$ is attacked and is inflated by $\delta'_j$.
                    Because $\delta'\in\Delta(\mathcal{A}_a)$, only original arcs in $\mathcal{A}_a$ are attackable.
                    Therefore, given $T$ and $\delta$, for any $c$, $d$, and $f$, we can identify the submodularity or supermodularity of $\phi(x)$ defined in \eqref{eq:MCF-capacity-imply} by the conditions in Proposition \ref{pps:general-MCF-cost}.
                    Considering that \eqref{eq:MCF-capacity-imply} is equivalent to \eqref{eq:MCF-capacity}, the conditions in Proposition \ref{pps:general-MCF-cost} also applies to \eqref{eq:MCF-capacity}, which implies Proposition \ref{pps:general-MCF-capacity}.
                \end{proof}

            \subsection{Proof of Corollary \ref{crl:MCF-imply-demand}}\label{apd:crl:MCF-imply-demand}
                \begin{proof}
                    Starting from \eqref{eq:MCF-demand}, we introduce an auxiliary variable $y'\in\mathbb{R}^{n_v}$ and reformulate \eqref{eq:MCF-demand} as
                    \begin{equation}\label{eq:MCF-demand-imply1}
                        \begin{aligned}
                            \phi(x)=\min_{y\in\mathbb{R}^{n_{a}},y'\in\mathbb{R}^{n_{v}}}~&c^{\top}y\\
                            \text{s.t.}~&Ty\geq d+\delta-y'\\
                            &y'\leq\delta-\text{diag}(\delta)x\\
                            &0\leq y\leq f\\
                            &0\leq y'\leq \delta.
                        \end{aligned}
                    \end{equation}
                    Given $x\in\{0,1\}^{n_v}$, we have $\{y'\in\mathbb{R}^{n_v}:y'\leq\delta-\text{diag}(\delta)x,0\leq y'\leq \delta\}=\{y'\in\mathbb{R}^{n_v}:x_{i}y'_{i}=0, \forall i\in[n_v],0\leq y'\leq \delta\}$.
                    Hence, \eqref{eq:MCF-demand-imply1} is equivalent to
                   \begin{equation}\label{eq:MCF-demand-imply2}
                        \begin{aligned}
                            \phi(x)=\min_{y\in\mathbb{R}^{n_{a}},y'\in\mathbb{R}^{n_{v}}}~&c^{\top}y\\
                            \text{s.t.}~&Ty\geq d+\delta-y'\\
                            &x_{i}y'_{i}=0, \forall i\in[n_v]\\
                            &0\leq y\leq f\\
                            &0\leq y'\leq \delta.
                        \end{aligned}
                    \end{equation}
                    Then, we relax the constraint $x_{i}y'_{i}=0, \forall i\in[n_v]$ and modify \eqref{eq:MCF-demand-imply2} as
                    \begin{equation}\label{eq:MCF-demand-imply3}
                        \begin{aligned}
                            \phi(x)=\min_{y\in\mathbb{R}^{n_{a}},y'\in\mathbb{R}^{n_{v}}}~&c^{\top}y+\sum_{i\in[n_v]}Mx_{i}y'_{i}\\
                            \text{s.t.}~&Ty\geq d+\delta-y'\\
                            &0\leq y\leq f\\
                            &0\leq y'\leq \delta,
                        \end{aligned}
                    \end{equation}
                    where $M$ is a constant.
                    When $M$ is sufficiently large, \eqref{eq:MCF-demand-imply3} is equivalent to \eqref{eq:MCF-demand-imply2}.
                    We rewrite \eqref{eq:MCF-demand-imply3} as
                    \begin{equation}\label{eq:MCF-demand-imply}
                        \begin{aligned}
                            \phi(x)=\min_{y\in\mathbb{R}^{n_{a}},y'\in\mathbb{R}^{n_{v}}}~&c^{\top}y+(0+\text{diag}(M)x)^{\top}y'\\
                            \text{s.t.}~&Ty+T'y'\geq d+\delta\\
                            &T_{d}y+T'_{d}y'\geq -\delta\\
                            &0\leq y\leq f\\
                            &0\leq y'\leq \delta,
                        \end{aligned}
                    \end{equation}
                    where $T'=I_{n_v}$, $T_d=0_{n_v\times n_a}$, and $T'_d=-I_{n_v}$.
                    From the above analysis, we have that \eqref{eq:MCF-demand-imply} is equivalent to \eqref{eq:MCF-demand}.

                    Furthermore, we observe that \eqref{eq:MCF-demand-imply} aligns with the form of \eqref{eq:MCF-cost}.
                    Topologically, in \eqref{eq:MCF-demand-imply}, for each node $v_i$ in network $\mathcal{G}$, a dummy supply node $v'_i$ and a dummy arc from $v'_i$ to $v_i$ are added.
                    The demand on node $v_i$ is modified into $d_i+\delta_i$ and the supply capacity on node $v'_i$ is set to be $\delta_i$.
                    $y'$ indicate the flow on dummy arcs with zero cost coefficient and flow capacity $\delta$.
                    $T'$, $T_d$, and $T'_d$ means the incidence matrix between original nodes and dummy arcs, incidence matrix between dummy nodes and real arcs, and incidence matrix between dummy nodes and dummy arcs, respectively.
                    For all $i\in[n_v]$, $x_i=1$ indicates that the cost coefficient of dummy arc $i$ is attacked and is inflated by $M$.
                    Therefore, given $T$ and $\delta$, for any $c$, $d$, and $f$, we can identify the submodularity or supermodularity of $\phi(x)$ defined in \eqref{eq:MCF-demand-imply} by the conditions in Proposition \ref{pps:general-MCF-cost}.
                    Because each dummy arc has only one incident node in the original network and the incident node of each dummy arc is a head node, according to Lemma \ref{lem:pps:general-MCF-capacity-opposite}, any two dummy arcs have opposite directions.
                    Following Proposition \ref{pps:general-MCF-cost}, we have $\phi(x)$ defined in \eqref{eq:MCF-demand-imply} is supermodular.
                    Considering that \eqref{eq:MCF-demand-imply} is equivalent to \eqref{eq:MCF-demand}, $\phi(x)$ defined in \eqref{eq:MCF-demand} is supermodular, which implies Proposition \ref{pps:general-MCF-demand}.
                \end{proof}

            \subsection{Proof of Proposition \ref{pps:general-MCF-mixed}}\label{apd:pps:general-MCF-mixed}
                \begin{proof}
                    Because $\mathcal{A}^{c}_{a}\cap\mathcal{A}^{f}_{a}=\emptyset$, we can rewrite \eqref{eq:MCF-mixed} as
                    \begin{equation}\label{eq:MCF-mixed-imply1}
                        \begin{aligned}
                            \phi(x^c,x^f)=\min_{y\in\mathbb{R}^{n_{a}}}~&\sum_{j\in[n_a]:a_j\in\mathcal{A}\backslash(\mathcal{A}^c_{a}\cup\mathcal{A}^f_{a})}c_{j}y_{j}+
                            \sum_{j\in[n_a]:a_j\in\mathcal{A}^c_{a}}(c_j+\delta^c_j x^c_j)y_j+
                            \sum_{j\in[n_a]:a_j\in\mathcal{A}^f_{a}}c_{j}y_{j}\\
                            \text{s.t.}~&\sum_{j\in[n_a]:a_j\in\mathcal{A}\backslash\mathcal{A}^f_{a}}T_{j}y_j+\sum_{j\in[n_a]:a_j\in\mathcal{A}^f_{a}}T_{j}y_j\geq d\\
                            &0\leq y_j\leq f_j,j\in[n_a]:a_j\in\mathcal{A}\backslash\mathcal{A}^f_{a}\\
                            &0\leq y_j\leq f_j-\delta^f_j x^f_j,j\in[n_a]:a_j\in\mathcal{A}^f_{a},
                        \end{aligned}
                    \end{equation}
                    where $T_j$ means the $j$th column of $T$.
                    We replace $y_{j}$ by $y_j+y'_j$ for all $j\in[n_a]:a_j\in\mathcal{A}^f_{a}$ and rewrite \eqref{eq:MCF-mixed-imply1} as
                    \begin{equation}\label{eq:MCF-mixed-imply2}
                        \begin{aligned}
                            \phi(x^c,x^f)=\min_{y\in\mathbb{R}^{n_{a}}}~&\sum_{j\in[n_a]:a_j\in\mathcal{A}\backslash(\mathcal{A}^c_{a}\cup\mathcal{A}^f_{a})}c_{j}y_{j}+
                            \sum_{j\in[n_a]:a_j\in\mathcal{A}^c_{a}}(c_j+\delta^c_j x^c_j)y_j+\\
                            &\sum_{j\in[n_a]:a_j\in\mathcal{A}^f_{a}}c_{j}(y_{j}+y'_j)\\
                            \text{s.t.}~&\sum_{j\in[n_a]:a_j\in\mathcal{A}\backslash\mathcal{A}^f_{a}}T_{j}y_j+\sum_{j\in[n_a]:a_j\in\mathcal{A}^f_{a}}T_{j}(y_j+y'_j)\geq d\\
                            &0\leq y_j\leq f_j,j\in[n_a]:a_j\in\mathcal{A}\backslash\mathcal{A}^f_{a}\\
                            &0\leq y_j+y'_j\leq f_j-\delta^f_j x^f_j,j\in[n_a]:a_j\in\mathcal{A}^f_{a}.
                        \end{aligned}
                    \end{equation}
                    Similar to the proof of Corollary \ref{crl:MCF-imply-capacity}, given $M\in\mathbb{R}$ such that $M$ is sufficiently large for all $j\in[n_a]:a_j\in\mathcal{A}^f_{a}$, \eqref{eq:MCF-mixed-imply2} is equivalent to
                    \begin{equation}\label{eq:MCF-mixed-imply}
                        \begin{aligned}
                            \phi(x^c,x^f)=\min_{y\in\mathbb{R}^{n_{a}}}~&\sum_{j\in[n_a]:a_j\in\mathcal{A}\backslash\mathcal{A}^c_{a}}c_{j}y_{j}+\sum_{j\in[n_a]:a_j\in\mathcal{A}^c_{a}}(c_j+\delta^c_j x^c_j)y_j+
                            \sum_{j\in[n_a]:a_j\in\mathcal{A}^f_{a}}(c_{j}+Mx^f_j)y'_j\\
                            \text{s.t.}~&Ty+\sum_{j\in[n_a]:a_j\in\mathcal{A}^f_{a}}T_{j}y'_j\geq d\\
                            &0\leq y\leq f-\delta^f\\
                            &0\leq y'_j\leq \delta^f_j,j\in[n_a]:a_j\in\mathcal{A}^f_{a}.
                        \end{aligned}
                    \end{equation}
                    From the above analysis, we have that \eqref{eq:MCF-mixed-imply} is equivalent to \eqref{eq:MCF-mixed}.

                    Furthermore, we observe that \eqref{eq:MCF-mixed-imply} aligns with the form of \eqref{eq:MCF-cost}.
                    The topological modifications are similar to those in the proof of Corollary \ref{crl:MCF-imply-capacity} and only original arcs in $\mathcal{A}^c_a\cup\mathcal{A}^f_a$ are attackable.
                    Therefore, given $T$, $\delta^c$,  and$\delta^f$, for any $c$, $d$, and $f$, we can identify the submodularity or supermodularity of $\phi(x)$ defined in \eqref{eq:MCF-mixed-imply} by the conditions in Proposition \ref{pps:general-MCF-cost}.
                    Considering that \eqref{eq:MCF-mixed-imply} is equivalent to \eqref{eq:MCF-mixed}, the conditions in Proposition \ref{pps:general-MCF-cost} applies to \eqref{eq:MCF-mixed}, which implies Proposition \ref{pps:general-MCF-mixed}.
                \end{proof}

            \subsection{Proof of Proposition \ref{pps:general-SP-distance}}\label{apd:pps:general-SP-distance}
                \begin{proof}\color{black}
                    Because \eqref{eq:SP-distance} is a parametric LP, we reformulate \eqref{eq:SP-distance} as
                    \begin{equation}\label{eq:SP-distance2}
                        \begin{aligned}
                            -\phi(x)=-\max_{\lambda\in\mathbb{R}^{n_{v}}}~&d^{\top}\lambda\\
                            \text{s.t.}~&T^{\top}\lambda\leq c+\text{diag}(\delta)x\\
                            =\min_{\lambda\in\mathbb{R}^{n_{v}}}~&(-d)^{\top}\lambda\\
                            \text{s.t.}~&Ax+B\lambda\leq h
                        \end{aligned}
                    \end{equation}
                    with dual variable $\lambda\in\mathbb{R}^{n_{v}}$ and
                    \begin{equation*}
                        A=-\delta'\in\mathbb{R}^{m\times n_{a}},~
                        B=T'\in\mathbb{R}^{m\times n_{v}},~
                        h=c\in\mathbb{R}^{m},
                    \end{equation*}
                    where $m=n_{a}$, $\delta':=\text{diag}(\delta)$, $T':=T^{\top}$.
                    Then, we have $\text{rank}(B)=n_{v}-1<m$.
                    Hence, we consider the second condition of Theorem \ref{trm:condition-generalLP-sub} or Theorem \ref{trm:condition-generalLP-super} in this proof.
                    Specifically, we consider a set $\mathcal{I}\subseteq[m]$ with $|\mathcal{I}|=\text{rank}(B)+1=n_{v}$, and denote $A_{\mathcal{I}}$ and $B_{\mathcal{I}}$ as
                    \begin{equation*}
                        A_{\mathcal{I}}=-\delta'_{\mathcal{I}}\in\mathbb{R}^{n_{v}\times n_{a}},~
                        B_{\mathcal{I}}=T'_{\mathcal{I}}\in\mathbb{R}^{n_{v}\times n_{v}}.
                    \end{equation*}

                    Following the second condition of Theorem \ref{trm:condition-generalLP-sub} or Theorem \ref{trm:condition-generalLP-super}, we suppose $\text{rank}(B_{\mathcal{I}})=\text{rank}(B)=n_{v}-1$.
                    It indicates that a matrix composed of any $n_{v}-1$ columns of $B_{\mathcal{I}}$ is of full column rank, and hence, for any $r\in[n_{v}]$, $B_{\mathcal{I}}\lambda=0$ with $\lambda_{r}=0$ have a unique solution $\lambda=0$.
                    Specifically, for any $r\in[n_{v}]$, the linear system
                    \begin{equation}\label{eq:SP-distance-I1}
                        \left\{\begin{aligned}
                            &T'_{j}\lambda=0,\forall j\in\mathcal{J}(\mathcal{I})\\
                            &\lambda_{r}=0
                        \end{aligned}\right.
                    \end{equation}
                    has a unique solution $\lambda=0$, where $\mathcal{J}(\mathcal{I})=\mathcal{I}$.
                    Then, we introduce an auxiliary variable $p\in\mathbb{R}^{n_{a}}$ and reformulate \eqref{eq:SP-distance-I1} as
                    \begin{equation}\label{eq:SP-distance-I2}
                        \left\{\begin{aligned}
                            &T'\lambda=-p\\
                            &p_{j}=0, \forall j\in\mathcal{J}(\mathcal{I})\\
                            &\lambda_{r}=0
                        \end{aligned}\right..
                    \end{equation}
                    Obviously, \eqref{eq:SP-distance-I1} has a unique solution $\lambda=0$ if and only if \eqref{eq:SP-distance-I2} has a unique solution $\lambda=0,p=0$.
                    From the perspective of network potential, $\lambda$ means node potential and $p$ indicates arc potential difference.
                    $T'\lambda=-p$ restricts the cycle potential consistency with potential difference $p$.
                    For $j\in\mathcal{J}(\mathcal{I})$, $p_{j}=0$ means zero potential difference on arc $j$.
                    For $j\notin\mathcal{J}(\mathcal{I})$, the potential difference on arcs $j$ is unlimited and is thus free.
                    Hence, we can remove arcs $j,\forall j\notin\mathcal{J}(\mathcal{I})$ from the network.
                    We denote the set of remaining arcs as $\mathcal{A}'(\mathcal{I})$ and obtain a directed graph $\mathcal{G}'(\mathcal{I})=\big(\mathcal{V},\mathcal{A}'(\mathcal{I})\big)$.
                    $\lambda_{r}=0$ means zero potential at node $v_r$.
                    Such nodes with zero potential are called reference nodes and we denote the set of reference nodes as $\mathcal{V}_{\text{ref}}(\mathcal{I})$.
                    According to \eqref{eq:SP-distance-I2}, $\mathcal{V}_{\text{ref}}(\mathcal{I})=\{v_r\}$.
                    With the above understanding in mind, for any $r\in[n_{v}]$, a unique solution $\lambda=0,p=0$ of \eqref{eq:SP-distance-I2} indicates:
                    under cycle potential consistency limits, given the reference node set $\mathcal{V}_{\text{ref}}(\mathcal{I})$, whenever there is no arc potential difference in $\mathcal{G}'(\mathcal{I})$, the potential of each node in $\mathcal{G}'(\mathcal{I})$ is zero.
                    Therefore, according to Lemma \ref{lem:pps:general-MCF-cost}, for any $r\in[n_{v}]$, each component of $\mathcal{G}'(\mathcal{I})$ contains at least one reference node in $\mathcal{V}_{\text{ref}}(\mathcal{I})$.
                    And because there is only one reference node $v_r$ in $\mathcal{V}_{\text{ref}}(\mathcal{I})$, $\mathcal{G}'(\mathcal{I})$ must be a connected graph.

                    In the following, we derive the necessary and sufficient conditions for the submodularity and supermodularity of $\phi(x)$ defined by \eqref{eq:SP-distance}, respectively.

                    \vspace{2ex}
                    \noindent\textbf{(i) Submodularity.}

                    $\phi(x)$ defined by \eqref{eq:SP-distance} is submodular if and only if $-\phi(x)$ in \eqref{eq:SP-distance2} is supermodular.
                    Following the second condition of Theorem \ref{trm:condition-generalLP-sub}, we suppose an $x\in\mathbb{R}_{+}^{n_{a}}$ such that $A_{\mathcal{I}}x\in\text{span}(B_{\mathcal{I}})$.
                    It implies that there exist a $\lambda\in\mathbb{R}^{n_{v}}$ such that $A_{\mathcal{I}}x=B_{\mathcal{I}}\lambda$.
                    Specifically, the linear system
                    \begin{equation}\label{eq:SP-distance-sub-x}
                        \left\{\begin{aligned}
                            &T'\lambda=-p\\
                            &p_{j}=\delta_{j}x_{j}, \forall j\in\mathcal{J}(\mathcal{I})
                        \end{aligned}\right.
                    \end{equation}
                    is feasible.
                    According to the second condition of Theorem 1, to guarantee the supermodularity of \eqref{eq:SP-distance2}, we need to guarantee $A_{\mathcal{I},k}x_{k}\in\text{span}(B_{\mathcal{I}})$ for any $k\in[n_{a}]$.
                    It implies that for any $k\in[n_{a}]$, there exists $\lambda'\in\mathbb{R}^{n_{v}}$ such that $A_{\mathcal{I},k}x_{k}=B_{\mathcal{I}}\lambda'$.
                    If $k\notin\mathcal{J}(\mathcal{I})$, according to the structure of $A_{\mathcal{I}}$, $A_{\mathcal{I},k}=0$, and hence, $\lambda'=0$ is feasible.
                    If $k\in\mathcal{J}(\mathcal{I})$, $A_{\mathcal{I},k}x_{k}=B_{\mathcal{I}}\lambda'$ implies that the linear system
                    \begin{equation}\label{eq:SP-distance-sub-y}
                        \left\{\begin{aligned}
                            &T'\lambda'=-p'\\
                            &p'_{k}=\delta_{k}x_{k}\\
                            &p'_{j}=0, \forall j\in\mathcal{J}(\mathcal{I})\backslash\{k\}
                        \end{aligned}\right.
                    \end{equation}
                    is feasible.
                    If $k\notin\mathcal{J}_{a}$, where $\mathcal{J}_{a}:=\{j\in[n_{a}]:\delta_{j}>0\}$ is corresponding to the set $\mathcal{A}_{a}$ of attackable arcs, we have $\delta_{k}=0$ and thus $\lambda'=0,p'=0$ is feasible.
                    Therefore, we only consider $k\in\mathcal{J}_{a}\cap\mathcal{J}(\mathcal{I})$ in the following.
                    Comparing \eqref{eq:SP-distance-sub-x} and \eqref{eq:SP-distance-sub-y} to \eqref{eq:SP-distance-I2}, we find:
                    \eqref{eq:SP-distance-sub-x} modifies the potential difference at arc $j$ from 0 to $\delta_{j}x_{j}$ for all $j\in\mathcal{J}(\mathcal{I})$ and does not set reference node;
                    \eqref{eq:SP-distance-sub-y} modifies the potential difference at the only arc $k$ from 0 to $\delta_{k}x_{k}$ and does not set reference node.
                    Either \eqref{eq:SP-distance-sub-x} or \eqref{eq:SP-distance-sub-y} is feasible if and only if under such modifications, the cycle potential consistency still holds for the connected graph $\mathcal{G}'(\mathcal{I})$.
                    Considering that there is no reference node after the above modifications, the cycle consistency holds when for any cycle $P$ in $\mathcal{G}'(\mathcal{I})$, we have
                    \begin{equation}\label{eq:SP-distance-sub-cycle}
                        \sum_{j\in\mathcal{J}_{P}}\alpha_{j}(P)p_{j}=0,
                    \end{equation}
                    where $\mathcal{J}_{P}\subseteq\mathcal{J}(\mathcal{I})$ is the index set of arcs in $P$;
                    $\alpha_{j}\in\{-1,0,1\}$ indicates the direction of arc $j$.
                    $\alpha_{j}(P)=0$ means that arc $j$ is not in $P$;
                    $\alpha_{j}(P)=1$ and $\alpha_{j}(P)=-1$ indicate arc $j$ aligns with the forward and reverse sequence of $P$, respectively.
                    Given \eqref{eq:SP-distance-sub-x}, \eqref{eq:SP-distance-sub-cycle} reduces to
                    \begin{equation}\label{eq:SP-distance-sub-cycle-x}
                        \sum_{j\in\mathcal{J}_{a}\cap\mathcal{J}_{P}}\alpha_{j}(P)\delta_{j}x_{j}=0.
                    \end{equation}
                    Given \eqref{eq:SP-distance-sub-y}, \eqref{eq:SP-distance-sub-cycle} reduces to
                    \begin{equation}\label{eq:SP-distance-sub-cycle-y}
                        \sum_{j\in\mathcal{J}_{a}\cap\mathcal{J}_{P}\cap\{k\}}\alpha_{j}(P)\delta_{j}x_{j}=0.
                    \end{equation}

                    From the above analysis, to guarantee the supermodularity of \eqref{eq:SP-distance2}, we need to guarantee that \eqref{eq:SP-distance-sub-cycle-y} holds for any $k\in\mathcal{J}_{a}\cap\mathcal{J}(\mathcal{I})$ whenever \eqref{eq:SP-distance-sub-cycle-x} holds.
                    When $k\notin\mathcal{J}_{P}$, \eqref{eq:SP-distance-sub-cycle-y} always holds.
                    Hence, we only consider $k\in\mathcal{J}_{a}\cap\mathcal{J}_{P}$, and \eqref{eq:SP-distance-sub-cycle-y} reduces to
                    \begin{equation}\label{eq:SP-distance-sub-cycle-yy}
                        \delta_{k}x_{k}=0.
                    \end{equation}
                    We can prove that \eqref{eq:SP-distance-sub-cycle-yy} holds for any $k\in\mathcal{J}_{a}\cap\mathcal{J}_P$ whenever \eqref{eq:SP-distance-sub-cycle-x} holds, if and only if any two arcs in $A_a\cap A_{P}$ have the same direction.
                    The proof is similar to that in the proof of Proposition \ref{pps:general-MCF-cost} and is thus omitted.

                    When $\mathcal{I}$ varies, $\mathcal{G}'(\mathcal{I})$ may be any connected subgraph of $\mathcal{G}$, and hence, a cycle $P$ in $\mathcal{G}'(\mathcal{I})$ may be any cycle in $\mathcal{G}$.
                    Therefore, $\forall\mathcal{I}\subseteq[m]$ with $|\mathcal{I}|=\text{rank}(B)+1$ such that $\text{rank}(B_{\mathcal{I}})=\text{rank}(B)$, for any cycle $P$ in $\mathcal{G}'(\mathcal{I})$, any two arcs in $\mathcal{A}_{a}\cap\mathcal{A}_{P}$ have the same direction, if and only if
                    \begin{condition}\label{cdt:SP-sub}
                        for any cycle $P$ in $\mathcal{G}$, any two arcs in $\mathcal{A}_{a}\cap\mathcal{A}_{P}$ have the same direction in $P$.
                    \end{condition}

                    Through the above proof and by the second condition of Theorem \ref{trm:condition-generalLP-sub}, we prove that $-\phi(x)$ defined by \eqref{eq:SP-distance2} is supermodular for any $-d$ and $h$, if and only if Condition \ref{cdt:SP-sub} holds.
                    Therefore, $\phi(x)$ defined by \eqref{eq:SP-distance} is submodular for any $c$ and $d$, if and only if Condition \ref{cdt:SP-sub} holds.

                    \vspace{2ex}
                    \noindent\textbf{(ii) Supermodularity.}

                    $\phi(x)$ defined by \eqref{eq:SP-distance} is supermodular if and only if $-\phi(x)$ in \eqref{eq:SP-distance2} is submodular.
                    Following the second condition of Theorem \ref{trm:condition-generalLP-super}, we suppose an $x\in\mathbb{R}^{n_{a}}$ such that $A_{\mathcal{I}}x\in\text{span}(B_{\mathcal{I}})$.
                    It implies that there exist a $\lambda\in\mathbb{R}^{n_{v}}$ such that $A_{\mathcal{I}}x=B_{\mathcal{I}}\lambda$.
                    Specifically, the linear system
                    \begin{equation}\label{eq:SP-distance-super-x}
                        \left\{\begin{aligned}
                            &T'\lambda=-p\\
                            &p_{j}=\delta_{j}x_{j}, \forall j\in\mathcal{J}(\mathcal{I})
                        \end{aligned}\right.
                    \end{equation}
                    is feasible.
                    According to the second condition of Theorem \ref{trm:condition-generalLP-super}, to guarantee the submodularity of \eqref{eq:SP-distance2}, we need to guarantee $A_{\mathcal{I}}x^{+}\in\text{span}(B_{\mathcal{I}})$.
                    It implies that there exists $\lambda'\in\mathbb{R}^{n_{v}}$ such that $A_{\mathcal{I}}x^{+}=B_{\mathcal{I}}\lambda'$.
                    Specifically, the linear system
                    \begin{equation}\label{eq:SP-distance-super-y}
                        \left\{\begin{aligned}
                            &T'\lambda'=-p'\\
                            &p'_{j}=\delta_{j}x^{+}_{j}, \forall j\in\mathcal{J}(\mathcal{I})
                        \end{aligned}\right.
                    \end{equation}
                    is feasible.
                    Comparing \eqref{eq:SP-distance-super-x} and \eqref{eq:SP-distance-super-y} to \eqref{eq:SP-distance-I2}, we find:
                    \eqref{eq:SP-distance-super-x} modifies the potential difference at arc $j$ from 0 to $\delta_{j}x_{j}$ for all $j\in\mathcal{J}(\mathcal{I})$ and does not set reference node;
                    \eqref{eq:SP-distance-super-y} modifies the potential difference at arc $j$ from 0 to $\delta_{j}x^{+}_{j}$ for all $j\in\mathcal{J}(\mathcal{I})$ and does not set reference node.
                    Either \eqref{eq:SP-distance-super-x} or \eqref{eq:SP-distance-super-y} is feasible if and only if under such modifications, the cycle potential consistency still holds for the connected graph $\mathcal{G}'(\mathcal{I})$.
                    Considering that there is no reference node after the above modifications, the cycle consistency holds when for any cycle $P$ in $\mathcal{G}'(\mathcal{I})$, we have
                    \begin{equation}\label{eq:SP-distance-super-cycle}
                        \sum_{j\in\mathcal{J}_{P}}\alpha_{j}(P)p_{j}=0,
                    \end{equation}
                    where $\mathcal{J}_{P}\subseteq\mathcal{J}(\mathcal{I})$ is the index set of arcs in $P$;
                    $\alpha_{j}\in\{-1,0,1\}$ indicates the direction of arc $j$.
                    $\alpha_{j}(P)=0$ means that arc $j$ is not in $P$;
                    $\alpha_{j}(P)=1$ and $\alpha_{j}(P)=-1$ indicate arc $j$ aligns with the forward and reverse sequence of $P$, respectively.
                    Given \eqref{eq:SP-distance-super-x}, \eqref{eq:SP-distance-super-cycle} reduces to
                    \begin{equation}\label{eq:SP-distance-super-cycle-x}
                        \sum_{j\in\mathcal{J}_{a}\cap\mathcal{J}_{P}}\alpha_{j}(P)\delta_{j}x_{j}=0,
                    \end{equation}
                    where $\mathcal{J}_{a}:=\{j\in[n_{a}]:\delta_{j}>0\}$ is corresponding to the set $\mathcal{A}_{a}$ of attackable arcs.
                    Given \eqref{eq:SP-distance-super-y}, \eqref{eq:SP-distance-super-cycle} reduces to
                    \begin{equation}\label{eq:SP-distance-super-cycle-y}
                        \sum_{j\in\mathcal{J}_{a}\cap\mathcal{J}_{P}}\alpha_{j}(P)\delta_{j}x^{+}_{j}=0.
                    \end{equation}

                    From the above analysis, to guarantee the submodularity of \eqref{eq:SP-distance2}, we need to guarantee that \eqref{eq:SP-distance-super-cycle-y} holds whenever \eqref{eq:SP-distance-super-cycle-x} holds.
                    We can prove that \eqref{eq:SP-distance-super-cycle-y} holds whenever \eqref{eq:SP-distance-super-cycle-x} holds, if and only if any two arcs in $A_a\cap A_{P}$ have opposite directions.
                    The proof is similar to that in the proof of Proposition \ref{pps:general-MCF-cost} and is thus omitted.

                    When $\mathcal{I}$ varies, $\mathcal{G}'(\mathcal{I})$ may be any connected subgraph of $\mathcal{G}$, and hence, a cycle $P$ in $\mathcal{G}'(\mathcal{I})$ may be any cycle in $\mathcal{G}$.
                    Therefore, $\forall\mathcal{I}\subseteq[m]$ with $|\mathcal{I}|=\text{rank}(B)+1$ such that $\text{rank}(B_{\mathcal{I}})=\text{rank}(B)$, for any cycle $P$ in $\mathcal{G}'(\mathcal{I})$, any two arcs in $\mathcal{A}_{a}\cap\mathcal{A}_{P}$ have opposite directions, if and only if
                    \begin{condition}\label{cdt:SP-super}
                        for any cycle $P$ in $\mathcal{G}$, any two arcs in $\mathcal{A}_{a}\cap\mathcal{A}_{P}$ have opposite directions in $P$.
                    \end{condition}

                    Through the above proof and by the second condition of Theorem \ref{trm:condition-generalLP-super}, we prove that $-\phi(x)$ defined by \eqref{eq:SP-distance2} is submodular for any $-d$ and $h$, if and only if Condition \ref{cdt:SP-super} holds.
                    Therefore, $\phi(x)$ defined by \eqref{eq:SP-distance} is supermodular for any $c$ and $d$, if and only if Condition \ref{cdt:SP-super} holds.
                \end{proof}

            \subsection{Proof of Proposition \ref{pps:check-MCF}}\label{apd:pps:check-MCF}
                \begin{proof}
                    We first analyze the meaning of $\tilde{\eta}$ and $\eta$ in Algorithm \ref{alg:check-MCF}, which depend on $\eta^+$ and $\eta^-$.
                    For the $k$th pair of arcs $a'$ and $a''$ in $\mathcal{A}_a$, $\eta^+_k=1$ indicates that \eqref{eq:check-path} is feasible when either $v_{i_1}=v'_h,v_{i_2}=v''_t$ or $v_{i_1}=v''_h,v_{i_2}=v'_t$.
                    Following Lemma \ref{lem:check-path2} and then Lemma \ref{lem:check-path}, we have that there exists a path such that $a'$ and $a''$ have the same direction in the path.
                    $\eta^+_k=0$ indicates that \eqref{eq:check-path} is infeasible when $v_{i_1}=v'_h,v_{i_2}=v''_t$ and $v_{i_1}=v''_h,v_{i_2}=v'_t$.
                    Following Lemma \ref{lem:check-path2} and then Lemma \ref{lem:check-path}, we have that there exists no path such that $a'$ and $a''$ have the same direction in the path.
                    Similarly, $\eta^-_k=1$ indicates that ;
                    $\eta^-_k=0$ indicates that there exists no path such that $a'$ and $a''$ have opposite directions in the path.

                    Then, we consider four combinations of $\eta^+_k$ and $\eta^-_k$.

                    Case 1: $\eta^+_k=0$ and $\eta^-_k=0$, and hence, $\eta_k=0$ and $\tilde{\eta}_k=0$.
                    In this case, there exists no path such that $a'$ and $a''$ have the same direction or opposite directions in the path.
                    Hence, these exists no path containing $a'$ and $a''$.

                    Case 2: $\eta^+_k=1$ and $\eta^-_k=0$, and hence, $\eta_k=1$ and $\tilde{\eta}_k=0$.
                    In this case, there exists a path such that $a'$ and $a''$ have the same direction in the path, but there exists no path such that $a'$ and $a''$ have opposite directions in the path.
                    Hence, for any path $P$ containing $a'$ and $a''$, $a'$ and $a''$ have the same direction in $P$.

                    Case 3: $\eta^+_k=0$ and $\eta^-_k=1$, and hence, $\eta_k=-1$ and $\tilde{\eta}_k=0$.
                    In this case, there exists no path such that $a'$ and $a''$ have the same direction in the path, but there exists a path such that $a'$ and $a''$ have opposite directions in the path.
                    Hence, for any path $P$ containing $a'$ and $a''$, $a'$ and $a''$ have opposite directions in $P$.

                    Case 4: $\eta^+_k=1$ and $\eta^-_k=1$, and hence, $\eta_k=0$ and $\tilde{\eta}_k=1$.
                    In this case, $a'$ and $a''$ may have the same direction or opposite directions in different paths.
                    Hence, $a'$ and $a''$ do not have clear topological relationship with respect to paths.

                    Algorithm \ref{alg:check-MCF} outputs 1 when $\tilde{\eta}=0$ and $\eta\geq0$.
                    It indicates that each pair of attackable arcs is under either Case 1 or Case 2.
                    Therefore, for any pair of attackable arcs and for any path containing them, they have the same direction in the path.

                    Algorithm \ref{alg:check-MCF} outputs $-1$ when $\tilde{\eta}=0$ and $\eta\leq0$.
                    It indicates that each pair of attackable arcs is under either Case 1 or Case 3.
                    Therefore, for any pair of attackable arcs and for any path containing them, they have opposite directions in the path.
                \end{proof}

            \subsection{Proof of Proposition \ref{pps:specific-MCF-arc-sufficient}}\label{apd:pps:specific-MCF-arc-sufficient}
                \begin{proof}\color{black}
                    Following \eqref{eq:MCF-cost} and Definition \ref{def:interdiction-independent}, we assume that for all $j\in[n_a]$, $x_j=1$ if $a_j\in\mathcal{X}\subseteq\mathcal{A}_a$.
                    Then, we have $\psi_{\mathcal{G}}(\mathcal{X})=\phi_{\mathcal{G}}(x)$, where a subscript $\mathcal{G}$ is introduced for $\phi$ to indicate the considered network.
                    We define $\mathcal{G}'=\mathcal{G}\backslash\mathcal{A}_r$, $\mathcal{J}_1:=\{j\in[n_a]:a_j\notin\mathcal{A}_r\}$, and $\mathcal{J}_2:=[n_a]\backslash\mathcal{J}_1$.
                    Then, we have $\psi_{\mathcal{G}'}(\mathcal{X}\backslash\mathcal{A}_r)=\phi_{\mathcal{G}'}(x_{\mathcal{J}_1})$.
                    Following Definition \ref{def:interdiction-independent}, we have
                    \begin{equation}\label{eq:interdiction-independent2}
                        \phi_{\mathcal{G}}(x)=\phi_{\mathcal{G}'}(x_{\mathcal{J}_1}),\forall x\in\{0,1\}^{n_a}.
                    \end{equation}

                    In the following, we prove the submodularity claim in Proposition \ref{pps:specific-MCF-arc-sufficient} and the supermodularity claim can be similarly proved.

                    We first prove that $\phi_{\mathcal{G}}(x)$ is submodular in $x$ if and only if $\phi_{\mathcal{G}'}(x_{\mathcal{J}_1})$ is submodular in $x_{\mathcal{J}_1}$.\\
                    $\Leftarrow$:
                    If $\phi_{\mathcal{G}'}(x_{\mathcal{J}_1})$ is submodular in $x_{\mathcal{J}_1}$, for any $x',x''$, we have
                    \begin{equation*}
                        \begin{aligned}
                            \phi_{\mathcal{G}}(x')+\phi_{\mathcal{G}}(x'')
                            &=\phi_{\mathcal{G}'}(x'_{\mathcal{J}_1})+\phi_{\mathcal{G}'}(x''_{\mathcal{J}_1})\\
                            &\geq\phi_{\mathcal{G}'}(x^{\vee}_{\mathcal{J}_1})+\phi_{\mathcal{G}'}(x^{\wedge}_{\mathcal{J}_1})\\
                            &=\phi_{\mathcal{G}}(x^{\vee})+\phi_{\mathcal{G}}(x^{\wedge}).
                        \end{aligned}
                    \end{equation*}
                    Hence, $\phi_{\mathcal{G}}(x)$ is submodular in $x$.\\
                    $\Rightarrow$:
                    If $\phi_{\mathcal{G}}(x)$ is submodular in $x$, for any $x',x''$, we have
                    \begin{equation*}
                        \begin{aligned}
                            \phi_{\mathcal{G}'}(x'_{\mathcal{J}_1})+\phi_{\mathcal{G}'}(x''_{\mathcal{J}_1})
                            &=\phi_{\mathcal{G}}(x')+\phi_{\mathcal{G}}(x'')\\
                            &\geq\phi_{\mathcal{G}}(x^{\vee})+\phi_{\mathcal{G}}(x^{\wedge})\\
                            &=\phi_{\mathcal{G}'}(x^{\vee}_{\mathcal{J}_1})+\phi_{\mathcal{G}'}(x^{\wedge}_{\mathcal{J}_1}).
                        \end{aligned}
                    \end{equation*}
                    Hence, $\phi_{\mathcal{G}'}(x_{\mathcal{J}_1})$ is submodular in $x_{\mathcal{J}_1}$.

                    Following Proposition \ref{pps:general-MCF-cost}, $\phi_{\mathcal{G}'}(x_{\mathcal{J}_1})$ is submodular in $x_{\mathcal{J}_1}$ for any $c_{\mathcal{J}_1}$, $d$, and $f_{\mathcal{J}_1}$,
                    if and only if
                    \begin{condition}\label{cdt:cost-sub-reduce}
                        for any path $P$ in $\mathcal{G}'$, for any two arcs $a',a''\in(\mathcal{A}_{a}\backslash\mathcal{A}_r)\cap\mathcal{A}_{P}$, $a'$ and $a''$ have the \emph{same} direction in $P$.
                    \end{condition}

                    Under some known network parameters ($c$, $d$, and $f$), $\phi_{\mathcal{G}'}(x_{\mathcal{J}_1})$ is submodular in $x_{\mathcal{J}_1}$ if Condition \ref{cdt:cost-sub-reduce} holds.
                    Therefore, $\phi_{\mathcal{G}}(x)$ is submodular in $x$ if Condition \ref{cdt:cost-sub-reduce} holds.
                \end{proof}

            \subsection{Proof of Proposition \ref{pps:SPN-SP-arc}}\label{apd:pps:SPN-SP-arc}
                \begin{proof}\color{black}
                    When $\mathcal{A}_a\cap\mathcal{A}_{P_s}=\emptyset$, the original shortest path is not attacked, and hence, $\phi(x)=\phi(0)$ holds for any $x\in\{0,1\}^{n_a}$.
                    Therefore, $\phi(x)$ is both supermodular and submodular.
                    In the following, we only consider the nontrivial case where $\mathcal{A}_a\cap\mathcal{A}_{P_s}\neq\emptyset$ holds.

                    \vspace{2ex}
                    \noindent\textbf{(i) Supermodularity}

                    $\Leftarrow$:

                    According to the first condition, we remove all $(\mathcal{A}'|\emptyset)$-parallel arcs for all $\mathcal{A}'\subseteq\mathcal{A}_a\cap\mathcal{A}_{P_s}$ with $|\mathcal{A}'|=2$.
                    Then, we get a reduced SPN $\mathcal{G}'$, where any arc is in parallel with at most one arc in $\mathcal{A}_a\cap\mathcal{A}_{P_s}$.
                    Without loss of generality, we assume that no attackable arc is removed, and hence, $\phi_{\mathcal{G}}(x)=\phi_{\mathcal{G}'}(x)$.

                    We first prove that under any attack, the shortest path $P'_s$ in $\mathcal{G}'$ must pass through the source and sink nodes of $g_k$ or (if $g_k$ does not exist) contain the arc $a_k$, for all $a_k\in\mathcal{A}_a\cap\mathcal{A}_{P_s}$.
                    Suppose the contrary that under a certain attack, there exists an arc $a_k\in\mathcal{A}_a\cap\mathcal{A}_{P_s}$ such that $P'_s$ does not pass through the source or sink node of $g_k$, or does not contain $a_k$ if $g_k$ does not exist.
                    For the former case, without loss of generality, we consider for example that $P'_s$ does not pass through the source node of $g_k$.
                    Hence, $P'_s$ does not contain $a_k$, and there must exist an arc $a_j$ in $P'_s$ that is in parallel with $a_k$.
                    Due to the removal of arcs, $a_j$ must be $(\{a_k\}|(\mathcal{A}_a\cap\mathcal{A}_{P_s}))$-parallel and is thus in $g_k$, which contradicts that $P'_s$ does not pass through the source node of $g_k$.
                    For the latter case, because $P'_s$ does not contain $a_k$, there must exist an arc $a_j$ in $P'_s$ that is in parallel with $a_k$.
                    And due to the removal of arcs, $a_j$ must be $(\{a_k\}|(\mathcal{A}_a\cap\mathcal{A}_{P_s}))$-parallel, which contradicts that $g_k$ does not exist.
                    We finish this part of proof.

                    According to the above analysis, the shortest distance $\phi_{\mathcal{G}'}(x)$ can be represented by
                    \begin{equation}\label{eq:SPN-SP-arc-super}
                        \phi_{\mathcal{G}'}(x)=\sum_{k\in\mathcal{K}_1}\phi_{g_k}(x_k,x_{g_k})+\sum_{k\in\mathcal{K}_2}(c_k+\delta_k x_k)+c_0,
                    \end{equation}
                    where $\mathcal{K}:=\{k\in[n_a]:a_k\in\mathcal{A}_a\cap\mathcal{A}_{P_s}\}$, $\mathcal{K}_1:=\{k\in\mathcal{K}:~g_k~\text{exists}\}$, and $\mathcal{K}_2:=\mathcal{K}\backslash\mathcal{K}_1$;
                    $x_k$ and $x_{g_k}$ indicate the interdiction decision on arc $a_k$ and on attackable arcs in $g_k$, respectively;
                    $\phi_{g_k}(x_k,x_{g_k})$ means the shortest distance between the source and sink nodes of $g_k$;
                    $c_0$ is the sum of the distance of the other arcs in $P_s$, which is a constant because these arcs are non-attackable.

                    We next prove that $\phi_{g_k}(\cdot)$ is supermodular.
                    To analyze the relationship between $\phi_{g_k}(x'_k,x'_{g_k})+\phi_{g_k}(x''_k,x''_{g_k})$ and $\phi_{g_k}(x^{\wedge}_k,x^{\wedge}_{g_k})+\phi_{g_k}(x^{\vee}_k,x^{\vee}_{g_k})$, we consider three cases.

                    \noindent
                    Case 1: $x'_k=x''_k=0$.
                    Because $\phi_{g_k}(0,x_{g_k})$ is independent of $x_{g_k}$, and hence, we have
                    \begin{equation*}
                        \phi_{g_k}(0,x'_{g_k})+\phi_{g_k}(0,x''_{g_k})=\phi_{g_k}(0,x^{\wedge}_{g_k})+\phi_{g_k}(0,x^{\vee}_{g_k}).
                    \end{equation*}

                    \noindent
                    Case 2: $x'_k=0, x''_k=1$.
                    Because $\phi_{g_k}(1,x_{g_k})$ is non-decreasing in $x_{g_k}$, we have
                    \begin{equation*}
                        \phi_{g_k}(0,x'_{g_k})+\phi_{g_k}(1,x''_{g_k})\leq\phi_{g_k}(0,x^{\wedge}_{g_k})+\phi_{g_k}(1,x^{\vee}_{g_k}).
                    \end{equation*}

                    \noindent
                    Case 3: $x'_k=x''_k=1$.
                    According to the second condition, $\phi_{g_k}(1,x_{g_k})$ is supermodular in $x_{g_k}$.
                    Consequently, we have
                    \begin{equation*}
                        \phi_{g_k}(1,x'_{g_k})+\phi_{g_k}(1,x''_{g_k})\leq\phi_{g_k}(1,x^{\wedge}_{g_k})+\phi_{g_k}(1,x^{\vee}_{g_k}).
                    \end{equation*}
                    Hence, we finish this part of proof.

                    Because the shortest distances $\phi_{g_k}(\cdot)$ for different $g_k$ are independent and the other part of the right-hand side of \eqref{eq:SPN-SP-arc-super} is linear, $\phi_{\mathcal{G}'}(x)$ is supermodular.
                    Therefore, $\phi_{\mathcal{G}}(x)$ is supermodular.

                    $\Rightarrow$:

                    When $|\mathcal{A}_a\cap\mathcal{A}_{P_s}|=1$, we suppose $\mathcal{A}_a\cap\mathcal{A}_{P_s}=\{a_k\}$ and the first condition naturally holds.
                    If $g_k$ does not exist, the second condition naturally holds.
                    If $g_k$ exists, because $|\mathcal{A}_a\cap\mathcal{A}_{P_s}|=1$, we can rewrite $\phi(x)$ as
                    \begin{equation*}
                        \phi(x)=\phi_{g_k}(x_k,x_{g_k})+c_0,
                    \end{equation*}
                    where $c_0$ is a constant.
                    Moreover, because $\phi(x)$ is supermodular, we set $x_k=1$ and have
                    \begin{equation*}
                        \phi_{g_k}(1,x'_{g_k})+\phi_{g_k}(1,x''_{g_k})\leq\phi_{g_k}(1,x^{\wedge}_{g_k})+\phi_{g_k}(1,x^{\vee}_{g_k}),
                    \end{equation*}
                    which implies that the second condition holds.

                    When $|\mathcal{A}_a\cap\mathcal{A}_{P_s}|\geq2$, suppose the contrary that the conditions do not hold.

                    If the first condition does not hold, there exist at least two arcs in $\mathcal{A}_a\cap\mathcal{A}_{P_s}$ such that under a certain attack $\hat{x}$, the shortest path $P'_s$ does not pass through these arcs.
                    We define $\mathcal{J}_1:=\{j\in[n_a]:a_j\in\mathcal{A}_a\cap\mathcal{A}_{P_s}\cap\mathcal{A}_{P'_s}\}$,
                    $\mathcal{J}_2:=\{j\in[n_a]:a_j\in(\mathcal{A}_a\cap\mathcal{A}_{P_s})\backslash\mathcal{A}_{P'_s}\}$,
                    $\mathcal{J}_3:=\{j\in[n_a]:a_j\in\mathcal{A}_a\backslash\mathcal{A}_{P_s}\}$.
                    For ease of exposition, we ignore the entries of $x$ related to non-attackable arcs and have $x=[x_{\mathcal{J}_1}^{\top},x_{\mathcal{J}_2}^{\top},x_{\mathcal{J}_3}^{\top}]^{\top}$.
                    Moreover, we have $\hat{x}_{\mathcal{J}_2}\neq0$.
                    We then consider two cases.

                    Case 1: There exists a $k\in\mathcal{J}_2$ such that $\hat{x}_k=0$.
                    We rewrite $\phi(x)$ as $\phi(x_{\mathcal{J}_1},x_k,x_{\mathcal{J}_2\backslash\{k\}},x_{\mathcal{J}_3})$ and consider two scenarios.

                    \noindent
                    Scenario 1:
                    When $x_{\mathcal{J}_1}=\hat{x}_{\mathcal{J}_1}$, $x_k=0$, $x_{\mathcal{J}_2\backslash\{k\}}=\hat{x}_{\mathcal{J}_2\backslash\{k\}}$, and $x_{\mathcal{J}_3}=\hat{x}_{\mathcal{J}_3}$, i.e., $x=\hat{x}$, the shortest path $P'_s$ does not contain arc $a_j$ for any $j\in\mathcal{J}_2$.
                    Hence, an additional attack on $a_k$ makes no sense, and we have
                    \begin{equation*}
                        \phi(\hat{x}_{\mathcal{J}_1},0,\hat{x}_{\mathcal{J}_2\backslash\{k\}},\hat{x}_{\mathcal{J}_3})=\phi(\hat{x}_{\mathcal{J}_1},1,\hat{x}_{\mathcal{J}_2\backslash\{k\}},\hat{x}_{\mathcal{J}_3}).
                    \end{equation*}

                    \noindent
                    Scenario 2:
                    When $x_{\mathcal{J}_1}=\hat{x}_{\mathcal{J}_1}$, $x_k=0$, $x_{\mathcal{J}_2\backslash\{k\}}=0$, and $x_{\mathcal{J}_3}=\hat{x}_{\mathcal{J}_3}$, the shortest path contains arc $a_j$ for all $j\in\mathcal{J}_2$.
                    Hence, attacking $a_k$ must lead to a distance increase and we have
                    \begin{equation*}
                        \phi(\hat{x}_{\mathcal{J}_1},1,0_{\mathcal{J}_2\backslash\{k\}},\hat{x}_{\mathcal{J}_3})>\phi(\hat{x}_{\mathcal{J}_1},0,0_{\mathcal{J}_2\backslash\{k\}},\hat{x}_{\mathcal{J}_3}).
                    \end{equation*}

                    \noindent
                    Therefore, we combine the above two scenarios and find a counterexample as
                    \begin{equation*}
                        \phi(\hat{x}_{\mathcal{J}_1},0,\hat{x}_{\mathcal{J}_2\backslash\{k\}},\hat{x}_{\mathcal{J}_3})+\phi(\hat{x}_{\mathcal{J}_1},1,0_{\mathcal{J}_2\backslash\{k\}},\hat{x}_{\mathcal{J}_3})>
                        \phi(\hat{x}_{\mathcal{J}_1},0,0_{\mathcal{J}_2\backslash\{k\}},\hat{x}_{\mathcal{J}_3})+\phi(\hat{x}_{\mathcal{J}_1},1,\hat{x}_{\mathcal{J}_2\backslash\{k\}},\hat{x}_{\mathcal{J}_3}),
                    \end{equation*}
                    which contradicts that $\phi(\cdot)$ is supermodular.

                    Case 2: There exists no $k\in\mathcal{J}_2$ such that $\hat{x}_k=0$.
                    We divide $\mathcal{J}_2$ into two partitions, i.e., $\mathcal{K}\subset\mathcal{J}_2$ and $\mathcal{J}_2\backslash\mathcal{K}$.
                    We rewrite $\phi(x)$ as $\phi(x_{\mathcal{J}_1},x_{\mathcal{K}},x_{\mathcal{J}_2\backslash\mathcal{K}},x_{\mathcal{J}_3})$ and consider two cases.

                    \noindent
                    Case 2.1:
                    There exists a $\mathcal{K}$ such that when $x_{\mathcal{J}_1}=\hat{x}_{\mathcal{J}_1}$, $x_{\mathcal{K}}=0$, $x_{\mathcal{J}_2\backslash\mathcal{K}}=\hat{x}_{\mathcal{J}_2\backslash\{k\}}$, and $x_{\mathcal{J}_3}=\hat{x}_{\mathcal{J}_3}$, the shortest path does not contain arc $a_j$ for any $j\in\mathcal{J}_2$.
                    This scenario reduces to Case 1.

                    \noindent
                    Case 2.2:
                    When $x_{\mathcal{J}_1}=\hat{x}_{\mathcal{J}_1}$ and $x_{\mathcal{J}_3}=\hat{x}_{\mathcal{J}_3}$, only if $x_{\mathcal{K}}=1$ and $x_{\mathcal{J}_2\backslash\mathcal{K}}=1$, the shortest path does not contain arc $a_j$ for any $j\in\mathcal{J}_2$.
                    Hence, we have
                    \begin{equation*}
                        \phi(\hat{x}_{\mathcal{J}_1},0_{\mathcal{K}},1_{\mathcal{J}_2\backslash\{k\}},\hat{x}_{\mathcal{J}_3})<\phi(\hat{x}_{\mathcal{J}_1},1_{\mathcal{K}},1_{\mathcal{J}_2\backslash\{k\}},\hat{x}_{\mathcal{J}_3})
                    \end{equation*}
                    and
                    \begin{equation*}
                        \phi(\hat{x}_{\mathcal{J}_1},1_{\mathcal{K}},0_{\mathcal{J}_2\backslash\{k\}},\hat{x}_{\mathcal{J}_3})<\phi(\hat{x}_{\mathcal{J}_1},1_{\mathcal{K}},1_{\mathcal{J}_2\backslash\{k\}},\hat{x}_{\mathcal{J}_3}).
                    \end{equation*}
                    And due to the structure of SPNs, we have
                    \begin{equation*}
                        \begin{aligned}
                            &\Big(\phi(\hat{x}_{\mathcal{J}_1},0_{\mathcal{K}},1_{\mathcal{J}_2\backslash\{k\}},\hat{x}_{\mathcal{J}_3})-\phi(\hat{x}_{\mathcal{J}_1},0_{\mathcal{K}},0_{\mathcal{J}_2\backslash\{k\}},\hat{x}_{\mathcal{J}_3})\Big)+\\
                            &\Big(\phi(\hat{x}_{\mathcal{J}_1},1_{\mathcal{K}},0_{\mathcal{J}_2\backslash\{k\}},\hat{x}_{\mathcal{J}_3})-\phi(\hat{x}_{\mathcal{J}_1},0_{\mathcal{K}},0_{\mathcal{J}_2\backslash\{k\}},\hat{x}_{\mathcal{J}_3})\Big)>\\
                            &\Big(\phi(\hat{x}_{\mathcal{J}_1},1_{\mathcal{K}},1_{\mathcal{J}_2\backslash\{k\}},\hat{x}_{\mathcal{J}_3})-\phi(\hat{x}_{\mathcal{J}_1},0_{\mathcal{K}},0_{\mathcal{J}_2\backslash\{k\}},\hat{x}_{\mathcal{J}_3})\Big).
                        \end{aligned}
                    \end{equation*}
                    It implies a counterexample as
                    \begin{equation*}
                        \phi(\hat{x}_{\mathcal{J}_1},0_{\mathcal{K}},1_{\mathcal{J}_2\backslash\{k\}},\hat{x}_{\mathcal{J}_3})+\phi(\hat{x}_{\mathcal{J}_1},1_{\mathcal{K}},0_{\mathcal{J}_2\backslash\{k\}},\hat{x}_{\mathcal{J}_3})>
                        \phi(\hat{x}_{\mathcal{J}_1},0_{\mathcal{K}},0_{\mathcal{J}_2\backslash\{k\}},\hat{x}_{\mathcal{J}_3})+\phi(\hat{x}_{\mathcal{J}_1},1_{\mathcal{K}},1_{\mathcal{J}_2\backslash\{k\}},\hat{x}_{\mathcal{J}_3}),
                    \end{equation*}
                    which contradicts that $\phi(\cdot)$ is supermodular.

                    Therefore, the first condition is necessary, and hence, we can represent $\phi(x)$ by \eqref{eq:SPN-SP-arc-super}.
                    Then, if the second condition does not hold, there exists an $a_k\in\mathcal{A}_a\cap\mathcal{A}_{P_s}$ such that $\phi_{g_k}(1,x_{g_k})$ is not supermodular in $x_{g_k}$.
                    More specifically, there exist $x'_{g_k}$ and $x''_{g_k}$ such that
                    \begin{equation*}
                        \phi_{g_k}(1,x'_{g_k})+\phi_{g_k}(1,x''_{g_k})>\phi_{g_k}(1,x^{\wedge}_{g_k})+\phi_{g_k}(1,x^{\vee}_{g_k}).
                    \end{equation*}
                    By setting the other entries of $x$ to be constants, we can easily find a counterexample that $\phi_{\mathcal{G}'}(x)$ defined in \eqref{eq:SPN-SP-arc-super} is not supermodular,
                    which contradicts that $\phi(\cdot)$ is supermodular.

                    \vspace{2ex}
                    \noindent\textbf{(ii) Submodularity}

                    $\Leftarrow$:

                    According to the condition, we remove all arcs in $\mathcal{A}_a\backslash\mathcal{A}_{P_s}$ and get a reduced SPN $\mathcal{G}'$.
                    Without loss of generality, we assume that no attackable arc is removed, and hence, $\phi_{\mathcal{G}}(x)=\phi_{\mathcal{G}'}(x)$.
                    We can find that all remaining attackable arcs in $\mathcal{G}'$ are in $P_s$.
                    According to the structure of SPNs, all remaining attackable arcs have the same direction (from source $s$ to sink $t$) in any cycle of $\mathcal{G}'$.
                    Following Proposition \ref{pps:general-SP-distance}, $\phi_{\mathcal{G}'}(x)$ is submodular, and hence, $\phi_{\mathcal{G}}(x)$ is submodular.

                    $\Rightarrow$:

                    Suppose the contrary that there exists $a_k\in\mathcal{A}_a\backslash\mathcal{A}_{P_s}$ such that under a certain attack $\hat{x}$, the shortest path $P'_s$ contains $a_k$.
                    We define $\mathcal{J}_1:=\{j\in[n_a]:a_j\in\mathcal{A}_a\cap\mathcal{A}_{P_s}\cap\mathcal{A}_{P'_s}\}$,
                    $\mathcal{J}_2:=\{j\in[n_a]:a_j\in(\mathcal{A}_a\cap\mathcal{A}_{P_s})\backslash\mathcal{A}_{P'_s}\}$,
                    $\mathcal{J}_3:=\{j\in[n_a]:a_j\in(\mathcal{A}_a\backslash\mathcal{A}_{P_s})\backslash\{k\}\}$.
                    For ease of exposition, we ignore the entries of $x$ related to non-attackable arcs and have $x=[x_{\mathcal{J}_1}^{\top},x_{\mathcal{J}_2}^{\top},x_{\mathcal{J}_3}^{\top},x_k]^{\top}$.
                    Moreover, we have $\hat{x}_{\mathcal{J}_2}\neq0$.
                    We consider two scenarios.

                    \noindent
                    Scenario 1:
                    When $x_{\mathcal{J}_1}=\hat{x}_{\mathcal{J}_1}$, $x_{\mathcal{J}_2}=0$, and $x_{\mathcal{J}_3}=\hat{x}_{\mathcal{J}_3}$, the shortest path contains arcs $a_j$ for all $j\in\mathcal{J}_2$, and $\phi$ is independent of $x_k$.
                    Hence, we have
                    \begin{equation*}
                        \phi(\hat{x}_{\mathcal{J}_1},0_{\mathcal{J}_2},\hat{x}_{\mathcal{J}_3},1)=\phi(\hat{x}_{\mathcal{J}_1},0_{\mathcal{J}_2},\hat{x}_{\mathcal{J}_3},0).
                    \end{equation*}

                    \noindent
                    Scenario 2:
                    When $x_{\mathcal{J}_1}=\hat{x}_{\mathcal{J}_1}$, $x_{\mathcal{J}_2}=\hat{x}_{\mathcal{J}_2}$, $x_{\mathcal{J}_3}=\hat{x}_{\mathcal{J}_3}$, and $x_k=0$, i.e., $x=\hat{x}$, the shortest path $P'_s$ contains $a_k$.
                    Hence, attacking $a_k$ must lead to a distance increase and we have
                    \begin{equation*}
                        \phi(\hat{x}_{\mathcal{J}_1},\hat{x}_{\mathcal{J}_2},\hat{x}_{\mathcal{J}_3},0)<\phi(\hat{x}_{\mathcal{J}_1},\hat{x}_{\mathcal{J}_2},\hat{x}_{\mathcal{J}_3},1).
                    \end{equation*}

                    Therefore, we combine the above two scenarios and find a counterexample as
                    \begin{equation*}
                        \phi(\hat{x}_{\mathcal{J}_1},0_{\mathcal{J}_2},\hat{x}_{\mathcal{J}_3},1)+\phi(\hat{x}_{\mathcal{J}_1},\hat{x}_{\mathcal{J}_2},\hat{x}_{\mathcal{J}_3},0)<
                        \phi(\hat{x}_{\mathcal{J}_1},0_{\mathcal{J}_2},\hat{x}_{\mathcal{J}_3},0)+\phi(\hat{x}_{\mathcal{J}_1},\hat{x}_{\mathcal{J}_2},\hat{x}_{\mathcal{J}_3},1),
                    \end{equation*}
                    which contradicts that $\phi(\cdot)$ is submodular.
                \end{proof}

            \subsection{Proof of Corollary \ref{crl:redundancy-SP}}\label{apd:crl:redundancy-SP}
                \begin{proof}
                    Following Definition \ref{def:interdiction-independent}, $a_j$ is $\mathcal{A}_a$-robust redundant if for all $x\in\{0,1\}^{n_a}$, we have $y^*_j(x)=0$, where $y^*_j(x)$ is optimal to $\phi(x)$.
                    Therefore, we need to prove that if $y^*_j(\hat{x})=0$, we have $y^*_j(x)=0$ for all $x\in\{0,1\}^{n_a}$.

                    We suppose the contrary that there exist an $\tilde{x}\in\{0,1\}^{n_a}$ such that $y^*_j(\tilde{x})\ne0$.
                    Then, we have: 1) under attack $\hat{x}$, the shortest path $\hat{P}_s$ does not contain arc $a_j$;
                    2) under attack $\tilde{x}$, the shortest path $\tilde{P}_s$ contains arc $a_j$.
                    We rewrite $x=[x_j,x_s^\top,x_p^\top,x_o^\top]^\top$, where $x_s$ and $x_p$ consist of the entries of $x$ corresponding to arcs in $(\mathcal{A}_{\tilde{P}_s}\backslash\mathcal{A}_{\hat{P}_s})\backslash\{a_j\}$ and $\mathcal{A}_{\hat{P}_s}\backslash\mathcal{A}_{\tilde{P}_s}$, respectively;
                    $x_o$ consists of the other entries of $x$.
                    Hence, in an SPN, we have
                    \begin{equation}\label{eq:crl:redundancy1}
                        \phi_p(\hat{x}_p)<\phi_s(\hat{x}_j,\hat{x}_s)
                    \end{equation}
                    and
                    \begin{equation}\label{eq:crl:redundancy2}
                        \phi_p(\tilde{x}_p)>\phi_s(\tilde{x}_j,\tilde{x}_s),
                    \end{equation}
                    where $\phi_p(\hat{x}_p)$ and $\phi_s(\tilde{x}_j,\tilde{x}_s)$ are the distance of the unshared segments of $\hat{P}_s$ and $\tilde{P}_s$, respectively.
                    We note that $\hat{x}_j=0$, $\hat{x}_s=0$, and $\hat{x}_p=1$, and hence, \eqref{eq:crl:redundancy1} reduces to
                    \begin{equation}\label{eq:crl:redundancy3}
                        \phi_p(1)<\phi_s(0,0).
                    \end{equation}
                    Moreover, both $\phi_p(\cdot)$ and $\phi_s(\cdot)$ are non-decreasing, and hence, \eqref{eq:crl:redundancy2} implies
                    \begin{equation}\label{eq:crl:redundancy4}
                        \phi_p(1)>\phi_s(0,0),
                    \end{equation}
                    which contradicts with \eqref{eq:crl:redundancy3}.
                    Therefore, we finish the proof.
                \end{proof}

            \subsection{Proof of Proposition \ref{pps:SPN-MF-arc}}\label{apd:pps:SPN-MF-arc}
                \begin{proof}\color{black}
                    When $\mathcal{A}_a\subseteq\mathcal{A}_{C_m}$, all attackable arcs are in the original minimum cut, and hence, we have $\phi(x)=\phi(0)-\sum_{j\in[n_a]}\delta_j x_j$.
                    Therefore, $\phi(x)$ is both supermodular and submodular.
                    In the following, we only consider the nontrivial case where $\mathcal{A}_a\backslash\mathcal{A}_{C_m}\neq\emptyset$ holds.

                    \vspace{2ex}
                    \noindent\textbf{(i) Submodularity}

                    $\Leftarrow$:

                    For any subnetwork in the simplified SPN $\mathcal{G}$, we consider the following cases.

                    Case 1: If a subnetwork $g$ is a parallel composition of two arcs $a_{j_1}$ and $a_{j_2}$.
                    Then, the maximum flow between the source and sink nodes of the sub-SPN is
                    \begin{equation*}
                        \phi_g(x_{j_1},x_{j_2})=f_{j_1}-\delta_{j_1}x_{j_1}+f_{j_2}-\delta_{j_2}x_{j_2},
                    \end{equation*}
                    which is non-increasing and submodular in $(x_{j_1},x_{j_2})$.

                    Case 2: If a subnetwork $g$ is a series composition of two arcs $a_{j_1}$ and $a_{j_2}$.
                    According to the condition, all attackable arcs are in parallel with each other, and hence, at least one of $a_{j_1}$ and $a_{j_2}$ is non-attackable.
                    Without loss of generality, we assume $a_{j_2}$ is non-attackable and thus $\delta_{j_2}=0$.
                    Then, the maximum flow between the source and sink nodes of the sub-SPN is
                    \begin{equation*}
                        \phi_g(x_{j_1},x_{j_2})=\min\{f_{j_1}-\delta_{j_1}x_{j_1},f_{j_2}\}
                        =h(f_{j_1}-\delta_{j_1}x_{j_1})
                    \end{equation*}
                    where $h(x):=\min\{x,f_{j_2}\}$.
                    We note that $h(\cdot)$ is non-decreasing and concave and $f_{j_1}-\delta_{j_1}x_{j_1}$ is non-increasing and submodular in $(x_{j_1},x_{j_2})$.
                    Following Proposition 2.2.5(c) in \citet{simchi2005logic}, we have that $h(f_{j_1}-\delta_{j_1}x_{j_1})$ is submodular in $(x_{j_1},x_{j_2})$.
                    Hence, $\phi_g(x_{j_1},x_{j_2})$ is non-increasing and submodular in $(x_{j_1},x_{j_2})$.

                    Case 3: If a subnetwork $g$ is a parallel composition of two sub-subnetworks $g_1$ and $g_2$.
                    Then, the maximum flow between the source and sink nodes of the sub-SPN is
                    \begin{equation*}
                        \phi_g(x_{g_1},x_{g_2})=\phi_{g_1}(x_{g_1})+\phi_{g_2}(x_{g_2}),
                    \end{equation*}
                    where $x_{g_1}$ and $x_{g_2}$ consist of the entries of $x$ corresponding to arcs in $g_1$ and $g_2$, respectively.
                    When both $\phi_{g_1}(x_{g_1})$ and $\phi_{g_2}(x_{g_2})$ are non-increasing and submodular, we have that $\phi_g(x_{g_1},x_{g_2})$ is non-increasing and submodular in $(x_{g_1},x_{g_2})$.

                    Case 4: If a subnetwork $g$ is a series composition of two sub-subnetworks $g_1$ and $g_2$.
                    According to the condition, all attackable arcs are in parallel with each other, and hence, either $g_1$ or $g_2$ must be a non-attackable arc in the simplified SPN.
                    Without loss of generality, we assume $g_2$ is a non-attackable arc $a_{j_2}$ and thus $\delta_{j_2}=0$.
                    Then, the maximum flow between the source and sink nodes of the sub-SPN $g$ is
                    \begin{equation*}
                        \phi_g(x_{g_1},x_{g_2})=\min\{\phi_{g_1}(x_{g_1}),f_{j_2}\}
                        =h(\phi_{g_1}(x_{g_1})).
                    \end{equation*}
                    Following Proposition 2.2.5(c) in \citet{simchi2005logic}, when $\phi_{g_1}(x_{g_1})$ is non-increasing and submodular in $x_{g_1}$, we have that $h(\phi_{g_1}(x_{g_1}))$ is submodular in $x_{g_1}$, and hence, $\phi_g(x_{g_1},x_{g_2})$ is non-increasing and submodular in $(x_{g_1},x_{g_2})$.

                    According to the structure of SPNs, an SPN is a series composition or a parallel composition of two subnetworks that are also SPNs.
                    Hence, the maximum flow $\phi(x)$ can be recursively represented through either of the above four cases.
                    Therefore, $\phi(x)$ is submodular in $x$.

                    $\Rightarrow$:

                    Suppose the contrary that there are two attackable arcs $a_{k_1},a_{k_2}\in\mathcal{A}_a$ such that $a_{k_1}$ and $a_{k_2}$ are in series.
                    Because $\mathcal{G}$ is a simplified SPN, all attackable arc are possibly binding.
                    Then, we consider two cases.

                    Case 1: Either $a_{k_1}$ or $a_{k_2}$ is contained in $\mathcal{A}_{C_m}$.
                    Without loss of generality, we assume $a_{k_1}\in\mathcal{A}_{C_m}$.
                    Because $a_{k_2}$ is possibly binding, we assume that under a certain attack $\hat{x}$, $a_{k_2}\in\mathcal{A}_{C'_m}$ is binding, where $C'_m$ is the new minimum cut.
                    We define $\mathcal{J}_1:=\{j\in[n_a]:a_j\notin\mathcal{A}_{C'_m}\}$ and $\mathcal{J}_2:=\{j\in[n_a]:a_j\in\mathcal{A}_{C'_m}\}$.
                    Because $a_{k_1}$ and $a_{k_2}$ are in series, we have that $k_1\in\mathcal{J}_1$ and $k_2\in\mathcal{J}_2$.
                    Then, we rewrite $x$ as $[x_{k_1},x^{\top}_{\mathcal{J}_1\backslash k_1},x^{\top}_{\mathcal{J}_2}]^{\top}$ and have $\hat{x}_{\mathcal{J}_2}>0$.
                    We denote $\phi(0,0,0)=\hat{\phi}$ and thus $\phi(1,0,0)=\hat{\phi}-\delta_{k_1}$.
                    We denote $\phi(\hat{x}_{k_1},\hat{x}_{\mathcal{J}_1\backslash k_1},\hat{x}_{\mathcal{J}_2})=\hat{\phi}-\Delta$ and thus $\phi(0,0,\hat{x}_{\mathcal{J}_2})=\hat{\phi}-\Delta$.
                    Furthermore, we have $\phi(1,0,\hat{x}_{\mathcal{J}_2})=\min\{\hat{\phi}-\delta_{k_1},\hat{\phi}-\Delta\}$.
                    Due to $\delta_{k_1}>0$ and $\Delta>0$, we find a counterexample as
                    \begin{equation*}
                        \phi(1,0,0)+\phi(0,0,\hat{x}_{\mathcal{J}_2})=\hat{\phi}-\delta_{k_1}+\hat{\phi}-\Delta<\hat{\phi}+\min\{\hat{\phi}-\delta_{k_1},\hat{\phi}-\Delta\}=\phi(0,0,0)+\phi(1,0,\hat{x}_{\mathcal{J}_2}),
                    \end{equation*}
                    which contradicts that $\phi(\cdot)$ is submodular.

                    Case 2: Both $a_{k_1}$ and $a_{k_2}$ are not contained in $\mathcal{A}_{C_m}$.
                    We assume that under attack $\hat{x}'$ or $\hat{x}''$, $a_{k_1}$ or $a_{k_2}$ is binding, where $C'_m$ or $C''_m$ is the corresponding minimum cut.
                    Because $a_{k_1}$ and $a_{k_2}$ are in series, $C'_m\neq C''_m$.
                    We define $\mathcal{J}_0:=\{j\in[n_a]:a_j\in\mathcal{A}_{C'_m}\cap\mathcal{A}_{C''_m}\}$, $\mathcal{J}_1:=\{j\in[n_a]:a_j\in\mathcal{A}_{C'_m}\backslash\mathcal{A}_{C''_m}\}$, $\mathcal{J}_2:=\{j\in[n_a]:a_j\in\mathcal{A}_{C''_m}\backslash\mathcal{A}_{C'_m}\}$, and $\mathcal{J}_3:=\{j\in[n_a]:a_j\notin\mathcal{A}_{C'_m}\cup\mathcal{A}_{C''_m}\}$.
                    Then, we rewrite $x$ as $[x^{\top}_{\mathcal{J}_0},x^{\top}_{\mathcal{J}_1},x^{\top}_{\mathcal{J}_1},x^{\top}_{\mathcal{J}_3}]^{\top}$.
                    We denote $\phi(1,0,0,0)=\hat{\phi}$.
                    Because $a_{k_1}$ is binding under attack $\hat{x}'$, we have that $a_{k_1}$ is also binding under attack $[1^{\top}_{\mathcal{J}_0},1^{\top}_{\mathcal{J}_1},0^{\top}_{\mathcal{J}_1},0^{\top}_{\mathcal{J}_3}]^{\top}$ and denote $\phi(1,1,0,0)=\hat{\phi}-\Delta_1$.
                    Because $a_{k_2}$ is binding under attack $\hat{x}''$, we have that $a_{k_2}$ is also binding under attack $[1^{\top}_{\mathcal{J}_0},0^{\top}_{\mathcal{J}_1},1^{\top}_{\mathcal{J}_1},0^{\top}_{\mathcal{J}_3}]^{\top}$ and denote $\phi(1,0,1,0)=\hat{\phi}-\Delta_2$.
                    Furthermore, we have $\phi(1,1,1,0)=\min\{\hat{\phi}-\Delta_1,\hat{\phi}-\Delta_2\}$.
                    Due to $\Delta_1>0$ and $\Delta_2>0$, we find a counterexample as
                    \begin{equation*}
                        \phi(1,1,0,0)+\phi(1,0,1,0)=\hat{\phi}-\Delta_1+\hat{\phi}-\Delta_2<\hat{\phi}+\min\{\hat{\phi}-\Delta_1,\hat{\phi}-\Delta_2\}=\phi(1,0,0,0)+\phi(1,1,1,0),
                    \end{equation*}
                    which contradicts that $\phi(\cdot)$ is submodular.

                    \vspace{2ex}
                    \noindent\textbf{(ii) Supermodularity}

                    $\Leftarrow$:

                    Following Definition \ref{def:conditional parallel}, we say that
                    an arc $a\in\mathcal{A}$ is $(\mathcal{A}_1|\mathcal{A}_2)$-series if $a$ is in series with all arcs in $\mathcal{A}_1$ but is not in series with any arc in $\mathcal{A}_2\backslash\mathcal{A}_1$;
                    and if $|\mathcal{A}_1|=1$, an $(\mathcal{A}_1|\mathcal{A}_2)$-series subnetwork $g$ consists of the arc in $\mathcal{A}_1$ and all $(\mathcal{A}_1|\mathcal{A}_2)$-series arcs, as well as their incident nodes.
                    A node $v$ in $g$ is the source or sink node of $g$ if every $s$--$t$ path containing an arc in $g$ passes through $v$.

                    Given any $a_k\in\mathcal{A}_{C_m}$ and its corresponding $(\{a_k\}|\mathcal{A}_{C_m})$-series subnetwork $g_k$, we denote the source and sink nodes of $g_k$ by $s_k$ and $t_k$, respectively.
                    Because all attackable arcs in $g_k$ is $(\{a_k\}|\mathcal{A}_{C_m})$-series, according to the condition, all attackable arcs in $g_k$ are strictly in series.
                    Furthermore, due to the structure of $g_k$, the attackable arcs must have the same direction in any $\{s_k,t_k\}$-excluded path of $g_k$.
                    Following Proposition \ref{pps:general-MF-capacity}, $\phi_{g_k}(x_{g_k})$ is supermodular in $x_{g_k}$.
                    Then, we recover the SPN by the series or parallel composition of different $g_k$ and the other arcs.
                    We consider the following cases.

                    Case 1: If a subnetwork $g$ is the parallel composition of two sub-subnetworks $g_1$ and $g_2$.
                    Then, the maximum flow between the source and sink nodes of $g$ is
                    \begin{equation*}
                        \phi_g(x_{g_1},x_{g_2})=\phi_{g_1}(x_{g_1})+\phi_{g_2}(x_{g_2}).
                    \end{equation*}
                    When both $\phi_{g_1}(x_{g_1})$ and $\phi_{g_2}(x_{g_2})$ are supermodular, we have that $\phi_g(x_{g_1},x_{g_2})$ is supermodular.

                    Case 2: If a subnetwork $g$ is the series composition of an arc $a_{j_1}$ and a sub-subnetwork $g_2$.
                    We note $g_2$ must contains two arcs in $\mathcal{A}_{C_m}$.
                    If $a_{j_1}$ is non-attackable, $a_{j_1}$ must be non-binding and does not affect the maximum flow.
                    If $a_{j_1}$ is attackable, according to the condition, $a_{j_1}$ must be binding whenever attacked, and hence, the maximum flow between the source and sink nodes of $g$ is
                    \begin{equation*}
                        \phi_g(x_{j_1},x_{g_2})=(1-x_{j_1})\phi_{g_2}(x_{g_2})+x_{j_1}(f_{j_1}-\delta_{j_1}).
                    \end{equation*}
                    We next prove that $\phi_g(x_{j_1},x_{g_2})$ is supermodular when $\phi_{g_2}(x_{g_2})$ is supermodular.
                    When $x_{j_1}=1$, $\phi_g(x_{j_1},x_{g_2})$ is independent of $x_{g_2}$, and we have
                    \begin{equation*}
                        \phi_g(1,x'_{g_2})+\phi_g(1,x''_{g_2})=\phi_g(1,x^{\wedge}_{g_2})+\phi_g(1,x^{\vee}_{g_2}).
                    \end{equation*}
                    When $x_{j_1}=0$, $\phi_g(0,x_{g_2})=\phi_{g_2}(x_{g_2})$, and when $\phi_{g_2}(x_{g_2})$ is supermodular, we have
                    \begin{equation*}
                        \phi_g(0,x'_{g_2})+\phi_g(0,x''_{g_2})\leq\phi_g(0,x^{\wedge}_{g_2})+\phi_g(0,x^{\vee}_{g_2}).
                    \end{equation*}
                    In addition, $\phi_g(0,x_{g_2})=\phi_{g_2}(x_{g_2})$ is non-increasing, and we have
                    \begin{equation*}
                        \phi_g(0,x'_{g_2})+\phi_g(1,x''_{g_2})\leq\phi_g(0,x^{\wedge}_{g_2})+\phi_g(1,x^{\vee}_{g_2}).
                    \end{equation*}
                    Hence, we finish this part of proof.

                    According to the structure of SPNs, an SPN is a series composition or a parallel composition of two subnetworks that are also SPNs.
                    Hence, the maximum flow $\phi(x)$ can be recursively represented through either of the above two cases.
                    Therefore, $\phi(x)$ is supermodular in $x$.

                    $\Rightarrow$:

                    Suppose the contrary and we consider two cases.

                    Case 1: there exists an arc $a_{j}\in\mathcal{A}_a\backslash\mathcal{A}_{C_m}$ such that $a_{j}$ is in series with at least two arcs $a_{k_1},a_{k_2}\in\mathcal{A}_{C_m}$ and may be non-binding when $a_{j}$ is attacked.
                    Without loss of generality, we assume that $a_{k_1},a_{k_2}$ are attackable.
                    Arc $a_{j}$ is attackable and is thus possibly binding in the simplified SPN.
                    Suppose that under a certain attack $\hat{x}$ with $\hat{x}_{j}=1$, $a_{j}$ is binding and the new minimum cut is $C'_m$.
                    We define $\mathcal{J}_0:=\{j\in[n_a]:a_j\in\mathcal{A}_{C_m}\cap\mathcal{A}_{C'_m}\}$,
                    $\mathcal{J}_1:=\{j\in[n_a]:a_j\in\mathcal{A}_{C_m}\backslash\mathcal{A}_{C'_m}\}$,
                    $\mathcal{J}_2:=\{j\in[n_a]:a_j\in\mathcal{A}_{C'_m}\backslash\mathcal{A}_{C_m}\}$, and
                    $\mathcal{J}_3:=\{j\in[n_a]:a_j\notin\mathcal{A}_{C_m}\cup\mathcal{A}_{C'_m}\}$.
                    Without loss of generality, we assume $\mathcal{J}_1=\{k_1,k_2\}$ and $\mathcal{J}_2=\{j\}$.
                    Then, we rewrite $x$ as $[x^{\top}_{\mathcal{J}_0},x_{k_1},x_{k_2},x_j,x^{\top}_{\mathcal{J}_3}]^{\top}$.
                    We denote
                    \begin{equation*}
                        \phi(\hat{x}_{\mathcal{J}_0},0,0,0,0)=\hat{\phi}
                    \end{equation*}
                    and thus
                    \begin{equation*}
                        \phi(\hat{x}_{\mathcal{J}_0},1,1,0,0)=\hat{\phi}-\delta_{k_1}-\delta_{k_2}.
                    \end{equation*}
                    We denote
                    \begin{equation*}
                        \phi(\hat{x}_{\mathcal{J}_0},\hat{x}_{k_1},\hat{x}_{k_2},1,0)=\hat{\phi}-\Delta_j
                    \end{equation*}
                    and thus
                    \begin{equation*}
                        \phi(\hat{x}_{\mathcal{J}_0},0,0,1,0)=\hat{\phi}-\Delta_j.
                    \end{equation*}
                    Then, we have
                    \begin{equation*}
                        \begin{aligned}
                            \phi(\hat{x}_{\mathcal{J}_0},1,0,1,0)=&\min\{\hat{\phi}-\delta_{k_1},\hat{\phi}-\Delta_j\}\\
                            \phi(\hat{x}_{\mathcal{J}_0},0,1,1,0)=&\min\{\hat{\phi}-\delta_{k_2},\hat{\phi}-\Delta_j\}\\
                            \phi(\hat{x}_{\mathcal{J}_0},1,1,1,0)=&\min\{\hat{\phi}-\delta_{k_1}-\delta_{k_2},\hat{\phi}-\Delta_j\}.
                        \end{aligned}
                    \end{equation*}
                    If $\Delta_j\geq\delta_{k_1}+\delta_{k_2}$, whether $a_{k_1},a_{k_2}$ are attacked or not, $a_{j}$ is always binding, which contradicts that $a_{j}$ may be non-binding.
                    Hence, we have $\Delta_j<\delta_{k_1}+\delta_{k_2}$ and
                    \begin{equation*}
                        \phi(\hat{x}_{\mathcal{J}_0},1,1,1,0)=\hat{\phi}-\delta_{k_1}-\delta_{k_2}.
                    \end{equation*}
                    Therefore, because $\delta_{k_1}>0$, $\delta_{k_2}>0$, and $\Delta_{j}>0$, the counterexample
                    \begin{equation*}
                        \phi(\hat{x}_{\mathcal{J}_0},1,0,1,0)+\phi(\hat{x}_{\mathcal{J}_0},0,1,1,0)>\phi(\hat{x}_{\mathcal{J}_0},0,0,1,0)+\phi(\hat{x}_{\mathcal{J}_0},1,1,1,0)
                    \end{equation*}
                    always holds, which contradicts that $\phi(\cdot)$ is supermodular.

                    Case 2: there exists an arc $a_{j_1}\in\mathcal{A}_a\backslash\mathcal{A}_{C_m}$ such that $a_{j_1}$ is in series with only one arc $a_{k}\in\mathcal{A}_{C_m}$ and there exist another arc $a_{j_2}\in\mathcal{A}_a$ such that $a_{j_2}$ is in parallel with $a_{j_1}$ and is in series with $a_k$.
                    Without loss of generality, we assume that $a_{k},a_{j_2}$ are attackable.
                    Suppose that under a certain attack $\hat{x}$, $a_{j_1},a_{j_2}$ are in the new minimum cut is $C'_m$.
                    We define $\mathcal{J}_0:=\{j\in[n_a]:a_j\in\mathcal{A}_{C_m}\cap\mathcal{A}_{C'_m}\}$,
                    $\mathcal{J}_1:=\{j\in[n_a]:a_j\in\mathcal{A}_{C_m}\backslash\mathcal{A}_{C'_m}\}$,
                    $\mathcal{J}_2:=\{j\in[n_a]:a_j\in\mathcal{A}_{C'_m}\backslash\mathcal{A}_{C_m}\}$, and
                    $\mathcal{J}_3:=\{j\in[n_a]:a_j\notin\mathcal{A}_{C_m}\cup\mathcal{A}_{C'_m}\}$.
                    Without loss of generality, we assume $\mathcal{J}_1=\{k\}$ and $\mathcal{J}_2=\{j_1,j_2\}$.
                    Then, we rewrite $x$ as $[x^{\top}_{\mathcal{J}_0},x_{k},x_{j_1},x_{j_2},x^{\top}_{\mathcal{J}_3}]^{\top}$.
                    We have either $\hat{x}_{j_1}=1$ or $\hat{x}_{j_2}=1$.
                    We denote
                    \begin{equation*}
                        \phi(\hat{x}_{\mathcal{J}_0},0,0,0,0)=\hat{\phi}+f_{k},
                    \end{equation*}
                    where $f_{k}<f_{j_1}+f_{j_2}$.
                    We denote
                    \begin{equation*}
                        \phi(\hat{x}_{\mathcal{J}_0},0,1,1,0)=\hat{\phi}+f_{j_1}+f_{j_2}-\delta_{j_1}-\delta_{j_2},
                    \end{equation*}
                    where $f_{j_1}+f_{j_2}-\delta_{j_1}-\delta_{j_2}<f_k$.
                    Then, we have
                    \begin{equation*}
                        \begin{aligned}
                            \phi(\hat{x}_{\mathcal{J}_0},0,0,1,0)=&\min\{\hat{\phi}+f_{k},\hat{\phi}+f_{j_1}+f_{j_2}-\delta_{j_2}\}\\
                            \phi(\hat{x}_{\mathcal{J}_0},0,1,0,0)=&\min\{\hat{\phi}+f_{k},\hat{\phi}+f_{j_1}+f_{j_2}-\delta_{j_1}\}.
                        \end{aligned}
                    \end{equation*}
                    Because $\delta_{k}>0$, $\delta_{j_1}>0$, and $\delta_{j_2}>0$, the counterexample
                    \begin{equation*}
                        \phi(\hat{x}_{\mathcal{J}_0},0,0,1,0)+\phi(\hat{x}_{\mathcal{J}_0},0,1,0,0)>\phi(\hat{x}_{\mathcal{J}_0},0,0,0,0)+\phi(\hat{x}_{\mathcal{J}_0},0,1,1,0)
                    \end{equation*}
                    always holds, which contradicts that $\phi(\cdot)$ is supermodular.
                \end{proof}

            \subsection{Proof of Corollary \ref{crl:non-binding-MF}}\label{apd:crl:non-binding-MF}
                \begin{proof}
                    Following Definition \ref{def:non-binding}, $a_j$ is $\mathcal{A}_a$-robust non-binding if for all $x\in\{0,1\}^{n_a}$, we have $y^*_j(x)<f_j-\delta_j x_j$, where $y^*_j(x)$ is optimal to $\phi(x)$.
                    Therefore, we need to prove that if $y^*_j(\hat{x})<f_j-\delta_j \hat{x}_j$, we have $y^*_j(x)<f_j-\delta_j x_j$ for all $x\in\{0,1\}^{n_a}$.

                    We suppose the contrary that there exist an $\tilde{x}\in\{0,1\}^{n_a}$ such that $y^*_j(\tilde{x})\geq f_j-\delta_j\tilde{x}_j$.
                    Then, we have: 1) under attack $\hat{x}$, the minimum cut $\hat{C}_m$ does not contain arc $a_j$;
                    2) under attack $\tilde{x}$, the minimum cut $\tilde{C}_m$ contains arc $a_j$.
                    We rewrite $x=[x_j,x_s,x_p^\top,x_o^\top]^\top$, where
                    $x_s$ corresponds to the arc $a_s\in\mathcal{A}_{\hat{C}_m}$ that is in series with $a_j$;
                    $x_p$ consists of the entries of $x$ corresponding to arcs in $\mathcal{A}_{\tilde{C}_m}\backslash\{a_j\}$ that are in series with $a_s$;
                    $x_o$ consists of the other entries of $x$.
                    Hence, in an SPN, we have
                    \begin{equation}\label{eq:crl:non-binding1}
                        \phi_s(\hat{x}_s)<\phi_p(\hat{x}_j,\hat{x}_p)
                    \end{equation}
                    and
                    \begin{equation}\label{eq:crl:non-binding2}
                        \phi_s(\tilde{x}_s)>\phi_p(\tilde{x}_j,\tilde{x}_p),
                    \end{equation}
                    where $\phi_s(\hat{x}_s)$ is the capacity of $a_s$ and $\phi_p(\tilde{x}_j,\tilde{x}_p)$ are the total capacity of arcs in $\mathcal{A}_{\tilde{C}_m}$ that are in series with $a_s$.
                    We note that $\hat{x}_j=1$, $\hat{x}_s=0$, and $\hat{x}_p=1$, and hence, \eqref{eq:crl:non-binding1} reduces to
                    \begin{equation}\label{eq:crl:non-binding3}
                        \phi_s(0)<\phi_p(1,1).
                    \end{equation}
                    Moreover, both $\phi_p(\cdot)$ and $\phi_s(\cdot)$ are non-increasing, and hence, \eqref{eq:crl:non-binding2} implies
                    \begin{equation}\label{eq:crl:non-binding4}
                        \phi_p(0)>\phi_s(1,1),
                    \end{equation}
                    which contradicts with \eqref{eq:crl:non-binding3}.
                    Therefore, we finish the proof.
                \end{proof}

            \subsection{Proof of Proposition \ref{pps:MCF-cost-repair-sub}}\label{apd:pps:MCF-cost-repair-sub}
                \begin{proof}
                    We first prove that $\varphi(x,\tilde{z})$ is jointly submodular in $x$ and $\tilde{z}$ if and only if
                    \begin{condition}\label{cdt:repair-cost-sub}
                        for any path $P$ in $\mathcal{G}$, for any two arcs $a',a''\in\mathcal{A}_{a}\cap\mathcal{A}_{P}$, $a'$ and $a''$ have the \emph{same} direction in $P$.
                    \end{condition}
                    Starting from \eqref{eq:MCF-cost-repair}, we introduce an auxiliary variable $y'=y$ and rewrite \eqref{eq:MCF-cost-repair} as
                    \begin{equation}\label{eq:MCF-cost-repair-imply1}
                        \begin{aligned}
                            \varphi(x,\tilde{z})=\min_{y,y'\in\mathbb{R}^{n_{a}}}~&(c+\text{diag}(\delta)x)^{\top}y+(-\delta+\text{diag}(\delta)\tilde{z})^{\top}y'\\
                            \text{s.t.}~&T^-y+T^+y'\geq d\\
                            &T_d y+T'_d y'=0\\
                            &0\leq y\leq f\\
                            &0\leq y'\leq f,
                        \end{aligned}
                    \end{equation}
                    where $T_d=I_{n_a}$ and $T'_d=-I_{n_a}$.
                    We relax the second constraint of \eqref{eq:MCF-cost-repair-imply1} and have
                    \begin{equation}\label{eq:MCF-cost-repair-imply2}
                        \begin{aligned}
                            \varphi(x,\tilde{z})=\min_{y,y'\in\mathbb{R}^{n_{a}}}~&(c+\text{diag}(\delta)x)^{\top}y+(-\delta+\text{diag}(\delta)\tilde{z})^{\top}y'+\sum_{j\in[n_a]}M(y_j-y'_j)\\
                            \text{s.t.}~&T^-y+T^+y'\geq d\\
                            &T_d y+T'_d y'\geq0\\
                            &0\leq y\leq f\\
                            &0\leq y'\leq f.
                        \end{aligned}
                    \end{equation}
                    where $M\in\mathbb{R}$ is a constant.
                    When $M$ is sufficiently large, \eqref{eq:MCF-cost-repair-imply2} is equivalent to \eqref{eq:MCF-cost-repair-imply1}.
                    We rewrite \eqref{eq:MCF-cost-repair-imply2} as
                    \begin{equation}\label{eq:MCF-cost-repair-imply}
                        \begin{aligned}
                            \varphi(x,\tilde{z})=\min_{y,y'\in\mathbb{R}^{n_{a}}}~&(c+M+\text{diag}(\delta)x)^{\top}y+(-\delta-M+\text{diag}(\delta)\tilde{z})^{\top}y'\\
                            \text{s.t.}~&T^-y+T^+y'\geq d\\
                            &T_d y+T'_d y'\geq0\\
                            &0\leq y\leq f\\
                            &0\leq y'\leq f.
                        \end{aligned}
                    \end{equation}

                    Furthermore, we observe that \eqref{eq:MCF-cost-repair-imply} aligns with the form of \eqref{eq:MCF-cost}.
                    Topologically, in \eqref{eq:MCF-cost-repair-imply}, for each original arc $j$ in network $\mathcal{G}$, it is broken into two segments connected by a dummy node.
                    Without loss of generality, we assume that $y$ indicates the flow on the former segment with cost coefficient $c+M$ and flow capacity $f$;
                    $y'$ indicates the flow on the latter segment with cost coefficient $-\delta-M$ and flow capacity $f$.
                    For all $j\in[n_a]$, $x_j=1$ indicates that the cost coefficient of the former segment of original arc $j$ is attacked and is inflated by $\delta_j$;
                    for all $j\in[n_a]$, $\tilde{z}_j=1$ indicates that the cost coefficient of the latter segment of original arc $j$ is attacked and is inflated by $\delta_j$.
                    We note that the former segment and the latter segment of any original arc always have the same direction.
                    Therefore, following Proposition \ref{pps:general-MCF-cost}, given $T$ and $\delta$, for any $c$, $d$, and $f$, $\varphi(x,\tilde{z})$ defined in \eqref{eq:MCF-cost-repair-imply} is jointly submodular in $x$ and $\tilde{z}$ if and only if Condition \ref{cdt:repair-cost-sub} holds.
                    Considering that \eqref{eq:MCF-cost-repair-imply} is equivalent to \eqref{eq:MCF-cost-repair}, $\varphi(x,\tilde{z})$ defined in \eqref{eq:MCF-cost-repair} is jointly submodular in $x$ and $\tilde{z}$ if and only if Condition \ref{cdt:repair-cost-sub} holds.
                    We finish this part of proof.

                    Next, if $\varphi(x,\tilde{z})$ is jointly submodular in $x$ and $\tilde{z}$, we have that $r^{\top}(1-\tilde{z})+\varphi(x,\tilde{z})$ is also jointly submodular in $x$ and $\tilde{z}$.
                    Because $\{0,1\}^{|X|}$ is a lattice, following the Theorem 2.7.6 of \citet{topkis1998supermodularity}, we have that $\phi'(x)$ defined by~\eqref{eq:MCF-integer2} is submodular in $x$.
                \end{proof}

            \subsection{Proof of Proposition \ref{pps:MCF-cost-repair-super-corrected}}\label{apd:pps:MCF-cost-repair-super-corrected}
                \begin{proof}
                    We denote by $\mathcal{J}_a:=\{j\in[n_a]:a_j\in\mathcal{A}_a\}$ the index set of attackable arcs and denote by $N:=|\mathcal{J}_a|$ the number of attackable arcs.
                    We only consider the entries of $x$ related to the attackable arcs, and for ease of exposition, we use the subscript $(\cdot)_a$ to indicate variables or parameters related to the attackable arcs, such as $x_{a}$ and $r_a$.
                    $\tilde{\phi}(x_a)$ is supermodular if and only if for all unordered $x'_a,x''_a\in\{0,1\}^{N}$,
                    \begin{equation}\label{eq:MCF-cost-repair-super-corrected-1}
                        \tilde{\phi}(x'_a)+\tilde{\phi}(x''_a)\leq\tilde{\phi}(x_a^{\wedge})+\tilde{\phi}(x_a^{\vee}),
                    \end{equation}
                    which implies
                    \begin{equation}\label{eq:MCF-cost-repair-super-corrected-2}
                        \phi'(x'_a)+\rho\zeta(x'_a)+\phi'(x''_a)+\rho\zeta(x''_a)\leq\phi'(x_a^{\wedge})+\rho\zeta(x_a^{\wedge})+\phi'(x_a^{\vee})+\rho\zeta(x_a^{\vee}).
                    \end{equation}
                    Because $\zeta(x_a)$ is strictly supermodular, we have $\zeta(x_a^{\wedge})+\zeta(x_a^{\vee})-\zeta(x'_a)-\zeta(x''_a)>0$, and hence, \eqref{eq:MCF-cost-repair-super-corrected-2} is equivalent to
                    \begin{equation}\label{eq:MCF-cost-repair-super-corrected-3}
                        \rho\geq\frac{\phi'(x'_a)+\phi'(x''_a)-\phi'(x_a^{\wedge})-\phi'(x_a^{\vee})}{\zeta(x_a^{\wedge})+\zeta(x_a^{\vee})-\zeta(x'_a)-\zeta(x''_a)}.
                    \end{equation}
                    When the numerator of \eqref{eq:MCF-cost-repair-super-corrected-3} is upper bounded, the right-hand side of \eqref{eq:MCF-cost-repair-super-corrected-3} is finite, and hence, there exist a $\rho\in\mathbb{R}_{+}$ such that
                    \begin{equation}\label{eq:MCF-cost-repair-super-corrected-4}
                        \rho\geq\max_{\text{unordered~} x'_a,x''_a\in\{0,1\}^{|\mathcal{A}_a|}}\frac{\phi'(x'_a)+\phi'(x''_a)-\phi'(x_a^{\wedge})-\phi'(x_a^{\vee})}{\zeta(x_a^{\wedge})+\zeta(x_a^{\vee})-\zeta(x'_a)-\zeta(x''_a)}
                    \end{equation}
                    holds and thus $\tilde{\phi}(x_a)$ is supermodular.

                    Then, we analyze the upper bound of the numerator.

                    We have that for all $z_1,z_2\in\{0,1\}^{N}$,
                    \begin{equation}\label{eq:MCF-cost-repair-super-2}
                        \begin{aligned}
                            &\phi'(x'_a)+\phi'(x''_a)-\phi'(x_a^{\wedge})-\phi'(x_a^{\vee})\\
                            \leq&r_a^{\top}z_1+\phi(x'_a-z_1)+r_a^{\top}z_2+\phi(x''_a-z_2)-r_a^{\top}z_3-\phi(x_a^{\wedge}-z_3)-r_a^{\top}z_4-\phi(x_a^{\vee}-z_4)\\
                            =&r_a^{\top}(z_1+z_2-z_3-z_4)+\phi(x'_a-z_1)+\phi(x''_a-z_2)-\phi(x_a^{\wedge}-z_3)-\phi(x_a^{\vee}-z_4),
                        \end{aligned}
                    \end{equation}
                    where $z_3,z_4\in\{0,1\}^N$ are optimal to $\phi'(x_a^{\wedge})$ and $\phi'(x_a^{\vee})$, respectively.
                    Comparing $x_a^{\wedge}-z_3$ and $x_a^{\vee}-z_4$, we observe that there exists a $\xi\in\{0,1\}^{N}$ such that
                    \begin{equation}\label{eq:MCF-cost-repair-super-3}
                        \begin{aligned}
                            x_a^{\wedge}-z_3\leq x_a^{\vee}-z_4+\xi.
                        \end{aligned}
                    \end{equation}
                    Especially, we set
                    \begin{equation}\label{eq:MCF-cost-repair-super-3.5}
                        \begin{aligned}
                            \xi_{j}:=(x^{\wedge}_{a,j}+1-x^{\vee}_{a,j})(1-z_{3,j})z_{4,j},~\forall j\in[N]
                        \end{aligned}
                    \end{equation}
                    such that $\xi=\max\{(x_a^{\wedge}-z_3)-(x_a^{\vee}-z_4),0\}$.
                    Furthermore, we introduce a vector $\Lambda\in\mathbb{R}^{n_a}$ to indicate the maximal cost increment of an additional attack, and hence, we have
                    \begin{equation}\label{eq:MCF-cost-repair-super-4}
                        \begin{aligned}
                            \phi(x_a^{\vee}-z_4+\xi)\leq\phi(x_a^{\vee}-z_4)+\Lambda_a^{\top}\xi,
                        \end{aligned}
                    \end{equation}
                    where for all $j\in[N]$,
                    \begin{equation}\label{eq:MCF-cost-repair-super-5}
                        \begin{aligned}
                            \Lambda_{a,j}:=\max_x~&\phi(x_a+e_j)-\phi(x_a)\\
                            \text{s.t.~}&x,x+e_j\in\{-1,0,1\}^{N}.
                        \end{aligned}
                    \end{equation}
                    Combining \eqref{eq:MCF-cost-repair-super-2} and \eqref{eq:MCF-cost-repair-super-4}, we have that for all $z_1,z_2\in\{0,1\}^{N}$,
                    \begin{equation}\label{eq:MCF-cost-repair-super-6}
                        \begin{aligned}
                            &\phi'(x'_a)+\phi'(x''_a)-\phi'(x_a^{\wedge})-\phi'(x_a^{\vee})\\
                            \leq&r_a^{\top}(z_1+z_2-z_3-z_4)+\phi(x'_a-z_1)+\phi(x''_a-z_2)-\phi(x_a^{\wedge}-z_3)-\phi(x_a^{\vee}-z_4+\xi)+\Lambda_a^{\top}\xi.
                        \end{aligned}
                    \end{equation}
                    Considering $0\leq z_3+z_4-\xi\leq2$, there exist $z_1,z_2\in\{0,1\}^{N}$ such that
                    \begin{equation}\label{eq:MCF-cost-repair-super-7}
                        \begin{aligned}
                            z_1+z_2=z_3+z_4-\xi,
                        \end{aligned}
                    \end{equation}
                    which implies $(x'_a-z_1)+(x''_a-z_2)=(x_a^{\wedge}-z_3)+(x_a^{\vee}-z_4+\xi)$.
                    Together with \eqref{eq:MCF-cost-repair-super-3}, we have
                    $(x'_a-z_1)\wedge(x''_a-z_2)=x_a^{\wedge}-z_3$ and $(x'_a-z_1)\vee(x''_a-z_2)=x_a^{\vee}-z_4+\xi$.
                    Because $\phi(\cdot)$ is supermodular, we have
                    \begin{equation}\label{eq:MCF-cost-repair-super-8}
                        \begin{aligned}
                            \phi(x'_a-z_1)+\phi(x''_a-z_2)\leq\phi(x_a^{\wedge}-z_3)+\phi(x_a^{\vee}-z_4+\xi).
                        \end{aligned}
                    \end{equation}
                    Combining \eqref{eq:MCF-cost-repair-super-6}--\eqref{eq:MCF-cost-repair-super-8}, we have
                    \begin{equation}\label{eq:MCF-cost-repair-super-9}
                        \begin{aligned}
                            \phi'(x'_a)+\phi'(x''_a)-\phi'(x_a^{\wedge})-\phi'(x_a^{\vee})\leq&(\Lambda_a-r_a)^{\top}\xi.
                        \end{aligned}
                    \end{equation}
                    Together with \eqref{eq:MCF-cost-repair-super-3.5}, we have
                    \begin{equation}\label{eq:MCF-cost-repair-super-10}
                        \begin{aligned}
                            \phi'(x'_a)+\phi'(x''_a)-\phi'(x_a^{\wedge})-\phi'(x_a^{\vee})\leq&(\Lambda_a-r_a)^{\top}(x_a^{\wedge}+1-x_a^{\vee})\\
                            =&(\Lambda_a-r_a)^{\top}(1-(x'_a-x_a^{\wedge}+x''_a-x_a^{\wedge})).
                        \end{aligned}
                    \end{equation}

                    Next, we consider the two examples of $\zeta(\cdot)$.
                    When $\zeta(x_a)=(1^\top x_a)^2$, we have
                    \begin{equation}
                        \zeta(x_a^{\wedge})+\zeta(x_a^{\vee})-\zeta(x'_a)-\zeta(x''_a)=2(1^{\top}(x'_a-x_a^{\wedge}))(1^{\top}(x''_a-x_a^{\wedge})).
                    \end{equation}
                    Together with \eqref{eq:MCF-cost-repair-super-10}, we restrict \eqref{eq:MCF-cost-repair-super-corrected-4} as
                    \begin{equation}
                        \rho\geq\max_{\text{unordered~} x'_a,x''_a\in\{0,1\}^{|\mathcal{A}_a|}}\frac{(\Lambda_a-r_a)^{\top}(1-(x'_a-x_a^{\wedge}+x''_a-x_a^{\wedge}))}{2(1^{\top}(x'_a-x_a^{\wedge}))(1^{\top}(x''_a-x_a^{\wedge}))}
                    \end{equation}
                    which can be imposed by
                    \begin{equation}
                        \rho\geq\frac{1}{2}\sum_{i=1}^{|\mathcal{A}_a|-2}(\Lambda_{a,\sigma_i}-r_{a,\sigma_i})^+.
                    \end{equation}
                    When $\zeta(x_a)=a^{|x_a|}$, we have
                    \begin{equation}
                        \zeta(x_a^{\wedge})+\zeta(x_a^{\vee})-\zeta(x'_a)-\zeta(x''_a)=(a^{1^{\top}(x'_a-x_a^{\wedge})}-1)(a^{1^{\top}(x''_a-x_a^{\wedge})}-1).
                    \end{equation}
                    Together with \eqref{eq:MCF-cost-repair-super-10}, we restrict \eqref{eq:MCF-cost-repair-super-corrected-4} as
                    \begin{equation}
                        \rho\geq\max_{\text{unordered~} x'_a,x''_a\in\{0,1\}^{|\mathcal{A}_a|}}\frac{(\Lambda-r)^{\top}(1-(x'_a-x_a^{\wedge}+x''_a-x_a^{\wedge}))}{(a^{1^{\top}(x'_a-x_a^{\wedge})}-1)(a^{1^{\top}(x''_a-x_a^{\wedge})}-1)},
                    \end{equation}
                    which can be imposed by
                    \begin{equation}
                        \rho\geq\frac{1}{(a-1)^2}\sum_{i=1}^{|\mathcal{A}_a|-2}(\Lambda_{a,\sigma_i}-r_{a,\sigma_i})^+.
                    \end{equation}

                    In addition, when $\phi(\cdot)$ is supermodular, for any $x'_{a}$, $x''_{a}$ with $x'_{a}<x''_{a}$, for any $j\in[N]$ with $x'_{a}\wedge e_{j}=0$ and $x''_{a}\wedge e_{j}=0$, we have
                    \begin{equation*}
                        \phi(x'_{a}+e_{j})-\phi(x'_{a})\leq\phi(x''_{a}+e_{j})-\phi(x''_{a}).
                    \end{equation*}
                    Hence, we have the closed-form representation of $\Lambda_a$, i.e., $\Lambda_{a,j}=\phi(1)-\phi(1-e_j)$ for all $j\in[N]$.
                \end{proof}

            \subsection{Proof of Corollary \ref{crl:MCF-cost-repair-super}}\label{apd:crl:MCF-cost-repair-super}
                \begin{proof}
                    As analyzed in the proof of Proposition \ref{pps:MCF-cost-repair-super-corrected}, for any unordered $x'_a,x''_a\in\{0,1\}^{N}$, we have
                    \begin{equation*}
                        \begin{aligned}
                            \phi'(x'_a)+\phi'(x''_a)-\phi'(x_a^{\wedge})-\phi'(x_a^{\vee})\leq&(\Lambda_a-r_a)^{\top}(x_a^{\wedge}+1-x_a^{\vee}).
                        \end{aligned}
                    \end{equation*}

                    Because $x_a^{\wedge}+1-x_a^{\vee}\geq0$, when $r_a\geq\Lambda_a$, the right-hand side is no larger than 0 and
                    \begin{equation*}
                        \begin{aligned}
                            \phi'(x'_a)+\phi'(x''_a)-\phi'(x_a^{\wedge})-\phi'(x_a^{\vee})\leq0
                        \end{aligned}
                    \end{equation*}
                    holds for any unordered $x'_a,x''_a\in\{0,1\}^{N}$.
                    Therefore, $\phi'(x_a)$ is supermodular.
                \end{proof}

        \section{Series-Parallel Networks}\label{apd:SPN}

            \subsection{Definitions Related to Series-Parallel Networks}\label{apd:SPN-definition}
                \begin{definition}\label{def:two-terminal graph}
                    A two-terminal graph is $\mathcal{G}=(\mathcal{V},\mathcal{A},s,t,\mathcal{T})$﻿, where $\mathcal{V}$﻿ is a finite set of nodes, $\mathcal{A}$﻿ is a finite set of arcs, $s,t\in\mathcal{V}$﻿ are the source node and sink node, respectively, and $\mathcal{T}\subseteq\mathcal{V}\times\mathcal{A}\times\mathcal{V}$ is the incidence relation that identifies the tail node and head node of each arc.
                    Here, $(v',a,v'')\in\mathcal{T}$ indicates that arc $a$ is from $v'$ to $v''$.
                \end{definition}

                \begin{definition}\label{def:series/parallel operation}
                    Let $\mathcal{G}_{1}=(\mathcal{V}_{1},\mathcal{A}_{1},s_{1},t_{1},\mathcal{T}_{1})$﻿ and $\mathcal{G}_{2}=(\mathcal{V}_{2},\mathcal{A}_{2},s_{2},t_{2},\mathcal{T}_{2})$﻿ be two-terminal graphs, with $|\mathcal{V}_{1}|\geq2,|\mathcal{V}_{2}|\geq2$﻿.\\
                    (i) The series composition of $\mathcal{G}_{1}$ and $\mathcal{G}_{2}$, denoted by $\mathcal{G}_{1}\cdot\mathcal{G}_{2}$,﻿ is $(\mathcal{V},\mathcal{A},s,t,\mathcal{T})$﻿ where
                    \begin{equation*}\label{eq:series composition}
                        \begin{aligned}
                            &\mathcal{V}=(\mathcal{V}_{1}\cup\mathcal{V}_{2})\backslash\{s_{2}\}\\
                            &\mathcal{A}=\mathcal{A}_{1}\cup\mathcal{A}_{2}\\
                            &s=s_{1},t=t_{2}\\
                            &\begin{aligned}
                                \mathcal{T}=&\big((\mathcal{T}_{1}\cup\mathcal{T}_{2})\cap(\mathcal{V}\times\mathcal{A}\times\mathcal{V})\big)\cup\\
                                &\{(t_{1},a,v): a\in\mathcal{A}_2, v\in\mathcal{V}_2\backslash\{s_{2}\}, (s_{2},a,v)\in\mathcal{T}_{2}\}\cup\\
                                &\{(u,a,t_{1}): a\in\mathcal{A}_2, v\in\mathcal{V}_2\backslash\{s_{2}\}, (u,a,s_{2})\in\mathcal{T}_{2}\}
                            \end{aligned}
                        \end{aligned}
                    \end{equation*}
                    (ii) The parallel composition of $\mathcal{G}_{1}$ and $\mathcal{G}_{2}$, denoted by $\mathcal{G}_{1}\|\mathcal{G}_{2}$,﻿ is $(\mathcal{V},\mathcal{A},s,t,\mathcal{T})$﻿ where
                    \begin{equation*}\label{eq:parallel composition}
                        \begin{aligned}
                            &\mathcal{V}=(\mathcal{V}_{1}\cup\mathcal{V}_{2})\backslash\{s_{2},t_{2}\}\\
                            &\mathcal{A}=\mathcal{A}_{1}\cup\mathcal{A}_{2}\\
                            &s=s_{1},t=t_{1}\\
                            &\begin{aligned}
                                \mathcal{T}=&\big((\mathcal{T}_{1}\cup\mathcal{T}_{2})\cap(\mathcal{V}\times\mathcal{A}\times\mathcal{V})\big)\cup\\
                                &\{(s_{1},a,t_{1}): a\in\mathcal{A}_2, (s_{2},a,t_{2})\in \mathcal{T}_{2}\}\cup\\
                                &\{(t_{1},a,s_{1}): a\in\mathcal{A}_2, (t_{2},a,s_{2})\in \mathcal{T}_{2}\}\cup\\
                                &\{(s_{1},a,v): a\in\mathcal{A}_2, v\in\mathcal{V}_2\backslash\{s_{2}, t_{2}\}, (s_{2},a,v)\in \mathcal{T}_{2}\}\cup\\
                                &\{(v,a,s_{1}): a\in\mathcal{A}_2, v\in\mathcal{V}_2\backslash\{s_{2}, t_{2}\}, (v,a,s_{2})\in \mathcal{T}_{2}\}\cup\\
                                &\{(t_{1},a,v): a\in\mathcal{A}_2, v\in\mathcal{V}_2\backslash\{s_{2}, t_{2}\}, (t_{2},a,v)\in \mathcal{T}_{2}\}\cup\\
                                &\{(v,a,t_{1}): a\in\mathcal{A}_2, v\in\mathcal{V}_2\backslash\{s_{2}, t_{2}\}, (v,a,t_{2})\in \mathcal{T}_{2}\}
                            \end{aligned}
                        \end{aligned}
                    \end{equation*}
                \end{definition}

                \begin{definition}\label{def:SPN}
                    A two-terminal series-parallel graph is defined inductively as follows:\\
                    (1) $\mathcal{G}=(\{s,t\},\{a\},s,t,\{(s,a,t)\})$ is a two-terminal series-parallel graph.\\
                    (2) If $\mathcal{G}_{1}$ and $\mathcal{G}_{2}$ are two-terminal series-parallel graphs, $G_{1}\cdot G_{2}$ and $G_{1}\|G_{2}$ are two-terminal series-parallel graphs.
                \end{definition}

                \begin{definition}\label{def:series/parallel}
                    Given an SPN $\mathcal{G}=(\mathcal{V},\mathcal{A},s,t,\mathcal{T})$ and two distinct arcs $a',a''\in\mathcal{A}$.\\
                    (i) Arcs $a'$ and $a''$ are in series if there exist an $s$--$t$ path that contains both $a'$ and $a''$.
                    Moreover, they are strictly in series if every $s$--$t$ path containing one of the two arcs also contains the other.\\
                    (ii) Arcs $a'$ and $a''$ are in parallel if there exist no $s$--$t$ path that contains both $a'$ and $a''$.\\
                    Hence, we have that $a'$ and $a''$ are either in series or in parallel.
                \end{definition}

            \subsection{Series-Parallel Reduction}\label{apd:SPN-reduction}
                \subsubsection{Shortest Path Interdiction}
                    We consider the context of SP interdiction.
                    We suppose an SPN $\mathcal{G}=(\mathcal{V},\mathcal{A},s,t,T)$ with parameters $c$, a set $\mathcal{A}_{a}\subseteq\mathcal{A}$ of attackable arcs, and a vector $\delta\in\Delta(\mathcal{A}_{a})$.
                    Then, for any non-attackable arc $a_j\in\mathcal{A}\backslash\mathcal{A}_{a}$, we introduce the series reduction and parallel reduction, respectively.

                    Series reduction:
                    If there exist an arc $a_k$ such that $a_j$ and $a_k$ are strictly in series, we can remove $a_j$ by merging its two incident nodes and modify $c_k$ into $c'_k$, where
                    \begin{equation*}
                        c'_k=c_k+c_j.
                    \end{equation*}

                    Parallel reduction:
                    If there exist an arc $a_k$ such that $a_j$ and $a_k$ are in parallel and share incident nodes, we can remove $a_j$ and modify $c_k$ and $\delta_k$ into $c'_k$ and $\delta'_k$, respectively, where
                    \begin{equation*}
                        c'_k=\min\{c_j,c_k\},~\delta'_k=\min\{c_j,c_k+\delta_k\}-\min\{c_j,c_k\}.
                    \end{equation*}

                    We conduct series reduction and parallel reduction until there is no removable arcs.
                    Then, we get the simplified SPN $\mathcal{G}'$ and the shortest distance remains the same.

                \subsubsection{Maximum Flow Interdiction}
                    We consider the context of MaxF interdiction.
                    We suppose an SPN $\mathcal{G}=(\mathcal{V},\mathcal{A},s,t,T)$ with parameters $f$, a set $\mathcal{A}_{a}\subseteq\mathcal{A}$ of attackable arcs, and a vector $\delta\in\Delta(\mathcal{A}_{a})$.
                    Then, for any non-attackable arc $a_j\in\mathcal{A}\backslash\mathcal{A}_{a}$, we introduce the series reduction and parallel reduction, respectively.

                    Series reduction:
                    If there exist an arc $a_k$ such that $a_j$ and $a_k$ are strictly in series, we can remove $a_j$ by merging its two incident nodes and modify $f_k$ and $\delta_k$ into $f'_k$ and $\delta'_k$, respectively, where
                    \begin{equation*}
                        f'_k=\min\{f_j,f_k\},~\delta'_k=\min\{f_j,f_k\}-\min\{f_j,f_k-\delta_k\}.
                    \end{equation*}

                    Parallel reduction:
                    If there exist a path $P$ between the two incident nodes of $a_j$ such that $a_j$ and any arc $a_k$ in $P$ are in parallel, we can remove $a_j$ and modify $f_k$ into $f'_k$ for all $a_k$ in $P$, where
                    \begin{equation*}
                        f'_k=f_k+f_j.
                    \end{equation*}

                    We conduct series reduction and parallel reduction until there is no removable arcs.
                    Then, we get the simplified SPN $\mathcal{G}'$ and the maximum flow remains the same.
    \end{appendices}

\end{document}